\documentclass[11pt,a4paper,leqno]{amsart}
\usepackage{graphicx} 
\usepackage{amssymb}
\usepackage{amsmath}
\usepackage{amsthm}
\usepackage{cancel}
\usepackage{amsfonts,amssymb}
\usepackage[abbrev,non-sorted-cites, alphabetic]{amsrefs}
\usepackage{latexsym}
\usepackage{color}
\usepackage{appendix}
\usepackage{hyperref}
\hypersetup{
    colorlinks=true,
    allcolors=black,
}

\usepackage{xr}
\oddsidemargin=.in
\evensidemargin=.in

\newtheorem{theorem}{Theorem}[section]
\newtheorem{lem}[theorem]{Lemma}
\theoremstyle{plain}

\newtheorem{claim}{Claim}

\newtheorem{corollary}[theorem]{Corollary}

\newtheorem{definition}[theorem]{Definition}

\newtheorem{lemma}[theorem]{Lemma}

\newtheorem{proposition}[theorem]{Proposition}
\newtheorem{remark}{Remark}

\newenvironment{claimproof}{\par \noindent \textit{Proof of Claim.}}{\hfill $\diamond$ \par}

\newcommand{\N}{\mathbb{N}}     
\newcommand{\R}{\mathbb{R}}     
\newcommand{\C}{\mathbb{C}}     

\newcommand{\ol}[1]{\overline{#1}}  

\newcommand\dd{{\mathrm d}}

\DeclareMathOperator\im{Im}

\DeclareMathOperator\dist{dist}     
\DeclareMathOperator\supp{supp}     
\DeclareMathOperator\osc{osc}     
\DeclareMathOperator*{\esssup}{ess\,sup} 

\title[Finite-Time Singularities of LMCF with Precise Dynamics]{Finite-Time Singularities of Lagrangian Mean Curvature Flow with Quantitatively Precise Dynamics}

\author{Maxwell Stolarski}
\address{Maxwell Stolarski: School of Mathematics and Statistics, University College Dublin, Science Centre North, Belfield, Dublin 4, Ireland}
\email{maxwell.stolarski@ucd.ie}

\author{Wei-Bo Su}
\address{Wei-Bo Su: Department of Mathematics, National Central University, No. 300, Zhongda Rd., Zhongli District, Taoyuan City 320317, Taiwan}
\email{weibosu@math.ncu.edu.tw}

\date{\today}

\begin{document}

\begin{abstract}
For each integer $K\geq2$ when $n\geq4$, and for $K=2,3,4$ when $n=3$, we construct an almost-calibrated Lagrangian mean curvature flow $L_K(t)$ in $\mathbb{C}^{n}$, starting from initial data arbitrarily close to being special Lagrangian, which develops a finite-time Type II singularity at time $T$ with the explicit curvature blow up rate
\[
    \sup_{L_{K}(t)} |\mathbf{A}_{L_{K}(t)}| \sim (T-t)^{-K/2}
    \qquad \text{as } t\nearrow T .
\]
The tangent flow at the singularity is a transverse pair of cohomogeneity-one
special Lagrangian cones, while the Type II blow-up limit is a smooth
cohomogeneity-one special Lagrangian desingularization.

This gives a quantitative construction of Type II blow-up for a fully
nonlinear parabolic PDE
arising from cohomogeneity-one Lagrangian mean curvature flow. Our construction is based on a modulation analysis around a shrinking family of cohomogeneity-one special Lagrangian desingularizations,
using the perturbative spectral theory developed in the companion paper~\cite{SS26I}. 
\end{abstract}

\maketitle

\tableofcontents

\section{Introduction}

A major question in differential geometry is to understand the special Lagrangian submanifolds of a Calabi-Yau manifold.
As special Lagrangian submanifolds minimize volume in their homology class, the negative gradient flow for volume, that is, the mean curvature flow of Lagrangian submanifolds provides a promising approach to this problem.
Indeed, the mean curvature flow in an ambient Calabi-Yau space preserves the Lagrangian condition \cite{Smoczyk96}, and the Thomas-Yau conjectures suggest that, if the initial Lagrangian $L(0)$ satisfies a stability property, then the Lagrangian mean curvature flow $L(t)$ will converge to the unique special Lagrangian in the Hamiltonian isotopy class of $L(0)$ as $t \to +\infty$ \cites{Thomas00, ThomasYau02, Joyce15}.

However, Lagrangian mean curvature flows often form finite-time singularities, and these singularities are generally unavoidable \cites{Neves07, Neves13}.
The Thomas-Yau conjectures therefore require a precise understanding of the flow through singularities.
There have been numerous important works examining singularities for Lagrangian mean curvature flows including, for example, 
characterizations of tangent flows \cite{Neves07},
structural and existence results for singularities of $O(n)$-invariant flows in $\C^n$ \cites{Neves07, Neves13, Lotay22, Wood24},
examples of infinite-time singularities \cites{STW, LT26},
uniqueness results for tangent flows and existence of Lawlor neck pinch \cites{LSS22, LS24},
generic singularities \cite{Szekelyhidi26},
and examples of finite-time singularities for flows with $U(1)$-symmetry \cites{LO24, LO25}.
All prior constructions of finite-time singularities for Lagrangian mean curvature flows are generally based on maximum principle techniques, blow-up arguments, and compactness results.
As such, they lack a \emph{quantitatively precise} description of finite-time singularity formation.
For example, while it's known that, under the zero-Maslov assumption, finite-time Lagrangian mean curvature singularities must be Type II (i.e. $ \sup_{x \in L(t)} |\mathbf{A}_{L(t)}| \gg \frac1{\sqrt{T-t}}$) \cites{Wang01, Neves07},
it's unknown what curvature blow-up rates may occur or whether the mean curvature can remain uniformly bounded.
Our main result remedies this gap by exhibiting Lagrangian mean curvature flow solutions that form finite-time singularities with quantitatively precise dynamics.

\begin{theorem} \label{main thm paper II intro}
    Let  $n,K \in \mathbb N$ be such that
        $$\left( n \ge 4 \text{ and } K \ge 2 \right) \qquad \text{or} \qquad \left( n =3 \text{ and } K \in \{2,3,4\} \right),$$
    and let $G \le SU(n)$ be a compact, connected Lie group such that generic $G$-orbits are (real) $(n-1)$-dimensional.
    Then there exists a smooth, non-compact, properly embedded, exact, almost calibrated, $G$-invariant, Lagrangian mean curvature flow $\big( L_K^n(t) \big)_{t \in [0, T)}$ in $\C^n$ such that 
    \begin{enumerate}
        \item \label{main thm paper II intro, item finite-time sing}
        the flow $L_K(t)$ forms a finite-time singularity at the spatial origin $\mathbf 0 \in \C^n$ at time $T = T(n,K) < \infty$,
        \item \label{main thm paper II intro, item A blow-up rate}
        the second fundamental form blows up as $|\mathbf{A}_{L_K(t)}| \sim \frac1{(T-t)^{K/2}}$, that is,
            $$0 < \liminf_{t \nearrow T}  \sup_{x \in L_K(t)}( T- t)^{K/2}   |\mathbf{A}_{L_K(t)}| \le \limsup_{t \nearrow T } \sup_{x \in L_K(t)} ( T- t)^{K/2}   |\mathbf{A}_{L_K(t)}| < \infty,$$ 
        and 

        \item \label{main thm paper II intro, item profile}
        the $G$-invariant flow $L_K(t)$ can be represented as the graph of a profile function\footnote{See Subsections \ref{subsect Cohom-1 SLs and LMCF}--\ref{Subsect Rescaled LMCF near SL Profile} for the definition of ``profile function'' and how it corresponds to a $G$-invariant Lagrangian mean curvature flow.} $h_K : \R \times [0, T) \to \R$ of the form
            $$h_K(x, t) =  f_{ \tilde a(t)}(x) + \partial_x \tilde u (x, t) = \tilde a(t) \cdot f_1\left( \frac x {\tilde a( t)} \right) + \partial_x \tilde u (x,t)$$
        where
        \begin{enumerate}
            \item $f_{\tilde a}$ denotes the profile function of the smooth, asymptotically conical, $G$-invariant special Lagrangian $\tilde a \overline L$ (see Lemma \ref{lem: profile function of Lawlor neck} and Subsection \ref{subsect Cohom-1 SLs and LMCF}),
            \item $\tilde a= \tilde a(t)$ is a $C^1$-function such that $\frac12 ( T-t)^{\frac K2} \le \tilde a (t) \le 2 ( T - t)^{\frac K2}$ for all $t$,
            and
            \item there exist $\epsilon = \epsilon(n, K) \in \left(0, \frac1{100 \cdot (n+K)} \right)$ and $C = C(n, K) > 1$ such that 
            $\partial_x \tilde u$ has global pointwise bounds
            \begin{equation}  \label{main thm paper II intro, eqn 1}
                |\partial_x \tilde u |(x,t) \le  \frac{C\tilde a(t)^2 \cdot  (T-t)^{- \frac \epsilon 2}}{\sqrt{\tilde a(t)^2 + x^2}^{1- \epsilon}} +  |x| \min \left\{ C |x|^{(2-\epsilon )(K-1)}, \epsilon \right\}
            \end{equation}
            for all $(x,t) \in \R \times [0, T)$.
        \end{enumerate}

    \end{enumerate}

    In particular, 
        
        \begin{enumerate}
        \setcounter{enumi}{3}
            \item \label{main thm paper II intro, item tangent flow}
            the tangent flow of $L_K(t)$ at $(\mathbf 0, T)$ is unique and given by the stationary flow of a pair of $G$-invariant special Lagrangian cones $\mathcal C_0 \cup \mathcal C_1$ (with multiplicity one),

            \item \label{main thm paper II intro, item Type II sing model}
            the Type II rescaled flow $\frac1{(T-t)^{K/2}} L_K(t)$ converges to the smooth, $G$-invariant special Lagrangian $\overline a \overline L$ (for some time-independent $\overline a = \overline a(n,K) > 0$) in the sense that 
                $$\frac1{(T-t)^{K/2}} L_K(t) \xrightarrow[t \nearrow T]{C^\infty_{loc}(\C^n) } \overline a \overline L,$$
            and
            \item \label{main thm paper II intro, item angle osc}
            the oscillation of the Lagrangian angle of the initial data can be taken to be arbitrarily small.
        \end{enumerate}    
\end{theorem}

\begin{remark}
    For concreteness, one can take, for example, $G = SO(n)$ in Theorem \ref{main thm paper II intro} to be the special orthogonal group acting on $\C^n$ via $A \cdot( \mathbf x, \mathbf y) = (A \mathbf x , A \mathbf y)$, where here $\mathbf x= (x^\alpha)$ and $\mathbf y = (y^\alpha)$ respectively denote the real and complex parts of the standard coordinates $\mathbf z = ( z^\alpha = x^\alpha + i y^\alpha)$ on $\C^n$.
    In this case, the tangent flow $\mathcal C_0 \cup \mathcal C_1$ is a transverse pair of $n$-dimensional planes and the desingularization $\overline L$ is an $SO(n)$-invariant Lawlor neck \cite{Law89}.
\end{remark}

\begin{remark}
    In fact, the proof of Theorem \ref{main thm paper II intro} obtains weighted $C^{2,\alpha}$ estimates for $\tilde u$ that are stronger than the derivative estimates listed in \eqref{main thm paper II intro, eqn 1} above.
    We refer the reader to the proof of Theorem \ref{main thm paper II intro} in Subsection \ref{Subsect Proof Main Thm Paper II} for details.
\end{remark}

\begin{remark}
    Theorem \ref{main thm paper II intro} continues to hold if ``\dots smooth, non-compact, properly embedded, exact, almost-calibrated \dots'' is replaced with ``\dots smooth, compact, boundaryless, immersed, zero-Maslov \dots.''
    We informally sketch the proof of how our arguments may be adapted to this case in Remark \ref{rem: compactification}. 
\end{remark}

Theorem \ref{main thm paper II intro} can be regarded as a completion of the research problem described in \cite{Joyce15}*{Problem 3.12(a)} for a particular class of special Lagrangian cones and their desingularizations in dimensions $n \ge 3$ (cf. ~\cite{LSS22}*{Theorem 8.3} for a related result in dimension $n=2$).
In our companion paper \cite{SS26III}, we prove the solutions from Theorem \ref{main thm paper II intro} with $n \ge 3$ and $K\ge 3$ have uniformly bounded mean curvature.
Aside from providing the first quantitatively precise dynamics and curvature blow-up rates for finite-time singularities of Lagrangian mean curvature flow, the solutions in Theorem \ref{main thm paper II intro} and their construction are significant for several geometric and analytic reasons we now describe.

\subsubsection{Singularity Formation for a Fully Nonlinear Parabolic PDE}
While the mean curvature flow for submanifolds of general codimension can be regarded as a quasilinear parabolic PDE system, the Lagrangian condition provides additional structure that makes it natural to regard the flow instead as a \emph{fully nonlinear}, scalar-valued PDE;
in effect, trading a quasilinear system for a fully nonlinear scalar-valued equation.
A scalar-valued equation allows for additional maximum principle techniques,
but the extra nonlinearity creates substantial analytic challenges that are particularly pronounced near the singularity.

Our overall strategy to prove Theorem \ref{main thm paper II intro} roughly resembles the approach Had{\v z}i{\'c}-Rapha{\"e}l initiated for the Stefan problem \cite{HadzicRaphael19}.
Namely, we use a time-dependent spectral basis associated to the Type II blow-up model $ a \overline L$ to establish the leading order behavior of nonlinear flow solutions.
Related constructions have been carried out for semilinear heat equations, the Keller-Segel system, and the Yang-Mills heat flow 
\cites{CollotMerleRaphael20, CGMN_Flow, BDG25, HouNguyenSong25}.
We also note that D{\'a}vila-del Pino-Wei initiated an alternative inner-outer gluing method to construct Type II singularities for the harmonic map heat flow \cite{DdPW19}, and this method has since been applied to a variety of nonlinear parabolic PDEs (see e.g. \cites{dPMW19, dPMWZ20, BDdPM26} and references therein).
Regardless of the overall strategy, the precise nonlinearity of the PDE plays a significant and essential role in the resulting analysis.
To the authors' knowledge, our Theorem \ref{main thm paper II intro} is the first time either strategy has been applied to a fully nonlinear parabolic PDE.

\subsubsection{Non-Integrable Deformations of Tangent Flows}
Theorem \ref{main thm paper II intro} has some resemblance to Vel{\'a}zquez's results for mean curvature flow \cite{Velazquez94} (see also \cites{GuoSesum18, Stolarski23, Liu24, Huang25}).
\cite{Velazquez94} constructed smooth solutions to the mean curvature flow of hypersurfaces $\Sigma^n(t) \subseteq \R^{n+1}$ ($n+1=2m \ge 8$) that form finite-time singularities at time $T < \infty$ 
modeled on the Simons cone $\mathcal C^n$ and a smooth, minimal desingularization $\overline \Sigma^n$ of the cone $\mathcal C$.
The minimal hypersurfaces $\overline \Sigma$ desingularize $\mathcal C$ in the sense that $\overline \Sigma$ is the normal graph of a function $u = O(|x|^{-\alpha})$ defined on $\mathcal C \setminus B_1$ where $\alpha = \alpha(n) =  \frac{n-2}2 - \frac12 \sqrt{ n^2 - 8n +8} \in (1,2]$.
The \emph{slow} decay rate of this function has two key consequences:
\begin{enumerate}
    \item It allows for the construction of barriers on the inner region $|x| \lesssim \sqrt{T-t}$ (see namely \cite{Velazquez94}*{Lemma 4.5} or \cite{Huang25}*{Subsection 5.3}).
    \item The nonlinear analysis lends itself to (Gaussian-weighted) $L^2$ estimates and methods.
    Indeed, the family of rescaled minimal hypersurfaces $\overline \Sigma_a := a \overline \Sigma$ ($a>0$)
    is generated by the function $\phi_0 = |x|^{-\alpha} \in L^2_{loc}(\mathcal C)$ on the cone $\mathcal C$.
    Informally, this deformation $\overline \Sigma_a$ of the cone $\mathcal C$ is integrable in some $L^2$ sense. 
\end{enumerate}

Theorem \ref{main thm paper II intro} constructs a Lagrangian mean curvature flow $L^n(t) \subseteq \mathbb C^n$ ($n\ge 3$) that forms a finite-time singularity modeled on a transverse pair of cones $\mathcal C_0 \cup \mathcal C_1$ and a smooth, special Lagrangian desingularization $\overline L$ of $\mathcal C_0 \cup \mathcal C_1$.
Away from the origin, the special Lagrangian $\overline L$ can be regarded as the graph of a 1-form $du : ( (\mathcal C_0 \cup \mathcal C_1 ) \setminus B_1) \to T^*(P_0 \cup P_1) $ on the cone $\mathcal C_0 \cup \mathcal C_1$ such that $du = O(|x|^{1-n})$.
The \emph{fast} decay rate of this 1-form means:
\begin{enumerate}
    \item There are no clear barriers to control the flow on the inner region $|x| \lesssim \sqrt{T-t}$ as in \cite{Velazquez94}.
    \item $L^2$ methods are no longer available since $|x|^{1-n} \notin L^2_{loc}(\mathcal C_0 \cup \mathcal C_1)$ for $n \ge 2$.
    Informally, the family of rescaled special Lagrangians $\overline L_a := a \overline L$ ($a > 0$) is a deformation of $\mathcal C_0 \cup \mathcal C_1$, and this deformation is \emph{not} integrable in an $L^2$ sense.
\end{enumerate}

The lack of $L^2(\mathcal C_0 \cup \mathcal C_1)$ methods is a substantial hurdle.
In essence, we replace $L^2(\mathcal C_0 \cup \mathcal C_1)$ spaces with the collection of spaces $\{ L^2( \overline L_a) \}_{a > 0}$ using the spectral analysis developed in our companion paper \cite{SS26I}.

\subsection{Outline of the Proof} \label{Subsect Proof Outline}

We now provide an informal description of the proof of Theorem \ref{main thm paper II intro} and further indicate how we address the challenges of full nonlinearity and non-integrability described above.
First, the $G$-invariance of the flow $L(t)$ allows us to recast the (parabolically rescaled) Lagrangian mean curvature flow as a PDE for the associated profile function $h(s,\tau)$.
We then use the scale-dependent eigenfunctions $\phi_{i,a}$ constructed in our companion paper \cite{SS26I} to expand this profile function as
\begin{equation} \label{outline of proof, eqn 0}
    h = f_{a} + \partial_s u = f_a + \partial_s\left( \sum_{i=1}^K b_i(\tau) (\phi_{i,a} - \phi_{0,a}) + v \right)
\end{equation}
where 
\begin{itemize}
    \item $0 < a = a(\tau) \ll 1$  is a small time-dependent scale and
    \item $v = v(s, \tau)$ represents a remainder term such that $\langle v, \phi_{i,a} \rangle_{L^2_a} = 0$ for all $0 \le i \le K$.
    (The space $L^2_a$ is defined in Definition \ref{defn Weigthed Sobolev spaces} to be a certain Gaussian-weighted $L^2$ space associated to $a \overline L$.)
\end{itemize}
The parabolically rescaled Lagrangian mean curvature flow is then a fully nonlinear parabolic PDE for $u(s,\tau)$ and $a(\tau)$.
Our companion paper \cite{SS26I}*{Section \ref{I-Section Approx Flow Solutions}} described an \emph{approximate} solution to this PDE given by
\begin{equation} \label{outline of proof, eqn 1}
    u \approx e^{(1-K)\tau} (\phi_{K,a} - \phi_{0,a}) \qquad\text{and} \qquad  a(\tau) \approx e^{\frac{1-K}2 \tau}
\end{equation}
for a fixed $K \ge 2$.

To prove Theorem \ref{main thm paper II intro}, we implement a Wa{\.z}ewski Box argument to show that this approximate solution \eqref{outline of proof, eqn 1} can be perturbed to an \emph{actual} solution of the nonlinear PDE.
This argument informally goes as follows.
Consider the family of initial data
\begin{multline} \label{outline of proof, eqn 2}
    u_{\mathbf c} (\cdot , \tau_0) = e^{(1-K)\tau_0} ( \phi_{K,a} -\phi_{0,a}) + \sum_{i=1}^K c_i e^{(1-K)\tau_0} (\phi_{i,a} - \phi_{0,a}), \qquad a_{\mathbf c}(\tau_0) = e^{(1-K)\tau_0},\\
    \text{parametrized by } \mathbf c = (c_1, \dots, c_K) \in \R^K.
\end{multline}
Let $\big(u_{\mathbf c}(s, \tau; \tau_0), a_{\mathbf c}(s, \tau; \tau_0)\big)$ denote the solution to the nonlinear parabolic PDE with initial data \eqref{outline of proof, eqn 2}.
We design a time-dependent set $\mathcal B(\tau)$ of $(u,a)$ that shrinks to $0$ as $\tau \nearrow +\infty$.
A delicate series of estimates in Sections \ref{Section Mod Eqns and Integral Ests}--\ref{Section Holder Estimates for the Potential} shows that, for an appropriate choice of parameters defining $\mathcal B$, if the solution $(u_{\mathbf c}, a_{\mathbf c})$ exits the set $\mathcal B$ at time $\tau < \infty$, then it must exit through a $(K-1)$-dimensional sphere $S^{K-1}$ in the boundary $\partial \mathcal B(\tau)$ of $\mathcal B(\tau)$.
Supposing, for the sake of contradiction, that every solution $(u_{\mathbf c}, a_{\mathbf c})$ exits $\mathcal B$ at some finite time $\tau^*(\mathbf c) < \infty$, then the exit map $\mathbf c \mapsto \big(u_{\mathbf c}( \tau^* (\mathbf c)), a_{\mathbf c}(\tau^* (\mathbf c)) \big) \in S^{K-1}$ can be used to construct a continuous retraction $B^K \to S^{K-1} $ of the $K$-dimensional ball to its boundary.
Such a continuous map, however, cannot exist by homology considerations.
This contradiction therefore proves there exists some $\mathbf c$ such that $(u_{\mathbf c}, a_{\mathbf c})$ remains in $\mathcal B(\tau)$ for all $\tau \ge \tau_0$.
By the particular choice of $\mathcal B$, this is the desired solution $(u,a)$ of the PDE with leading order behavior given by \eqref{outline of proof, eqn 1}.

The technical bulk of the paper consists in proving a series of estimates in Sections \ref{Section Mod Eqns and Integral Ests}--\ref{Section Holder Estimates for the Potential}.
These estimates say roughly that the remainder term $v$ remains small relative to the eigenfunction terms $\sum_{i=1}^K b_i(\tau) ( \phi_{i,a}- \phi_{0,a})$ in \eqref{outline of proof, eqn 0}, and so the dynamics are well-approximated by an ODE system for $a, b_1, \dots, b_K$.
Such estimates are most difficult to obtain near the singularity, and it is here where the full nonlinearity of the PDE and the non-integrability of the tangent flow presents the most serious challenges.

To overcome these challenges, we prove a weighted $C^0$ estimate (Proposition \ref{prop: weighted sup norm estimate intermediate}) that controls the PDE solution near the singularity by its values at parabolic scales $|x| \sim \sqrt{T-t}$.
One may regard the statement of this estimate as a weaker version of the maximum principle.
Its proof, on the other hand, comes from combining a blow-up argument with Liouville-type theorems for $\overline L$ and $\mathcal C_0 \cup \mathcal C_1$.
After obtaining this estimate, these $C^0$ bounds can then be bootstrapped to weighted $C^{2, \alpha}$ bounds near the singularity, which suffices to control the dynamics of the flow (see Subsection \ref{subsect inner region estimates}).

The paper is organized as follows.
In Section \ref{section Prelims}, we present background material on $G$-invariant Lagrangian mean curvature flows, recall the spectral analysis developed in \cite{SS26I}, and establish notation that will be used throughout the paper.
Sections \ref{Section Mod Eqns and Integral Ests}--\ref{Section Holder Estimates for the Potential} obtain the estimates required for the Wa{\.z}ewski box argument.
In particular, Section \ref{Section Mod Eqns and Integral Ests} focuses on integral estimates, and Section \ref{Section Holder Estimates for the Potential} focuses on weighted H{\"o}lder estimates.
Section \ref{Section Box Argument} then combines these estimates to prove Theorem \ref{main thm paper II intro}.
Appendix \ref{Appendix Assorted Formulas and Estimates} collects some formulas and estimates that are used throughout the paper.

\addtocontents{toc}{\protect\setcounter{tocdepth}{0}} 
\subsection*{Acknowledgments}
The authors thank Albert Wood, Charles Collot, and Felix Schulze for helpful conversations.
The first named author is supported by a Leverhulme Trust Early Career Fellowship (ECF-2023-182).
The second named author is supported by the Taiwan NSTC grant 114-2115-M-008-012-MY3. 
The authors also thank the University of Warwick and the National Center for Theoretical Sciences for the support and hospitality during the authors' visits.

\subsection*{Statement on AI use} AI tools were used only during the final stages of manuscript preparation to improve the exposition, ensure consistency of notation, and detect and correct minor errors. They were not used to develop the mathematical ideas presented in this paper.
\addtocontents{toc}{\protect\setcounter{tocdepth}{2}}

\section{Preliminaries} \label{section Prelims}
\subsection{Lagrangian Mean Curvature Flow in $\mathbb{C}^{n}$}
Throughout, we take the ambient manifold to be $\mathbb{C}^{n}$ equipped with its
standard flat Calabi--Yau structure $(\overline{g},J,\omega,\Omega)$. In terms of
the complex coordinates $z_{j}=x_{j}+iy_{j}$, $j=1,\dots,n$, these are given by
\begin{align*}
\overline{g}
=
\sum_{j=1}^{n}(dx_{j}^{2}+dy_{j}^{2}),\qquad
J=i\,\cdot,\qquad
\omega=
\sum_{j=1}^{n}dx_{j}\wedge dy_{j},\qquad
\Omega=dz_{1}\wedge\cdots\wedge dz_{n}.
\end{align*}
We denote by $\lambda$ the Liouville form
\begin{align*}
    \lambda
    =
    \frac{1}{2}\sum_{j=1}^{n}\big(y_{j}dx_{j}-x_{j}dy_{j}\big),
\end{align*}
such that $\omega = -d\lambda$.

An $n$-dimensional immersed submanifold $L\subset\mathbb{C}^{n}$ is called
\emph{Lagrangian} if $\omega\big|_{L}=0$. We say that a one-parameter family of
immersed submanifolds $\{L(t)\}_{t\in[0, T)}$ satisfies the \emph{Lagrangian mean curvature flow} (LMCF) if each time-slice
$L(t)$ is Lagrangian and
\begin{align}
    \left(\frac{d}{dt}\mathbf{x}_{L(t)}\right)^{\perp}
    =
    \mathbf{H}_{L(t)},
    \qquad t\in(0,T),
\end{align}
where $\mathbf{x}_{L(t)}\in\mathbb{C}^{n}$ is the position vector and $\mathbf{H}_{L(t)}$ denotes the mean curvature vector of $L(t)$.

If $L$ is oriented and Lagrangian, then
\begin{align*}
    \Omega|_{L}=e^{i\theta_{L}}\,dV_{L}
\end{align*}
for some function $\theta_{L}:L\to\mathbb{R}/2\pi\mathbb{Z}$, called the
\emph{Lagrangian angle}. The Lagrangian angle is a potential for the mean
curvature vector, in the sense that
\begin{align*}
    \mathbf{H}_{L}=J\nabla\theta_{L}.
\end{align*}
It follows that $L$ is minimal if and only if $\theta_{L}$ is constant $\overline{\theta}$. In this
case, we call $L$ a \emph{special Lagrangian (SL)} submanifold with phase $\overline{\theta}$.

We say a Lagrangian submanifold $L$ is \emph{zero-Maslov} if $\theta_{L}$ can be lifted to a real-valued function $\theta_{L}:L\to\mathbb{R}$. A zero-Maslov Lagrangian submanifold $L$ is \emph{almost-calibrated} if $\osc_{L}\theta_{L}<\pi$.
A Lagrangian submanifold $L$ is \emph{exact} if $\lambda\big|_{L} = d\beta_{L}$ for some $\beta_{L}:L\to\mathbb{R}$. In this case, the position vector $\mathbf{x}_{L}\in\mathbb{C}^{n}$ of $L$ satisfies $\mathbf{x}_{L}^{\perp} = 2J\nabla\beta_{L}$, where $(\cdot)^{\perp}$ denotes the orthogonal projection onto the normal space of $L$. In other words, if $L$ is exact, the normal part of the position vector of $L$ has a potential $2\beta_{L}$.
Exactness, zero-Maslov, and almost-calibrated conditions are preserved under LMCF.

Suppose that a LMCF $L(t)$ develops a finite-time singularity at
$(P,T)\in\mathbb{C}^{n}\times\mathbb{R}$. That is,
\begin{align*}
    \limsup_{\substack{x\in L(t),\, x\to P\\ t\nearrow T}}
    |\mathbf{A}_{L(t)}(x)|=\infty,
\end{align*}
where $\mathbf{A}_{L(t)}$ denotes the second fundamental form of $L(t)$ in
$\mathbb{C}^{n}$. A coarse asymptotic structure of the singularity is captured
by its \emph{tangent flows}, which arise as subsequential limits of the
parabolically rescaled flows
\begin{align*}
    L^{i}(t)
    :=
    \lambda_{i}\big(L(T+t/\lambda_{i}^{2})-P\big),
    \qquad t<0,
\end{align*}
where $\lambda_{i}\nearrow\infty$. Neves~\cite{Neves07} showed that if the LMCF is zero-Maslov, then
every tangent flow at a finite-time singularity is a union of special
Lagrangian cones. Equivalently, consider the \emph{rescaled flow}
\begin{align*}
    M(\tau)
    =
    e^{\tau/2}\big(L(T-e^{-\tau})-P\big),
\end{align*}
which satisfies
\begin{align}
    \left(\frac{d}{d\tau}\mathbf{x}_{M(\tau)}\right)^{\perp} = \mathbf{H}_{M(\tau)} + \frac{1}{2}\mathbf{x}_{M(\tau)}^{\perp},
\end{align}
then every subsequential limit of $M(\tau)$ as $\tau\to\infty$ is a union of SL cones. If, in addition,
$L(t)$ is \emph{exact}, then Neves~\cite{Neves07} showed that
every tangent flow is a union of SL cones of \emph{same phase}.

The tangent flow captures the geometry of the flow near the singularity at the
parabolic scale $\sim\sqrt{T-t}$. If the tangent flow is itself singular, then one
must look at smaller, \emph{Type II scales} in order to detect the smooth geometry
forming near the singularity. We call any smooth limiting flow obtained at such
a scale a \emph{Type II model} for the singularity.

Assuming that the tangent flow is a SL cone, one
expects the Type II model to be an asymptotically conical SL submanifold which desingularizes the tangent flow. More precisely, one expects the existence of a SL submanifold $\overline{L}$ asymptotic to the tangent flow
$\mathcal{C}$ at infinity, together with a scale function
\begin{align*}
    a(t)=o(\sqrt{T-t})\qquad\text{as }t\nearrow T,
\end{align*}
such that, in a neighborhood of the singularity, the flow is asymptotic to the
rescaled model $a(t)\overline{L}$, which collapses to $\mathcal{C}$ as
$t\nearrow T$.

\subsection{Cohomogeneity-One SLs and LMCF} \label{subsect Cohom-1 SLs and LMCF}

Now suppose that a compact connected Lie group $G\leq SU(n)$ acts linearly on
$\mathbb{C}^{n}$, preserving the standard flat Calabi--Yau structure, and that
the generic $G$-orbits are $(n-1)$-dimensional. 
Throughout the paper, we shall consider Lagrangian submanifolds in $\C^n$ that are invariant under this group action.
Such cohomogeneity-one constructions were studied by, for instance, Haskins~\cite{Haskins04},
Joyce~\cite{Joyce02}, and Madnick--Wood~\cite{MadnickWood25}. We thus direct the readers to those papers for further details. 

In this setting, if $\mathcal{C}\subset\mathbb{C}^{n}$ is a $G$-invariant SL cone with phase $0$, then, for any $a>0$, the submanifold
\begin{align}
    L
    =
    \left\{
    \lambda\boldsymbol{\sigma}
    \::\:
    \lambda\in\mathbb{C},\ 
    \im(\lambda^{n})=a^{n},\
    \arg(\lambda)\in(0,\pi/n),\
    \boldsymbol{\sigma}\in\Sigma=\mathcal{C}\cap\mathbb{S}^{2n-1}
    \right\}
\end{align}
is a smooth, embedded, SL submanifold in $\mathbb{C}^{n}$ with
phase $0$, asymptotic to the SL cone $\mathcal{C}\cup -e^{i\pi/n}\cdot\mathcal{C}$, where $e^{i\pi/n}\cdot$ denotes the $U(1)$ action on $\mathbb{C}^{n}$ via 
\begin{align*}
    e^{i\theta}\cdot (z_1, \cdots, z_n) := (e^{i\theta}z_1, \cdots, e^{i\theta}z_{n}),
\end{align*}
and the minus sign indicates the reverse of orientation so that $-e^{i\pi/n}\cdot\mathcal{C}$ has phase $0$.

The set
\begin{align*}
    \left\{
    \lambda\in\mathbb{C}
    \::\:
    \im(\lambda^{n})=a^{n},\
    \arg(\lambda)\in(0,\pi/n)
    \right\}
\end{align*}
is a connected, smooth, embedded curve in $\mathbb{C}$, called the \emph{profile curve}
of $L$. It can be parametrized by
\begin{align*}
    \gamma_{a}(\theta)
    =
    a\,\gamma(\theta)
    =
    a\,\sin(n\theta)^{-1/n}e^{i\theta},
    \qquad \theta\in(0,\pi/n),
\end{align*}
and satisfies
\begin{equation}\label{eq: equiv. SL profile curve}
    \boldsymbol{\kappa}(\gamma_{a})
    -
    (n-1)\frac{\gamma_{a}^{\perp}}{|\gamma_{a}|^{2}}
    =
    0,
\end{equation}
where $\boldsymbol{\kappa}(\gamma_a)$ denotes the curvature of the plane curve $\gamma_a \subseteq \C$.
Rotating the profile curve by $e^{i\phi}$ changes the phase of the corresponding
SL by $n\phi$. In particular, the rotated curve
\begin{align}
    \overline{\gamma}(\theta)
    =
    \sin(n\theta)^{-1/n}
    e^{i\left(\theta+\frac{\pi}{2}\left(1-\frac1n\right)\right)},
    \qquad \theta\in(0,\pi/n),
\end{align}
is reflection-symmetric with respect to the imaginary axis and gives rise to an
asymptotically conical, $G$-invariant, SL $\overline{L}\subset\mathbb{C}^{n}$ with phase
$(n-1)\pi/2$.

\begin{lemma}[\cite{SS26I}*{Lemma~\ref{I-lem: profile function of Lawlor neck}}] \label{lem: profile function of Lawlor neck}
    Let $n \ge 2$. For any $a>0$, the curve $\overline{\gamma}_{a}=a\overline{\gamma}$ can be
    written as the graph of an even function $f_a : \R \to \R_{> 0}$, which satisfies
    \begin{gather}\label{eq: profile function minimality}
        \frac{(f_{a})_{xx}}{1+(f_{a})_{x}^{2}}
        +
        (n-1)\frac{x(f_{a})_{x}-f_{a}}{x^{2}+(f_{a})^{2}}
        =
        0 \\
        \nonumber
        \text{with} \qquad  
        f_{a}(0)=a,\qquad
        f_{a}'(0)=0,\quad \text{and}\quad 
        f_{a}''(x)>0
        \quad\text{for all }x\in\mathbb{R}.
    \end{gather}
    Moreover, as $|x|\to\infty$, there holds\footnote{See Definition \ref{defn notation} below for the precise definition of the $O(\cdot; \cdot)$ notation.}
    \begin{gather}\label{eq: asymptotic of SL profile function}
        \left|
        \left(\frac{d}{dx}\right)^{k}
        \!\left(f_{a}(x)-\overline{c}_{0}|x|\right)
        \right|
        =
        O\!\left(a^{n}|x|^{1-n-k} ; k\right)
        \quad\text{for all }k\in\mathbb{N}\cup\{0\}, \\
        \nonumber
        \text{where }
        \overline{c}_{0}
        =
        \tan\!\left(\frac{\pi}{2}-\frac{\pi}{2n}\right)
        =
        \cot\!\left(\frac{\pi}{2n}\right).
    \end{gather}
\end{lemma}

Neves~\cite{Neves07}, Groh--Schwarz--Smoczyk--Zehmisch ~\cite{GSSZ07}, and
Madnick--Wood~\cite{MadnickWood25} showed that, under the above
cohomogeneity-one ansatz, LMCF reduces to the following evolution equation for
the profile curves $\gamma(t):\mathbb{R}\to\mathbb{C}$:
\begin{equation}\label{eq: equiv. LMCF profile curve}
    \left(\frac{d}{dt}\gamma(t)\right)^{\perp}
    =
    \boldsymbol{\kappa}(\gamma(t))
    -
    (n-1)\frac{\gamma(t)^{\perp}}{|\gamma(t)|^{2}},
\end{equation}
where $\boldsymbol{\kappa}(\gamma(t))$ denotes the curvature vector of
$\gamma(t)$.

\begin{remark}\label{rem: indep. of G}
    It is worth emphasizing that the reduced equations
    \eqref{eq: equiv. LMCF profile curve} and
    \eqref{eq: equiv. SL profile curve} are independent of the choice of the Lie
    group $G$. For example, when $G=SO(n)$ acts diagonally on $\mathbb{C}^{n}$,
    the orbits are copies of $\mathbb{S}^{n-1}$, and the resulting
    $SO(n)$-invariant SLs are diffeomorphic to
    $\mathbb{S}^{n-1}\times\mathbb{R}$. These are the
    \emph{Lagrangian catenoids}, which may be viewed as the $SO(n)$-invariant
    versions of the \emph{Lawlor necks}~\cite{Law89}.
\end{remark}

\subsection{Rescaled LMCF Near a SL Profile} \label{Subsect Rescaled LMCF near SL Profile}

Let $(\gamma(t))_{t\in I}$ be a solution of (\ref{eq: equiv. LMCF profile curve}), such that the associated $G$-invariant LMCF $L(t)$ is almost-calibrated. Denote the image of $\gamma(t):\mathbb{R}\to\mathbb{C}$ by $\Gamma(t) = \gamma(t)(\mathbb{R})$.

Suppose that for each $t\in I$, $\Gamma(t)$ is an entire graph with profile function $f(\cdot, t)$, that is, $\Gamma(t) = \{s + if(s, t)\:|\:s\in\mathbb{R}\}$. Then $f$ satisfies the quasi-linear equation
\begin{equation}\label{eq: LMCF of profile function}
    \partial_{t}f = \frac{f_{xx}}{1+f_{x}^{2}} + (n-1)\frac{xf_{x} - f}{x^{2}+f^{2}}.
\end{equation}

Suppose that $I = [0, T)$ for some $0 < T < \infty$, and $L(t)$ develops a finite-time singularity as $t\nearrow T$. 
Then, by \cite{MadnickWood25}*{Theorems 6.2 and 6.7}, the singularity must occur at the origin $\mathbf 0$, and the tangent flow at $(\mathbf 0, T)$ is a union of $G$-invariant SL cones. 
Consider the case that the tangent flow at $(\mathbf 0, T)$ is $\mathcal{C}_{0}\cup\mathcal{C}_{1}$ given by the reflection-symmetric profile curve $C_{0} := e^{-i(\frac{\pi}{2}+\frac{\pi}{2n})}\mathbb{R}_{+}\cup e^{i(\frac{\pi}{2}+\frac{\pi}{2n})}\mathbb{R}_{+}$. It follows that the rescaled flow $\widetilde{\gamma}(\tau) := e^{\tau/2}\gamma(T - e^{-\tau})$ converges to $C_{0}$ as $\tau\to \infty$ smoothly away from $\mathbf 0$. By the graphicality assumption on $\gamma(t)$, the rescaled flow $\widetilde{\gamma}(\tau)$ is also graphical, whose profile function $h(\cdot, \tau)$ satisfies
\begin{align}\label{eq: rescaled flow of profile function}
    \partial_{\tau}h = \frac{h_{ss}}{1+h_{s}^{2}} + (n-1)\frac{sh_{s} - h}{s^{2}+h^{2}} - \frac{1}{2}(sh_{s} - h).
\end{align}

Equation~(\ref{eq: rescaled flow of profile function}) is again a quasi-linear parabolic equation.
In \cite{SS26I}*{subsection~\ref{I-subsec: RLMCF near a SL profile}}, we linearize at a SL profile $f_{a(\tau)}$ given in Lemma~\ref{lem: profile function of Lawlor neck}, and integrate (\ref{eq: rescaled flow of profile function}) into a fully non-linear equation. The key observation is that the Lagrangian angle of $L(t)$ and the potential of $\lambda\big|_{L(t)}$ are given by 
\begin{align*}
    \Theta[h] = \arctan(h_{s}) + (n-1)\arctan(s^{-1}h),\quad\mbox{and}\quad \beta[h] = \frac{1}{2}\int_{0}^{s}(h-\sigma h_{\sigma}),
\end{align*}
respectively. Hence, right-hand side of (\ref{eq: rescaled flow of profile function}) can be written as
\begin{align*}
    \partial_{s}\left(\Theta[h] + \beta[h]\right) = \frac{h_{ss}}{1+h_{s}^{2}} + (n-1)\frac{sh_{s} - h}{s^{2}+h^{2}} - \frac{1}{2}(sh_{s} - h).
\end{align*}

\begin{proposition}[\cite{SS26I}*{Proposition~\ref{I-prop: PDE for the potential}}] \label{prop: PDE for the potential}
    Let $h(s, \tau) = f_{a(\tau)}(s) + u_{s}(s, \tau)$. Then $h$ is a solution of (\ref{eq: rescaled flow of profile function}) if and only if
    \begin{equation}\label{eq: linearization at Lawlor at potential level with constant c}
        \partial_{\tau}u = (n-1)\frac{\pi}{2} + \left(1-2\tfrac{\partial_\tau a(\tau)}{a(\tau)}\right)\beta[f_{a(\tau)}] + H_{s, a}u + Q_{s, a}(u) + c
    \end{equation}
    for some $\tau$-dependent constant $c(\tau) \in\mathbb{R}$, where
    \begin{equation}
        H_{s, a}u = L_{s, a}u - \frac{s}{2}u_{s} + u, \qquad 
        L_{s,a} u = \frac{u_{ss}}{1 + (f_a)_s^2} +(n-1) \frac{s u_s}{s^2 + f_a^2} 
        ,
    \end{equation}
    and
    \begin{multline}
        Q_{s, a}(u) = \Theta[f_{a}+u_{s}] - (n-1)\frac{\pi}{2} - L_{s, a}u \\
        = \arctan\left((f_{a} + u_{s})_{s}\right) - \frac{u_{ss}}{1+(f_{a})_{s}^{2}} + (n-1)\left[\arctan\left(\frac{f_{a}+u_{s}}{s}\right) - \frac{su_{s}}{s^{2}+f_{a}^{2}}\right]-(n-1)\frac{\pi}{2}.
    \end{multline}
\end{proposition}

Taking $c = -(n-1)\frac{\pi}{2}$, we obtain:
\begin{corollary}[\cite{SS26I}*{Corollary~\ref{I-cor: PDE for the potential}}] \label{cor: PDE for the potential}
    If $u$ solves the equation
    \begin{equation}\label{eq: main equation}
        \partial_{\tau}u = \left(1 - 2\tfrac{\partial_\tau a(\tau)}{a(\tau)}\right)\beta[f_{a(\tau)}] + H_{s, a}u + Q_{s, a}(u),
    \end{equation}
    then $h(s, \tau) := f_{a(\tau)}(s) + u_{s}(s, \tau)$ is a solution of (\ref{eq: rescaled flow of profile function}).
\end{corollary}

We collect the properties of the zeroth order term in (\ref{eq: main equation}).
\begin{lemma}[\cite{SS26I}*{Lemma~\ref{I-lem beta properties}}] \label{lem beta properties}
    Let $n \ge 3$ and $a>0$. The function $\beta_{a} := \beta[f_{a}]$ has the following properties:
    \begin{itemize}
        \item[(i)] $\beta_{a}(s) = a^{2}\beta_{1}(\frac{s}{a})$.
        \item[(ii)] $2a^{-1}\beta_{a}$ is the potential of the infinitesimal scaling of $f_{a}$.
        \item[(iii)] $L_{s, a}\beta_{a} = 0$.
        \item[(iv)] $\beta_a(s)$ is a smooth, odd function of $s \in \R$. 
        \item[(v)] There is a positive constant  $\overline{\beta}\in (0, \infty)$  such that 
        \begin{equation}
            \lim_{s\to\pm\infty}\beta_{1}(s) = \pm\overline{\beta}, 
        \end{equation}
        and $\beta_{a}$ satisfies the estimate\footnote{See Definition \ref{defn notation} below for the precise definition of the $O(\cdot; \cdot)$ notation.}
        \begin{align}
            \left|\partial_{s}^{k}(\beta_{a}(s) \pm a^{2}\overline{\beta})\right| = O(a^{n}|s|^{2-n-k};k)\quad\mbox{as}\quad s\to\pm\infty.
        \end{align}
        \item[(vi)] For $s > 0$, $\beta_a(s)$ is a positive, increasing, concave function of $s$.
    \end{itemize}
\end{lemma}

\subsection{Notation}
    For the remainder of the paper, we use the following notation unless otherwise noted.

    \begin{definition} \label{defn notation}
        For constants $a,b,c$ and functions $f = f(\vec x), g = g( \vec x),$ and $h = h(\vec x) \ge 0$ defined for all $\vec x \in \Omega$, 
        we write ``$f = g + O(h; a,b,c )$'' to mean there exists a constant $C$ (depending on $a,b,c$ and dimension $n$) such that 
            $$|f(\vec x) - g (\vec x )| \le C h(\vec x) \qquad \text{for all }  \, \vec x \in \Omega.$$
        We write ``$f = g + O(h; a,b,c)$ as $\vec x \to \infty$'' to mean there exists constants $C, R_0$ (depending on $a,b,c$ and dimension $n$) such that 
            $$|f(\vec x) - g (\vec x )| \le C h(\vec x) \qquad \forall \, |\vec x| \ge R_0.$$
        Similarly, ``$f = g + O(h; a,b,c)$ as $\vec x \to \vec x_0$'' means there exists $C, \epsilon$ (depending on $a,b,c$ and dimension $n$) such that 
            $$|f(\vec x) - g(\vec x ) | \le C h(\vec x) \qquad \forall \, |\vec x - \vec x_0| \le \epsilon.$$
        Analogous definitions apply for additional (or fewer) parameters $a,b,c,d,e, \dots$ or for $\vec x = (\vec y, \vec z) \to (\infty, \vec z_0)$.

        For nonnegative functions $f = f(\vec x) \ge 0 $ and $g = g(\vec x) \ge 0$, we use the notation ``$f \lesssim g$'' to mean there exists $C > 0$ such that $f(\vec x) \le C g (\vec x)$ for all $\vec x$.
        We write ``$f \lesssim_{a,b,c} g $'' to indicate that $f \lesssim g$ and the implicit constant $C = C(a,b,c)$ depends on $a,b,c$.
        We write ``$f \sim g$'' to indicate that $f \lesssim g$ and $g \lesssim f$.
        Analogous definitions apply for additional (or fewer) parameters.
    \end{definition}

    For any $a > 0$, $\overline \gamma_a : \R \to \C$ denotes the profile curve for the cohomogeneity-one special Lagrangian $\overline L_a$ asymptotic to $C = \R_{+} e^{i \left( \frac \pi 2 - \frac \pi {2n} \right)} \cup \R_{+} e^{i \left( \frac \pi 2 + \frac \pi {2n} \right)}$ described in Subsection \ref{subsect Cohom-1 SLs and LMCF}.
    It is scaled so that the image 
        $$\overline \Gamma_a := \overline \gamma_a (\R) \text{ contains } 0+ia \in \C.$$
    $f_a : \R \to \R$ denotes the profile function for $\overline \Gamma_a$ as in Lemma \ref{lem: profile function of Lawlor neck}, that is,
        $$\overline \Gamma_a = \{ s + i f_a(s) \, \colon \, s \in \R \}.$$
    $\overline c_0 := \tan\left( \frac \pi 2 - \frac \pi{2n} \right)$ denotes the asymptotic slope of $f_a(s)$ (Lemma \ref{lem: profile function of Lawlor neck}).

    We denote operators
    \begin{align}
        \label{eqn L_a defn}
        L_a u :={}& L_{s,a} u = \frac1{\sqrt{ ( 1 + (\partial_s f_a)^2 ) (s^2 + f_a^2)^{n-1} }} \frac \partial {\partial s} \left( \sqrt{ \frac{(s^2 + f_a^2)^{n-1} }{1 + (\partial_s f_a)^2 } } \frac{\partial u}{\partial s} \right) \\
        \nonumber
        ={}& \frac{u_{ss}}{1 + (f_a)_s^2} +(n-1) \frac{s u_s}{s^2 + f_a^2}, \\
        \label{eqn H_a defn}
        H_a u :={}& H_{s,a} u = L_{s,a} u - \frac s2 \partial_s u + u , \text{ and}\\ 
        \label{eqn Q_a defn}
        Q_a(u) :={}& Q_{s,a}(u) = \arctan ( \partial_s f_a + u_{ss} ) + (n-1) \arctan \left( \frac {f_a+u_s}s \right) -(n-1) \frac \pi 2 - L_{s, a} u .
    \end{align}

\subsubsection{Weighted Sobolev Spaces}

\begin{definition}  \label{defn Weigthed Sobolev spaces}
        For any $a > 0$, define measures $dV_a, d\mu_a$ on $\R$ by
        \begin{align}
            \label{eqn defn dV_a}
            dV_a &:= \sqrt{ ( 1 + (\partial_s f_a)^2 )( s^2 + f_a^2 )^{n-1} } \, ds \text{ and} \\
            \label{defn eqn mu_a}
            d \mu_a &:= e^{-F_a(s)} dV_a \\
            \label{eqn defn F_a}
            \text{where } F_a(s) &:= \frac{s^2}4 + \frac12 \int_0^s \tilde s ( \partial_s f_a)^2 d \tilde s.
        \end{align}

        For any $k \in \N$, $p \in [1, \infty]$, and $\Omega \subseteq \R$ open, we furthermore denote weighted Sobolev spaces with respect to $d \mu_a$ by
        \begin{equation}
            W^{k,p}_a(\Omega) := W^{k,p}( \Omega , d\mu_a ), \qquad
            H^k_a(\Omega) := W^{k, 2}(\Omega, d\mu_a ), \qquad
            L^2_a(\Omega) := L^2(\Omega, d\mu_a),
        \end{equation}
        where for example
        \begin{equation*}
            \| u \|_{H^k_a(\Omega)} = \left( \int_\Omega u^2 d \mu_a + \int_\Omega ( \partial_s u)^2 d \mu_a  \right)^{1/2}.
        \end{equation*}
        Often, we may simply write $H^k_a$ for example when the domain $\Omega$ is clear from context.
\end{definition}

\begin{remark}
    Note that the derivatives used in the definition of the $W^{k,p}_a, H^k_a$ norms are simply $\partial_s$ and \emph{not} covariant derivatives associated to the special Lagrangian $\overline L_a$.
\end{remark}

\begin{definition}
    Let $k \in \mathbb N$, $\Omega, (\tau_0, \tau_1)  \subseteq \R$ be open subsets, $a: (\tau_0, \tau_1) \to \R$, and $u : \Omega \times (\tau_0, \tau_1) \to \R$ be a function such that $u( \cdot, \tau) \in H^k_{a(\tau)}(\Omega)$ for all $\tau \in (\tau_0, \tau_1)$.
    We denote
    \begin{align*}
        \| u \|_{L^\infty H^k_a (\Omega \times (\tau_0, \tau_1))} &:= \esssup_{\tau \in (\tau_0, \tau_1)} \| u ( \cdot, \tau) \|_{H^k_{a(\tau)}(\Omega)} \text{ and}\\
        \| u \|_{L^2 H^k_a( \Omega \times (\tau_0, \tau_1) )} &:= \left( \int_{\tau_0}^{\tau_1} \| u ( \cdot, \tau) \|_{H^k_{a(\tau)} (\Omega)}^2 d \tau \right)^{\frac12}.
    \end{align*}
\end{definition}

\subsection{Perturbative Spectral Theory}

In \cite{SS26I}, we proved the following partial diagonalization result of $H_{a}$ in $L^{2}_{a}$: 

\begin{theorem}[{Eigenfunctions and Spectral Gap, \cite{SS26I}*{Theorem \ref{I-main thm intro}}}]\label{thm global eigenfunctions and spectral gap}
    Let $n\geq 3$, $K\in\mathbb{N}$, and $\epsilon>0$. 
    There exist
    $s_0^*=s_0^*(n,K)>0$ and, for each $0<s_0\leq s_0^*$,
    $a^*=a^*(n,K,s_0,\epsilon)>0$  and $C = C(n,K,s_0) > 0$ with the following property.

    For every $0<a\leq a^*$ and every integer $0\leq k\leq K$, there exists 
    $\widetilde\lambda_k=\widetilde\lambda_{k}(a;n,k,s_0)  \in \R$ 
    and a smooth eigenfunction $\phi_{k,a} = \phi_{k, a , s_0, n}$ such that 
    \begin{equation*}
         H_a\phi_{k,a}=(1-k + \widetilde \lambda_k) \phi_{k,a} \qquad \text{with} \qquad 
        |\widetilde\lambda_k(a)|\le C a
        \quad \text{and} \quad 
        |\partial_a\widetilde\lambda_k(a)|\le C.
    \end{equation*}
    
    Moreover, the eigenfunctions $\{ \phi_{k,a} \}_{0 \le k \le K}$ satisfy the following spectral gap estimate.
    If $v:\mathbb{R}\to\mathbb{R}$ is an odd $C^2$ function such that
    \begin{gather*}
        \|v\|_{L^2_a} + \| \partial_s v \|_{L^2_a} +\|H_av\|_{L^2_a}<\infty 
        \qquad \text{and} \qquad  
        \langle v,\phi_{k, a}\rangle_{L^2_a}=0
        \quad
        \text{for all } 0\leq k\leq K,
    \end{gather*}
    then
    \begin{equation}\label{eq: spectral gap}
        \langle v,H_av\rangle_{L^2_a}
        \leq
        (-K+\epsilon)\|v\|_{L^2_a}^2 .
    \end{equation}
\end{theorem}

The lowest mode $\phi_{0, a}$ is closely related to the infinitesimal scaling of the SL desingularization, as shown in the following proposition.

\begin{proposition}[{Identification of the Scaling Mode, \cite{SS26I}*{Proposition \ref{I-prop scaling mode intro}}}]
\label{prop scaling mode}
    Let $n\geq 3$ and $K\in\mathbb{N}$. There exist
    $s_0^*=s_0^*(n,K)>0$ and, for each $0<s_0\leq s_0^*$,
    $a^*=a^*(n,K,s_0)>0$ with the following property.

    Suppose $0<a\leq a^*$, and let
    $\phi_{k,a}=\phi_{k,a,s_0,n}$, $0\leq k\leq K$, be the eigenfunctions
    constructed in Theorem~\ref{thm global eigenfunctions and spectral gap}. Let $\beta_a$ be the potential for the scaling deformation as in
    Lemma~\ref{lem beta properties}. Then
    \begin{equation*}\label{eqn phi_0 and beta L^2_a close}
        \|\phi_{0,a}-a^{-2}\beta_a\|_{L^2_a}
        \leq
        Ca,
    \end{equation*}
    for some constant $C=C(n,s_0)>0$. In particular,
    \begin{align*}
        \langle \phi_{0,a},a^{-2}\beta_a\rangle_{L^2_a}
        &=
        \|\phi_{0,a}\|_{L^2_a}^2+O(a;n,s_0), \text{ and}\\
        \label{lem L^2_a est for phi with beta, eqn 2}
        \langle \phi_{k,a},a^{-2}\beta_a\rangle_{L^2_a}
        &=
        O(a;n,k,s_0) \qquad 
        \text{for } 1 \le k \le K .
    \end{align*}
\end{proposition}

We will also use several quantitative estimates for the eigenfunctions
$\phi_{k,a}$. Whenever such estimates are needed, we will refer directly to the corresponding
results in the companion paper~\cite{SS26I}.

\subsection{Geometric Decomposition} \label{Subsect Geom Decomp}

In the following sections, we will write the solution $u$ of
\eqref{eq: main equation} in the form
\begin{equation*}
u = \phi + v,
\end{equation*}
where the finite-dimensional component $\phi$ is chosen to be of the form
\begin{equation*}
\phi(s,\tau)
=
\sum_{k=1}^{K} b_k(\tau)
\big(\phi_{k,a(\tau)}(s)-\phi_{0,a(\tau)}(s)\big).
\end{equation*}
Here, $a(\tau)$ is the scale parameter of the desingularization, while
$b_1(\tau),\ldots,b_K(\tau)$ are modulation parameters controlling the leading
finite-dimensional dynamics. 
The remainder $v$ is required
to satisfy the orthogonality conditions
\begin{equation*}
\langle v,\phi_{k,a(\tau)}\rangle_{L^2_{a(\tau)}}=0,
\qquad \forall  0\leq k\leq K ,
\end{equation*}
which also implicitly couples the scaling parameter $a(\tau)$ to the solution $u$.
The scale $a(\tau)$ and modulation parameters $b_1(\tau) , \dots , b_K(\tau)$ will be controlled by the
modulation equations derived in the next section. 
The remainder $v$ will be controlled using the spectral gap estimate
\eqref{eq: spectral gap}, together with parabolic estimates.

The validity of the decomposition $u = \phi + v$ is justified by the following proposition.

\begin{proposition} \label{Prop Geometric Decomposition}
Given $K\in\mathbb{N}$ with $K\geq 2$, $a_{1}>0$, $\mathbf{b}_{1} = ((b_{1})_{1}, \cdots, (b_{K})_{1})\in\mathbb{R}^{K}$, define
\begin{align}
    \phi_{a_{1}, \mathbf{b}_{1}} := \sum_{i=1}^{K}(b_{i})_{1}(\phi_{i, a_{1}} - \phi_{0, a_{1}}).
\end{align}
Then there exists $a_{0} = a_{0}(n, K)>0$ with the following property: For any $a_{1}\in(0, a_{0})$, if $\sum_{i=1}^{K}|(b_{i})_{1}|<C_{1}a_{1}^{2}$ and $v$ is an odd function with $\|v_{1}\|_{L^{2}_{a_{1}}}<C_{1}a_{1}^{2+\eta}$ for some $C_{1}, \eta>0$, then there exist $a>0$, $\mathbf{b}\in\mathbb{R}^{K}$, and an odd function $v\in L^{2}_{a}$ such that
\begin{align}
    |a - a_1|^{2} + \sum_{i=1}^{K}|b_{i} - (b_{1})_{i}| + \|v\|_{L^{2}_{a}}\leq C\|v_{1}\|_{L^{2}_{a_{1}}}
\end{align}
for some $C>0$ independent of $a_{1}$,
satisfying
\begin{align}
    f_{a_1} + \partial_{s}(\phi_{a_1, \mathbf{b}_{1}} + v_{1}) = f_{a} + \partial_{s}(\phi_{a, \mathbf{b}} + v),\label{eq: change of profile function}
\end{align}
and the orthogonality condition
\begin{align}
    \langle v, \phi_{i, a}\rangle_{L^{2}_{a}} = 0,\quad\forall i = 0, 1, \cdots, K.\label{eq: orthogonality cond}
\end{align}

Moreover, $a, \mathbf b$, and $v$ are $C^1$-functions of $a_1, \mathbf b_1$, and $v_1$.

\end{proposition}

\begin{proof}
Given $(a_1, \mathbf{b}_{1})\in\mathbb{R}_{+}\times\mathbb{R}^{K}$ and $v_{1}\in L^{2}_{a_1}$, 
define $\Phi = (\Phi_{0}, \cdots, \Phi_{K}):\mathbb{R}_{+}\times\mathbb{R}^{K}\to\mathbb{R}^{K+1}$ by
\begin{align}
    \Phi(a, \mathbf{b})
    :=
    \left(
    \langle V(a, \mathbf{b})+v_{1}, \phi_{0, a}\rangle_{L^{2}_{a}},
    \cdots,
    \langle V(a, \mathbf{b})+v_{1}, \phi_{K, a}\rangle_{L^{2}_{a}}
    \right),
\end{align}
where
\begin{align} \label{eqn defn V}
    V(a, \mathbf{b})(s)
    :=
    \int_{0}^{s}(f_{a_1}(\tilde{s}) - f_{a}(\tilde{s}))\:d\tilde{s}
    + \phi_{a_{1}, \mathbf{b}_{1}}(s) - \phi_{a, \mathbf{b}}(s).
\end{align}
Then $\Phi$ is a smooth function since $f_{a}$ and $\phi_{i, a}$ are. Note that if $\Phi(a, \mathbf{b}) = 0$, then
\begin{align}
    v := V(a, \mathbf{b}) + v_1
\end{align}
is an odd function which satisfies (\ref{eq: change of profile function}) and (\ref{eq: orthogonality cond}). Thus, our goal is to solve $\Phi(a, \mathbf{b}) = 0$ near $(a_1, \mathbf{b}_{1})$ provided $\|v_1\|_{L^{2}_{a_1}}$ is sufficiently small.

We first estimate $|\Phi(a_1, \mathbf{b}_1)|$. Since $V(a_1, \mathbf{b}_{1}) = 0$, \cite{SS26I}*{Lemma~\ref{I-lem L^2_a est for global eigenfunctions}} implies
\begin{align*}
    |\Phi(a_1, \mathbf{b}_1)|^{2} = \sum_{i=0}^{K}\langle v_1, \phi_{i, a_1}\rangle_{L^{2}_{a_1}}^{2}\leq\left(\sum_{i=0}^{K}\|\phi_{i, a_1}\|^{2}_{L^{2}_{a_1}}\right)\|v_1\|^{2}_{L^{2}_{a_1}}\leq C(n, K)\|v_1\|^{2}_{L^{2}_{a_{1}}}.
\end{align*}

Next we compute the matrix $D_{(a, \mathbf{b})}\Phi\big|_{(a_1, \mathbf{b}_1)}$.
The derivatives of $V(a, \mathbf{b})$ at $(a_1, \mathbf{b}_1)$ are given by
    \begin{align*}
        &\partial_{a}\big|_{(a_1, \mathbf{b}_1)}V = -2a_{1}^{-1}\beta_{a_1} - \sum_{i=1}^{K}(b_{1})_{i}\partial_{a}\big|_{a=a_{1}}\left(\phi_{i, a} - \phi_{0, a}\right),\\
        &\partial_{b_{i}}\big|_{(a_1, \mathbf{b}_1)}V = -(\phi_{i, a_{1}} - \phi_{0, a_{1}}),\quad i=1, \cdots, K.
    \end{align*}
    Thus, the $a$-derivative of $\Phi$ is 
    \begin{align*}
        \partial_{a}\big|_{(a_1, \mathbf{b}_1)}\Phi_{i} &= \langle\partial_{a}\big|_{(a_1, \mathbf{b}_1)}V, \phi_{i, a_1}\rangle_{L^{2}_{a_1}} + \langle v_1, \partial_{a}\big|_{(a_1, \mathbf{b}_1)}\phi_{i, a}\rangle_{L^{2}_{a_1}} + \big\langle v_1, \phi_{i, a_1}\tfrac{\partial_{a}|_{a_1}d\mu_a}{d\mu_{a_1}} \big\rangle_{L^{2}_{a_1}}.
    \end{align*}
    By \cite{SS26I}*{Lemmas ~\ref{I-lem partial_a phi_a ests}, ~\ref{I-lem L^2_a est for phi with beta}} and the assumption that $|\mathbf{b}_{1}|\leq Ca_{1}^{2}$, the first term can be estimated by
    \begin{align*}
        \langle\partial_{a}\big|_{(a_1, \mathbf{b}_1)}V, \phi_{i, a_1}\rangle_{L^{2}_{a_1}} = -2a_{1}\|\phi_{0, a_1}\|^{2}_{L^{2}_{a_1}}\delta_{0i} + O(a_1^{2}, n, i, s_{0}).
    \end{align*}
    By \cite{SS26I}*{Lemmas \ref{I-lem L^2_a est for partial_a phi}, \ref{I-lem L^2_a est for global eigenfunctions}, \ref{I-lem L^2_a est for global eigenfunctions + weight}}
    and Proposition~\ref{prop computation of partial_a d mu_a}, the latter two terms combined can be estimated by
    \begin{align*}
        \langle v_1, \partial_{a}\big|_{(a_1, \mathbf{b}_1)}\phi_{i, a}\rangle_{L^{2}_{a_1}} + \big\langle v_1, \phi_{i, a_1}\tfrac{\partial_{a}|_{a_1}d\mu_a}{d\mu_{a_1}} \big\rangle_{L^{2}_{a_1}} = O(\|v_1\|_{L^{2}_{a_1}}; n, i, s_{0}).
    \end{align*}
    Therefore,
    \begin{align}
        \partial_{a}\big|_{(a_1, \mathbf{b}_1)}\Phi_{i} = -2a_{1}\|\phi_{0, a_1}\|^{2}_{L^{2}_{a_1}}\delta_{0i} + O(a_1^{2}; n, i, s_{0}) + O(\|v_1\|_{L^{2}_{a_1}}; n, i, s_{0}).
    \end{align}
    The $\mathbf{b}$-derivative of $\Phi$ is given by
    \begin{align*}
        \partial_{b_{j}}\big|_{(a_1, \mathbf{b}_{1})}\Phi_{i} = \langle -(\phi_{j, a_1} - \phi_{0, a_{1}}), \phi_{i, a_{1}} \rangle_{L^{2}_{a_1}} = -\|\phi_{i, a_1}\|^{2}_{L^{2}_{a_1}}(\delta_{ij} - \delta_{0i}).
    \end{align*}

    To obtain uniform invertibility, we perform a change of variables by setting
    \begin{align}
        x = \frac{a-a_{1}}{a_{1}},\quad \mathbf{y} = \frac{\mathbf{b} -\mathbf{b}_{1}}{a_{1}^{2}}.
    \end{align}
    Then, viewing $a$ and $\mathbf{b}$ as functions of $(x, \mathbf{y})$,
    \begin{align*}
        &\partial_{x}\big|_{(0, \mathbf{0})}\left(a_{1}^{-2}\Phi_{i}(a, \mathbf{b})\right) = -2\|\phi_{0, a_1}\|^{2}_{L^{2}_{a_{1}}}\delta_{0i} + O(a_{1}; n,i,s_{0}) + O(a_{1}^{-1}; n,i,s_{0})\cdot\|v_{1}\|_{L^{2}_{a_1}},\\
        &\partial_{y_{j}}\big|_{(0, \mathbf{0})}\left(a_{1}^{-2}\Phi_{i}(a, \mathbf{b})\right) = -\|\phi_{i, a_1}\|^{2}_{L^{2}_{a_1}}(\delta_{ij} - \delta_{0i}).
    \end{align*}
    Hence, if $|\mathbf{b}_{1}|\leq C_{1}a_{1}^{2}$ and $\|v_{1}\|_{L^{2}_{a_1}}\leq C_{1}a_{1}^{2+\eta}$ for some $C_{1}, \eta>0$, then 
    \begin{align*}
        \left|a_{1}^{-2}\Phi(a, \mathbf{b})\big|_{(0, \mathbf{0})}\right| = \left|a_{1}^{-2}\Phi(a_1, \mathbf{b}_1)\right| \leq C(n, K)a_{1}^{-2}\|v_1\|_{L^{2}_{a_1}}\leq C(n, K, C_1)a_{1}^{\eta}, 
    \end{align*}
    and also $D_{(x, \mathbf{y})}\left(a_{1}^{-2}\Phi(a, \mathbf{b})\right)\big|_{(0, 0)}$ is invertible (with the inverse bounded uniformly in $a_{1}$). The Inverse Function Theorem now implies that there exists $a_{0}>0$ such that if $a_{1}\in(0, a_{0})$, then there is a unique solution $(x, \mathbf{y})$ to $a_{1}^{-2}\Phi(a, \mathbf{b})\big|_{(x, \mathbf{y})} = \mathbf{0}$ with
    \begin{align*}
        x^{2} + |\mathbf{y}|^{2}\leq C\left|a_{1}^{-2}\Phi(a(0, 0), \mathbf{b}(0, 0))\right|^{2}\leq Ca_{1}^{-4}\|v_{1}\|^{2}_{L^{2}_{a_{1}}},
    \end{align*}
    for some $C = C\left(n, K, C_{1}, \|D(a_{1}^{-2}\Phi(a, \mathbf{b}))^{-1}\big|_{(0, 0)}\|\right)>0$, which is independent of $a_{1}$.
    Therefore, the solution $(a, \mathbf{b}) := (a(x, \mathbf{y}), \mathbf{b}(x, \mathbf{y}))$ satisfies
    \begin{align}\label{eq: closeness of solution to initial parameters}
        |a - a_{1}|\leq Ca_{1}^{-1}\|v_{1}\|_{L^{2}_{a_{1}}}\ll a_{1},\quad  \sum_{i=1}^{K}|b_{i} - (b_{1})_{i}|\leq C\|v_{1}\|_{L^{2}_{a_{1}}}\ll a_{1}^{2}.
    \end{align}
    By the Inverse Function Theorem, $(x, \mathbf y)$ and thus $(a, \mathbf b)$ are $C^1$ functions of $(a_1, \mathbf b_1, v_1)$.

    Finally, we estimate $v = V(a, \mathbf{b})+v_{1}$. Since $\beta_{a} = O(a^{2})$, using (\ref{eq: closeness of solution to initial parameters}) we get
    \begin{align*}
        \left\|\int_{0}^{s}(f_{a_1}(\tilde{s}) - f_{a}(\tilde{s}))\:d\tilde{s}\right\|_{L^{2}_{a}} &= \left\|\int_{a_{1}}^{a}\int_{0}^{s}\partial_{\tilde{a}}(f_{a_1}(\tilde{s}) - f_{\tilde{a}}(\tilde{s}))\:d\tilde{s}d\tilde{a}\right\|_{L^{2}_{a}}\\
        &=\left\|\int_{a_{1}}^{a}2\tilde{a}^{-1}\beta_{\tilde{a}}\:d\tilde{a}\right\|_{L^{2}_{a}}\\
        &\leq C|a^{2} - a_{1}^{2}|\\
        &\leq C\|v_{1}\|_{L^{2}_{a_{1}}}.
    \end{align*}
    We also have, from \cite{SS26I}*{Lemma~\ref{I-lem partial_a phi_a ests}} and (\ref{eq: closeness of solution to initial parameters}),
    \begin{align*}
        \|\phi_{a_1, \mathbf{b}_1} - \phi_{a, \mathbf{b}}\|_{L^{2}_{a}} &\leq\left\|\int_{a_{1}}^{a}\partial_{\tilde{a}}(\phi_{a_1, \mathbf{b}_1} - \phi_{\tilde{a}, \mathbf{b}_{1}})\:d\tilde{a}\right\|_{L^{2}_{a}} + \sum_{i=1}^{K}\left\|\int_{(b_{1})_{i}}^{b_{i}}\partial_{\tilde{b}_{i}}(\phi_{a, \mathbf{b}_1} - \phi_{a, \tilde{\mathbf{b}}})\:d\tilde{b}_{i}\right\|_{L^{2}_{a}}\\
        &\leq C\left[|a - a_{1}||\mathbf{b}_{1}| + \sum_{i=1}^{K}|b_{i} - (b_{1})_{i}|\right]\\
        &\leq C\|v_{1}\|_{L^{2}_{a_{1}}}.
    \end{align*}
    Combining these estimates yield
    \begin{align*}
        \|v\|_{L^{2}_{a}}&\leq\|V(a, \mathbf{b})\|_{L^{2}_{a}} + \|v_{1}\|_{L^{2}_{a}}\\
        &\leq C\|v_{1}\|_{L^{2}_{a_{1}}} + o_{|a_{1}-a|\to 0}(1)\|v_{1}\|_{L^{2}_{a_1}}\\
        &\leq C\|v_{1}\|_{L^{2}_{a_{1}}}.
    \end{align*}
    Additionally, $v = V(a, \mathbf b) + v_1$ is a $C^1$ function of $(a_1, \mathbf b_1, v_1)$ since $(a, \mathbf b)$ is.
    This finishes the proof.
\end{proof}

\section{Modulation Equations and Integral Estimates} \label{Section Mod Eqns and Integral Ests}

We now begin to develop a series of estimates for solutions $u$ of \eqref{eq: main equation} that decompose as 
$$u = \phi + v = \sum_{k=1}^K b_k(\tau) \left( \phi_{k, a(\tau)} - \phi_{0, a(\tau)} \right) + v$$ described in Subsection \ref{Subsect Geom Decomp} above.
These estimates will later be applied in Section \ref{Section Box Argument} to formulate a topological argument that proves the paper's main result (Theorem \ref{main thm paper II intro}).

\subsection{Modulation Equations}

This subsection obtains estimates for the evolution of the scale and modulation parameters $a(\tau), b_1(\tau), \dots, b_K(\tau)$ (Theorem \ref{thm modulation eqns}) as well as the $L^2_a$-norm of the remainder $v$ (Lemma \ref{lem L^2_a est for v}).

\begin{theorem}[Modulation Equations] \label{thm modulation eqns}
    Let $n \ge 3$, $K \in \mathbb N$, $\tau_1 < \tau_2$, and $C_0 > 0$.
    For all $0 < s_0 \le s_0^* (n,K) \ll 1$, there exist $0 < a^{**}(n,K, s_0, C_0) \le a^*(n,K, s_0) \ll 1$ such that the following holds:

    Assume
    \begin{enumerate}
        \item $0 < a = a(\tau) \le a^*(n, K, s_0)$ for all $\tau \in (\tau_1, \tau_2)$,
        \item $u = u(s,\tau)$ satisfies the equation
        \begin{equation} \label{thm modulation eqns, eqn 1}
            \partial_\tau u = \left( 1 - 2 \frac{\partial_\tau a}{a} \right) \beta_a + H_a u + Q_a(u)
            \qquad \forall (s, \tau) \in \R \times (\tau_1 , \tau_2) ,
        \end{equation}
        and
        \item $u$ is of the form
        \begin{gather}
            \label{thm modulation eqns, eqn 2}
            u = \phi + v = \sum_{k=1}^K b_k ( \phi_{k,a} - \phi_{0,a}) + v, 
            \qquad b_k = b_k (\tau) ,\qquad v = v(s,\tau) , \\
            \label{thm modulation eqns, eqn 3}
            \text{with } \langle v, \phi_{k,a} \rangle_{L^2_a} = 0 \quad \forall 0 \le k \le K , \quad \forall \tau \in (\tau_1, \tau_2) .
        \end{gather}
    \end{enumerate}
   Define
    \begin{equation} \label{eqn defn Mod}
        Mod := \sum_{k=1}^K | \partial_\tau b_k - (1 -k) b_k | + \left| \frac d{d\tau} (a^2) - a^2 + \sum_{k=1}^K k b_k \right| .
    \end{equation}
    Then there exists $C = C(n,K, s_0) > 0$ and $C' = C'(n, K) > 0$ such that 
    \begin{multline} \label{eqn Mod est}
        Mod \le C \left( |\partial_\tau a| \| v \|_{L^2_a} + a^3 + |\partial_\tau a | a^2 + \sum_{i=1}^K (|\partial_\tau a| + a) | b_i |   \right) 
        + C' \sum_{i=0}^K |\langle  Q_a(u) , \phi_{i,a} \rangle_{L^2_a} |  
        \\ \forall \tau \in (\tau_1, \tau_2) .
    \end{multline}

    If additionally
    \begin{multline} \label{thm modulation eqns, eqn 4}
        0 < a(\tau) \le a^{**}(n, K, s_0, C_0)  , \quad
        |b_k | \le C_0 a^2, \quad \text{and} \quad \| v \|_{L^2_a} \le a^2 \\ \forall \tau \in (\tau_1, \tau_2) , \, \forall 1 \le k \le K,
    \end{multline}
    then there exists $ C = C(n,K, s_0) > 0$ and $C' = C'(n,K) > 0$ such that 
    \begin{gather} \label{eqn Mod est 2}
        Mod \le C ( 1 + C_0)^2 a^3 + C' \sum_{i=0}^K |\langle  Q_a(u) , \phi_{i,a} \rangle)_{L^2_a} | \qquad \forall \tau \in (\tau_1, \tau_2).
    \end{gather}
\end{theorem}
\begin{proof}
    Throughout the proof, we use $'$ to denote $\partial_\tau$,
    $\langle \cdot , \cdot \rangle = \langle \cdot, \cdot \rangle_{L^2_a}$, and $\| \cdot \| = \| \cdot \|_{L^2_a}$.
    Additionally, unless indicated otherwise, $C = C(n, K, s_0) > 0$ is a positive constant which depends only on $n,K,s_0$ and which may change from line to line.
    Similarly, $O(\cdot ) = O(\cdot; n,K, s_0)$ unless indicated otherwise.

    Hypotheses \eqref{thm modulation eqns, eqn 1} and \eqref{thm modulation eqns, eqn 2} imply
    \begin{equation} \label{proof mod eqns, eqn 1}
        (\partial_\tau - H_a) v = \left( 1 - 2 \frac{\partial_\tau a}{a} \right) \beta_a -(\partial_\tau - H_a) \phi + Q_a(u).
    \end{equation}
    Taking an $L^2_a$ inner product with $\phi_{k,a}$ for some $0 \le k \le K$ then gives
    \begin{equation} \label{proof mod eqns, eqn 2}
        \langle ( \partial_\tau - H_a ) v , \phi_{k,a} \rangle 
        = \left( 1 - 2 \frac{\partial_\tau a}{a} \right) \langle \beta_a , \phi_{k,a} \rangle - \langle (\partial_\tau - H_a) \phi , \phi_{k,a} \rangle + \langle Q_a(u), \phi_{k,a} \rangle.
    \end{equation}
    We proceed to estimate the terms in \eqref{proof mod eqns, eqn 2}.

    (The $v$ Term) 

    Using the facts that $v \perp_{L^2_a} \phi_{k,a}$, $\phi_{k,a}$ is an eigenfunction of $H_a$, and Proposition \ref{prop H_a symmetric}, 
    it follows that 
    \begin{gather} \label{proof mod eqns, eqn 3} \begin{aligned}
        \langle (\partial_\tau - H_a ) v, \phi_{k,a} \rangle 
        ={}& \frac d{d\tau} \left( \langle v , \phi_{k,a} \rangle \right)
        - a' \langle v,  \partial_a \phi_{k,a} \rangle 
        - a' \left\langle v, \phi_{k,a} \cdot \frac{\partial_a d \mu_a}{d \mu_a} \right\rangle
        - \langle v, H_a \phi_{k,a} \rangle \\
        ={}& - a' \langle v,  \partial_a \phi_{k,a} \rangle 
        - a' \left\langle v, \phi_{k,a} \cdot \frac{\partial_a d \mu_a}{d \mu_a} \right\rangle
    \end{aligned}  \end{gather}
    where $ d\mu_a = e^{-F_a} \sqrt{ ( 1 + (\partial_s f_a)^2)(s^2 + f_a^2)^{n-1} } ds$ as in Definition \ref{defn Weigthed Sobolev spaces}.

    Moreover,
    \begin{gather} \label{proof mod eqns, eqn 4} \begin{aligned}
        &\left|  \left\langle  v, \phi_{k,a}  \frac{\partial_a d \mu_a}{d \mu_a} \right\rangle \right| \\
        \le{}& \left|\left\langle |v| , | \phi_{k,a} | \left( C a + \frac{C}a \frac1{1 + (|s|/a)^{n+1}} \right) \right \rangle \right| 
        && (\text{Proposition \ref{prop computation of partial_a d mu_a}} ) \\
        \le{}& C a \| v \| \| \phi_{k,a} \| + \frac{C}a \| v \| \left\| \frac{ \phi_{k,a} }{1 + (|s|/a)^{n+1}} \right\|  \\
        \le{}& C a \| v \| + C a^{\frac n2 -1} \| v \| 
        && (\text{\cite{SS26I}*{Lemmas \ref{I-lem L^2_a est for global eigenfunctions}, \ref{I-lem L^2_a est for global eigenfunctions + weight}}} ) \\
        \le{}& C a^{\frac12} \| v \| && ( n\ge 3) .
    \end{aligned} \end{gather}
    Additionally, by \cite{SS26I}*{Lemma \ref{I-lem L^2_a est for partial_a phi}},
    \begin{equation}  \label{proof mod eqns, eqn 5}
        \left| \langle v, \partial_a \phi_{k,a} \rangle \right| 
        \le \| v \| \| \partial_a \phi_{k,a} \| 
        \le C \| v \| .
    \end{equation} 

    Combining \eqref{proof mod eqns, eqn 3}--\eqref{proof mod eqns, eqn 5} yields
    \begin{equation} \label{proof mod eqns, eqn 6}
        \left| \langle (\partial_\tau - H_a ) v, \phi_{k,a} \rangle \right| 
        \le C | a'| \| v \| .
    \end{equation}

    (The $\beta_a$ Term) 

    Using \cite{SS26I}*{Lemma \ref{I-lem L^2_a est for phi with beta}}, it follows that
    \begin{gather} \label{proof mod eqns, eqn 7} \begin{aligned}
        \left( 1 - 2 \frac{a'}a \right) \langle \beta_a , \phi_{k,a} \rangle 
        ={}& ( a^2 - 2 a' a) \langle a^{-2} \beta_a , \phi_{k,a} \rangle \\
        ={}& ( a^2 - 2a' a) \left[ \delta_{k0} \| \phi_{0,a} \|^2 + O(a) \right] \\
        ={}& (a^2 - 2a' a) \delta_{k0} \| \phi_{0,a} \|^2 + O(a^3) + O(|a'| a^2 ) 
    \end{aligned} \end{gather}
    where $\delta_{k0}$ here and throughout the proof is the Kronecker Delta Function.
    
    (The $\phi$ Term)
    
    By \eqref{thm modulation eqns, eqn 2} and Theorem \ref{thm global eigenfunctions and spectral gap},
    \begin{equation} \label{proof mod eqns, eqn 8}
       ( \partial_\tau - H_a ) \phi
       = \sum_{i=1}^K b_i' ( \phi_{i,a} - \phi_{0,a}) - b_i ( \lambda_{i,a} \phi_{i,a} - \lambda_{0,a} \phi_{0,a} ) + b_i a' ( \partial_a \phi_{i,a} - \partial_a \phi_{0,a} ) ,
    \end{equation}
    where $\lambda_{i,a} = 1 - i + \tilde \lambda_{i}(a)$ (Theorem \ref{thm global eigenfunctions and spectral gap}).
    Note also $\langle \phi_{i,a} , \phi_{j,a} \rangle = 0$ for all $i \ne j$.
    It follows that, when $1 \le k \le K$,  
    \begin{equation} \label{proof mod eqns, eqn 9}
        \langle (\partial_\tau - H_a ) \phi , \phi_{k,a} \rangle 
        = (b_k' - \lambda_{k,a} b_k ) \| \phi_{k,a} \|^2 
        + \sum_{i=1}^K b_i a' \langle \partial_a \phi_{i,a} - \partial_a \phi_{0,a} , \phi_{k,a} \rangle 
    \end{equation}
    and, when $k = 0$,
    \begin{equation} \label{proof mod eqns, eqn 10}
        \langle ( \partial_\tau - H_a ) \phi, \phi_{0,a} \rangle 
        = - \sum_{i=1}^K (b_i' - b_i \lambda_{0,a} ) \| \phi_{0,a} \|^2 
        + \sum_{i=1}^K b_i a' \langle \partial_a \phi_{i,a} - \partial_a \phi_{0,a} , \phi_{0,a} \rangle .
    \end{equation}  

    Note that, for any $0 \le k \le K$,
    \begin{equation} \label{proof mod eqns, eqn 11}
        \left| \sum_{i=1}^K b_i a' \langle \partial_a \phi_{i,a} - \partial_a \phi_{0,a} , \phi_{k,a} \rangle \right| 
        \le \sum_{i=1}^K |b_i | |a'| \| \partial_a \phi_{i,a} - \partial_a \phi_{k,a} \| \| \phi_{k,a} \| 
        \le C \sum_{i=1}^K |b_i | |a' | ,
    \end{equation} 
    where the last inequality uses \cite{SS26I}*{Lemmas \ref{I-lem L^2_a est for partial_a phi}, \ref{I-lem L^2_a est for global eigenfunctions}}.

    Combining \eqref{proof mod eqns, eqn 9}--\eqref{proof mod eqns, eqn 11} yields
    \begin{align}
        \label{proof mod eqns, eqn 12}
        \langle (\partial_\tau - H_a ) \phi, \phi_{0,a} \rangle
        ={}& - \sum_{i=1}^K( b_i ' - \lambda_{0,a} b_i ) \| \phi_{0,a} \|^2 + O \left( \sum_{i=1}^K |b_i| |a'| \right) && \text{and} \\
        \label{proof mod eqns, eqn 13}
        \langle (\partial_\tau - H_a ) \phi, \phi_{k,a} \rangle
        ={}& ( b_k ' - \lambda_{k,a} b_k ) \| \phi_{k,a} \|^2 + O \left( \sum_{i=1}^K |b_i| |a'|  \right)  &&  \text{for } 1 \le k \le K.
    \end{align}

    (The $Q_a$ Term)
    
    The $Q_a(u)$ term we simply estimate as $|\langle Q_a(u), \phi_{k,a} \rangle | $.

    Having estimated each term in \eqref{proof mod eqns, eqn 2}, we now proceed to combine the estimates to obtain the claimed inequality for $Mod$ \eqref{eqn Mod est}.
    Inserting estimates \eqref{proof mod eqns, eqn 6}, \eqref{proof mod eqns, eqn 7}, and \eqref{proof mod eqns, eqn 12}
    into \eqref{proof mod eqns, eqn 2} yields that for $1 \le k \le K$
    \begin{multline} \label{proof mod eqns, eqn 15}
        O(|a'| \| v\| ) = O(a^3 ) + O(|a'| a^2 ) - ( b_k' - \lambda_{k,a} b_k ) \| \phi_{k,a} \|^2 + O \left(\sum_{i=1}^K |b_i| |a'| \right) \\ + |\langle Q_a(u) , \phi_{k,a} \rangle |.
    \end{multline}
    Using Theorem \ref{thm global eigenfunctions and spectral gap} and \cite{SS26I}*{Lemma \ref{I-lem L^2_a est for global eigenfunctions}} to estimate 
    \begin{equation} \label{proof mod eqns, eqn 15.1}
        \lambda_{k,a} b_k \| \phi_{k,a} \|^2 =(1-k) b_k \| \phi_{k,a} \|^2 +  O \left( a |b_k| \right),
    \end{equation}
    we deduce from \eqref{proof mod eqns, eqn 15} and \eqref{proof mod eqns, eqn 15.1} that 
    \begin{multline} \label{proof mod eqns, eqn 15.2}
        O(|a'| \| v\| ) = O(a^3 ) + O(|a'| a^2 ) - ( b_k' - (1-k) b_k ) \| \phi_{k,a} \|^2 + O(|b_k| a) \\
        +  O \left(\sum_{i=1}^K |b_i| |a'| \right)  + |\langle Q_a(u), \phi_{k,a} \rangle |.
    \end{multline}
    Rearranging terms and using \cite{SS26I}*{Lemma \ref{I-lem L^2_a est for global eigenfunctions}}, it follows that 
    \begin{multline} \label{proof mod eqns, eqn 16}
        | b_k' - (1-k) b_k | 
        \le C \left( |a'| \| v \| + a^3 + |a' | a^2 + |b_k| a +  \sum_{i=1}^K | b_i | |a'| \right) 
        + C(n, K) |\langle Q_a(u), \phi_{k,a} \rangle |  \\
         (\forall 1 \le k \le K).
    \end{multline}
    
    A similar argument in the case of $k = 0$ yields that
    \begin{multline} \label{proof mod eqns, eqn 17}
        \left| a^2 - 2a' a + \sum_{i=1}^K (b_i' - \lambda_{0,a} b_i )  \right| 
        \le C \left( |a'| \| v \| + a^3 + |a' | a^2 + \sum_{i=1}^K | b_i | |a'| \right)  \\
        + C(n, K) |\langle Q_a(u) , \phi_{0,a} \rangle | .
    \end{multline}
    It follows that 
    \begin{gather} \label{proof mod eqns, eqn 18} \begin{aligned}
        &\left| \frac d{d\tau} a^2 - a^2 + \sum_{i=1}^K i b_i \right| \\
        \le{}& \left| 2 a' a - a^2 - \sum_{i=1}^K (b_i' - \lambda_{0,a} b_i )\right| +  \sum_{i=1}^K | i b_i + b_i' - \lambda_{0,a} b_i |  \\ 
        \le{}& \left| 2 a' a - a^2 - \sum_{i=1}^K (b_i' - \lambda_{0,a} b_i )\right|  
        + \sum_{i=1}^K | b_i'- \lambda_{i,a} b_i | + \sum_{i=1}^K | (\tilde \lambda_{i,a} - \tilde \lambda_{0,a} ) b_i | \\
        & ( \text{since } \lambda_{i,a} = 1 - i + \tilde \lambda_{i,a} ) \\
        \le{}& C \left( |a'| \| v \| + a^3 + |a' | a^2 + \sum_{i=1}^K | b_i | |a'| \right) 
        + C(n, K) \sum_{i=0}^K |\langle  Q_a(u) , \phi_{i,a} \rangle |  + C \sum_{i=1}^K |b_i | a
    \end{aligned} \end{gather}
    where the last inequality follows from \eqref{proof mod eqns, eqn 16}, \eqref{proof mod eqns, eqn 17}, and Theorem \ref{thm global eigenfunctions and spectral gap}.
    Summing \eqref{proof mod eqns, eqn 16} and \eqref{proof mod eqns, eqn 18} completes the proof of \eqref{eqn Mod est}.

    Now, assume additionally that
    \begin{equation} \label{proof mod eqns, eqn 19}
        | b_i | \le C_0 a^2  \quad \text{and} \quad \| v \|_{L^2_a} \le a^2 \qquad \forall \tau \in (\tau_1, \tau_2) .
    \end{equation}
    Then \eqref{eqn Mod est} and \eqref{proof mod eqns, eqn 19} imply
    \begin{gather} \label{proof mod eqns, eqn 20} \begin{aligned}
        & \sum_{k=1}^K | b_k' - (1-k) b_k | + \left| \frac d{d\tau} a^2 - a^2 + \sum_{i=1}^K i b_i \right| \\
        \le{}&  C \left( |a'| \| v \| + a^3 + |a' | a^2 + \sum_{i=1}^K | b_i | ( |a'| + a) \right) 
        + C(n, K) \sum_{i=0}^K |\langle  Q_a(u) , \phi_{i,a} \rangle |  \\
        \le{}& C(1+ C_0)  \left( |a'| a^2 + a^3   \right) + C(n,K) \sum_{i=0}^K |\langle  Q_a(u) , \phi_{i,a} \rangle |  \\
        ={}& C ( 1+ C_0) \frac a2 \left| 2 a' a -a^2 + \sum_{k=1}^K k b_k + a^2 - \sum_{k=1}^K k b_k \right| \\
        &+C ( 1+ C_0) a^3 + C(n,K)\sum_{i=0}^K |\langle  Q_a(u) , \phi_{i,a} \rangle | \\
        \le{}& C ( 1+ C_0) \frac a2 \left| 2 a' a -a^2 + \sum_{k=1}^K k b_k \right|  
        +C ( 1+ C_0 + C_0^2) a^3 + C(n,K)\sum_{i=0}^K |\langle  Q_a(u) , \phi_{i,a} \rangle |.
    \end{aligned} \end{gather}
    If $0 < a \le a^{**}(n, K,s_0, C_0)  \ll 1$ is sufficiently small so that $C ( 1 + C_0) \frac a2 \le \frac12$,
    then the $ C ( 1+ C_0) \frac a2 \left| 2 a' a -a^2 + \sum_{k=1}^K k b_k \right| $ term on the right-hand side of \eqref{proof mod eqns, eqn 20} can be absorbed into the left-hand side of \eqref{proof mod eqns, eqn 20}.
    In this case, we therefore deduce
    \begin{equation}
        \sum_{k=1}^K | b_k' - (1-k) b_k | + \left| \frac d{d\tau} a^2 - a^2 + \sum_{i=1}^K i b_i \right|
        \le C ( 1+ C_0 )^2 \cdot a^3 + C(n,K) \sum_{i=0}^K |\langle  Q_a(u) , \phi_{i,a} \rangle |
    \end{equation}
    as claimed.
\end{proof}

Next, we obtain an evolution inequality for $\| v \|_{L^2_a}$.
\begin{lem}[Evolution of the $L^2_a$-norm of $v$] \label{lem L^2_a est for v}
    Let $n \ge 3$, $K \in \mathbb N$, $\epsilon > 0$, $\tau_1 < \tau_2$, $C_0 > 0$, and $\gamma \le \frac{n-3}2$.
    For all $0 < s_0 \le s_0^* (n,K) \ll 1$, there exists $0 < a^*(n,K, s_0, C_0, \epsilon) \ll 1$ such that the following holds:

    Assume
    \begin{enumerate}
        \item $0 < a = a(\tau) \le a^*(n, K, s_0, C_0, \epsilon)$ for all $\tau \in (\tau_1, \tau_2)$,
        \item $u = u(s,\tau)$ satisfies the equation
        \begin{equation} \label{lem L^2_a est for v, eqn 1}
            \partial_\tau u = \left( 1 - 2 \frac{\partial_\tau a}{a} \right) \beta_a + H_a u + Q_a(u)
            \qquad \forall (s, \tau) \in \R \times (\tau_1 , \tau_2) ,
        \end{equation}
        \item $u$ is of the form
        \begin{gather}
            \label{lem L^2_a est for v, eqn 2}
            u = \phi + v = \sum_{k=1}^K b_k ( \phi_{k,a} - \phi_{0,a}) + v, 
            \qquad b_k = b_k (\tau) ,\qquad v = v(s,\tau) , \\
            \text{where }  v( \cdot, \tau) : \R \to \R \text{ is an odd $C^2$ function of $s$ with } \\
            \label{lem L^2_a est for v, eqn 3}
            \| v \|_{H^1_a} + \| H_a v \|_{L^2_a}  < \infty \qquad \text{and} \qquad 
            \langle v, \phi_{k,a} \rangle_{L^2_a} = 0 \quad \forall 0 \le k \le K , \quad \forall \tau \in (\tau_1, \tau_2) , 
        \end{gather}
        and
        \item there hold the estimates
        \begin{equation}
            |b_k | \le C_0 a^2 \quad \text{and} \quad \| v \|_{L^2_a} \le a^2 
            \qquad \forall \tau \in (\tau_1, \tau_2), \, 1 \le k \le K.
        \end{equation}
    \end{enumerate}
    Then there exists $C = C(n,K, s_0) > 0$ such that, for all $\tau \in (\tau_1, \tau_2)$,
    \begin{multline}
        \frac12 \frac d{d \tau} \left( \| v \|_{L^2_a(\R)}^2 \right)
        \le (-K + \epsilon ) \| v \|_{L^2_a(\R)}^2  \\
        + C  (1 + C_0)^3 \left(1 + a^{\frac{n-3}2 + \gamma} \sup_{|s|\le s_0} \rho_a^{-\gamma} |v|  \right) \left( a^3 + \| Q_a(u) \|_{L^2_a(\R)} \right) \| v \|_{L^2_a(\R)}.
    \end{multline}
\end{lem}
\begin{proof}
    Throughout the proof, we use $'$ to denote $\partial_\tau$,
    $\langle \cdot , \cdot \rangle = \langle \cdot, \cdot \rangle_{L^2_a}$, and $\| \cdot \| = \| \cdot \|_{L^2_a}$.
    Additionally, unless indicated otherwise, $C = C(n, K, s_0) > 0$ is a positive constant which depends only on $n,K,s_0$ and which may change from line to line.
    Similarly, $O(\cdot ) = O(\cdot; n,K, s_0)$ unless indicated otherwise.

    We first make the following claim.
    \begin{claim} \label{proof L^2_a est for v, claim 1}
        \begin{multline} \label{proof L^2_a est for v, claim 1 eqn}
            \frac12\frac{d}{d\tau}\left(  \| v \|^2 \right) 
            - \langle H_a v, v \rangle
            = \frac{a'}2 \left \langle v, v \frac{\partial_a d \mu_a}{d \mu_a} \right \rangle
            + \left( a^2 - 2 a a' \right)   \langle a^{-2} \beta_a - \phi_{0,a} , v \rangle\\
            - \sum_{i=1}^K b_i a' \langle \partial_a \phi_{i,a} - \partial_a \phi_{0,a} , v \rangle 
            + \langle Q_a (u) , v \rangle 
        \end{multline}
    \end{claim}
    \begin{claimproof}
        Recall \eqref{proof mod eqns, eqn 1} says
        \begin{equation} \tag{\ref{proof mod eqns, eqn 1}}
            (\partial_\tau - H_a) v = \left( 1 - 2 \frac{\partial_\tau a}{a} \right) \beta_a -(\partial_\tau - H_a) \phi + Q_a(u).
        \end{equation}
        Taking the $L^2_a$-inner product with $v$ and using that $\langle v, \phi_{k,a} \rangle = 0$ for all $0 \le k \le K$ implies
        \begin{multline} \label{proof L^2_a est for v, proof of claim 1, eqn 1}
            \langle \partial_\tau v, v \rangle - \langle H_a v , v \rangle 
            = ( a^2 - 2 a a' ) \langle a^{-2 } \beta_a - \phi_{0,a} , v \rangle \\
            - \sum_{i=1}^K b_i a' \langle \partial_a \phi_{i,a} - \partial_a \phi_{0,a} , v \rangle  + \langle Q_a (u) , v \rangle .
        \end{multline}
        Additionally,
        \begin{equation} \label{proof L^2_a est for v, proof of claim 1, eqn 2}
            \langle \partial_\tau v, v \rangle
            = \frac12 \frac{d}{d \tau} ( \| v \|^2 ) - \frac{a'}2 \left\langle v, v \frac{\partial_a d\mu_a}{d \mu_a} \right \rangle .
        \end{equation}
        Inserting \eqref{proof L^2_a est for v, proof of claim 1, eqn 2} into \eqref{proof L^2_a est for v, proof of claim 1, eqn 1} and rearranging terms then proves the claim.
    \end{claimproof}

    We proceed to estimate the terms on the right-hand side of \eqref{proof L^2_a est for v, claim 1 eqn}.

    (The $\frac{\partial_a d \mu_a}{d \mu_a}$ Term) 
    
    This is the most delicate term to estimate.
    First note that Proposition \ref{prop computation of partial_a d mu_a} implies the pointwise estimate
    \begin{equation} \label{proof L^2_a est for v, eqn 1}
        \left| \frac{\partial_a d \mu_a} {d \mu_a} \right| \le C_n a + \frac{C_n}a \frac1{1 + (|s|/a)^{n+1}}
    \end{equation}
    for some dimensional constant $C_n > 0$.
    Hence,
    \begin{gather} \label{proof L^2_a est for v, eqn 2} \begin{aligned}
        &\left| \frac{a'}2 \left\langle v, v \frac{\partial_a d \mu_a}{d \mu_a} \right \rangle \right| \\
        \le{}& C_n a |a'| \| v \|^2 
        + C_n |a'| a^{-1} \int_{\{ |s| \le s_0 \} } \frac{v^2}{1 + (|s|/a)^{n+1}} d \mu_a \\
        &+ C_n |a'| a^{-1} \int_{ \{ |s| > s_0 \} } \frac{v^2}{1 + (|s|/a)^{n+1}} d \mu_a \\
        \le{}& C_n a |a'| \| v \|^2 
        + C_n |a'| a^{-1} \int_{\{ |s| \le s_0\} } \frac{v^2}{1 + (|s|/a)^{n+1}} d \mu_a 
        + C(n,s_0) |a'| a^{n} \| v \|^2,  
    \end{aligned} \end{gather}
    where, to deduce the last estimate, we used the pointwise estimate
    \begin{equation} \label{proof L^2_a est for v, eqn 3}
        \frac1{1+(|s|/a)^{n+1}} \le \frac{a^{n+1}}{s^{n+1}} \le s_0^{-n-1} a^{n+1} \qquad \forall |s| \ge s_0.
    \end{equation}
    The middle term is estimated as follows
    {\allowdisplaybreaks 
    \begin{align} \label{proof L^2_a est for v, eqn 4} 
        &C_n |a'| a^{-1} \int_{|s| \le s_0} \frac{ v^2}{1 + (|s|/a)^{n+1}} d \mu_a \\
        \le{}& C |a'| a^n \int_{|s| \le s_0} \frac{v^2}{\rho_a(s)^{n+1}} d\mu_a \nonumber\\
        \le{}& C |a'| a^n \left( \int_{\{ |s| \le s_0 \}}  v^2 d \mu_a \right)^{1/2} \left(\int_{\{ |s| \le s_0\} } v^2 \rho_{a}^{-2n-2} d \mu_a \right)^{1/2} \nonumber\\
        \le{}& C |a'| a^n \| v \| \left(\int_{\{ |s| \le s_0\} } (\rho_a^{-\gamma} v)^2  \rho_a^{2 \gamma} \rho_{a}^{-2n-2} \rho_a^{n-1} \, ds \right)^{1/2}  \nonumber\\
        &  ( \text{by Proposition \ref{prop vol element bounds}} ) \nonumber\\
        \le{}& C |a'| a^{n}\| v \| \left( \sup_{|s| \le s_0}  \rho_a(s)^{-\gamma} |v(s, \tau)| \right) 
        \left( \int_{\{ |s| \le s_0 \}} \rho_a^{2\gamma -n-3} ds  \right)^{1/2} 
        \nonumber\\
        \le{}& C |a' | a^{n} \| v \|\left( \sup_{|s| \le s_0} \rho_a^{-\gamma} |v| \right)
        a^{\gamma -\frac n2 - \frac32} 
        \nonumber\\
        & \left(\text{since }  \rho_a \ge a \text{ and }  \gamma \le \frac{n-3}2  \right) \nonumber\\
        ={}& C |a'| a^{\frac{n-3}2  + \gamma } \| v \| \left( \sup_{|s| \le s_0}  \rho_a^{-\gamma} |v| \right) .\nonumber
    \end{align} }
    Inserting \eqref{proof L^2_a est for v, eqn 4} into \eqref{proof L^2_a est for v, eqn 2} and using $\| v \| \le a^2$ thereby gives
    \begin{multline}\label{proof L^2_a est for v, eqn 5}
        \left| \frac{a'}2 \left\langle v, v \frac{\partial_a d \mu_a}{d \mu_a} \right \rangle \right|
        \le C |a'| a \| v \|^2 + C |a'| a^{\frac{n-3}2 + \gamma}  \left( \sup_{|s| \le s_0} \rho_a(s)^{-\gamma} |v(s, \tau)| \right) \| v\| \\
        \le C |a'|  a^3 \| v \| + C |a'| a^{\frac{n-3}2 + \gamma }  \left( \sup_{|s| \le s_0} \rho_a(s)^{-\gamma} |v(s, \tau)| \right) \| v \|.
    \end{multline}

    (The $a^{-2} \beta_a - \phi_{0,a}$ Term)

    Using \cite{SS26I}*{Lemma \ref{I-lem L^2_a est for phi with beta}}, it follows that
    \begin{multline} \label{proof L^2_a est for v, eqn 6}
        \left| ( a^2 - 2 a a' )  \langle a^{-2} \beta_a - \phi_{0,a} , v \rangle \right| 
        \le | a^2 - 2 a' a| \| a^{-2} \beta_a - \phi_{0,a} \| \| v \| 
        \le C a | a^2 - 2 a' a | \| v \|  \\
        \le C ( a^3 + |a'| a^2 ) \| v \|.
    \end{multline}        

    (The $\partial_a \phi_{i,a} - \partial_a \phi_{0,a}$ Term)

    Using \cite{SS26I}*{Lemma \ref{I-lem partial_a phi_a ests}} and $|b_k| \le C_0 a^2$ for all $1 \le k \le K$, it follows that
    \begin{multline} \label{proof L^2_a est for v, eqn 7}
        \left| \sum_{i=1}^K b_i a' \langle \partial_a \phi_{i,a} - \partial_a \phi_{0,a} , v \rangle \right|
        \le \sum_{i=1}^K |b_i| |a'| \| \partial_a \phi_{i,a} - \partial_a \phi_{0,a} \| \| v \| 
        \le C \left( \sum_{i=1}^K |b_i| |a'|  \right) \| v\|  \\
        \le C C_0 |a'| a^2 \| v \| .
    \end{multline}

    (The $Q_a(u)$ Term)
    
    Cauchy-Schwartz implies
    \begin{equation} \label{proof L^2_a est for v, eqn 8}
        | \langle Q_a(u) , v \rangle | \le \| Q_a (u) \| \| v \|.
    \end{equation}

    Combining the estimates \eqref{proof L^2_a est for v, eqn 5}--\eqref{proof L^2_a est for v, eqn 8} with \eqref{proof L^2_a est for v, claim 1 eqn} then yields
    \begin{multline} \label{proof L^2_a est for v, eqn 9}
        \left| \frac12 \frac d{d\tau} \left( \| v \|^2 \right) - \langle H_a v, v \rangle \right| \\
        \le C\| v \|  \left\{ 
        (1+C_0) |a'|a^2  + a^3 + \| Q_a \|
        +  |a'| a^{\frac{n-3}2 + \gamma} \left( \sup_{|s| \le s_0}  \rho_a(s)^{-\gamma} |v(s, \tau)| \right)\right\}.
    \end{multline}
    
    Using Theorem \ref{thm modulation eqns}, we can estimate
    \begin{gather} \label{proof L^2_a est for v, eqn 10} \begin{aligned}
        &a^2 |a'|  \\
        ={}& \frac a2 \left| 2 aa' - a^2 + \sum_{k=1}^K kb_k + a^2 - \sum_{k=1}^K k b_k \right| \\
        \le{}& \frac a2 \left| \frac d{d\tau} (a^2) - a^2 + \sum_{k=1}^K k b_k \right| 
        + C (1+ C_0) a^3 
        && ( \text{since }  |b_k | \le C_0 a^2 ) \\
        \le{}& a \cdot Mod (\tau) + C ( 1 + C_0) a^3 \\
        \le{}& C ( 1+ C_0)^2 a^4 + C  a \| Q_a (u) \| + C ( 1 + C_0) a^3 
        && ( \text{by Theorem \ref{thm modulation eqns}} ) \\
        \le{}& C ( 1 + C_0)^2 a^3 + C  \| Q_a(u) \| .
    \end{aligned} \end{gather}
    Inserting \eqref{proof L^2_a est for v, eqn 10} into \eqref{proof L^2_a est for v, eqn 9}, it follows that
    \begin{equation} 
        \left|\frac12 \frac{d}{d \tau} ( \| v \|^2) - \langle H_a v, v\rangle \right| \\
        \le C \| v \| (1 + C_0)^3 \left(1 + a^{\frac{n-3}2 + \gamma} \sup_{|s|\le s_0} \rho_a^{-\gamma} |v|  \right) \left( a^3 + \| Q_a(u) \|\right).
    \end{equation}
    This estimate combined with the spectral gap result Theorem \ref{thm global eigenfunctions and spectral gap} now implies
    \begin{align*}
        &\frac12 \frac d{d \tau} \left( \| v \|^2 \right)  \\
        \le{}& \langle H_a v, v \rangle 
        + C \| v \| (1 + C_0)^3 \left(1 + a^{\frac{n-3}2 + \gamma} \sup_{|s|\le s_0} \rho_a^{-\gamma} |v|  \right) \left( a^3 + \| Q_a(u) \|\right) \\
        \le{}& (-K + \epsilon ) \| v \|^2 + C \| v \| (1 + C_0)^3 \left(1 + a^{\frac{n-3}2 + \gamma} \sup_{|s|\le s_0} \rho_a^{-\gamma} |v|  \right) \left( a^3 + \| Q_a(u) \|\right),
    \end{align*}
    which completes the proof.
\end{proof}

\subsection{Energy Estimates}

This subsection obtains $L^\infty H^k_a \cap L^2 H^{k+1}_a$ estimates for $u$ and the remainder $v$.
The main results of this subsection, namely Corollary \ref{cor local higher order energy ests for the potential+} and Lemma \ref{lem local higher order energy ests for v}, will follow from related energy estimates we now describe for more general functions $u$.

\begin{lem}[$L^\infty L^2_a \cap L^2 H^1_a$ Estimates] \label{lem energy ests}
    Let $\Omega' \Subset \Omega \subset \R$ be intervals\footnote{Throughout, $\Omega' \Subset \Omega$ means the closure $\overline{\Omega'}$ of $\Omega' \subseteq \R$ is compact and $\overline{\Omega'} \subseteq \Omega$.} and $\tau_0 \le \tau_1 < \tau_2$.
    Let $a = a(\tau)$ be $C^1$ and assume $0 < a \le 1$ and  $\left| \frac {\partial_\tau a}a \right| \le C_0 < \infty$.
    Assume $u, \psi, b : \Omega \times (\tau_0, \tau_2) \to \R$  satisfy
        $$\partial_\tau u = H_a u + b \partial_s u + \psi \quad \text{on } \Omega \times (\tau_0, \tau_2) \qquad 
        \text{and} \qquad  \| b \|_{L^\infty(\Omega \times (\tau_0 , \tau_2) )} \le C_1 < \infty.$$
    \begin{enumerate}
        \item If $\tau_0 < \tau_1$, then there exists $C = C\left(n, C_0, C_1,  \dist(\Omega' , \partial \Omega) , \tau_1 - \tau_0  \right) $ such that
        \begin{equation}
            \| u \|_{L^\infty L^2_a (\Omega' \times (\tau_1, \tau_2)) } + \| u \|_{L^2 H^1_a (\Omega' \times (\tau_1, \tau_2) )}
            \le C \left( \| u \|_{L^2 L^2_a(\Omega \times (\tau_0, \tau_2)) } + \| \psi \|_{L^2 L^2_a(\Omega \times (\tau_0, \tau_2 ))}  \right).
        \end{equation}

        \item If $\tau_0 = \tau_1$, then there exists $C = C \left( n, C_0, C_1, \dist(\Omega', \partial \Omega)  \right)$ such that 
        \begin{multline}
            \| u \|_{L^\infty L^2_a (\Omega' \times (\tau_0, \tau_2)) } + \| u \|_{L^2 H^1_a (\Omega' \times (\tau_0, \tau_2) )} \\
            \le C \left( \| u (\cdot, \tau_0) \|_{L^2_a(\Omega ) } +  \| u \|_{L^2 L^2_a(\Omega \times (\tau_0, \tau_2)) } + \| \psi \|_{L^2 L^2_a(\Omega \times (\tau_0, \tau_2 ))}  \right).
        \end{multline}
    \end{enumerate}
\end{lem}
\begin{proof}
    Let $\eta : \Omega \times [\tau_0, \tau_2] \to [0,1]$ be a smooth bump function such that $\eta \equiv 1$ in $\Omega' \times[ \tau_1, \tau_2]$ and $\supp \eta \subset \Omega \times [\tau_0, \tau_2]$.
    Multiplying $\partial_\tau u = H_a u + b \partial_s u + \psi$ by $\eta^2 u$ and integrating in space $\int_\R d \mu_a$ gives
    \begin{equation} \label{energy est 1 proof eqn 1}
        \int \eta^2 u \partial_\tau u \, d \mu_a = \int \eta^2 u H_a u \, d \mu_a + \int \eta^2 u b \partial_s u \, d \mu_a + \int \eta^2 u \psi \, d \mu_a.
    \end{equation}  
    Using Proposition \ref{prop computation of partial_a d mu_a}, the left-hand side of \eqref{energy est 1 proof eqn 1} can be rewritten as
    \begin{multline} \label{energy est 1 proof eqn 2}
        \int \eta^2 u \partial_\tau u \, d \mu_a 
        = \frac12 \frac d{d \tau} \int \eta^2 u^2 d \mu_a 
        - \int \eta (\partial_\tau \eta) u^2 d \mu_a  \\
        - \frac{a'}2 \int \eta^2 u^2 \left[ \frac1a \int_0^s \tilde s^2 (\partial_s f_a) (\partial_{ss}f_a) d \tilde s + \frac{n-1}a \frac{ (s \partial_s f_a - f_a )^2 }{s^2 + f_a^2} \right] d \mu_a.
    \end{multline}
    Using Proposition \ref{prop H_a symmetric}, the right-hand side of \eqref{energy est 1 proof eqn 1} can be rewritten as
    \begin{multline} \label{energy est 1 proof eqn 3}
        \int \eta^2 u H_a u \, d \mu_a + \int \eta^2 u \psi \, d \mu_a
        = - \int \eta^2 \frac{(\partial_s u)^2 }{1 + (\partial_s f_a)^2} \, d\mu_a 
        - \int \frac{2 \eta u (\partial_s \eta) (\partial_s u ) }{1 + (\partial_s f_a)^2 } d \mu_a \\
        + \int \eta^2 u^2 \, d \mu_a + \int \eta^2 u b \partial_s u \, d \mu_a + \int \eta^2 u \psi \, d \mu_a .
    \end{multline}
    Combining equations \eqref{energy est 1 proof eqn 1}--\eqref{energy est 1 proof eqn 3} and integrating in time $\int_{\tau_0}^{\tilde \tau} d \tau$ for $\tilde \tau \in [\tau_1, \tau_2]$ implies
    \begin{gather} \label{energy est 1 proof eqn 4}
    \begin{aligned} 
        &\left. \frac12 \int\eta^2 u^2 \, d\mu_a \right|_{\tau_0}^{\tilde \tau} 
        + \iint \eta^2 \frac{(\partial_s u)^2 }{1 + (\partial_s f_a)^2 } d \mu_a d \tau \\
        ={}& - \iint \frac{2 \eta u (\partial_s \eta) (\partial_s u ) }{1 + (\partial_s f_a)^2 } d \mu_a d\tau
        + \iint \eta^2 u^2 \, d \mu_a d\tau 
        + \iint \eta^2 u b \partial_s u \, d \mu_a d\tau \\
        &+ \iint \eta^2 u \psi \, d \mu_a d\tau
        + \iint \eta (\partial_\tau \eta) u^2 d \mu_a  d \tau\\
        &+  \iint \frac{a'}2 \eta^2 u^2 \left[ \frac1a \int_0^s \tilde s^2 (\partial_s f_a) (\partial_{ss}f_a) d \tilde s + \frac{n-1}a \frac{ (s \partial_s f_a - f_a )^2 }{s^2 + f_a^2} \right] d \mu_a d\tau
    \end{aligned}
    \end{gather}
    where all the $d \tau$ integrals are over $[\tau_0, \tilde \tau] \subset [\tau_0, \tau_2]$.

    We estimate each term on the right-hand side of \eqref{energy est 1 proof eqn 4}.
    \begin{align*}
        &\left| - \iint \frac{2 \eta u (\partial_s \eta) (\partial_s u ) }{1 + (\partial_s f_a)^2 } d \mu_a d\tau \right| \\
        \le{}& \frac14 \iint \eta^2 \frac{(\partial_s u)^2 }{1 + (\partial_s f_a)^2} d \mu_a d \tau + 4 \iint \frac{(\partial_s \eta)^2 u^2 }{1 + (\partial_s f_a)^2 } d \mu_a d\tau 
        && ( \text{Young's inequality})\\
        \le{}& \frac14 \int_{\tau_0}^{\tilde \tau} \int  \eta^2 \frac{(\partial_s u)^2 }{1 + (\partial_s f_a)^2} d \mu_a d \tau + 4 \| \partial_s \eta \|_{L^\infty(\Omega \times (\tau_0, \tau_2) )}^2 \| u \|_{L^2 L^2_a (\Omega \times (\tau_0, \tau_2) )}^2
        && \left( \frac1{1 + (\partial_s f_a)^2} \le 1 \right) .
    \end{align*}

    \begin{equation*}
        \left| \iint \eta^2 u^2 d\mu_a d\tau \right| \le \| u \|_{L^2 L^2_a(\Omega \times (\tau_0, \tau_2) ) }^2
    \end{equation*}

    For $\epsilon > 0$ to be determined, Young's inequality gives
    \begin{align*}
        &\left| \iint \eta^2 u b \partial_s u \, d\mu_a d \tau \right| \\
        \le{}& \frac1{2\epsilon} \iint \eta^2 u^2 d \mu_a d \tau + \frac\epsilon 2 \iint \eta^2 b^2 (\partial_s u)^2 d \mu_a d \tau \\
        \le{}& \frac1{2 \epsilon} \| u \|^2_{L^2 L^2_a(\Omega \times (\tau_0, \tau_2))} + \frac\epsilon 2 \| b^2 (1 + (\partial_s f_a)^2 ) \|_{L^\infty(\Omega \times (\tau_0, \tau_2))} \iint \eta^2 \frac{(\partial_s u)^2 }{1 + (\partial_s f_a)^2} d \mu_a d\tau \\
        \le{}& \frac1{2\epsilon} \| u \|^2_{L^2 L^2_a(\Omega \times (\tau_0, \tau_2))} + \frac\epsilon 2 C_1^2 \| 1 + (\partial_s f_a)^2 \|_{L^\infty (\Omega \times (\tau_0, \tau_2))} \iint \eta^2 \frac{(\partial_s u)^2 }{1 + (\partial_s f_a)^2} d \mu_a d\tau.
    \end{align*}
    There exists a dimensional constant $C_n \ge 1$ such that $1 \le 1 + (\partial_s f_a)^2 \le C_n$ (see Lemma \ref{lem: profile function of Lawlor neck} for example).
    Thus, by taking $\epsilon = \frac1{2 C_1^2 C_n}$, this term can further be estimated as
    \begin{equation*}
        \left| \iint \eta^2 u b \partial_s u \, d\mu_a d \tau \right| 
        \le \frac14 \iint \eta^2 \frac{(\partial_s u)^2 }{1 + (\partial_s f_a)^2} d \mu_a d\tau +  C_n C_1^2  \| u \|_{L^2L^2_a (\Omega \times (\tau_0, \tau_2))}^2 .
    \end{equation*}

    \begin{align*}
        \left| \iint \eta^2 u \psi \, d \mu_a d\tau \right|
        &\le \frac12 \iint \eta^2 u^2 d \mu_a d \tau + \frac12 \iint \eta^2 \psi^2 d \mu_a d \tau && (\text{Young's inequality})\\
        &\le \frac12 \| u \|_{L^2 L^2_a(\Omega \times (\tau_0, \tau_2))}^2 + \frac12 \| \psi \|_{L^2 L^2_a (\Omega \times (\tau_0, \tau_2))}^2.
    \end{align*}

    \begin{equation*}
        \left|  \iint \eta (\partial_\tau \eta) u^2 d \mu_a  d \tau \right| 
        \le \| \partial_\tau \eta \|_{L^\infty(\Omega \times (\tau_0, \tau_2))} \| u \|_{L^2 L^2_a (\Omega \times (\tau_0, \tau_2))}^2 
    \end{equation*}

    \begin{align*}
        &\left| \iint \frac{a'}2 \eta^2 u^2 \left[ \frac1a \int_0^s \tilde s^2 (\partial_s f_a) (\partial_{ss}f_a) d \tilde s + \frac{n-1}a \frac{ (s \partial_s f_a - f_a )^2 }{s^2 + f_a^2} \right] d \mu_a d\tau \right| \\
        \le{}& C_n \iint \left(  |aa'| + \left| \frac{a'}{a} \right| \right) \eta^2 u^2 d \mu_a d \tau 
        && (\text{by \eqref{prop computation of partial_a d mu_a est 1} and \eqref{prop computation of partial_a d mu_a est 2}} ) \\
        \le{}& C_n C_0 \| u \|_{L^2 L^2_a( \Omega \times (\tau_0, \tau_2))}^2 
        && ( |a' | \le C_0 a \le C_0 ) 
    \end{align*}
    where $C_n \ge 1$ here and below denotes a dimensional constant that may change from line to line.

    Combining these estimates and absorbing the $\left(\frac14 + \frac14 \right) \iint \eta^2 \frac{ (\partial_s u)^2 }{ 1 + (\partial_s f_a)^2 } d \mu_a d \tau$ terms into the left-hand side of \eqref{energy est 1 proof eqn 4},
    it follows from \eqref{energy est 1 proof eqn 4} that
    \begin{gather} \label{energy est 1 proof eqn 4.5}
    \begin{aligned}
        &\max \left\{ \frac12 \int \eta(\tilde \tau)^2 u(\tilde \tau)^2 d \mu_{a(\tilde \tau)}, \, 
        \frac12 \int_{\tau_0}^{\tilde \tau} \int \eta^2 \frac{(\partial_s u)^2 }{1 + (\partial_s f_a)^2 } d \mu_a d \tau \right\} \\
        \le{}& \frac12 \int \eta(\tilde \tau)^2 u(\tilde \tau)^2 d \mu_{a(\tilde \tau)} 
        + \frac12 \int_{\tau_0}^{\tilde \tau} \int \eta^2 \frac{(\partial_s u)^2 }{1 + (\partial_s f_a)^2 } d \mu_a d \tau \\
        \le{}& \frac12 \int \eta(\tau_0)^2 u(\tau_0)^2 d \mu_{a (\tau_0)} 
        + \frac12 \| \psi \|_{L^2 L^2_a (\Omega \times (\tau_0, \tau_2))}^2\\
        &+C_n \left( 1 + \| \partial_s \eta \|^2_{L^\infty (\Omega \times (\tau_0, \tau_2))} + \| \partial_\tau \eta \|_{L^\infty (\Omega \times (\tau_0, \tau_2))} + C_0 + C_1^2 \right)  \| u \|^2_{L^2 L^2_a(\Omega \times (\tau_0, \tau_2))} .
    \end{aligned}
    \end{gather}
    Taking a supremum over $\tilde \tau \in [\tau_1, \tau_2]$, using that $\eta \equiv 1$ on $\Omega' \times [\tau_1, \tau_2]$,
    and using that $\frac{1}{1 + (\partial_s f_a)^2 } \ge C_n^{-1} > 0$ (see Lemma \ref{lem: profile function of Lawlor neck}),
    it therefore follows that
    \begin{gather} \label{energy est 1 proof eqn 5}
    \begin{aligned}
        &\| u \|_{L^\infty L^2_a (\Omega' \times (\tau_1, \tau_2))}^2 
        + \left\| \partial_s u  \right\|_{L^2 L^2_a(\Omega' \times (\tau_1, \tau_2))}^2  
        + \left\|  u  \right\|_{L^2 L^2_a(\Omega' \times (\tau_1, \tau_2))}^2\\
        \le{}&  C_n  \| \eta(\tau_0) u(\tau_0) \|_{L^2_a(\Omega)}^2 
        +C_n \| \psi \|_{L^2 L^2_a(\Omega \times (\tau_0, \tau_2))}^2  \\
        &+ C_n\left( 1 + \| \partial_s \eta \|_{L^\infty (\Omega \times (\tau_0, \tau_2))}^2 + \| \partial_\tau \eta \|_{L^\infty (\Omega \times (\tau_0, \tau_2))} + C_0 +C_1^2 \right)  \| u \|^2_{L^2 L^2_a(\Omega \times (\tau_0, \tau_2))}.
    \end{aligned}
    \end{gather}

    In the case that $\tau_1 > \tau_0$, the bump function $\eta$ can be chosen so that additionally
        $$\eta \equiv 0 \text{ on } \Omega \times \{ \tau_0 \} \quad \text{and} \quad 
        \| \partial_s \eta\|_{L^\infty(\Omega \times (\tau_0, \tau_2) )} , \| \partial_\tau \eta \|_{L^\infty(\Omega \times (\tau_0, \tau_2))} \le C(n, \dist (\Omega' , \partial \Omega), \tau_1 - \tau_0 ). $$
    Statement (1) of the lemma now follows from inserting these estimates into \eqref{energy est 1 proof eqn 5}.

    In the case that $\tau_1 = \tau_0$, the bump function $\eta$ can be chosen so that additionally $\eta$ is independent of $\tau$ and 
    $$\| \partial_s \eta\|_{L^\infty(\Omega \times (\tau_0, \tau_2) )} \le C(n, \dist(\Omega' , \partial \Omega ) ) .$$
    Statement (2) of the lemma now follows from inserting these estimates into \eqref{energy est 1 proof eqn 5}.
\end{proof}

\begin{lem}[Local Higher Order Energy Estimates] \label{lem local higher order energy ests}
    Let $\Omega', \Omega $ be intervals such that $\Omega' \Subset \Omega \Subset \R \setminus \{ 0 \}$.
    Let $\tau_0 \le \tau_1 < \tau_2$.
    Let $a = a(\tau)$ be $C^1$ and assume that $0 < a \le 1$ and $\left| \frac{\partial_\tau a}a \right| \le C_0$.
    Assume $u : \Omega \times (\tau_0, \tau_2) \to \R$ and $\psi : \Omega \times (\tau_0, \tau_2) \to \R$ are smooth functions that satisfy
        $$\partial_\tau u = H_a u + \psi \qquad \text{on } \Omega \times (\tau_0, \tau_2).$$
    \begin{enumerate}
        \item If $k \in \mathbb{N}$ and $\tau_0 < \tau_1$, then there exists $C = C\left(n, k, C_0, \Omega,  \dist(\Omega' , \partial \Omega) , \tau_1 - \tau_0  \right) $ such that
        \begin{multline} \label{lem local higher order energy ests eqn 1}
            \| u \|_{L^\infty H^k_a (\Omega' \times (\tau_1, \tau_2)) } + \| u \|_{L^2 H^{k+1}_a (\Omega' \times (\tau_1, \tau_2) )} \\
            \le C \left( \| u \|_{L^2 H^k_a(\Omega \times (\tau_0, \tau_2)) } + \| \psi \|_{L^2 H^k_a(\Omega \times (\tau_0, \tau_2 ))}  \right).
        \end{multline}

        \item If $k \in \mathbb{N}$ and $\tau_0 = \tau_1$, then there exists $C = C \left( n, k, C_0, \Omega, \dist(\Omega', \partial \Omega)  \right)$ such that 
        \begin{multline} \label{lem local higher order energy ests eqn 2}
            \| u \|_{L^\infty H^k_a (\Omega' \times (\tau_0, \tau_2)) } + \| u \|_{L^2 H^{k+1}_a (\Omega' \times (\tau_0, \tau_2) )} \\
            \le C \left( \| u (\tau_0, \cdot) \|_{H^k_a(\Omega ) } +  \| u \|_{L^2 H^k_a(\Omega \times (\tau_0, \tau_2)) } + \| \psi \|_{L^2 H^k_a(\Omega \times (\tau_0, \tau_2 ))}  \right).
        \end{multline}
    \end{enumerate}    
\end{lem}

\begin{proof}
    We first show the following claim.
    \begin{claim} \label{claim evol eqn of u derivatives}
        For any $k \in \mathbb{N}$, $\partial_s^k u$ satisfies an equation of the form
        \begin{equation} \label{local higher order energy ests proof eqn 1}
            \partial_\tau \partial_s^k u = H_a \partial_s^k u + \partial_s^k \psi + \sum_{j=0}^{k+1} B_j^k(a(\tau), s) \cdot \partial_s^j u \qquad \text{on } \Omega \times (\tau_0, \tau_2)
        \end{equation}
        where $B_j^k = B_j^k(a,s)$ are functions of $(a, s) \in (0,1] \times \Omega$ such that
        \begin{equation} \label{local higher order energy ests proof eqn 2}
            \| B_j^k \|_{L^\infty((0,1] \times \Omega ) } \le C (n, k, \Omega) < \infty .
        \end{equation}
    \end{claim}
    \begin{claimproof}
        To simplify the notation in this proof, we use $'$ to denotes $s$-derivatives $\partial_s$.
        By the definition of $H_a$ \eqref{eqn H_a defn}, we can write
        \begin{align}
            H_a u &= A(a,s) u'' + B (a,s) u' + u \\
            \text{where } A(a,s) &:= \frac1{1 + f_a'^2} \\ \text{and }  
            B(a,s) &:= \frac1{\sqrt{(1+f_a'^2)(s^2 + f_a^2)^{n-1}}} \frac \partial{\partial s} \left( \sqrt{ \frac{ (s^2 + f_a^2)^{n-1} }{1 +f_a'^2} } \right) - \frac s2 .
        \end{align}
        Applying $\partial_s^k$ to both sides of $\partial_\tau u = H_a u + \psi$, it follows that
        \begin{equation*}
            \partial_\tau \partial_s^k u 
            = H_a \partial_s^k u + \partial_s^k \psi 
            + \sum_{j=0}^{k-1} C_{k,j} \left( \partial_s^{k-j} A \cdot \partial_s^{j+2} u + \partial_s^{k-j} B \cdot \partial_s^{j+1} u \right)
        \end{equation*}
        for some constants $C_{k,j} \in \R$.
        Because $f_a(s)$ is a smooth function of $(a,s) \in (0, \infty) \times \R$,
        $f_a$ converges to $\overline c_0 |s|$ in $C^\infty_{loc}(\R \setminus \{ 0 \} )$ as $a \searrow 0$ (Lemma \ref{lem: profile function of Lawlor neck}),
        and $\Omega \Subset \R \setminus \{ 0 \}$, 
        there exists $C = C(n, k, \Omega )$ such that 
            $$\| \partial_s^{k-j} A \|_{L^\infty((0,1] \times \Omega ) } + \| \partial_s^{k-j} B \|_{L^\infty((0,1] \times \Omega ) } \le C<\infty \qquad \forall 0 \le j \le k-1.$$
        The statement of the claim now follows.
    \end{claimproof} 
    By \eqref{local higher order energy ests proof eqn 1}, 
    \begin{equation} \label{local higher order energy ests proof eqn 3}
    \partial_\tau \partial_s^k u = H_a \partial_s^k u + B^k_{k+1} \cdot \partial_s^{k+1} u + \tilde \psi \text{ on } \Omega \times (\tau_0, \tau_2), \quad \text{where }   \tilde \psi := \partial_s^k \psi + \sum_{j=0}^k B^k_j \cdot \partial_s^j u 
    \end{equation}
    and where here the $B^k_j(a(\tau), s)$ are regarded as functions of $(s,\tau) \in \Omega \times (\tau_0, \tau_2)$.
    By \eqref{local higher order energy ests proof eqn 2},
    \begin{equation}
        \label{local higher order energy ests proof eqn 4}
        \| B^k_{k+1} \|_{L^\infty(\Omega \times (\tau_0, \tau_2) )} \le C(n, k, \Omega) 
    \end{equation}
    and
    \begin{multline} \label{local higher order energy ests proof eqn 5}        
        \| \tilde \psi \|_{L^2 L^2_a ( \Omega \times (\tau_0, \tau_2))} 
        \le \| \partial_s^k \psi \|_{L^2 L^2_a ( \Omega \times (\tau_0, \tau_2))} + \sum_{j = 0}^k \| B_j^k \|_{L^\infty (\Omega \times (\tau_0, \tau_2))} \| \partial_s^j u \|_{L^2 L^2_a (\Omega \times (\tau_0, \tau_2))} \\
        \le \| \psi \|_{L^2 H^k_a (\Omega \times (\tau_0, \tau_2))} + C(n, k, \Omega ) \| u \|_{L^2 H^k_a (\Omega \times (\tau_0, \tau_2))} .
    \end{multline}
    By applying the $L^\infty L^2_a \cap L^2 H^1_a$-estimate (Lemma \ref{lem energy ests}) to $\partial_s^k u$ and \eqref{local higher order energy ests proof eqn 3}, using the estimates \eqref{local higher order energy ests proof eqn 4} and \eqref{local higher order energy ests proof eqn 5},
    and summing over $0 \le k \le k_0$,
    this proves the statement of the lemma (with $k$ replaced by $k_0$).
\end{proof}

\begin{lem} \label{lem local higher order energy ests for the potential}
    Let $\Omega', \Omega$ be intervals such that $\Omega' \Subset \Omega \Subset \R \setminus \{ 0 \}$.
    Let $\tau_0 \le \tau_1 < \tau_2$.
    Let $a = a(\tau)$ be $C^1$ and assume that $0 < a \le 1$ and $\left| \frac{\partial_\tau a}{a} \right| \le C_0$.
    Assume $u : \Omega \times (\tau_0, \tau_2) \to \R$ satisfies
        $$\partial_\tau u = H_a u + \left( 1 - 2 \frac{\partial_\tau a}{a} \right) \beta_a + Q_a(u) \qquad \text{on } \Omega \times (\tau_0, \tau_2).$$

    For any $k \in \mathbb{N}$, there exists $0 < \epsilon = \epsilon(n, k, \Omega) \ll 1$ sufficiently small such that, when 
        $$\sup_{\tau \in (\tau_0, \tau_2)} \| u(\cdot , \tau) \|_{C^{k+3}(\Omega  )} < \epsilon ,$$
    the following estimates hold:
    \begin{enumerate}
        \item If $\tau_0 < \tau_1$, then there exists $C = C(n, k, C_0, \Omega, \dist(\Omega', \partial \Omega) , \tau_1 - \tau_0 )$ such that
        \begin{equation} \label{lem local higher order energy ests for the potential, eqn 1}
            \| u \|_{L^\infty H^k_a (\Omega' \times (\tau_1, \tau_2) )} + \| u \|_{L^2 H^{k+1}_a (\Omega' \times (\tau_1, \tau_2))} 
            \le C \left( \| u \|_{L^2 H^k_a (\Omega \times (\tau_0, \tau_2))} + \| a^2 \|_{L^2((\tau_0, \tau_2))}  \right) .
        \end{equation}
            
        \item If $\tau_0 = \tau_1$, then there exists $C = C(n,k, C_0, \Omega, \dist(\Omega', \partial \Omega) ) $ such that 
        \begin{multline} \label{lem local higher order energy ests for the potential, eqn 2}
            \| u \|_{L^\infty H^k_a (\Omega' \times (\tau_0, \tau_2) )} + \| u \|_{L^2 H^{k+1}_a (\Omega' \times (\tau_0, \tau_2))} \\
            \le C \left( \| u (\cdot, \tau_0) \|_{H^k_a(\Omega)} +  \| u \|_{L^2 H^k_a (\Omega \times (\tau_0, \tau_2))} + \| a^2 \|_{L^2((\tau_0, \tau_2))}  \right) .
        \end{multline}
    \end{enumerate}
\end{lem}
\begin{proof}
    Let $0 \le l \le k$.
    By Claim \ref{claim evol eqn of u derivatives} in the proof of Lemma \ref{lem local higher order energy ests}, it follows that 
    \begin{equation} \label{proof local higher order energy ests for the potential, eqn 1}
        \partial_\tau \partial_s^l u = H_a \partial_s^l u + \sum_{j = 0}^{l+1} B^l_j ( a(\tau) , s) \cdot \partial_s^j u + \left( 1 - 2 \frac{\partial_\tau a}a \right) \partial_s^l \beta_a + \partial_s^l Q_a(u)
    \end{equation}
    where the $B_j^l = B_j^l(a,s)$ are functions of $(a,s) \in (0, 1] \times \Omega$ such that
    \begin{equation} \label{proof local higher order energy ests for the potential, eqn 1.5}
        \| B_j^l \|_{L^\infty((0,1] \times \Omega)} \le C(n,k, \Omega) < \infty.
    \end{equation}
    Rewrite \eqref{proof local higher order energy ests for the potential, eqn 1} as
    \begin{gather} 
        \label{proof local higher order energy ests for the potential, eqn 2}
        \partial_\tau \partial_s^l u 
        = H_a \partial_s^l u + B^l_{l+1}(a(\tau), s) \cdot \partial_s^l u + \partial_s^l Q_a(u) + \tilde \psi_l, \\
        \text{where } \tilde \psi_l := \left( 1 - 2 \frac{\partial_\tau a}a \right) \partial_s^l \beta_a + \sum_{j=0}^l B^l_j \cdot \partial_s^j u.
    \end{gather}

    Let $\eta : \Omega \times [\tau_0, \tau_2] \to [0,1]$ be a smooth bump function such that $\eta \equiv 1$ in $\Omega' \times[ \tau_1, \tau_2]$ and $\supp \eta \subset \Omega \times [\tau_0, \tau_2]$.
    Multiply \eqref{proof local higher order energy ests for the potential, eqn 2} by $\eta^2 \partial_s^l u $, integrate in space $\int d\mu_a$, and then in time $\int_{\tau_0}^{\tilde \tau} d \tau$ for an arbitrary $\tilde \tau \in [\tau_1, \tau_2]$.
    By similar logic as in the beginning of the proof of Lemma \ref{lem energy ests}, we obtain \eqref{energy est 1 proof eqn 4} but with $u$ replaced by $\partial_s^l u$, $\psi$ replaced by $\tilde \psi$, $b$ replaced by $B^l_{l+1}$, and additional terms from the $\partial_s^l Q_a$.
    Namely,
    \begin{gather} \label{proof local higher order energy ests for the potential, eqn 3}
    \begin{aligned} 
        &\left. \frac12 \int\eta^2 (\partial_s^l u)^2 \, d\mu_a \right|_{\tau_0}^{\tilde \tau} 
        + \iint \eta^2 \frac{(\partial_s^{l+1} u)^2 }{1 + (\partial_s f_a)^2 } d \mu_a d \tau \\
        ={}& - \iint \frac{2 \eta (\partial_s^l u) (\partial_s \eta) (\partial_s^{l+1} u ) }{1 + (\partial_s f_a)^2 } d \mu_a d\tau 
        + \iint \eta^2 (\partial_s^l u)^2 \, d \mu_a d\tau \\
        &+ \iint \eta^2 B^l_{l+1} (\partial_s^l u)  (\partial_s^{l+1} u) \, d \mu_a d\tau \\
        &+ \iint \eta^2 (\partial_s^l u) \tilde \psi \, d \mu_a d\tau
        + \iint \eta (\partial_\tau \eta) (\partial_s^l u)^2 d \mu_a  d \tau\\
        &+  \iint \frac{a'}2 \eta^2 (\partial_s^l u)^2 \left[ \frac1a \int_0^s \tilde s^2 (\partial_s f_a) (\partial_{ss}f_a) d \tilde s + \frac{n-1}a \frac{ (s \partial_s f_a - f_a )^2 }{s^2 + f_a^2} \right] d \mu_a d\tau \\
        &+ \iint \eta^2 (\partial_s^l u) (\partial_s^l Q_a(u)) \, d \mu_a d\tau.
    \end{aligned}
    \end{gather}
    where all the $d \tau$ integrals are over $[\tau_0, \tilde \tau] \subset [\tau_0, \tau_2]$.
    The same estimates as in the proof of Lemma \ref{lem energy ests} apply to deduce that (cf. \eqref{energy est 1 proof eqn 4.5})
    \begin{gather} \label{proof local higher order energy ests for the potential, eqn 4}
    \begin{aligned}
        & \frac12 \int \eta(\tilde \tau)^2 (\partial_s^l u(\tilde \tau))^2 d \mu_{a(\tilde \tau)} 
        + \frac12 \int_{\tau_0}^{\tilde \tau} \int \eta^2 \frac{(\partial_s^{l+1} u)^2 }{1 + (\partial_s f_a)^2 } d \mu_a d \tau \\
        \le{}& \frac12 \int \eta(\tau_0)^2 (\partial_s^l u(\tau_0))^2 d \mu_{a (\tau_0)} 
        + \frac12 \| \tilde \psi \|_{L^2 L^2_a (\Omega \times (\tau_0, \tau_2))}^2\\
        & + C_n
        \left( 1 + \| \partial_s \eta \|^2_{L^\infty (\Omega \times (\tau_0, \tau_2))} + \| \partial_\tau \eta \|_{L^\infty (\Omega \times (\tau_0, \tau_2))} + C_0 \right. \\
        & \left. \qquad \qquad \qquad + \| B_{l+1}^l(a(\tau), s) \|_{L^\infty(\Omega \times(\tau_0, \tau_2))}^2 \right) 
        \cdot \| \partial_s^l u \|_{L^2 L^2_a(\Omega \times (\tau_0, \tau_2))}^2 
         \\
        &+ \int_{\tau_0}^{\tilde \tau} \int \eta^2 |\partial_s^l u| |\partial_s^l Q_a(u)| \, d \mu_a d\tau.
    \end{aligned}
    \end{gather}
    
    It remains to estimate the integral involving $Q_a(u)$.
    We first deduce a pointwise estimate for $\partial_s^l Q_a(u)$.
    In what follows, we shall use the notation 
    $$|w|_{C^l}  = |w|_{C^l}(s,\tau) := 
    \left( \sum_{i=0}^l (\partial_s^i w(s,\tau))^2 \right)^{1/2}$$
    for any function $w = w(s,\tau)$.
    \begin{claim} \label{proof local higher order energy ests for the potential, claim 1}
        There exists $C = C(n, l, \Omega)$ such that 
        \begin{equation} \label{proof local higher order energy ests for the potential, eqn 4.5}
            |\partial_s^l Q_a(u) |(s, \tau) 
            \le C(n, l, \Omega) \left( (\partial_s^{l+2} u)^2(s, \tau) + (\partial_s^{l+1} u )^2(s, \tau) + | u |_{C^l}^2(s,\tau) \right)
        \end{equation}
        for all $(s, \tau) \in \Omega \times (\tau_0, \tau_2)$.
    \end{claim}
    \begin{claimproof}
        Consider the case that $1 \le l$.
        In this case, Lemma \ref{lem: error term Ck} implies that, for any $(s, \tau) \in \Omega \times (\tau_0, \tau_2)$, 
    \begin{gather} \label{proof local higher order energy ests for the potential, eqn 5}
    \begin{aligned}
        &|\partial_s^l Q_a(u) |(s, \tau) \\
        \le{}& C(n,l) \left( | u_{ss} |_{C^l}^2(s, \tau) + \left| \frac{u_s}s \right|^2_{C^l}(s, \tau) + | u_{ss} |_{C^l}^{l+1}(s, \tau) + \left| \frac{u_s}s \right|^{l+1}_{C^l}(s, \tau) \right) \\
        &\cdot \left(1 + |\partial_s f_a |_{C^l}^{l}(s, \tau)  + \left| \frac{f_a}{s} \right|_{C^l}^{l}(s, \tau) \right) \\
        \le{}& C(n,l) \left( |u|^2_{C^{l+2}} + | u|_{C^{l+2}}^{l+1} \right) 
        \cdot \left( 1 + | f_a |_{C^{l+1}}^l \right) 
        && (  \Omega \Subset \R \setminus \{ 0 \} ) \\
        \le{}& C(n,l)  |u |_{C^{l+2}}^2 \cdot \left( 1 + | f_a |_{C^{l+1}}^l \right) 
    \end{aligned}    
    \end{gather}
    where the last inequality follows from the fact that $$|u|_{C^{l+2}}(s, \tau) \le \sup_{\tau \in (\tau_0, \tau_2)} \| u(\cdot , \tau) \|_{C^{k+3}(\Omega  )}< \epsilon < 1.$$
    Recall that $f_a(s)$ is a smooth function of $(a, s)$ and $f_a(s) \xrightarrow[a \searrow 0]{C^\infty_{loc}(\R \setminus \{ 0 \} ) } \overline c_0 |s|$ by Lemma \ref{lem: profile function of Lawlor neck}.
    Since $\Omega \Subset \R \setminus \{ 0\}$ and $0 < a \le 1$, it then follows that 
    \begin{equation} \label{proof local higher order energy ests for the potential, eqn 5.5}
        |f_a|_{C^{l+1}}(s, \tau) \le C(n,l,\Omega) < \infty \qquad \text{for all } (s, \tau) \in \Omega \times (\tau_0, \tau_2).
    \end{equation}
    With this estimate, we deduce from \eqref{proof local higher order energy ests for the potential, eqn 5} above that
    \begin{multline} \label{proof local higher order energy ests for the potential, eqn 6}
        |\partial_s^l Q_a(u) |(s, \tau) 
        \le  C(n,l)  |u |_{C^{l+2}}^2(s, \tau) \cdot \left( 1 + | f_a |_{C^{l+1}}^l(s, \tau) \right) \\
        \le C(n, l, \Omega) \left( (\partial_s^{l+2} u)^2(s, \tau) + (\partial_s^{l+1} u )^2(s, \tau) + | u |_{C^l}^2(s,\tau) \right) 
    \end{multline}
    for any $(s, \tau) \in \Omega \times (\tau_0, \tau_2)$.
    This completes the proof of the claim when $l \ge 1$.
    In the case that $l = 0$, \eqref{proof local higher order energy ests for the potential, eqn 4.5} follows from Lemma \ref{lem: error term C0} together with the fact that $\Omega \Subset \R \setminus \{ 0\}$.
    \end{claimproof}
    
    Claim \ref{proof local higher order energy ests for the potential, claim 1} now implies
    \begin{gather} \label{proof local higher order energy ests for the potential, eqn 7}
    \begin{aligned}
        &\int_{\tau_0}^{\tilde \tau} \int \eta^2 | \partial_s^l u| |\partial_s^l Q_a(u)| d \mu_a d \tau \\
        \le{}& C(n, l, \Omega) 
        \left( 
        \int_{\tau_0}^{\tilde \tau} \int \eta^2 | \partial_s^l u| (\partial_s^{l+2}u)^2  d \mu_a d \tau 
        + \int_{\tau_0}^{\tilde \tau} \int \eta^2 | \partial_s^l u| (\partial_s^{l+1}u)^2 d \mu_a d \tau  \right. \\
        &\left. + \int_{\tau_0}^{\tilde \tau} \int \eta^2 | \partial_s^l u| |u|_{C^l}^2 d \mu_a d \tau \right). \\
        :={}& C(n,l, \Omega) \cdot \left( I_1 + I_2 + I_3 \right).
    \end{aligned}    
    \end{gather}
    We estimate each term $I_1, I_2, I_3$.

    ($I_1$ Estimate)
    By \eqref{eqn defn dV_a}--\eqref{eqn defn F_a}, write
    \begin{gather*}
        d \mu_a = e^{-F_a(s)} J_a(s) \, ds \\
        \text{where } F_a(s) = \frac{s^2}4 + \frac12 \int_0^s \tilde s ( \partial_s f_a)^2 d \tilde s \quad 
        \text{and } J_a(s) = \sqrt{ ( 1 + (\partial_s f_a)^2 )( s^2 + f_a^2 )^{n-1} }.
    \end{gather*}
    Then integration by parts in $s$ implies
    \begin{align*}
        |I_1| 
        ={}&  \left| \int_{\tau_0}^{\tilde \tau} \int \eta^2 | \partial_s^l u| (\partial_s^{l+2}u)^2  d \mu_a d \tau \right| \\
        ={}& \left|  - \int_{\tau_0}^{\tilde \tau} \int 
        (\partial_s^{l+1} u) \frac \partial{\partial s} \left( \eta^2 | \partial_s^l u|  (\partial_s^{l+2} u) e^{-F_a} J_a \right) ds d \tau \right|\\
        \le{}&
        \int_{\tau_0}^{\tilde \tau} \int 2 \eta |\partial_s \eta| | \partial_s^{l+1} u |  |\partial_s^l u| |\partial_s^{l+2} u | d \mu_a d \tau
        + \int_{\tau_0}^{\tilde \tau} \int \eta^2 | \partial_s^{l+1} u | \, | \partial_s  |\partial_s^l u|| \,  |\partial_s^{l+2} u | d\mu_a d \tau \\
        &+ \int_{\tau_0}^{\tilde \tau} \int \eta^2 | \partial_s^{l+1} u |   |\partial_s^l u|  |\partial_s^{l+3} u | d\mu_a d \tau
        + \int_{\tau_0}^{\tilde \tau} \int \eta^2 | \partial_s^{l+1} u |   |\partial_s^l u|  |\partial_s^{l+2} u |\left| \frac{\partial_s (e^{-F_a} J_a ) }{e^{-F_a} J_a} \right|  d\mu_a d \tau \\
        :={}& I_{1a} + I_{1b} + I_{1c} + I_{1d}.
    \end{align*}
    It follows that
    \begin{align*}
        I_{1a} 
        ={}& \int_{\tau_0}^{\tilde \tau} \int 2 \eta |\partial_s \eta| | \partial_s^{l+1} u |  |\partial_s^l u| |\partial_s^{l+2} u | d \mu_a d \tau \\
        \le{}&  \epsilon \int_{\tau_0}^{\tilde \tau} \int 2 \eta | \partial_s \eta| |\partial_s^{l+1} u | | \partial_s^l u | d \mu_a d \tau
        && \left( \sup_{\tau} \| u \|_{C^{k+3}(\Omega  )} \le \epsilon \right) \\
        \le{}&  \epsilon \int_{\tau_0}^{\tilde \tau} \int \eta^2 (\partial_s^{l+1} u)^2 d \mu_a d \tau +  \epsilon \int_{\tau_0}^{\tilde \tau} \int  (\partial_s \eta)^2 (\partial_s^l u )^2 d \mu_a d \tau
        && ( \text{Young's inequality} ) \\
        \le{}& C(n)\epsilon  \int_{\tau_0}^{\tilde \tau} \int \eta^2 \frac{(\partial_s^{l+1} u)^2}{1 + (\partial_s f_a)^2} d \mu_a d \tau \\
        &+  \| \partial_s \eta\|_{L^\infty(\Omega \times (\tau_0, \tau_2))}^2 \| u \|_{L^2 H^l_a (\Omega \times (\tau_0, \tau_2))}^2
    \end{align*}
    where the last line uses that fact that $\epsilon < 1$ and $| 1 + (\partial_s f_a)^2 | \le C(n)$ (see Lemma \ref{lem: profile function of Lawlor neck}).
    
    Next,
    \begin{align*}
        I_{1b} 
        ={}& \int_{\tau_0}^{\tilde \tau} \int \eta^2 | \partial_s^{l+1} u | \, | \partial_s  |\partial_s^l u|| \,  |\partial_s^{l+2} u | d\mu_a d \tau \\
        \le{}& \int_{\tau_0}^{\tilde \tau} \int \eta^2 | \partial_s^{l+1} u | |\partial_s^{l+1} u | | \partial_s^{l+2} u | d \mu_a d \tau
        && (\text{Kato's inequality}) \\
        \le{}& C(n) \epsilon \int_{\tau_0}^{\tilde \tau} \int \eta^2 \frac{(\partial_s^{l+1} u)^2}{1 + (\partial_s f_a)^2} d \mu_a d \tau 
    \end{align*}    
    where the last line uses the fact that $\sup_{\tau \in (\tau_0, \tau_2)} \| u(\cdot , \tau) \|_{C^{k+3}(\Omega  )} \le \epsilon$ and that $1 + (\partial_s f_a)^2 \le C(n)$.

    Additionally,
    \begin{align*}
        I_{1c}
        ={}&  \int_{\tau_0}^{\tilde \tau} \int \eta^2 | \partial_s^{l+1} u |   |\partial_s^l u|  |\partial_s^{l+3} u | d\mu_a d \tau \\
        \le{}& \epsilon \int_{\tau_0}^{\tilde \tau} \int \eta^2 |\partial_s^{l+1} u| |\partial_s^l u | d \mu_a d \tau
        && \left( \sup_{\tau } \| u \|_{C^{k+3}(\Omega  )} \le \epsilon \right) \\
        \le{}&  \epsilon  \int_{\tau_0}^{\tilde \tau} \int \eta^2 (\partial_s^{l+1} u)^2 d \mu_a d \tau
        + \epsilon \int_{\tau_0}^{\tilde \tau} \int \eta^2 (\partial_s^l u)^2 d \mu_a d \tau
        && ( \text{Young's inequality} )\\
        \le{}& C(n) \epsilon  \int_{\tau_0}^{\tilde \tau} \int \eta^2 \frac{(\partial_s^{l+1} u)^2}{1 + (\partial_s f_a)^2} d \mu_a d \tau
        + \| u \|_{L^2 H^l_a (\Omega \times (\tau_0, \tau_2))}^2 
    \end{align*}
    where the last line uses the fact that $\epsilon < 1$ and $ 1+ (\partial_sf_a)^2 \le C(n)$.

    To estimate $I_{1d}$, first observe that
        $$J_a(s) = \sqrt{ (1 + (\partial_s f_a)^2) (s^2 + f_a^2)^{n-1}} \ge |s|^{n-1}$$
    is uniformly bounded below for all $(a,s) \in (0, 1] \times \Omega$ since $\Omega \Subset \R \setminus \{ 0\}$.
    Moreover, by the proof of Claim \ref{proof local higher order energy ests for the potential, claim 1} (see \eqref{proof local higher order energy ests for the potential, eqn 5.5} in particular),
    $$\| f_a \|_{C^2(\Omega \times (\tau_0, \tau_2))} \le C(n, \Omega) < \infty.$$
    It therefore follows from the definition of $F_a, J_a$ that 
    $$\left| \frac{\partial_s (e^{-F_a} J_a ) }{e^{-F_a} J_a} \right| \le C(n, \Omega) < \infty \qquad \text{for all } (s, \tau) \in \Omega \times (\tau_0, \tau_2).$$
    Hence,
    \begin{align*}
        I_{1d}
        ={}&\int_{\tau_0}^{\tilde \tau} \int \eta^2 | \partial_s^{l+1} u |   |\partial_s^l u|  |\partial_s^{l+2} u |\left| \frac{\partial_s (e^{-F_a} J_a ) }{e^{-F_a} J_a} \right|  d\mu_a d \tau \\
        \le{}& C(n, \Omega) \int_{\tau_0}^{\tilde \tau} \int \eta^2 | \partial_s^{l+1} u |   |\partial_s^l u|  |\partial_s^{l+2} u |  d\mu_a d \tau \\
        \le{}& C(n, \Omega ) \epsilon \int_{\tau_0}^{\tilde \tau} \int \eta^2 | \partial_s^{l+1} u | |\partial_s^l u| d \mu_a d \tau 
        && \left( \sup_{\tau } \| u \|_{C^{k+3}(\Omega  )} < \epsilon \right) \\
        \le{}& C(n, \Omega) \epsilon \left( \int_{\tau_0}^{\tilde \tau} \int \eta^2 \frac{(\partial_s^{l+1} u)^2}{1 + (\partial_s f_a)^2} d \mu_a d \tau 
        +  \int_{\tau_0}^{\tilde \tau} \int \eta^2 | \partial_s^l u|^2 d \mu_a d \tau \right)\\
        \le{}&  C(n, \Omega) \epsilon \int_{\tau_0}^{\tilde \tau} \int \eta^2 \frac{(\partial_s^{l+1} u)^2}{1 + (\partial_s f_a)^2} d \mu_a d \tau
        + C(n, \Omega) \| u \|_{L^2 H^l_a(\Omega \times (\tau_0, \tau_2))}^2 
        && ( \epsilon < 1).
    \end{align*}
    
    Combining the estimates for $I_{1a}, I_{1b}, I_{1c},$ and $I_{1d}$ above, it follows that
    \begin{multline}  \label{proof local higher order energy ests for the potential, eqn 8}
        |I_1| 
        \le C(n, \Omega) \epsilon \int_{\tau_0}^{\tilde \tau} \int \eta^2 \frac{(\partial_s^{l+1} u)^2}{1 + (\partial_s f_a)^2} d \mu_a d \tau \\
        + C(n, \Omega)\left( 1 + \| \partial_s \eta \|_{L^\infty(\Omega \times (\tau_0, \tau_2))}^2 \right)  \| u \|_{L^2 H^l_a(\Omega \times (\tau_0, \tau_2))}^2 .
    \end{multline}

    ($I_2$ Estimate)
    Using $\sup_{\tau} \| u \|_{C^{k+3}(\Omega  )}< \epsilon$, it follows that
    \begin{multline} \label{proof local higher order energy ests for the potential, eqn 9}
        |I_2| 
        = \int_{\tau_0}^{\tilde \tau} \int \eta^2 | \partial_s^l u| (\partial_s^{l+1}u)^2 d \mu_a d \tau 
        \le \epsilon \int_{\tau_0}^{\tilde \tau} \int \eta^2 (\partial_s^{l+1} u)^2 d \mu_a d \tau 
        \\
        \le C(n) \epsilon \int_{\tau_0}^{\tilde \tau} \int \eta^2 \frac{ ( \partial_s^{l+1} u)^2}{1 + (\partial_s f_a)^2} d \mu_a d \tau.
    \end{multline}  
    
    ($I_3$ Estimate) Again, using $\sup_{\tau} \| u \|_{C^{k+3}(\Omega  )}< \epsilon$, it follows that
    \begin{equation} \label{proof local higher order energy ests for the potential, eqn 10}
        |I_3| = \int_{\tau_0}^{\tilde \tau} \int \eta^2 | \partial_s^l u| |u|_{C^l}^2 d \mu_a d \tau 
        \le  \int_{\tau_0}^{\tilde \tau} \int \eta^2 |u|_{C^l}^2 d \mu_a d \tau 
        \le   \| u \|^2_{L^2 H^l_a ( \Omega \times (\tau_0, \tau_2))}.
    \end{equation} 
    
    Inserting the estimates \eqref{proof local higher order energy ests for the potential, eqn 8}, \eqref{proof local higher order energy ests for the potential, eqn 9}, \eqref{proof local higher order energy ests for the potential, eqn 10} for $I_1, I_2, I_3$ respectively into \eqref{proof local higher order energy ests for the potential, eqn 7} yields that
    \begin{multline} \label{proof local higher order energy ests for the potential, eqn 10.5}
        \int_{\tau_0}^{\tilde \tau} \int \eta^2 | \partial_s^l u | | \partial_s^l Q_a(u)| d \mu_a d \tau 
        \le C(n, l, \Omega) \epsilon \int_{\tau_0}^{\tilde \tau} \int \eta^2 \frac{ ( \partial_s^{l+1} u)^2}{1 + (\partial_s f_a)^2} d \mu_a d \tau \\
        + C(n, l, \Omega) \left( 1 + \| \partial_s \eta \|_{L^\infty(\Omega \times (\tau_0, \tau_2))}^2 \right)  \| u \|_{L^2 H^l_a(\Omega \times (\tau_0, \tau_2))}^2 .
    \end{multline}
    Applying this estimate to \eqref{proof local higher order energy ests for the potential, eqn 4} then implies that 
    \begin{gather} \label{proof local higher order energy ests for the potential, eqn 11}
    \begin{aligned}
        & \frac12 \int \eta(\tilde \tau)^2 (\partial_s^l u(\tilde \tau))^2 d \mu_{a(\tilde \tau)} 
        + \frac12 \int_{\tau_0}^{\tilde \tau} \int \eta^2 \frac{(\partial_s^{l+1} u)^2 }{1 + (\partial_s f_a)^2 } d \mu_a d \tau \\
        \le{}& \frac12 \int \eta(\tau_0)^2 (\partial_s^l u(\tau_0))^2 d \mu_{a (\tau_0)} 
        + \frac12 \| \tilde \psi \|_{L^2 L^2_a (\Omega \times (\tau_0, \tau_2))}^2\\
        & + C(n, l, \Omega)
        \left( 1 + \| \partial_s \eta \|^2_{L^\infty (\Omega \times (\tau_0, \tau_2))} + \| \partial_\tau \eta \|_{L^\infty (\Omega \times (\tau_0, \tau_2))} + C_0 \right. \\
        & \left. \qquad \qquad \qquad + \| B_{l+1}^l(a(\tau), s) \|_{L^\infty(\Omega \times(\tau_0, \tau_2))}^2 \right) 
        \cdot \|  u \|_{L^2 H^l_a(\Omega \times (\tau_0, \tau_2))}^2 
         \\
        &+ C(n, l, \Omega) \epsilon \int_{\tau_0}^{\tilde \tau} \int \eta^2 \frac{ ( \partial_s^{l+1} u)^2}{1 + (\partial_s f_a)^2} d \mu_a d \tau .
    \end{aligned}
    \end{gather}
    If $\epsilon = \epsilon (n, l , \Omega) \ll 1$ is sufficiently small (depending only on $n,l, \Omega$) such that $C(n, l, \Omega) \epsilon < \frac14$,
    then subtracting $C(n, l, \Omega ) \epsilon  \int_{\tau_0}^{\tilde \tau} \int \eta^2 \frac{ ( \partial_s^{l+1} u)^2}{1 + (\partial_s f_a)^2} d \mu_a d \tau $
    from both sides of \eqref{proof local higher order energy ests for the potential, eqn 11} gives that
    \begin{gather} \label{proof local higher order energy ests for the potential, eqn 12}
    \begin{aligned}
        & \frac12 \int \eta(\tilde \tau)^2 (\partial_s^l u(\tilde \tau))^2 d \mu_{a(\tilde \tau)} 
        + \frac14 \int_{\tau_0}^{\tilde \tau} \int \eta^2 \frac{(\partial_s^{l+1} u)^2 }{1 + (\partial_s f_a)^2 } d \mu_a d \tau \\
        \le{}& \frac12 \int \eta(\tau_0)^2 (\partial_s^l u(\tau_0))^2 d \mu_{a (\tau_0)} 
        + \frac12 \| \tilde \psi \|_{L^2 L^2_a (\Omega \times (\tau_0, \tau_2))}^2\\
        & + C(n, l, \Omega)
        \left( 1 + \| \partial_s \eta \|^2_{L^\infty (\Omega \times (\tau_0, \tau_2))} + \| \partial_\tau \eta \|_{L^\infty (\Omega \times (\tau_0, \tau_2))} + C_0 \right. \\
        & \left. \qquad \qquad \qquad + \| B_{l+1}^l(a(\tau), s) \|_{L^\infty(\Omega \times(\tau_0, \tau_2))}^2 \right) 
        \cdot \|  u \|_{L^2 H^l_a(\Omega \times (\tau_0, \tau_2))}^2 \\
        \le{}& \frac12 \int \eta(\tau_0)^2 (\partial_s^l u(\tau_0))^2 d \mu_{a (\tau_0)} 
        + \frac12 \| \tilde \psi \|_{L^2 L^2_a (\Omega \times (\tau_0, \tau_2))}^2\\
        & + C(n, k, \Omega)
        \left( 1 + \| \partial_s \eta \|^2_{L^\infty (\Omega \times (\tau_0, \tau_2))} + \| \partial_\tau \eta \|_{L^\infty (\Omega \times (\tau_0, \tau_2))} + C_0 \right) 
        \cdot \|  u \|_{L^2 H^l_a(\Omega \times (\tau_0, \tau_2))}^2
    \end{aligned}
    \end{gather}
    where the last equality follows from \eqref{proof local higher order energy ests for the potential, eqn 1.5} and $0 \le l \le k$.
    Additionally, 
    \begin{gather}\label{proof local higher order energy ests for the potential, eqn 13}
    \begin{aligned}
        \| \tilde \psi\|_{L^2L^2(\Omega \times (\tau_0, \tau_2))}^2 
        ={}& \left\| \left( 1 - 2 \frac{\partial_\tau a}a \right) \partial_s^l \beta_a + \sum_{j=0}^l B^l_j \cdot \partial_s^j u \right\|_{L^2L^2(\Omega \times (\tau_0, \tau_2))}^2 \\
        \le{}& C(l) \left\| 1 - 2 \frac {\partial_\tau a}a \right\|_{L^\infty(\Omega \times (\tau_0, \tau_2))}^2 \| \partial_s^l \beta_a \|_{L^2L^2(\Omega \times (\tau_0, \tau_2))}^2 \\
        &+ C(l) \sum_{j=0}^l \| B_j^l \|_{L^\infty(\Omega \times (\tau_0, \tau_2))}^2 \| \partial_s^j u \|_{L^2L^2(\Omega \times (\tau_0, \tau_2))}^2 \\
        \le{}& C(l) (1 + 2 C_0) ^2 \| \partial_s^l \beta_a \|_{L^2L^2(\Omega \times (\tau_0, \tau_2))}^2 + C(n,k, \Omega)  \|  u \|_{L^2H^l_a(\Omega \times (\tau_0, \tau_2))}^2 
    \end{aligned}
    \end{gather}
    where the last line uses \eqref{proof local higher order energy ests for the potential, eqn 1.5} and $\left|\frac{\partial_\tau a}a \right| \le C_0$.
    To estimate this term further, we claim the following pointwise bounds on $\beta_a$.
    \begin{claim}
        For $\Omega \Subset \R \setminus \{ 0 \}$, $\tau_0 \le \tau_1 < \tau_2$,  $a = a(\tau) \in (0,1]$, and $k \in \mathbb N$ as in the statement of the lemma, 
        there exists $C = C(n, k, \Omega)$ such that $\beta_a = \beta_{a(\tau)}(s)$ and its derivatives satisfy the following pointwise bounds 
        \begin{equation} \label{proof local higher order energy ests for the potential, eqn 14}
            |\beta_a| \le C a^2  \qquad \text{and} \qquad 
            \sum_{l=1}^k | \partial_s^l \beta_a | \le C a^{n} \le C a^2 
            \qquad \forall (s, \tau) \in \Omega \times [\tau_0, \tau_2].
        \end{equation}        
    \end{claim}
    \begin{claimproof}
        This follows directly from Lemma \ref{lem beta properties}. 
    \end{claimproof}
    
    Inserting the pointwise estimate \eqref{proof local higher order energy ests for the potential, eqn 14} into \eqref{proof local higher order energy ests for the potential, eqn 13}, it follows that 
    \begin{equation} \label{proof local higher order energy ests for the potential, eqn 15}
        \| \tilde \psi\|_{L^2L^2(\Omega \times (\tau_0, \tau_2))}^2
        \le C(n, k, \Omega, C_0) \int_{\tau_0}^{\tau_2} a^4(\tau) \,  d \tau + C(n, k, \Omega) \| u \|_{L^2 H^l_a (\Omega \times (\tau_0, \tau_2))}^2 
    \end{equation}
    where we used Propositions \ref{prop exp -F_a bounds} and \ref{prop vol element bounds} to estimate $\sup_{a \in (0, 1]} \int_{\Omega} d \mu_a \le C(n, \Omega)$.    

    Now, inserting \eqref{proof local higher order energy ests for the potential, eqn 15} into \eqref{proof local higher order energy ests for the potential, eqn 12} implies that, for all $0 \le l \le k$ and all $\tilde \tau \in [\tau_1, \tau_2]$,
    \begin{align*}
        &\max \left\{ \int \eta(\tilde \tau)^2 (\partial_s^l u(\tilde \tau))^2 d \mu_{a(\tilde \tau)}, \, 
        \int_{\tau_0}^{\tilde \tau} \int \eta^2 \frac{(\partial_s^{l+1} u)^2 }{1 + (\partial_s f_a)^2 } d \mu_a d \tau \right\}\\
        \le{}& \int \eta(\tilde \tau)^2 (\partial_s^l u(\tilde \tau))^2 d \mu_{a(\tilde \tau)} 
        +  \int_{\tau_0}^{\tilde \tau} \int \eta^2 \frac{(\partial_s^{l+1} u)^2 }{1 + (\partial_s f_a)^2 } d \mu_a d \tau \\
        \le{}& 
        C(n, k, \Omega, C_0) \left\{ \int \eta(\tau_0)^2 (\partial_s^l u(\tau_0))^2 d \mu_{a (\tau_0)} + \int_{\tau_0}^{\tau_2} a^4 (\tau) \, d \tau \right\} 
        \\
        & + C(n, k, \Omega, C_0 )
        \left( 1 + \| \partial_s \eta \|^2_{L^\infty (\Omega \times (\tau_0, \tau_2))} + \| \partial_\tau \eta \|_{L^\infty (\Omega \times (\tau_0, \tau_2))}  \right) 
        \cdot \|  u \|_{L^2 H^l_a(\Omega \times (\tau_0, \tau_2))}^2.
    \end{align*}
    Taking a supremum over $\tilde \tau \in [\tau_1, \tau_2]$,
    summing over $0 \le l \le k$,
    using that $\eta \equiv 1$ on $\Omega' \times [\tau_1, \tau_2]$,
    and using that $\frac{1}{1 + (\partial_s f_a)^2 } \ge C_n^{-1} > 0$ (see Lemma \ref{lem: profile function of Lawlor neck}),
    it therefore follows that
    \begin{align*}
        &\| u \|_{L^\infty H^k_a (\Omega' \times (\tau_1, \tau_2))}^2 + \| u \|_{L^2 H^{k+1}_a(\Omega' \times (\tau_1, \tau_2))}^2 \\
        \le{}& C(n, k, \Omega, C_0) \left\{ \sum_{l=0}^k
        \int \eta(\tau_0)^2 (\partial_s^l u(\tau_0))^2 d \mu_{a (\tau_0)} + \int_{\tau_0}^{\tau_2} a^4 (\tau) \, d \tau \right. \\
        & \left. + \left( 1 + \| \partial_s \eta \|^2_{L^\infty (\Omega \times (\tau_0, \tau_2))} + \| \partial_\tau \eta \|_{L^\infty (\Omega \times (\tau_0, \tau_2))}  \right) 
        \cdot \|  u \|_{L^2 H^k_a(\Omega \times (\tau_0, \tau_2))}^2
        \right\} .
    \end{align*}
    The statement of the lemma now follows from taking suitably chosen bump functions $\eta$ as in the end of the proof of Lemma \ref{lem energy ests}.
\end{proof}

By iterating the previous lemma over a suitable sequence of spacetime domains, the $L^2H^k_a$-norms of $u$ on the right-hand sides of estimates \eqref{lem local higher order energy ests for the potential, eqn 1} and \eqref{lem local higher order energy ests for the potential, eqn 2} can be replaced with $L^2L^2_a$-norms.
The proof is left as an exercise for the reader.

\begin{corollary} \label{cor local higher order energy ests for the potential+}
    Let $\Omega', \Omega$ be intervals such that $\Omega' \Subset \Omega \Subset \R \setminus \{ 0 \}$.
    Let $\tau_0 \le \tau_1 < \tau_2$.
    Let $a = a(\tau)$ be $C^1$ and assume that $0 < a \le 1$ and $\left| \frac{\partial_\tau a}{a} \right| \le C_0$.
    Assume $u : \Omega \times (\tau_0, \tau_2) \to \R$ satisfies
        $$\partial_\tau u = H_a u + \left( 1 - 2 \frac{\partial_\tau a}{a} \right) \beta_a + Q_a(u) \qquad \text{on } \Omega \times (\tau_0, \tau_2).$$

    For any $k \in \mathbb{N}$, there exists $0 < \epsilon = \epsilon(n, k, \Omega) \ll 1$ sufficiently small such that, when 
        $$\sup_{\tau \in (\tau_0, \tau_2)} \| u(\cdot , \tau) \|_{C^{k+3}(\Omega  )}< \epsilon ,$$
    the following estimates hold:
    \begin{enumerate}
        \item If $\tau_0 < \tau_1$, then there exists $C = C(n, k, C_0, \Omega, \dist(\Omega', \partial \Omega) , \tau_1 - \tau_0 )$ such that
        \begin{equation}
            \| u \|_{L^\infty H^k_a (\Omega' \times (\tau_1, \tau_2) )} + \| u \|_{L^2 H^{k+1}_a (\Omega' \times (\tau_1, \tau_2))} 
            \le C \left( \| u \|_{L^2 L^2_a (\Omega \times (\tau_0, \tau_2))} + \| a^2 \|_{L^2((\tau_0, \tau_2))}  \right) .
        \end{equation}
            
        \item If $\tau_0 = \tau_1$, then there exists $C = C(n,k, C_0, \Omega, \dist(\Omega', \partial \Omega) ) $ such that 
        \begin{multline}
            \| u \|_{L^\infty H^k_a (\Omega' \times (\tau_0, \tau_2) )} + \| u \|_{L^2 H^{k+1}_a (\Omega' \times (\tau_0, \tau_2))} \\
            \le C \left( \| u (\tau_0) \|_{H^k_a(\Omega)} +  \| u \|_{L^2 L^2_a (\Omega \times (\tau_0, \tau_2))} + \| a^2 \|_{L^2((\tau_0, \tau_2))}  \right) .
        \end{multline}
    \end{enumerate}
\end{corollary}

We also establish local energy estimates for the remainder $v = u - \phi$.
\begin{lem} \label{lem local higher order energy ests for v}
    Let $n \ge 3$, $K \in \mathbb N$, $\tau_0 \le \tau_1 < \tau_2$, and $C_0 > 0$.
    Let $\Omega', \Omega$ be intervals such that $\Omega' \Subset \Omega \Subset \R \setminus \{ 0 \}$.
    For all $0 < s_0 \le s_0^* (n, K) \ll 1$ and all $k \in \mathbb N$,
    there exists $0 < \epsilon(n, k, \Omega) \ll 1$ and $0 < a^*(n, K,s_0, k, \Omega, C_0) \ll 1$ such that the following holds:

    Assume
    \begin{enumerate}
        \item $a = a(\tau)$ is a $C^1$ function $[\tau_0, \tau_2) \to \R$ such that 
        \begin{equation} \label{lem local higher order energy ests for v, eqn 1}
            0 < a \le a^* (n, K,s_0, k, \Omega, C_0) \ll 1 \quad \text{and} \quad \left| \frac{\partial_\tau a}{a} \right| \le C_0 
            \qquad \forall \tau \in [\tau_0, \tau_2),
        \end{equation}
        \item $u : \R \times [\tau_0, \tau_2) \to \R$ satisfies
        \begin{equation} \label{lem local higher order energy ests for v, eqn 2}
            \partial_\tau u = H_a u + \left( 1 - 2 \frac{\partial_\tau a}{a} \right) \beta_a + Q_a(u) \qquad \text{on } \R \times (\tau_0, \tau_2),
        \end{equation}
        \item $u$ is of the form
        \begin{gather}
            \label{lem local higher order energy ests for v, eqn 3}
            u = \phi + v = \sum_{k=1}^K b_k ( \phi_{k,a} - \phi_{0,a}) + v, 
            \qquad b_k = b_k (\tau) ,\qquad v = v(s,\tau) , \\
            \label{lem local higher order energy ests for v, eqn 4}
            \text{and } \langle v, \phi_{k,a} \rangle_{L^2_a} = 0 \quad \forall 0 \le k \le K , \quad \forall \tau \in (\tau_1, \tau_2) ,
        \end{gather}
        where $\phi_{k,a}$ denote the $H_a$-eigenfunctions from Theorem \ref{thm global eigenfunctions and spectral gap},
        \item and
        \begin{gather} \label{lem local higher order energy ests for v, eqn 5}
            | b_k(\tau) | \le C_0 a^2 \qquad \forall 1 \le k \le K , \, \forall \tau \in (\tau_1, \tau_2) \\
            \label{lem local higher order energy ests for v, eqn 6}
            \text{and } \sup_{\tau \in (\tau_0, \tau_2)}\| u(\cdot, \tau) \|_{C^{k+3}(\Omega )} < \epsilon .
        \end{gather}
    \end{enumerate}
    
    Then the following estimates hold:
    \begin{enumerate}
        \item If $\tau_0 < \tau_1$, then there exists $C = C(n, K, s_0, k, C_0, \Omega, \dist (\Omega', \partial\Omega ), \tau_1 - \tau_0)$ such that
        \begin{multline} \label{lem local higher order energy ests for v, eqn 7}
            \| v \|_{L^\infty H^k_a(\Omega' \times (\tau_1, \tau_2) )} + \| v \|_{L^2 H^{k+1}_a ( \Omega' \times (\tau_1, \tau_2))} \\
            \le C \left( \| v \|_{L^2 H^k_a(\Omega \times (\tau_0, \tau_2))} + \| Mod\|_{L^2((\tau_0, \tau_2))} + \| a^3 \|_{L^2 ( (\tau_0, \tau_2))} \right).
        \end{multline}

        \item If $\tau_0 = \tau_1$, then there exists $C = C(n, K, s_0, k, C_0, \Omega, \dist (\Omega', \partial\Omega ))$ such that
        \begin{multline} \label{lem local higher order energy ests for v, eqn 8}
            \| v \|_{L^\infty H^k_a(\Omega' \times (\tau_1, \tau_2) )} + \| v \|_{L^2 H^{k+1}_a ( \Omega' \times (\tau_1, \tau_2))}  \\
            \le C \left( \| v(\tau_0) \|_{H^k_a(\Omega)}+ \| v \|_{L^2 H^k_a(\Omega \times (\tau_0, \tau_2))} + \| Mod\|_{L^2((\tau_0, \tau_2))} + \| a^3 \|_{L^2 ( (\tau_0, \tau_2))} \right).
        \end{multline}
    \end{enumerate}
\end{lem}

\begin{remark} \label{remark local higher order energy ests for v+}
    By similar logic as for Corollary \ref{cor local higher order energy ests for the potential+}, the $L^2H^k_a(\Omega \times (\tau_0, \tau_2))$-norms of $v$ on the right-hand sides of \eqref{lem local higher order energy ests for v, eqn 7}, \eqref{lem local higher order energy ests for v, eqn 8} can be replaced with $L^2 L^2_a(\Omega \times(\tau_0, \tau_2))$-norms of $v$. 
\end{remark}

\begin{proof}
    Throughout the proof, we use $'$ to denote $\frac{\partial}{\partial \tau}$.

    Let $0 \le l \le k$.
    It follows from \eqref{lem local higher order energy ests for v, eqn 2} (cf. \eqref{proof mod eqns, eqn 1}) that
    \begin{equation} \label{proof local higher order energy ests for v, eqn 1}
        \partial_\tau v = H_a v + \left( 1 - 2\frac{a'}a \right) \beta_a - ( \partial_\tau - H_a ) \phi + Q_a(u) .
    \end{equation}
    Denote
    \begin{equation} \label{proof local higher order energy ests for v, eqn 2}
        \hat \psi := \left( 1 - 2\frac{a'}a \right) \beta_a - ( \partial_\tau - H_a ) \phi
        = \left( 1 - 2\frac{a'}a \right) \beta_a - ( \partial_\tau - H_a ) \sum_{i=1}^K b_i (\phi_{i,a} - \phi_{0,a} ) .
    \end{equation}

    By Claim \ref{claim evol eqn of u derivatives} from the proof of Lemma \ref{lem local higher order energy ests}, it follows that $\partial_s^l v$ satisfies an equation of the form
    \begin{equation} \label{proof local higher order energy ests for v, eqn 3}
        \partial_\tau \partial_s^l v = H_a \partial_s^l v + \sum_{j=0}^{l+1} B_j^l(a(\tau), s) \cdot \partial_s^j v + \partial_s^l \hat \psi + \partial_s^l (Q_a (u)) \qquad \text{on } \Omega \times (\tau_0, \tau_2)
    \end{equation}
    where $B_j^l = B_j^l(a,s)$ are functions of $(a, s) \in (0,1] \times \Omega$ such that
    \begin{equation} \label{proof local higher order energy ests for v, eqn 4}
        \| B_j^l \|_{L^\infty((0,1] \times \Omega ) } \le C (n, l, \Omega) < \infty .
    \end{equation}
    Rewrite \eqref{proof local higher order energy ests for v, eqn 3} as 
    \begin{gather} \label{proof local higher order energy ests for v, eqn 5}
        \partial_\tau \partial_s^l v = H_a \partial_s^l v + B^l_{l+1}(a(\tau), s) \cdot \partial_s^{l+1} v + \tilde \psi_l + \partial_s^l (Q_a (u)) \\
        \label{proof local higher order energy ests for v, eqn 5.5}
        \text{where } \tilde \psi_l := \partial_s^l \hat \psi + \sum_{j=0}^{l} B^l_j (a(\tau), s) \cdot \partial_s^j v.
    \end{gather}
    
    Let $\eta : \Omega \times [\tau_0, \tau_2] \to [0,1]$ be a smooth bump function such that $\eta \equiv 1$ in $\Omega' \times[ \tau_1, \tau_2]$ and $\supp \eta \subset \Omega \times [\tau_0, \tau_2]$.
    Multiply \eqref{proof local higher order energy ests for v, eqn 5} by $\eta^2 \partial_s^l v $, integrate in space $\int d\mu_a$, and then in time $\int_{\tau_0}^{\tilde \tau} d \tau$ for an arbitrary $\tilde \tau \in [\tau_1, \tau_2]$.
    By similar logic as in the beginning of the proof of Lemma \ref{lem energy ests}, we obtain \eqref{energy est 1 proof eqn 4} but with $u$ replaced by $\partial_s^l v$, $\psi$ replaced by $\tilde \psi_l$, $b$ replaced by $B^l_{l+1}$, and additional terms from the $\partial_s^l Q_a$ (cf. \eqref{proof local higher order energy ests for the potential, eqn 3}).
    The same estimates as in the proof of Lemma \ref{lem energy ests} apply to deduce that (cf. \eqref{energy est 1 proof eqn 4.5}, \eqref{proof local higher order energy ests for the potential, eqn 4})
    \begin{gather} \label{proof local higher order energy ests for v, eqn 6} \begin{aligned}
        & \frac12 \int \eta(\tilde \tau)^2 (\partial_s^l v(\tilde \tau))^2 d \mu_{a(\tilde \tau)} 
        + \frac12 \int_{\tau_0}^{\tilde \tau} \int \eta^2 \frac{(\partial_s^{l+1} v)^2 }{1 + (\partial_s f_a)^2 } d \mu_a d \tau \\
        \le{}& \frac12 \int \eta(\tau_0)^2 (\partial_s^l v(\tau_0))^2 d \mu_{a (\tau_0)} 
        + \frac12 \| \tilde \psi_l \|_{L^2 L^2_a (\Omega \times (\tau_0, \tau_2))}^2\\
        & + C_n
        \left( 1 + \| \partial_s \eta \|^2_{L^\infty (\Omega \times (\tau_0, \tau_2))} + \| \partial_\tau \eta \|_{L^\infty (\Omega \times (\tau_0, \tau_2))} + C_0 \right. \\
        & \left. \qquad \qquad \qquad + \| B_{l+1}^l(a(\tau), s) \|_{L^\infty(\Omega \times(\tau_0, \tau_2))}^2 \right) 
        \cdot \| \partial_s^l v \|_{L^2 L^2_a(\Omega \times (\tau_0, \tau_2))}^2 
         \\
        &+ \int_{\tau_0}^{\tilde \tau} \int \eta^2 |\partial_s^l v| |\partial_s^l Q_a(u)| \, d \mu_a d\tau
    \end{aligned} \end{gather}
    for any $\tilde \tau \in [\tau_1, \tau_2]$.
    (Here $C_n >0 $ is a dimensional constant.)

    By Claim \ref{proof local higher order energy ests for the potential, claim 1} from the proof of Lemma \ref{lem local higher order energy ests for the potential}, there exists $C (n,l,\Omega) > 0$ such that 
    \begin{multline} \label{proof local higher order energy ests for v, eqn 7}
        | \partial_s^l Q_a(u) | \le \frac12 C(n, l, \Omega) \cdot \sum_{j=0}^{l+2} (\partial_s^j u)^2 
        \le  C(n, l, \Omega) \cdot \sum_{j=0}^{l+2}  \left[ (\partial_s^j \phi)^2 + (\partial_s^j v)^2 \right]
        \\ \forall (s, \tau) \in \Omega \times (\tau_0, \tau_2).
    \end{multline}
    Additionally,
    \begin{equation} \label{proof local higher order energy ests for v, eqn 8}
        \| \tilde \psi_l \|_{L^2 L^2_a(\Omega \times(\tau_0, \tau_2))}^2
        \le 2 \| \partial_s^l \hat \psi \|_{L^2 L^2_a (\Omega \times (\tau_0, \tau_2))}^2 + C(n,l, \Omega) \| v \|_{L^2 H^l_a(\Omega \times (\tau_0, \tau_2) )}^2 
    \end{equation}
    by \eqref{proof local higher order energy ests for v, eqn 4} and \eqref{proof local higher order energy ests for v, eqn 5.5}.
    Inserting these estimates into \eqref{proof local higher order energy ests for v, eqn 6} gives, for any $\tilde \tau \in [\tau_1, \tau_2]$,
    \begin{gather} \label{proof local higher order energy ests for v, eqn 9} \begin{aligned}
        & \frac12 \int \eta(\tilde \tau)^2 (\partial_s^l v(\tilde \tau))^2 d \mu_{a(\tilde \tau)} 
        + \frac12 \int_{\tau_0}^{\tilde \tau} \int \eta^2 \frac{(\partial_s^{l+1} v)^2 }{1 + (\partial_s f_a)^2 } d \mu_a d \tau \\
        \le{}& \frac12 \int \eta(\tau_0)^2 (\partial_s^l v(\tau_0))^2 d \mu_{a (\tau_0)} 
        +  \| \partial_s^l \hat \psi \|_{L^2 L^2_a (\Omega \times (\tau_0, \tau_2))}^2\\
        & + C_n
        \left( 1 + \| \partial_s \eta \|^2_{L^\infty (\Omega \times (\tau_0, \tau_2))} + \| \partial_\tau \eta \|_{L^\infty (\Omega \times (\tau_0, \tau_2))}  + C_0    \right) 
        \cdot \|  v \|_{L^2 H^l_a(\Omega \times (\tau_0, \tau_2))}^2 \\
         &+ C(n, l, \Omega)  \|  v \|_{L^2 H^l_a(\Omega \times (\tau_0, \tau_2))}^2 \\
        &+ C(n,l, \Omega) \sum_{j=0}^{l+2} \int_{\tau_0}^{\tilde \tau} \int \eta^2 |\partial_s^l v| |\partial_s^j \phi |^2  \, d \mu_a d\tau \\
        &+ C(n,l, \Omega) \sum_{j=0}^{l+2} \int_{\tau_0}^{\tilde \tau} \int \eta^2 |\partial_s^l v| |\partial_s^j v |^2 \, d \mu_a d\tau .
    \end{aligned} \end{gather}
    
    Before we can estimate the terms on the right-hand side of \eqref{proof local higher order energy ests for v, eqn 9}, we require a series of preparatory claims.
     \begin{claim} \label{proof local higher order energy ests for v, claim 1}
        For all $j \in \mathbb N$, there holds the pointwise estimate
        \begin{equation}
            | \partial_s^j (a^{-2} \beta_a - \phi_{0,a} ) | \le C(n, j, \Omega, s_0) \cdot a \qquad \forall s \in \Omega .
        \end{equation}
    \end{claim}
    \begin{claimproof}
        By Lemma \ref{lem beta properties} and Theorem \ref{thm global eigenfunctions and spectral gap},
        \begin{multline} \label{proof local higher order energy ests for v, claim 1, eqn 1}
            H_{s,a} ( a^{-2} \beta_a - \phi_{0,a} ) 
        = a^{-2} \cancel{L_{s,a} \beta_a} - \frac s2 a^{-2} \partial_s \beta_a + a^{-2} \beta_a
        - ( 1 + \tilde \lambda_{0,a} ) \phi_{0,a}  \\
        = ( a^{-2} \beta_a - \phi_{0,a}) 
        - \frac 12 \frac sa  \partial_s \beta_1(s/a)
        - \tilde \lambda_{0,a} \phi_{0,a}  .
        \end{multline}
        It then follows from the definition of $H_a$ \eqref{eqn H_a defn} that
        \begin{multline} \label{proof local higher order energy ests for v, claim 1, eqn 2}
            \partial_{ss} ( a^{-2} \beta_a - \phi_{0,a} ) 
            = \left( 1 + (\partial_s f_a)^2 \right) 
            \left\{ -  \frac{\frac \partial {\partial s}  \sqrt{ \frac{(s^2 + f_a^2)^{n-1} }{1 + (\partial_s f_a)^2 } }}{\sqrt{ ( 1 + (\partial_s f_a)^2 ) (s^2 + f_a^2)^{n-1} }} \partial_s \left( a^{-2} \beta_a - \phi_{0,a} \right) \right. \\
            \left. + \frac s2 \partial_s ( a^{-2} \beta_a - \phi_{0,a} ) - \frac12 \frac sa \partial_s \beta_1(s/a) - \tilde \lambda_{0,a} \phi_{0,a} \right\}.
        \end{multline}
        Since $f_a(s)$ converges to $\overline c_0|s|$ in $C^\infty_{loc}(\R \setminus \{ 0\} ) $ as $a \searrow 0$,
        \eqref{proof local higher order energy ests for v, claim 1, eqn 2} implies
        \begin{equation}\label{proof local higher order energy ests for v, claim 1, eqn 3}
            |\partial_{ss} ( a^{-2} \beta_a - \phi_{0,a} ) | 
            \le C(n, \Omega) \left\{ |\partial_s( a^{-2} \beta_a - \phi_{0,a}) | + \left| \frac sa \partial_s \beta_1 (s/a) \right| + |\tilde \lambda_{0,a} | |\phi_{0,a} | \right\} 
        \end{equation}
        for all $s \in \Omega \Subset \R \setminus \{ 0 \}$ and all $0 < a \le a^*$.
        The terms on the right-hand side of \eqref{proof local higher order energy ests for v, claim 1, eqn 3} can be estimated by Lemma \ref{lem beta properties}, Theorem \ref{thm global eigenfunctions and spectral gap}, and \cite{SS26I}*{I-Lemma \ref{I-lem tilde phi_k,a pointwise ests}} to give
        \begin{equation}
            |\partial_{ss} ( a^{-2} \beta_a - \phi_{0,a} ) | \le C(n, \Omega, s_0) a
            \qquad \forall s \in \Omega \Subset \R \setminus \{ 0 \} , \, 0 < a \le a^* .
        \end{equation}
        By repeatedly differentiating \eqref{proof local higher order energy ests for v, claim 1, eqn 2} in $s$ and applying similar estimates, it follows that
        for all $j \ge 2$
        \begin{equation}
            |\partial_s^j ( a^{-2} \beta_a - \phi_{0,a} ) | \le C(n,j, \Omega, s_0) a
            \qquad \forall s \in \Omega \Subset \R \setminus \{ 0 \} , \, 0 < a \le a^* .
        \end{equation}
        \cite{SS26I}*{Lemma \ref{I-lem tilde phi_k,a pointwise ests}} gives the same estimate for $j=0,1$ and thereby completes the proof of the claim.
    \end{claimproof}

    \begin{claim} \label{proof local higher order energy ests for v, claim 2}
        For all $0 \le i \le K$ and all $j \in \mathbb N$, there holds the pointwise estimate
        \begin{equation}
            | \partial_s^j  \phi_{i,a}  | \le C(n, j, i, \Omega, s_0)  \qquad \forall s \in \Omega .
        \end{equation}
    \end{claim}
    \begin{claimproof}
        The proof is similar to that of Claim \ref{proof local higher order energy ests for v, claim 1} so we merely sketch the argument.

        By Theorem \ref{thm global eigenfunctions and spectral gap},
        \begin{equation}
            H_a \phi_{i,a} = (1-i+ \tilde \lambda_{i,a} ) \phi_{i,a}.
        \end{equation}
        The definition of $H_a$ \eqref{eqn H_a defn} then gives 
        \begin{multline} \label{proof local higher order energy ests for v, claim 2, eqn 1}
            \partial_{ss} \phi_{i,a} 
            = \left( 1 + (\partial_s f_a)^2 \right) 
            \left\{ -  \frac{\frac \partial {\partial s}  \sqrt{ \frac{(s^2 + f_a^2)^{n-1} }{1 + (\partial_s f_a)^2 } }}{\sqrt{ ( 1 + (\partial_s f_a)^2 ) (s^2 + f_a^2)^{n-1} }} \partial_s \phi_{i,a} \right. \\
            \left. + \frac s2 \partial_s \phi_{i,a} + (-i + \tilde \lambda_{i,a} ) \phi_{i,a} \right\}.
        \end{multline}
        $f_a(s)$ converges to $\overline c_0 |s|$ in $C^\infty_{loc}(\R \setminus \{0 \} ) $ as $a \searrow 0$.
        Combined with pointwise estimates for $\phi_{i,a}$ and $\partial_s \phi_{i,a}$ from \cite{SS26I}*{Lemmas \ref{I-lem rewriting Laguerre poly}, \ref{I-lem tilde phi_k,a pointwise ests}}, it follows that 
        \begin{equation}
            |\partial_{ss} \phi_{i,a} | \le C(n, i, \Omega, s_0)  \qquad \forall s \in \Omega.
        \end{equation}
        Successively differentiating \eqref{proof local higher order energy ests for v, claim 2, eqn 1} in $s$ and applying similar estimates proves the claim.
    \end{claimproof}

    \begin{claim} \label{proof local higher order energy ests for v, claim 3}
        For all $0 \le i \le K$ and all $j \in \mathbb N$, there holds the pointwise estimate
        \begin{equation}
            |  \partial_s^j  \partial_a  \phi_{i,a}  | \le C(n, j, i, \Omega, s_0)  \qquad \forall s \in \Omega .
        \end{equation}
    \end{claim}
    \begin{claimproof}
        The proof is similar to that of Claim \ref{proof local higher order energy ests for v, claim 1} so we merely sketch the argument.

        For $l =0,1$, the claimed estimate is a direct consequence of \cite{SS26I}*{Lemma \ref{I-lem partial_a phi_a ests}}.
        Differentiating equation \eqref{proof local higher order energy ests for v, claim 2, eqn 1} with respect to $a$ gives an expression for $\partial_{ss} \partial_a \phi_{i,a}$ in terms of $\phi_{i,a}$, $\partial_s \phi_{i,a}$, $\partial_a \phi_{i,a}$, $\partial_s \partial_a \phi_{i,a}$, $\partial_a \tilde \lambda_{i,a}$, and derivatives of $f_a$ with respect to $s$ and $a$.
        These quantities can all be estimated pointwise to obtain
            $$|\partial_{ss} \partial_a \phi_{i,a} | \le C(n, i, \Omega, s_0) \qquad \forall s \in \Omega.$$
        Differentiating equation \eqref{proof local higher order energy ests for v, claim 2, eqn 1} with respect to $a$ and then successively with respect to $s$ yields an expression for $\partial_s^j \partial_a \phi_{i,a}$ in terms of lower order $s$-derivatives which can then be iteratively estimated to prove the claim.
    \end{claimproof}

    We can now proceed to estimate the terms on the right-hand side of \eqref{proof local higher order energy ests for v, eqn 9}.
    
    \begin{claim} \label{proof local higher order energy ests for v, claim 4}
        There exists $C(n, l, K, \Omega, s_0) > 0$ such that 
        \begin{equation}
            \| \partial_s^l \hat \psi \|^2_{L^2 L^2_a (\Omega \times (\tau_0, \tau_2))} \le C(n, l, K, \Omega, s_0) \| Mod + ( 1+ C_0)^2 a^3 \|_{L^2((\tau_0, \tau_2))}^2
        \end{equation}
        where $Mod(\tau) = \sum_{k=1}^K | \partial_\tau b_k - (1 -k) b_k | + \left| \frac d{d\tau} (a^2) - a^2 + \sum_{k=1}^K k b_k \right| $ is defined as in \eqref{eqn defn Mod}.
    \end{claim}
    \begin{claimproof}
        Using the fact that $H_a \phi_{i,a} = (1-i + \tilde \lambda_{i,a} ) \phi_{i,a}$ (Theorem \ref{thm global eigenfunctions and spectral gap}), a straightforward computation reveals that
        \begin{gather} \label{proof local higher order energy ests for v, claim 4, eqn 1} \begin{aligned}
            \hat \psi
            ={}& (a^2 - 2 a a' ) ( a^{-2} \beta_a - \phi_{0,a} ) 
            - \left( 2 a a' - a^2 + \sum_{i=1}^K i b_i \right) \phi_{0,a} \\
            &- \sum_{i=1}^K ( b_i' - (1-i) b_i) ( \phi_{i,a} - \phi_{0,a} ) \\
            &+ \sum_{i=1}^K b_i ( \tilde \lambda_{i,a} \phi_{i,a} - \tilde \lambda_{0,a} \phi_{0,a} ) 
            - a' \sum_{i=1}^K b_i ( \partial_a \phi_{i,a} - \partial_a \phi_{0,a} ) .
        \end{aligned} \end{gather}
        Thus,
        \begin{gather}\label{proof local higher order energy ests for v, claim 4, eqn 2} \begin{aligned}
            \partial_s^l \hat \psi
            ={}& (a^2 - 2 a a' ) ( a^{-2} \partial_s^l \beta_a - \partial_s^l \phi_{0,a} ) 
            - \left( 2 a a' - a^2 + \sum_{i=1}^K i b_i \right) \partial_s^l \phi_{0,a} \\
            &- \sum_{i=1}^K ( b_i' - (1-i) b_i) ( \partial_s^l \phi_{i,a} - \partial_s^l \phi_{0,a} ) \\
            &+ \sum_{i=1}^K b_i ( \tilde \lambda_{i,a} \partial_s^l \phi_{i,a} - \tilde \lambda_{0,a} \partial_s^l  \phi_{0,a} ) 
            - a' \sum_{i=1}^K b_i (  \partial_s^l \partial_a  \phi_{i,a} -  \partial_s^l  \partial_a \phi_{0,a} ) .
        \end{aligned} \end{gather}
        Combining Claims \ref{proof local higher order energy ests for v, claim 1}--\ref{proof local higher order energy ests for v, claim 3} with \eqref{proof local higher order energy ests for v, claim 4, eqn 2} then gives the integral estimate
        \begin{multline*}
            \| \partial_s^l \hat \psi(\cdot , \tau) \|_{L^2_a ( \Omega) }
            \le C(n, l, K, \Omega, s_0) \left\{ |a^2 - 2 a a'| a + \left| 2 a a' - a^2 + \sum_{i=1}^K i b_i \right|  \right. \\
            \left.+ \sum_{i=1}^K | b_i' - (1-i) b_i|  + \sum_{i=1}^K |b_i| ( |\tilde \lambda_{i,a}| + |\tilde \lambda_{0,a}| + |a'| ) \right\} .
        \end{multline*}
        Using that $|a'| \le C_0a$ and $|b_i | \le C_0 a^2$ for all $1 \le i \le K$, it follows that
        \begin{equation} \label{proof local higher order energy ests for v, claim 4, eqn 3}
            \| \partial_s^l \hat \psi(\cdot , \tau) \|_{L^2_a ( \Omega) } \le C(n, l, K, \Omega, s_0) \left\{ Mod(\tau) + ( 1 + C_0 )^2 a^3 \right\}.
        \end{equation}
        Thus,
        \begin{equation}
            \| \partial_s^l \hat  \psi \|^2_{L^2L^2_a(\Omega \times (\tau_0, \tau_2) )} 
            \le C(n, l, K, \Omega, s_0) \| Mod + ( 1+ C_0)^2 a^3 \|^2_{L^2((\tau_0, \tau_2))}.
        \end{equation}
    \end{claimproof}

    \begin{claim} \label{proof local higher order energy ests for v, claim 4.5}
        \begin{equation}
            \sum_{j=0}^{l+2} \int_{\tau_0}^{\tau_2} \int \eta^2 | \partial_s^l v | | \partial_s^j \phi |^2 d \mu_a d \tau
            \le \| v \|_{L^2 H^l_a(\Omega \times (\tau_0, \tau_2)}^2 
            + C(n, K, l, \Omega, s_0)  C_0^4 \| a^4 \|_{L^2 ((\tau_0, \tau_2))}^2
        \end{equation}
    \end{claim}
    \begin{claimproof}
        Recall $\phi = \sum_{i=1}^K b_i ( \phi_{i,a} - \phi_{0,a})$. We compute that
        \begin{align*}
            &\sum_{j=0}^{l+2} \int_{\tau_0}^{ \tau_2} \int \eta^2 | \partial_s^l v | | \partial_s^j \phi |^2 d \mu_a d \tau \\
            \le{}& \int_{\tau_0}^{ \tau_2} \int \eta^2 | \partial_s^l v |^2 d \mu_a d\tau
            + C(l) \sum_{j=0}^{l+2}\int_{\tau_0}^{ \tau_2} \int \eta^2 | \partial_s^j \phi |^4 d \mu_a d \tau 
            && ( \text{Young's inequality}) \\
            \le{}& \| v \|_{L^2 H^l_a(\Omega \times (\tau_0, \tau_2)}^2 
            + C(n, K, l, \Omega, s_0) \sum_{i=0}^K \int_{\tau_0}^{\tau_2} | b_i |^4 d\tau 
            && (\text{Claim \ref{proof local higher order energy ests for v, claim 2}} )\\
            \le{}& \| v \|_{L^2 H^l_a(\Omega \times (\tau_0, \tau_2)}^2 
            + C(n, K, l, \Omega, s_0)  C_0^4 \| a^4 \|_{L^2 ((\tau_0, \tau_2))}^2  
            && ( |b_i| \le C_0 a^2 ),
        \end{align*}
        which proves the claim.
    \end{claimproof}

    \begin{claim} \label{proof local higher order energy ests for v, claim 5}
        For all $\tilde \tau \in [\tau_0, \tau_2]$,
        \begin{gather} \label{proof local higher order energy ests for v, claim 5 eqn} \begin{aligned}
            &\sum_{j=0}^{l+2} \int_{\tau_0}^{\tilde \tau} \int \eta^2 |\partial_s^l v| |\partial_s^j v |^2 \, d \mu_a d\tau \\
            \le{}&
            C(n, l, \Omega) \epsilon \int_{\tau_0}^{\tilde \tau} \int \eta^2 \frac{ ( \partial_s^{l+1} v)^2}{1 + (\partial_s f_a)^2} d \mu_a d \tau \\
            &+ C(n, K, k, \Omega, s_0)C_0 a^2  \int_{\tau_0}^{\tilde \tau} \int \eta^2 \frac{ ( \partial_s^{l+1} v)^2}{1 + (\partial_s f_a)^2} d \mu_a d \tau \\
            &+ C(n, l, \Omega) \left( 1 + \| \partial_s \eta \|_{L^\infty(\Omega \times (\tau_0, \tau_2))}^2 \right)  \| v \|_{L^2 H^l_a(\Omega \times (\tau_0, \tau_2))}^2.
        \end{aligned} \end{gather}
    \end{claim}
    \begin{claimproof}
        First note that $v = u - \phi = u - \sum_{i=1}^K b_i ( \phi_{i,a} - \phi_{0,a})$ implies 
        \begin{align*}
            \sup_{\tau \in (\tau_0, \tau_2)} \| v \|_{C^{k+3}(\Omega)} 
            \le{}& \sup_{\tau \in (\tau_0, \tau_2) } \| u \|_{C^{k+3}(\Omega)}  + \sup_{\tau \in (\tau_0, \tau_2) } \| \phi \|_{C^{k+3}(\Omega)}  \\
            <{}& \epsilon + C(n, K, k, \Omega, s_0) \sum_{i =1}^K \sup_{\tau \in (\tau_0, \tau_2)} |b_i (\tau) | 
            && ( \text{Claim \ref{proof local higher order energy ests for v, claim 2}} ) \\
            \le{}& \epsilon + C(n, K, k, \Omega, s_0) C_0 a^2  
            && ( |b_i| \le C_0 a^2 ) ,
        \end{align*}
        and this can be made arbitrarily small by taking $\epsilon $ and $a^*$ sufficiently small.
        Therefore, the same argument used to estimate \eqref{proof local higher order energy ests for the potential, eqn 7} in the proof of Lemma \ref{lem local higher order energy ests for the potential} applies (with $\epsilon$ replaced by $\epsilon + C(n, K, k, \Omega, s_0) C_0 a^2$) to give the claimed bound \eqref{proof local higher order energy ests for v, claim 5 eqn} (cf. \eqref{proof local higher order energy ests for the potential, eqn 10.5}).
    \end{claimproof}

    Combining Claims \ref{proof local higher order energy ests for v, claim 4}--\ref{proof local higher order energy ests for v, claim 5} with \eqref{proof local higher order energy ests for v, eqn 9} therefore gives
    \begin{gather} \label{proof local higher order energy ests for v, eqn 10} \begin{aligned}
        & \frac12 \int \eta(\tilde \tau)^2 (\partial_s^l v(\tilde \tau))^2 d \mu_{a(\tilde \tau)} 
        + \frac12 \int_{\tau_0}^{\tilde \tau} \int \eta^2 \frac{(\partial_s^{l+1} v)^2 }{1 + (\partial_s f_a)^2 } d \mu_a d \tau \\
        \le{}& \frac12 \int \eta(\tau_0)^2 (\partial_s^l v(\tau_0))^2 d \mu_{a (\tau_0)} 
        + C(n,K,l,\Omega, s_0) \| Mod + (1 + C_0)^2 a^3 \|_{L^2((\tau_0, \tau_2))}^2\\
        & + C(n,l,\Omega) 
        \left( 1 + \| \partial_s \eta \|^2_{L^\infty (\Omega \times (\tau_0, \tau_2))} + \| \partial_\tau \eta \|_{L^\infty (\Omega \times (\tau_0, \tau_2))}  + C_0    \right) 
        \cdot \|  v \|_{L^2 H^l_a(\Omega \times (\tau_0, \tau_2))}^2 \\
        &+ C(n,K, l, \Omega,s_0) C_0^4 \| a^4 \|_{L^2((\tau_0, \tau_2))}^2  \\
        &+ C(n,l, \Omega) \epsilon \int_{\tau_0}^{\tilde \tau} \int \eta^2   \frac{(\partial_s^{l+1} v)^2 }{1 + (\partial_s f_a)^2 } d \mu_a d \tau \\
        &+ C(n,K, k, \Omega, s_0) C_0 a^2 \int_{\tau_0}^{\tilde \tau} \int \eta^2 \frac{(\partial_s^{l+1} v)^2 }{1 + (\partial_s f_a)^2 } d \mu_a d \tau
    \end{aligned} \end{gather}
    for all $\tilde \tau \in [\tau_0, \tau_2]$.
    If $0 < \epsilon (n,k, \Omega) \ll 1$ is sufficiently small (depending only on $n, k, \Omega$) and $0 < a^*(n, K, k, \Omega, s_0, C_0) \ll 1$ is sufficiently small (depending only on $n, K, k, \Omega ,s_0, C_0$) so that 
    \begin{equation*}
        C(n,l, \Omega) \epsilon + C(n,K, k, \Omega, s_0) C_0 a^2 \le \frac14,
    \end{equation*}
    then the last two terms in \eqref{proof local higher order energy ests for v, eqn 10} can be absorbed into the left-hand side to give
    \begin{gather} \begin{aligned}
        &\max \left\{ \int \eta(\tilde \tau)^2 (\partial_s^l v(\tilde \tau))^2 d \mu_{a(\tilde \tau)}, \int_{\tau_0}^{\tilde \tau} \int \eta^2 \frac{(\partial_s^{l+1} v)^2 }{1 + (\partial_s f_a)^2 } d \mu_a d \tau  \right\}  \\
        \le{}& C(n, K, k, \Omega, s_0, C_0) \left( \int \eta(\tau_0)^2 (\partial_s^l v(\tau_0))^2 d \mu_{a (\tau_0)} + \| Mod \|_{L^2((\tau_0, \tau_2))}^2 + \|  a^3 \|_{L^2((\tau_0, \tau_2))}^2 \right) \\
        &+ C(n, K, k, \Omega, s_0, C_0) \left( 1 + \| \partial_s \eta \|_{L^\infty(\Omega \times (\tau_0, \tau_2))}^2 + \| \partial_\tau \eta \|_{L^\infty( \Omega \times (\tau_0, \tau_2))}  \right) \\
        & \qquad \cdot \| v \|_{L^2 H^l_a(\Omega \times (\tau_0, \tau_2))}^2 
    \end{aligned} \end{gather}
    for all $\tilde \tau \in [\tau_0, \tau_2]$ and all $0 \le l \le k$.
    The statement of the lemma now follows by similar logic as the end of the proof of Lemma \ref{lem energy ests}.
    Namely, take a supremum over $\tilde \tau \in [\tau_0, \tau_2]$, sum over $0 \le l \le k$, use $\frac1{1 + (\partial_s f_a)^2} \ge C(n) > 0$ for all $(a, s) \in (0, 1] \times \Omega$ (see Lemma \ref{lem: profile function of Lawlor neck}), and take suitably chosen bump functions $\eta$.
    This completes the proof.
\end{proof}

\section{H\"older Estimates} \label{Section Holder Estimates for the Potential}

In this section, we obtain weighted H{\"o}lder estimates for $u$.
The spatial domain $\R$ is divided into the parabolic region $|s| \sim 1$, the outer region $|s| \gtrsim 1$, and the inner region $|s| \lesssim 1$.
Different weighted H{\"o}lder estimates are obtained in each of these regions.

\subsection{Parabolic Region Estimates}

We begin with H{\"o}lder estimates in the parabolic region $|s| \sim 1$ since this region is the simplest to handle.
As this region is compactly supported away from the origin, combining energy estimates from the previous section with Sobolev embedding provides the desired H{\"o}lder estimates.

\begin{theorem}[Preservation of $C^{2, \alpha}$ Bounds in the Parabolic Region] \label{lem preserving Holder bounds in parab region+}
    For any $n \ge 3$, $\alpha \in (0,1)$,
    $0 < s_0 < \Upsilon$,
    $\kappa > 0$,
    $C_0, C_1, C_2, C_3 > 0$, and
    $0 < \epsilon \le \epsilon^* (n, \alpha, s_0, \Upsilon ) \ll 1$,
    the following holds:

    Assume $a : [\tau_0, \tau_1) \to (0, \infty)$ is $C^1$ and $u : \R \times [\tau_0, \tau_1) \to \R$ is an odd function of $s$ for every $\tau$.
    If $u$ is a solution of
    \begin{equation} \label{lem preserving Holder bounds in parab region+, eqn 1}
        \partial_\tau u = \left( 1 - 2 \frac {\partial_\tau a} a\right) \beta_a + H_a u + Q_a(u) \qquad \forall (s, \tau) \in \R \times (\tau_0, \tau_1)
    \end{equation}
    such that 
    \begin{enumerate}
        \item (Global $C^1$-Bounds on $a$)
        \begin{equation} \label{lem preserving Holder bounds in parab region+, eqn 2}
            0 < a = a(\tau) \le \min \{ C_1 e^{- \frac \kappa 2 \tau}, \epsilon \} 
            \quad \text{and} \quad |\partial_\tau a| \le C_0 a \quad \forall \tau \in (\tau_0, \tau_1) ,
        \end{equation}

        \item (Coarse H{\"o}lder\footnote{In this statement and its proof, $C^{k,\alpha}$ denotes the standard, unweighted H{\"o}lder norm.} Bounds for $u$ in the Parabolic Region)
        \begin{equation}  \label{lem preserving Holder bounds in parab region+, eqn 3}
            \sup_{\tau \in (\tau_0, \tau_1)} \| u (\cdot, \tau) \|_{C^{2, \alpha}((\frac{s_0}4, 4 \Upsilon )) } + \| u ( \cdot , \tau_0) \|_{C^{7, \alpha}( ( \frac{s_0}4, 4 \Upsilon) )} \le \epsilon ,
        \end{equation}

        \item (Fine Global Integral Bounds for $u$)
        \begin{align} 
            \label{lem preserving Holder bounds in parab region+, eqn 4}
            \| u ( \cdot, \tau) \|_{L^2_a(\R)} &\le C_2 e^{- \kappa \tau} \qquad \forall \tau \in (\tau_0, \tau_1) , \\
            \label{lem preserving Holder bounds in parab region+, eqn 5}
            \text{and } \| u ( \cdot, \tau_0) \|_{H^4_a\left( \left( \frac{s_0}2 , 2 \Upsilon \right) \right) } &\le C_3 e^{-\kappa \tau_0} ,
        \end{align}
    \end{enumerate}
    then, for some $C = C(n, s_0, \Upsilon , C_0, C_1, C_2, C_3) > 0$,
    \begin{equation}
        \| u ( \cdot, \tau) \|_{C^{2, \alpha}((s_0, \Upsilon))} \le C e^{-\kappa \tau}  .
    \end{equation}
\end{theorem}
\begin{proof}
    We first claim that
    \begin{claim} \label{proof preserving Holder bounds in parab region+, claim 1}
        There exists $C = C(n, \alpha , s_0, \Upsilon) > 1$ such that, for all $0 < \epsilon \le \epsilon^* ( n, \alpha, s_0, \Upsilon) \ll 1$, 
            $$\sup_{\tau \in (\tau_0, \tau_1)} \| u \|_{C^7 ( (\frac{s_0}2 , 2 \Upsilon ) )} \le C \epsilon.$$
    \end{claim}
    \begin{claimproof}
        By Proposition \ref{prop: PDE for the potential}, the profile function $h(s, \tau) = f_{a(\tau)} (s) + u_s(s, \tau)$ satisfies the PDE \eqref{eq: rescaled flow of profile function}.
        Set $w (s,\tau) := \frac{ h- \overline c_0 s}{\epsilon} = \frac{ f_{a(\tau)}(s) - \overline c_0 s + u_s(s, \tau)}{\epsilon }$.
        Then
        \begin{multline} \label{proof preserving Holder bounds in parab region+, claim 1, eqn 1}
            \partial_\tau w = \frac{w_{ss}}{1 + ( \overline c_0 +  \epsilon  w_s)^2 } + (n-1) \frac{ s w_s - w}{s^2 + ( \overline c_0 + \epsilon  w)^2} + \frac12 ( w - sw_s)
            \\ \forall (s, \tau) \in \left(\frac{s_0}4, 4\Upsilon\right) \times (\tau_0,\tau_1).
        \end{multline}
        Recall $f_a(s) \xrightarrow[a \searrow 0]{C^\infty\left( \left( \frac{s_0}4 ,  4 \Upsilon \right)\right) } \overline c_0 s .$
        Lemma \ref{lem: profile function of Lawlor neck}, the assumed estimate $|a(\tau)| \le  \epsilon$ \eqref{lem preserving Holder bounds in parab region+, eqn 2}, and the assumed H{\"o}lder estimates on $u$ \eqref{lem preserving Holder bounds in parab region+, eqn 3} imply
        \begin{equation} \label{proof preserving Holder bounds in parab region+, claim 1, eqn 2}
            \sup_{\tau \in (\tau_0, \tau_1)} \| w \|_{C^{1, \alpha}\left(\left( \frac{s_0}4 , 4 \Upsilon \right) \right) } + \| w (\cdot, \tau_0) \|_{C^{6, \alpha}\left( \left(\frac{s_0}4 ,  4 \Upsilon\right) \right)} < C = C(n,\alpha, s_0, \Upsilon)    .
        \end{equation}
        In particular, $$\sup_{(s, \tau) \in \left( \frac{s_0}4, 4 \Upsilon \right) \times (\tau_0, \tau_1) } | w| + \sup_{(s, \tau) \in \left( \frac{s_0}4, 4 \Upsilon \right) \times (\tau_0, \tau_1) } |w_s | < C.$$
        Interior estimates as in \cite{LSU88}*{Ch. VI, Theorem 1.1}
        for the quasilinear PDE \eqref{proof preserving Holder bounds in parab region+, claim 1, eqn 1} then imply spacetime $C^{1, \alpha'}$ bounds of the form
        \begin{equation} \label{proof preserving Holder bounds in parab region+, claim 1, eqn 3}
            \| w \|_{C^{1, \alpha'} \left( \left( \frac{s_0}{3.5}, 3.5 \Upsilon \right) \times (\tau_0, \tau_1) \right)} < C 
        \end{equation}
        for some $\alpha' = \alpha'(n,s_0, \Upsilon, \alpha)  \in (0, 1)$.
        By replacing $\alpha'$ with $\min \{ \alpha', \alpha\}$, we may additionally assume without loss of generality that $0 < \alpha' \le \alpha$.
        With the $C^{1, \alpha'}$ estimate \eqref{proof preserving Holder bounds in parab region+, claim 1, eqn 3}, we may then regard \eqref{proof preserving Holder bounds in parab region+, claim 1, eqn 1} as a homogeneous, linear PDE for $w$ with coefficients that have $C^{0, \alpha'}$ norms bounded by $C = C(n, \alpha, s_0, \Upsilon ) > 0$.
        Applying interior Schauder estimates (see e.g. \cite{LSU88}*{Ch. IV, Theorem 10.1})
        to \eqref{proof preserving Holder bounds in parab region+, claim 1, eqn 1} then implies that
        \begin{equation}
            \| w \|_{C^{2, \alpha'}\left( \left( \frac{s_0}{3}, 3 \Upsilon \right) \times (\tau_0, \tau_1) \right) } 
            < C \left( \| w \|_{C^{0}\left( \left( \frac{s_0}{3.5}, 3.5 \Upsilon \right) \times (\tau_0, \tau_1) \right) } 
            + \| w (\cdot, \tau_0) \|_{C^{2, \alpha'}\left( \left( \frac{s_0}{3.5}, 3.5 \Upsilon \right)  \right) } \right) 
            < C
        \end{equation}
        where the last inequality follows from \eqref{proof preserving Holder bounds in parab region+, claim 1, eqn 2}.
        By successively differentiating \eqref{proof preserving Holder bounds in parab region+, claim 1, eqn 1} in $s$ and repeatedly applying interior Schauder estimates, it follows that
        \begin{equation}
            \| w \|_{C^{6, \alpha'}\left( \left( \frac{s_0}{2}, 2 \Upsilon \right) \times (\tau_0, \tau_1) \right) } < C \left( \| w \|_{C^{0}\left( \left( \frac{s_0}{4}, 4 \Upsilon \right) \times (\tau_0, \tau_1) \right) } 
            + \| w (\cdot, \tau_0) \|_{C^{6, \alpha'}\left( \left( \frac{s_0}{4}, 4 \Upsilon \right)  \right) } \right)
             < C.
        \end{equation}
        Since $u_s = \epsilon w - f_{a(\tau)}( s) + \overline c_0 s$, it follows that
        \begin{align*}
            &\sup_{\tau \in (\tau_0, \tau_1)} \| u_s \|_{C^{6, \alpha'}\left( \left(\frac{s_0}2 , 2 \Upsilon \right) \right)} \\
            \le{}& \epsilon \| w \|_{C^{6, \alpha'}\left( \left( \frac{s_0}{2}, 2 \Upsilon \right) \times (\tau_0, \tau_1) \right) } + \sup_{\tau \in (\tau_0, \tau_1)} \| f_{a(\tau)} - \overline c_0 s\|_{C^{6, \alpha'} \left( \left( \frac{s_0}{2}, 2 \Upsilon \right)\right)}\\
            <{}& C\epsilon  + C \sup_{\tau \in (\tau_0, \tau_1)} |a(\tau)| 
            && ( \text{by Lemma \ref{lem: profile function of Lawlor neck}} )\\
            <{}& C\epsilon + C \epsilon 
            && ( |a(\tau) |  \le \epsilon ).
        \end{align*}
        Using also $\sup_{\tau \in (\tau_0, \tau_1)} \| u \|_{C^{2, \alpha}\left( \left( \frac{s_0}4 , 4 \Upsilon \right) \right) } < \epsilon$, this completes the proof of the claim.
    \end{claimproof}

    Claim \ref{proof preserving Holder bounds in parab region+, claim 1} now allows us to incorporate the $L^\infty H^4_a$-energy estimates from Corollary \ref{cor local higher order energy ests for the potential+} so long as $0 < \epsilon \le \epsilon^* (n, \alpha, s_0, \Upsilon) \ll 1$.
    Let $\tau \in (\tau_0, \tau_1)$.
    If $\tau > \tau_0+1$, then
    \begin{align*}
        &\| u (\cdot, \tau) \|_{C^{2, \alpha}( (s_0, \Upsilon))} \\
        \le{}& \| u (\cdot, \tau) \|_{C^3 ((s_0, \Upsilon))} \\
        \le{}& C \| u (\cdot, \tau) \|_{H^4_a( (s_0, \Upsilon))} 
        && ( \text{by Sobolev embedding, $C = C(n, s_0, \Upsilon)$)} \\
        \le{}& C  \left\{  \| u \|_{L^2 L^2_a \left ( \left( \frac{s_0}2 , 2 \Upsilon \right) \times (\tau-1 , \tau) \right)}
        + \| a^2 \|_{L^2( (\tau-1, \tau))} \right\}
        && ( \text{by Corollary \ref{cor local higher order energy ests for the potential+}, $C = C(n, s_0, \Upsilon, C_0)$} )\\
        \le{}& C C_2 e^{- \kappa \tau} + C C_1^2 e^{-\kappa \tau}
        && (\text{by \eqref{lem preserving Holder bounds in parab region+, eqn 2}, \eqref{lem preserving Holder bounds in parab region+, eqn 4}}).
    \end{align*}
    If $\tau \in (\tau_0, \tau_0+1)$, then similarly
    \begin{align*}
        &\| u (\cdot, \tau) \|_{C^{2, \alpha}( (s_0, \Upsilon))} \\
        \le{}& \| u (\cdot, \tau) \|_{C^3 ((s_0, \Upsilon))} \\
        \le{}& C \| u (\cdot, \tau) \|_{H^4_a( (s_0, \Upsilon))} 
        && ( \text{by Sobolev embedding, $C = C(n, s_0, \Upsilon)$)} \\
        \le{}& C  \left\{  \| u \|_{L^2 L^2_a \left ( \left( \frac{s_0}2 , 2 \Upsilon \right) \times (\tau_0 , \tau) \right)}
        + \| a^2 \|_{L^2( (\tau_0, \tau))} \right. \\
        & \left. + \| u ( \cdot, \tau_0)\|_{H^4_a\left( \left( \frac{s_0}2, 2 \Upsilon \right) \right) } \right\}
        && ( \text{by Corollary \ref{cor local higher order energy ests for the potential+}, $C = C(n, s_0, \Upsilon, C_0)$} )\\
        \le{}& C C_2 e^{- \kappa \tau} + C C_1^2 e^{-\kappa \tau} + C C_3 e^{- \kappa \tau}
        && (\text{by \eqref{lem preserving Holder bounds in parab region+, eqn 2}--\eqref{lem preserving Holder bounds in parab region+, eqn 5}}).
    \end{align*}
    In summary, there exists $C = C(n, s_0, \Upsilon, C_0, C_1, C_2, C_3) > 0$ such that
    $$\| u (\cdot, \tau) \|_{C^{2, \alpha}( (s_0, \Upsilon))} 
    \le C  e^{-\kappa \tau} \qquad \forall \tau \in  ( \tau_0, \tau_1).$$
\end{proof}

\subsection{Outer Region Estimates} \label{Subsect Outer Region Ests}

This subsection obtains weighted H{\"o}lder estimates in the outer region $|s| \gtrsim 1$.
First, we construct a collection of barriers that provide weighted $C^0$ estimates in this region.
These barriers will then be combined with a rescaling argument to obtain weighted $C^{2, \alpha}$ estimates.

\subsubsection{Outer Region Barriers}

\begin{lem}[Barrier 1 in the Outer Region] \label{lem outer region barrier 1}
    For all $n \ge 3$ and $C_0 > 0$, there exists a dimensional constant $C_n \ge 1$ such that 
    the function
    \begin{gather}
        u(s) = \alpha s^2 - \beta s \text{ is a supersolution} \\
         \partial_\tau u - \left( H_a u + Q_a(u) + \left( 1 - 2 \frac{\partial_\tau a}a \right) \beta_a \right) \ge 0 \qquad \forall (s, \tau) \in (\Gamma, \infty) \times (\tau_1, \tau_2)    
    \end{gather}
    if
    \begin{gather*}
        0 < \alpha , \beta \le 1 ,\\
        0 < a = a(\tau) < 1  \quad\text{and} \quad  |\partial_\tau a | \le C_0 a \quad \forall \tau \in (\tau_1, \tau_2) , \\
        \Gamma \ge C_n , \quad \beta \Gamma \ge C_n \alpha, \quad \text{and} \quad \beta \Gamma \ge C_n (1+C_0) \sup_{\tau \in (\tau_1, \tau_2)} a^2 .
    \end{gather*}
    Under the same conditions, $-u$ is a subsolution.

    Moreover, $\alpha \Gamma > \beta$ implies $u(s) = \alpha s^2 - \beta s > 0$ for all $s \in (\Gamma, \infty)$.
\end{lem}
\begin{proof}
    Throughout, $C = C(n) \ge 1$ denotes a dimensional constant that may change from line to line.

    With $u = \alpha s^2 - \beta s$, we estimate that, for all $(s, \tau) \in (\Gamma , \infty) \times ( \tau_1, \tau_2)$,
    \begin{align*}
        | (H_a - H_0)  u | 
        \le{}& C(  a^n s^{-n} | u_{ss} | + a^n s^{-n-1} |u_s| ) && ( \text{Lemma \ref{lem H_a - H_0 pointwise est}} )\\
        \le{}& C \left( \alpha + \frac \beta s \right) && ( s \ge \Gamma \ge 1, \, |a | \le 1 ), 
    \end{align*}
    \begin{equation*}
        |L_0 u |
        = \left| \left[ \frac{\partial_{ss}}{1 + \overline c_0^2 } +  \frac{n-1}{(1 + \overline c_0^2) s }  \partial_s  \right] ( \alpha s^2 - \beta s) \right| 
        \le C \left( \alpha + \frac \beta s \right) , 
    \end{equation*}
    \begin{align*}
        |Q_a(u) | 
        \le{}& C \left( u_{ss}^2 + \frac{u_s^2}{s^2} \right) 
        && ( \text{Lemma \ref{lem: error term C0}} ) \\
        \le{}& C \left( \alpha^2 + \frac {\beta^2}{s^2} \right) \\
        \le{}& C \left( \alpha + \frac \beta s \right) 
        && ( 0 < \alpha, \beta  \le 1 \text{ and }  s \ge \Gamma \ge 1) ,
    \end{align*}
    and
    \begin{align*}
        \left| \left( 1 - 2 \frac{\partial_\tau a}a \right) \beta_a \right|
        \le{}& ( 1 + 2 C_0 ) | \beta_a | \\
        \le{}& ( 1 + 2 C_0) \left( a^2 \overline \beta + C a^n s^{2-n} \right) 
        && (\text{Lemma \ref{lem beta properties}} ) \\
        \le{}& C(1+ C_0) a^2 && ( s \ge \Gamma \ge C(n)).
    \end{align*}
    Combining the estimates above, it follows that for all $(s, \tau) \in (\Gamma, \infty) \times (\tau_1, \tau_2)$,
    \begin{align*}
        &\partial_\tau u - \left( H_a u + Q_a(u) + \left( 1 - 2 \frac{\partial_\tau a}a \right) \beta_a \right) \\
        ={}& \left( \frac s2 u_s - u \right) - L_0 u   + (H_0 - H_a) u - Q_a(u) - \left( 1 - 2 \frac{\partial_\tau a}a \right) \beta_a  \\
        \ge{}& \frac \beta 2 s - C \left( \alpha + \frac \beta s \right) - C ( 1 + C_0) a^2 \\
        \ge{}& \frac \beta 2 \Gamma - C \left( \alpha + \frac \beta \Gamma \right)  - C ( 1 + C_0) a^2 .
    \end{align*}
    The last line is indeed positive if
    $$\Gamma^2 \ge 8 C, \quad \beta \Gamma \ge 8 C \alpha , \quad \text{and } \beta \Gamma \ge 8 C ( 1 + C_0 ) \sup_{\tau \in (\tau_1, \tau_2)} a^2 .$$
    An analogous argument applies to show $-u$ is a subsolution.

    Clearly, for all $s \in (\Gamma, \infty)$, $u(s) > \alpha \Gamma^2 - \beta \Gamma > 0$  if $\alpha \Gamma > \beta$.    
    This completes the proof.
\end{proof}

\begin{lem}[Barrier 2 in the Outer Region] \label{lem outer region barrier 2}
    For all $n \ge 3$, $C_0, C_1> 0$, and $\kappa > 0$, there exists $C = C(\kappa, n) \ge 1$ such that the function
    \begin{gather}
        u(s, \tau) = e^{-\kappa \tau} ( A s^{2 \kappa + 2} - B s^{2 \kappa +1} ) \text{ is a supersolution} \\
        \label{lem outer region barrier 2, eqn 2}
        \partial_\tau u - \left( H_a u + Q_a(u) + \left( 1 - 2 \frac{\partial_\tau a}a \right) \beta_a \right) \ge 0 \qquad \forall \Gamma < s < \gamma e^{\tau /2} , \, \tau_1 < \tau < \tau_2
    \end{gather}
    if
    \begin{gather*}
        0 < A, B ,\qquad 0 < \gamma \le \gamma^* (\kappa, n, A, B) \ll 1, \\
        0 < a = a(\tau) < \min\{ C_1 e^{-\frac \kappa 2 \tau} , 1 \}   \quad\text{and} \quad  |\partial_\tau a | \le C_0 a \quad \forall \tau \in (\tau_1, \tau_2) , \\
        \Gamma \ge C , \quad B \Gamma \ge CA , \quad \text{and} \quad B \Gamma \ge  C ( 1 + C_0) C_1^2 .
    \end{gather*}
    Under the same hypotheses, $-u$ is a subsolution of \eqref{lem outer region barrier 2, eqn 2} on $\{ (s,\tau ) \in \R \times (\tau_1, \tau_2)  \colon \Gamma < s < \gamma e^{\tau/2} \} $.

    Moreover, $A \Gamma > B$ implies $u(s, \tau) > 0$ for all $(s,\tau) \in (\Gamma, \infty) \times (\tau_1, \tau_2)$.
\end{lem}
\begin{proof}
    Throughout, $C = C(\kappa , n) \ge 1$ is a constant depending only on $\kappa, n$ and that may change from line to line.

    Note that
    \begin{gather} \label{proof barrier 2 outer region, eqn 1} \begin{aligned}
        |u_{ss} | + \frac{|u_s|}{|s|} 
        \le{}& C e^{- \kappa \tau} \left(  A s^{2 \kappa }  + B s^{2 \kappa - 1} \right) \\
        \le{}& C \left( A + \frac B\Gamma \right) e^{-\kappa \tau}  s^{2\kappa} && ( s \ge \Gamma ) \\
        \le{}& C ( A + B ) \gamma^{2 \kappa} && ( s \le \gamma e^{\tau/2} \text{ and } \Gamma \ge 1) .
    \end{aligned} \end{gather}

    Thus, when $ 1 \le \Gamma \le s \le \gamma e^{\tau/2} $,
    \begin{align*}
        | ( H_0 - H_a) u  |
        \le{}& C \left( a^n s^{-n} | u_{ss} | + a^n s^{-n-1} |u_s| \right) && (\text{Lemma \ref{lem H_a - H_0 pointwise est}} ) \\
        \le{}& C a^n \Gamma^{-n} \left( |u_{ss} | + \frac{|u_s|}{|s|} \right) && ( s \ge \Gamma ) \\
        \le{}& C  \left( A + \frac B \Gamma \right) e^{-\kappa \tau}  s^{2\kappa} 
        && ( \text{by \eqref{proof barrier 2 outer region, eqn 1} and } a \le 1 \le \Gamma ) ,
    \end{align*}
    and
    \begin{align*}
        | Q_a(u) |  
        \le{}& C \left( u_{ss}^2 + \frac{u_s^2}{s^2} \right)  && (\text{Lemma \ref{lem: error term C0}}) \\
        \le{}& C (A + B) \gamma^{2\kappa} e^{-\kappa \tau} \left( A s^{2\kappa} + B s^{2 \kappa - 1} \right)  
        && (\text{by \eqref{proof barrier 2 outer region, eqn 1}} )\\
        \le{}& C ( A + B )^2 \gamma^{2 \kappa} e^{-\kappa \tau} s^{2 \kappa} .
    \end{align*}
    We also have 
    \begin{align*}
        \left| \left( 1 - 2 \frac{\partial_\tau a}a \right) \beta_a \right|
        \le{}& ( 1 + 2 C_0 ) | \beta_a | \\
        \le{}& ( 1 + 2 C_0) \left( a^2 \overline \beta + C a^n s^{2-n} \right) 
        && (\text{Lemma \ref{lem beta properties}} ) \\
        \le{}& C(1+ C_0) a^2 && ( \text{when } s \ge \Gamma \ge C(n)) \\
        \le{}& C ( 1 + C_0) C_1^2 e^{-\kappa \tau} && ( 0 < a \le C_1 e^{- \frac \kappa 2 \tau} ) .
    \end{align*}
    A direct computation with $u = e^{-\kappa \tau} ( A s^{2\kappa + 2} - B s^{2 \kappa + 1} )$ gives, for all $s \ge \Gamma$,
    \begin{align*}
        &\partial_\tau u - H_0 u  \\
        ={}& A e^{-\kappa \tau} s^{2\kappa + 2} \cancel{\left[- \kappa + \frac12 ( 2 \kappa + 2) - 1 \right]}
        + A e^{-\kappa \tau} s^{2\kappa } \left[ - \frac{(2 \kappa + 2 )(2 \kappa + 1 )}{1 + \overline c_0^2}  - \frac{(n-1)(2 \kappa + 2) }{1 + \overline c_0^2} \right] \\
        &+ B e^{-\kappa \tau } s^{2 \kappa + 1} \left[ \kappa - \frac12 ( 2 \kappa + 1 ) + 1 \right]
        + B e^{-\kappa \tau } s^{2 \kappa - 1} \left[ \frac{(2 \kappa + 1)(2 \kappa)}{1 + \overline c_0^2 } + \frac{(n-1)(2 \kappa + 1) }{1 + \overline c_0^2} \right]  \\
        \ge{}& \frac12 B e^{-\kappa \tau} s^{2 \kappa+ 1} - C \left( A + \frac B \Gamma  \right) e^{-\kappa \tau} s^{2 \kappa} .
    \end{align*}

    Combining the estimates above, it follows that
    \begin{align*}
        &\partial_\tau u -\left( H_a u + Q_a(u) + \left( 1 - 2 \frac{\partial_\tau a}a \right) \beta_a \right)  \\
        ={}& \partial_\tau u - H_0 u + (H_0- H_a)  u - Q_a(u) - \left( 1 - 2 \frac{\partial_\tau a}a \right) \beta_a \\
        \ge{}& \frac12 B e^{-\kappa \tau } s^{2 \kappa + 1 } - C \left( A + \frac B \Gamma \right) e^{-\kappa \tau } s^{2\kappa } 
        - C ( A + B )^2 \gamma^{2 \kappa} e^{-\kappa \tau} s^{2 \kappa} 
        - C ( 1 + C_0) C_1^2 e^{-\kappa \tau}  \\
        \ge{}&  e^{-\kappa \tau } s^{2 \kappa +1} \left[ \frac B 4 - \frac C \Gamma  \left(  A  + \frac B \Gamma  \right) - C ( A + B)^2 \gamma^{2 \kappa}  \right] 
        + \frac B4 \Gamma^{2 \kappa + 1} e^{-\kappa \tau} - C ( 1 + C_0) C_1^2 e^{-\kappa \tau} \\
        \ge{}& e^{-\kappa \tau} s^{2 \kappa +1} \left[ \frac B 5 - \frac C \Gamma \left( A + \frac B \Gamma \right)  \right] 
        + \frac B 4 \Gamma e^{-\kappa \tau} - C ( 1 + C_0 ) C_1^2 e^{-\kappa \tau} .
    \end{align*}
    where the last inequality uses $0 < \gamma \le \gamma^* ( \kappa, n, A, B ) \ll 1$.
    The last line is indeed positive if
    $$\Gamma^2 > 10 C , \quad B \Gamma > 10 C A , \quad \text{and} \quad  B \Gamma > 4 C ( 1 + C_0) C_1^2 .$$
    An analogous argument argument applies to show $-u$ is a subsolution.

    Clearly, $u = e^{-\kappa \tau} ( A s^{2 \kappa + 2 } - B s^{2 \kappa +1} ) > 0$ for all $s \ge \Gamma $ if $A \Gamma > B$.
    This completes the proof.
\end{proof}

We now combine the above two barriers to form a stronger barrier that suffices for later arguments.

\begin{lem}[Barrier 3 in the Outer Region] \label{lem outer region barrier 3}
    For all $n \ge 3$, $C_0, C_1, A, B > 0$, and $\kappa > 0$, 
    there exists $C = C(\kappa, n ) \ge 1$ and $0 < \gamma^* = \gamma^* (\kappa , n, A , B) \le 1$ such that the function
    \begin{gather*}
       u(s, \tau) 
        = \begin{cases}
            u_1  & \Gamma < s < \frac \gamma 4 e^{\tau / 2}   \\
            \min (u_1, u_2)  & \frac \gamma 4 e^{\tau/2} \le s \le \gamma e^{\tau/2} \\
            u_2  & \gamma e^{\tau /2} < s 
        \end{cases} \\
        \text{with } u_1 =  e^{-\kappa \tau} (  As^{2 \kappa + 2 } -  Bs^{2 \kappa + 1} ) \quad \text{ and } \quad  
        u_2 = A \left( \frac \gamma 2 \right)^{2 \kappa} \left( s^2 - s \right)
    \end{gather*}
    satisfies $u = \min (u_1, u_2)$ on $(\Gamma, \infty) \times (\tau_1, \tau_2)$ and is a (viscosity) supersolution
    \begin{gather}
        \partial_\tau u - \left( H_a u + Q_a(u) + \left( 1 - 2 \frac{\partial_\tau a}a \right) \beta_a \right) \ge 0 \qquad \text{on }  (s, \tau) \in (\Gamma, \infty) \times (\tau_1, \tau_2)
    \end{gather} 
    if
    \begin{gather}
        \label{lem outer region barrier 3, item 0.5}
        0 < \gamma \le \gamma^* (\kappa, n, A, B)  \\
        \label{lem outer region barrier 3, item 1}
        0 < a = a(\tau) < \min\{ C_1 e^{-\frac \kappa 2 \tau} , 1 \}   \quad\text{and} \quad  |\partial_\tau a | \le C_0 a \quad \forall \tau \in (\tau_1, \tau_2) , \\
        \label{lem outer region barrier 3, item 2}
        \Gamma \ge C , \quad B \Gamma \ge CA , \quad \text{and} \quad B \Gamma \ge  C ( 1 + C_0) C_1^2 , \\
        \label{lem outer region barrier 3, item 3}
        \tau_1 \ge \tau_* (\kappa, n,C_0, A ,B, \gamma, \Gamma) \gg 1.
    \end{gather}
    Under the same conditions, $-u$ is a (viscosity) subsolution.

    If additionally $A \Gamma > B$, then $u(s, \tau) > 0$ for all $(s, \tau) \in (\Gamma, \infty) \times (\tau_1, \tau_2)$.
\end{lem}
\begin{proof}
    With assumptions \eqref{lem outer region barrier 3, item 0.5}--\eqref{lem outer region barrier 3, item 2}
    and by taking $\gamma^* = \gamma^* (\kappa, n, A,B) $ sufficiently small (depending only on $\kappa, n, A, B$), 
    Lemma \ref{lem outer region barrier 2} applies to show $u_1$ is a supersolution on the region $\Gamma < s < \gamma e^{\tau/2}$.
    If also $\gamma^* = \gamma^*(\kappa,n, A, B)$ is chosen sufficiently small (depending only $\kappa, A$) so that
        $$A \left( \frac {\gamma^*} 2 \right)^{2 \kappa} \le 1$$
    and $\tau_1 \ge \tau_* (\kappa, n, \gamma, A , C_0) \gg 1$,
    then  Lemma \ref{lem outer region barrier 1} applies (with $\alpha = \beta = A \left( \frac \gamma 2 \right)^{2 \kappa}$ and $\Gamma = \frac \gamma 4 e^{\tau_1/2}$) to show $u_2$ is a supersolution on the region where $\tau \in (\tau_1, \tau_2)$ and $\frac \gamma 4 e^{\tau/2} < s$.
    Note also that 
        $$\frac \gamma 4 e^{\tau_1/ 2} \ge \Gamma \qquad \text{for } \tau_1 \ge \tau_* ( \gamma, \Gamma) \gg 1.$$
        
    It is straightforward to verify that
    \begin{align}
      u_1(s, \tau) &< u_2 (s, \tau)  &&\text{for all } \Gamma \le s \le\frac \gamma 4 e^{\tau/2} \\
      \text{and } u_1(s, \tau) &> u_2(s, \tau)  &&\text{for all } s \ge \gamma e^{\tau/2}
    \end{align}
    for $\tau  \ge \tau_*( \kappa, A, B, \gamma) \gg 1$ and $\Gamma \ge C(\kappa,n)$.
    Thus, $u = \min (u_1, u_2)$ on $(\Gamma, \infty) \times (\tau_1, \tau_2)$.
    By continuity of $u_1$ and $u_2$, in fact $\min (u_1, u_2) = u_1$ in a spacetime neighborhood of $s = \frac \gamma 4 e^{\tau/2}$ and $\min (u_1, u_2) = u_2$ in a neighborhood of $s = \gamma e^{\tau /2}$.
    Since the minimum of two supersolutions is a viscosity supersolution, 
    it follows that $u$ is a continuous viscosity supersolution on $(s, \tau) \in (\Gamma, \infty) \times ( \tau_1, \tau_2)$.
    An analogous argument applies to show $-u$ is a viscosity subsolution.

    When $A \Gamma > B$ and $\tau_1 \ge \tau_*( \gamma ) \gg 1$, Lemmas \ref{lem outer region barrier 1} and \ref{lem outer region barrier 2} can be used to show
    \begin{align*}
        u_1(s, \tau) &> 0 && \text{for }  \Gamma < s < \gamma e^{\tau/2} , \, \tau \in (\tau_1, \tau_2) \\
        \text{and } u_2(s, \tau) &> 0 && \text{for } \frac \gamma 4 e^{\tau/2} < s , \, \tau \in (\tau_1, \tau_2),
    \end{align*}
    which then implies $u$ is positive for all $(s, \tau) \in (\Gamma, \infty) \times ( \tau_1, \tau_2)$.
\end{proof}

\subsubsection{Outer Region Interior Estimates}

We proceed to show how $C^0$ estimates can be upgraded to interior $C^{2, \alpha}$ estimates in the outer region.
We will need the following non-standard Schauder estimate of \cite{Brandt69} (see also \cite{CS21}*{Theorem 3.6} and references therein).
\begin{theorem}[\cite{Brandt69}*{Theorem 3}]\label{thm: nonstand. Sch est.}
    Suppose that $\alpha \in (0,1)$, $B_{2}\subset\mathbb{R}^{n}$, and we are given coefficients $a^{ij}, b^{i}, c:B_{2}\times[-2, 0]\to\mathbb{R}$ and functions $u, h:B_{2}\times[-2, 0]\to\mathbb{R}$ so that $u$ is a classical solution of 
    \begin{align*}
    \partial_t u = a^{ij}\partial_i \partial_j u + b^{i}\partial_i u + cu + h \qquad \text{on } B_2 \times (-2,0).
    \end{align*}
    Assume $a^{ij}, b^{i}, c, h:B_{2}\times[-2, 0]\to\mathbb{R}$ are continuous functions such that 
    \begin{gather*}
    \sup_{t\in [-2, 0]} \|a^{ij}\|_{C^{0, \alpha}(B_2)} + \sup_{t \in [-2,0]} \|b^{i}\|_{C^{0, \alpha}(B_2)} + \sup_{t \in [-2,0]} \|c\|_{C^{0, \alpha}(B_2)} \leq \Lambda < \infty, \\
    \text{and } a^{ij}(x, t)\xi_{i}\xi_{j}\geq\lambda|\xi|^{2} > 0, \qquad \forall\xi\in\mathbb{R}^{n}\setminus \{ 0 \} , \ \forall (x,t) \in B_2 \times [-2, 0].
    \end{gather*}
    Then, for some $ C = C(n, \alpha, \lambda , \Lambda) > 0$,
    \begin{equation}
        \sup_{t \in [-1, 0]} \| u \|_{C^{2, \alpha} ( B_1) } 
        \le C \left\{ \sup_{t \in [-2, 0]} \| u \|_{C^0(B_2)} + \sup_{t \in [-2, 0]} \| h \|_{C^{0, \alpha} (B_2)} \right\} .
    \end{equation}
\end{theorem}

We also need an interior in space non-standard Schauder estimate up to the initial time.
\begin{theorem} \label{thm: nonstand. Sch est. init data}
    Suppose that $\alpha \in (0,1)$, $B_2 \subset \R^n$, $ T \in (0,1]$, and we are given coefficients $a^{ij}, b^i , c: B_2 \times [0, T] \to \R$ and functions $u, h : B_2 \times [0, T] \to \R$ and $u_0 : B_2 \to \R$ so that $u: B_2 \times[0,T] \to \R$ is a classical solution of 
    \begin{gather} \label{thm: nonstand. Sch est. init data, eqn 1}
    \left\{ \begin{aligned}
        \partial_t u &= a^{ij} \partial_i \partial_j u + b^i \partial_i u + c u + h && \text{on }   B_2 \times [0, T],  \\
        u(\cdot , 0) &= u_0  && \text{in } B_2
        \end{aligned} \right.
    \end{gather}
        
    Assume $a^{ij} , b^i, c, h : B_2 \times [ 0,T] \to \R$ are continuous functions such that 
        $$\sup_{t \in [0,T]} \| a^{ij} \|_{C^{0, \alpha}(B_2)} + \sup_{t \in [0,T]} \| b^{i} \|_{C^{0, \alpha}(B_2)} + \sup_{t \in [0,T]} \| c \|_{C^{0, \alpha}(B_2)} + 
        \sup_{t \in [0,T]} \| h \|_{C^{0, \alpha}(B_2)}
        \le \Lambda < \infty, $$
        $$a^{ij} (x,t) \xi_i \xi_j \ge \lambda | \xi |^2 > 0 \qquad \forall \xi \in \R^n \setminus \{ 0 \} , \ \forall (x,t) \in B_2 \times [0,T],$$
    and $u_0 \in C^{2, \alpha} ( B_2) $ with
        $$\| u_0 \|_{C^{2, \alpha}(B_2)} \le \Lambda < \infty.$$
    Then
    \begin{equation}
        \sup_{t \in [0, T]} \| u \|_{C^{2, \alpha} (B_1)} \le C   \left\{  \| u_0 \|_{C^{2, \alpha}(B_{2})} + \sup_{t \in [0,T]} \| u \|_{C^0(B_2)} +  \sup_{t \in [0, T]} \| h \|_{C^{0, \alpha}(B_{2})} \right\}.
    \end{equation}
    for some $C = C(n, \alpha, \lambda, \Lambda)$.
\end{theorem}
\begin{proof}
    Throughout the proof, $C = C(n, \alpha, \lambda, \Lambda) > 0$ denotes a constant depending only on $n, \alpha, \lambda, \Lambda$, and which may change from line to line.

    By multiplying \eqref{thm: nonstand. Sch est. init data, eqn 1} by a suitable cutoff function $\eta : B_2 \to [0,1]$,
    \cite{Lorenzini00}*{Theorem 4.1} (see in particular \cite{Lorenzini00}*{Step 1 of the Proof of Theorem 4.1}) shows that
    \begin{equation} \label{proof non-std Schauder 2, eqn 1}
        \sup_{t \in [0, T]} \| u \|_{C^{2, \alpha}(B_{1})}
        \le C \left\{  \| u_0 \|_{C^{2, \alpha}(B_{3/2})} + \sup_{t \in [0, T]} \| u \|_{C^{1, \alpha} (B_{3/2})} +  \sup_{t \in [0, T]} \| h \|_{C^{0, \alpha}(B_{3/2})}  \right\}.
    \end{equation}
    Spacetime $W^{2,p}$ estimates as in \cite{LSU88}*{Ch. IV, Theorem 9.1}
    and Sobolev embedding furthermore imply that
    \begin{equation} \label{proof non-std Schauder 2, eqn 2}
        \sup_{t \in [0,T]} \| u \|_{C^{1,\alpha}(B_{3/2}) } 
        \le C \left\{  \| u_0 \|_{C^{2, \alpha}(B_{2})} +  \| u \|_{C^0(B_2 \times [0,T])} +  \| h \|_{C^{0}(B_{2} \times [0, T])}  \right\}.
    \end{equation}
    Combining \eqref{proof non-std Schauder 2, eqn 1} and \eqref{proof non-std Schauder 2, eqn 2} completes the proof.
\end{proof}

\begin{remark}
    The example of \cite{SVW88}*{Section 3} shows that an analogous $C^0_t C^{2, \alpha}_x$ bound up to the lateral boundary $\partial B_2 \times (0,1)$ does \emph{not} hold.
    \cite{Lieberman92}*{Lemma 16.8} does however obtain a weaker boundary regularity result that holds up to the lateral boundary.
\end{remark}

In what follows, we will use certain weighted H{\"o}lder norms adapted to the outer region.
\begin{definition} \label{defn weighted Holder norms, outer region}
Let $\alpha\in (0, 1)$, $\gamma\in\mathbb{R}$, and let $\Omega\subseteq\mathbb{R}$. For any function $u:\Omega\to\mathbb{R}$,
define the following spatial weighted H\"older (semi-)norms:
\begin{align}
    \label{eqn weighted Holder norm 1, outer}
    \|u\|^{(\gamma)}_{k;\Omega} &:= \sum_{j=0}^{k}\sup_{s\in\Omega}|s|^{-\gamma+j}|\partial^{j}_{s}u|,\\
    \label{eqn weighted Holder norm 2, outer}
    [\partial^{k}_{s}u]^{(\gamma)}_{\alpha; \Omega} &:= \sup_{s \in \Omega} |s|^{-\gamma+\alpha} \sup_{\substack{s'\neq s''\in\Omega\\ |s' - s|, |s'' - s| < \frac12 |s|}} \frac{|\partial^{k}_{s}u(s') - \partial^{k}_{s}u(s'')|}{|s'-s''|^{\alpha}}
    \\
    \label{eqn weighted Holder norm 3, outer}
    \|u\|^{(\gamma)}_{k, \alpha;\Omega} &:= \|u\|^{(\gamma)}_{k;\Omega} + [\partial^{k}_{s}u]^{( \gamma-k)}_{k, \alpha; \Omega}.
\end{align}
\end{definition}

A rescaling argument and Theorem~\ref{thm: nonstand. Sch est. init data} now gives the following.

\begin{proposition} \label{prop outer region interior est}
    Let $\Upsilon>1, \upsilon \in (0, \infty], 0 < \tau_1 < \tau_2$, $\gamma, \kappa \in \R $ with $\gamma\geq 2\kappa$, and let $u$ satisfy
    \begin{align*}
        \partial_{\tau}u = H_{s, a}u + \psi
        \qquad \text{on } \left\{ (s, \tau ) \in \R \times [\tau_1, \tau_2) \,  \colon \Upsilon < s < \upsilon e^{\tau/2}  \right\} 
    \end{align*}
    where also $a(\tau) \in (0, 1]$ for all $\tau \in [\tau_1, \tau_2)$.
    Then there exists $C = C(n, \alpha, \gamma)>0$ such that
    \begin{multline}
        \sup_{\tau \in [\tau_1, \tau_2)} e^{\kappa \tau} \|u(\cdot, \tau)\|^{(\gamma+2)}_{2, \alpha; \left\{4\Upsilon < s < \frac14 \upsilon e^{\tau/2} \right\}}\\
        \leq C \left[ \sup_{\tau\in[\tau_{1}, \tau_2)} e^{\kappa \tau}  \|u(\cdot, \tau)\|^{(\gamma+2)}_{0; \{\Upsilon< s < \upsilon e^{\tau/2} \}} + \sup_{\tau\in[\tau_{1}, \tau_2)}  e^{\kappa \tau} \|\psi(\cdot, \tau)\|^{(\gamma)}_{0, \alpha; \{\Upsilon< s < \upsilon e^{\tau/2} \}} \right. \\
        \left. +  e^{\kappa \tau_1} \|u(\cdot, \tau_{1})\|^{(\gamma+2)}_{2, \alpha;\{\Upsilon< s < \upsilon e^{\tau_1/2} \}} + \sup_{\tau\in[\tau_{1}, \tau_{2})}e^{\kappa\tau}\|u(\cdot, \tau)\|^{(\gamma+2)}_{2,\alpha; (\Upsilon, 8\Upsilon)\cap \{\Upsilon< s < \upsilon e^{\tau/2} \}}\: \right].
    \end{multline}
\end{proposition}


\begin{proof}
Fix $\tau_*\in(\tau_1,\tau_2)$ and
$s_*\in(4\Upsilon,\frac14\upsilon e^{\tau_*/2})$.
To deal with the unbounded coefficients in $H_{s,a}$, we scale back
the rescaled flow by setting
\begin{align*}
 U(x,t)&:=(-t)\cdot u\left(\frac{x}{\sqrt{-t}},
                              \tau_*-\log(-ts_*^2)\right),\\
 \Psi(x,t)&:=\psi\left(\frac{x}{\sqrt{-t}},
                              \tau_*-\log(-ts_*^2)\right).
\end{align*}
Then $U$ satisfies
\begin{multline*}
 \partial_tU=L_{x,\tilde a}U+\Psi
 =\frac{\partial_{xx}U}{1+(f'_{\tilde a}(x))^2}
    +(n-1)\frac{x\partial_xU}{x^2+f_{\tilde a}^2(x)}+\Psi\\
 \text{on the spacetime domain }
 \left\{(x,t):\sqrt{-t}\,\Upsilon<x<\frac{\upsilon e^{\tau_*/2}}{s_*},\quad
 -\frac{e^{\tau_* -\tau_1}}{s_*^2}
 \leq t< -\frac{e^{\tau_*-\tau_2}}{s_*^2}\right\},
\end{multline*}
where $\tilde a(t):=\sqrt{-t}\cdot
a(\tau_*-\log(-ts_*^2))$.

Set
\[
 t_1:=\max\left\{-\frac{e^{\tau_*-\tau_1}}{s_*^2},
                  -\frac1{4\Upsilon^2}\right\},\qquad
 \Omega=\left(\frac12,\frac32\right)\times[t_1,-s_*^{-2}].
\]
Note that $\Omega\neq\emptyset$ and $\Omega$ is contained in the
spacetime domain above. Moreover,
\[
 0<\tilde a(t)\leq\sqrt{-t}\leq\frac1{2\Upsilon}<\frac12,
 \qquad 0<-s_*^{-2}-t_1<\frac1{4\Upsilon^2}<1.
\]
The $C^0_tC^1_x(\Omega)$ norms of the coefficients of
$L_{x,\tilde a}$ are therefore uniformly bounded by a dimensional
constant $\Lambda(n)$, and the operator is uniformly elliptic, by
the smooth convergence of $f_a$ to the cone $\overline c_0s$ on
$(1/2,3/2)$ as $a\searrow0$.

The non-standard Schauder estimates now apply. There are two cases
to consider based on whether
$t_1=-e^{\tau_*-\tau_1}/s_*^2$ or
$-1/(4\Upsilon^2)$.

First, consider the case $t_1=-\frac{e^{\tau_*-\tau_1}}{s_*^2}$.
The interior-in-space non-standard Schauder estimate up to the
initial time (Theorem~\ref{thm: nonstand. Sch est. init data}) implies
that there is $C=C(n,\alpha)>0$, which may vary from line to line
below, such that
\begin{gather}\label{proof outer region interior est, eqn 1}
\begin{aligned}
 &\|U(\cdot,-s_*^{-2})\|_{2,\alpha;(3/4,5/4)}\\
 &\leq C\Bigg[
 \sup_{t_1<t\leq-s_*^{-2}}\|U(\cdot,t)\|_{0;(1/2,3/2)}
 +\sup_{t_1<t\leq-s_*^{-2}}\|\Psi(\cdot,t)\|_{0,\alpha;(1/2,3/2)}\\
 &\quad \quad \quad+\|U(\cdot,t_1)\|_{2,\alpha;(1/2,3/2)}\Bigg]\\
 &=C\Bigg[
 \sup_{-s_*^{-2}e^{\tau_*-\tau_1}<t\leq-s_*^{-2}}
 (-t)\left\|u\left(\frac{x}{\sqrt{-t}},
                      \tau_*-\log(-ts_*^2)\right)\right\|_{0;(1/2,3/2)}\\
 &\quad\quad\quad+
 \sup_{-s_*^{-2}e^{\tau_*-\tau_1}<t\leq-s_*^{-2}}
 \left\|\psi\left(\frac{x}{\sqrt{-t}},
                      \tau_*-\log(-ts_*^2)\right)\right\|_{0,\alpha;(1/2,3/2)}\\
 &\quad\quad\quad+(-t_1)\left\|u\left(\frac{x}{\sqrt{-t_1}},\tau_1\right)
                     \right\|_{2,\alpha;(1/2,3/2)}\Bigg].
\end{aligned}
\end{gather}
Note that, since $U(x,-s_*^{-2})=s_*^{-2}u(s_*x,\tau_*)$, we have
\begin{gather}\label{proof outer region interior est, eqn 2}
\begin{aligned}
 &\|U(\cdot,-s_*^{-2})\|_{2,\alpha;(3/4,5/4)}=s_*^{-2}\|u(s_*x,\tau_*)\|_{2,\alpha;(3/4,5/4)}\\
 &=\sum_{j=0}^2\sup_{3s_*/4<s<5s_*/4}
                  s_*^{-2+j}|\partial_s^ju(s,\tau_*)|+\sup_{s\neq s'\in(3s_*/4,5s_*/4)}s_*^\alpha
       \frac{|\partial_s^2u(s,\tau_*)-\partial_s^2u(s',\tau_*)|}
            {|s-s'|^\alpha}\\
 &\geq\sum_{j=0}^2s_*^{-2+j}|\partial_s^ju(s_*,\tau_*)|+\sup_{s\neq s'\in(3s_*/4,5s_*/4)}s_*^\alpha
       \frac{|\partial_s^2u(s,\tau_*)-\partial_s^2u(s',\tau_*)|}
            {|s-s'|^\alpha}.
\end{aligned}
\end{gather}
For $\tau=\tau_*-\log(-ts_*^2)$, the spatial interval
$\bigl(1/(2\sqrt{-t}),3/(2\sqrt{-t})\bigr)$ is contained in
$(\Upsilon,\upsilon e^{\tau/2})$. The spatial scaling gives the
factor $(-t)^{-\gamma/2}$ when estimating $U$ in terms of the
weighted norms of $u$, or $\Psi$ in terms of those of $\psi$.
Since $\gamma\geq2\kappa$ and $\tau\leq\tau_*$,
\begin{equation}\label{proof outer region interior est, weight factor}
 e^{\kappa\tau_*}s_*^{-\gamma}(-t)^{-\gamma/2}e^{-\kappa\tau}
       =e^{(\kappa-\gamma/2)(\tau_*-\tau)}\leq1.
\end{equation}
Thus, combining \eqref{proof outer region interior est, eqn 1} and
\eqref{proof outer region interior est, eqn 2}, and multiplying by
$e^{\kappa\tau_*}s_*^{-\gamma}$, gives
\begin{gather}\label{proof outer region interior est, eqn 3}
\begin{aligned}
 &\sum_{j=0}^2e^{\kappa\tau_*}s_*^{-2-\gamma+j}
                         |\partial_s^ju(s_*,\tau_*)|+\sup_{s\neq s'\in(3s_*/4,5s_*/4)}
 e^{\kappa\tau_*}s_*^{\alpha-\gamma}
 \frac{|\partial_s^2u(s,\tau_*)-\partial_s^2u(s',\tau_*)|}{|s-s'|^\alpha}\\
 &\leq C\Bigg[
 \sup_{\tau_1<\tau\leq\tau_*}e^{\kappa\tau}
 \|u\|^{(\gamma+2)}_{0;(\Upsilon,\upsilon e^{\tau/2})}+\sup_{\tau_1<\tau\leq\tau_*}e^{\kappa\tau}
 \|\psi\|^{(\gamma)}_{0,\alpha;(\Upsilon,\upsilon e^{\tau/2})}+e^{\kappa\tau_1}
 \|u(\cdot,\tau_1)\|^{(\gamma+2)}_{2,\alpha;
                  (\Upsilon,\upsilon e^{\tau/2})}\Bigg],
\end{aligned}
\end{gather}
where $C=C(n,\alpha,\gamma)>0$ now also depends on $\gamma$.

In the case
\[
 t_1=-\frac1{4\Upsilon^2}
                     >-\frac{e^{\tau_*-\tau_1}}{s_*^2},
\]
$x\in(1/2,3/2)$ implies
$x/\sqrt{-t_1}\in(\Upsilon,3\Upsilon)\subset(\Upsilon,8\Upsilon)$.
Denote
\[
 \overline\tau_1:=\tau_*-
                   \log\left(\frac{s_*^2}{4\Upsilon^2}\right).
\]
Then $\tau_1<\overline\tau_1<\tau_*$ and $\upsilon e^{\overline\tau_1/2}
   =\upsilon e^{\tau_*/2}\frac{2\Upsilon}{s_*}>8\Upsilon$.
Since $U(x,t_1)=(2\Upsilon)^{-2}
u(2\Upsilon x,\overline\tau_1)$ and $\gamma\geq 2\kappa$, it follows that
\begin{align*}
e^{\kappa\tau_*}s_*^{-\gamma}
                  \|U(\cdot,t_1)\|_{2,\alpha;(1/2,3/2)} &\leq C e^{\kappa\tau_*}
       \left(\frac{2\Upsilon}{s_*}\right)^\gamma
       \|u(\cdot,\overline\tau_1)\|^{(\gamma+2)}_{2,\alpha;(\Upsilon,3\Upsilon)}\\
 &=C e^{(\kappa-\gamma/2)(\tau_*-\overline\tau_1)}
       e^{\kappa\overline\tau_1}
       \|u(\cdot,\overline\tau_1)\|^{(\gamma+2)}_{2,\alpha;(\Upsilon,3\Upsilon)}\\
 &\leq C\sup_{\tau\in[\tau_1,\tau_2)}e^{\kappa\tau}
 \|u(\cdot,\tau)\|^{(\gamma+2)}_{2,\alpha;
             (\Upsilon,8\Upsilon)\cap(\Upsilon,\upsilon e^{\tau/2})}.
\end{align*}
A similar argument using the interior-in-space non-standard
Schauder estimate up to the initial time
(Theorem~\ref{thm: nonstand. Sch est. init data}) gives
\begin{gather}\label{proof outer region interior est, eqn 4}
\begin{aligned}
 &\sum_{j=0}^2e^{\kappa\tau_*}s_*^{-2-\gamma+j}
                          |\partial_s^ju(s_*,\tau_*)|+\sup_{s\neq s'\in(3s_*/4,5s_*/4)}
 e^{\kappa\tau_*}s_*^{\alpha-\gamma}
 \frac{|\partial_s^2u(s,\tau_*)-\partial_s^2u(s',\tau_*)|}{|s-s'|^\alpha}\\
 &\leq C\Bigg[
 \sup_{\overline\tau_1<\tau\leq\tau_*}e^{\kappa\tau}
 \|u\|^{(\gamma+2)}_{0;{(\Upsilon,\upsilon e^{\tau/2})}}+\sup_{\overline\tau_1<\tau\leq\tau_*}e^{\kappa\tau}
 \|\psi\|^{(\gamma)}_{0,\alpha;(\Upsilon,\upsilon e^{\tau/2})}\\
 &\qquad\quad +\sup_{\tau\in[\tau_1,\tau_2)}e^{\kappa\tau}
 \|u(\cdot,\tau)\|^{(\gamma+2)}_{2,\alpha;
       (\Upsilon,8\Upsilon)\cap(\Upsilon,\upsilon e^{\tau/2})}
 \Bigg]
\end{aligned}
\end{gather}
for some $C=C(n,\alpha,\gamma)>0$.

Combining \eqref{proof outer region interior est, eqn 3} and
\eqref{proof outer region interior est, eqn 4}, and taking a
supremum over $s_*\in(4\Upsilon,\frac14\upsilon e^{\tau_*/2})$
and $\tau_*\in(\tau_1,\tau_2)$, then gives
\begin{multline*}
 \sup_{\tau\in(\tau_1,\tau_2)}e^{\kappa\tau}
 \|u(\cdot,\tau)\|^{(\gamma+2)}_{2,\alpha;
                                (4\Upsilon,\frac14\upsilon e^{\tau/2})}\\
 \leq C\Bigg[
 \sup_{\tau\in(\tau_1,\tau_2)}e^{\kappa\tau}
 \|u(\cdot,\tau)\|^{(\gamma+2)}_{0;(\Upsilon,\upsilon e^{\tau/2})}
 +\sup_{\tau\in(\tau_1,\tau_2)}e^{\kappa\tau}
 \|\psi(\cdot,\tau)\|^{(\gamma)}_{0,\alpha;(\Upsilon,\upsilon e^{\tau/2})}\\
 +e^{\kappa\tau_1}\|u(\cdot,\tau_1)\|^{(\gamma+2)}_{2,\alpha;
                                      (\Upsilon,\upsilon e^{\tau_1/2})}\\
 +\sup_{\tau\in[\tau_1,\tau_2)}e^{\kappa\tau}
 \|u(\cdot,\tau)\|^{(\gamma+2)}_{2,\alpha;
       (\Upsilon,8\Upsilon)\cap(\Upsilon,\upsilon e^{\tau/2})}
 \Bigg].
\end{multline*}
This completes the proof.
\end{proof}

\subsubsection{H{\"o}lder Estimates in the Outer Region Are Preserved}

We are now prepared to prove the main result of Subsection \ref{Subsect Outer Region Ests}.

\begin{lem} \label{lem preserving Holder bounds in outer region}
    For any $n \ge 3$, any $\alpha \in (0,1)$, and any $\kappa, C_0, C_1, A> 0$,
    there exists a constant $C = C(n, \alpha, \kappa) >1$ such that
    for all $0 < \epsilon \le \epsilon^* ( n, \alpha, \kappa, A) \ll 1$,
    all $\Upsilon \ge \Upsilon_*( n, \alpha, \kappa, C_0, C_1, A)  \gg 1$, 
    and all $\tau_0 > \tau_*( n, \alpha, \kappa, C_0, C_1, A, \epsilon, \Upsilon) \gg 1$, the following holds:

    If $a : [\tau_0, \tau_1) \to (0, \infty)$ and $u : \R \times [\tau_0, \tau_1) \to \R$ are such that
    \begin{equation} \label{lem preserving Holder bounds in outer region, eqn 1}
        \partial_\tau u = \left( 1 - 2 \frac {\partial_\tau a} a\right) \beta_a + H_a u + Q_a(u) \qquad \forall (s, \tau) \in \R \times (\tau_0, \tau_1),
    \end{equation}
    $a(\tau)$ has $C^1$-estimates
    \begin{gather} \label{lem preserving Holder bounds in outer region, eqn 1.5}
        0 <  a = a(\tau) \le \min \{ C_1 e^{-\frac \kappa 2 \tau} , 1 \} 
        \quad \text{and} \quad
        |\partial_\tau a | \le C_0 a \quad \forall \tau \in (\tau_0, \tau_1),
    \end{gather}
    $u$ has ``thickened'' lateral boundary estimates
    \begin{equation}
        \label{lem preserving Holder bounds in outer region, eqn 2}
        \sup_{\tau\in (\tau_0, \tau_1)} \| u(\cdot, \tau) \|_{2, \alpha; (\Upsilon, 8 \Upsilon)}^{(2)} < \epsilon 
        \quad \text{and} \quad 
        \sup_{\tau\in (\tau_0, \tau_1)} e^{\kappa \tau} \| u(\cdot, \tau) \|_{2, \alpha; (\Upsilon, 8 \Upsilon)}^{(2\kappa + 2)} < A,
    \end{equation}
    and $u$ has initial data estimates
    \begin{equation}
        \label{lem preserving Holder bounds in outer region, eqn 3}
        \| u(\cdot, \tau_0) \|_{2, \alpha; (\Upsilon, \infty)}^{(2)} < \epsilon 
        \quad \text{and} \quad 
        e^{\kappa \tau_0} \| u(\cdot, \tau_0) \|_{2, \alpha; (\Upsilon, \infty)}^{(2\kappa + 2)} < A,
    \end{equation}
    then $u$ has interior estimates
    \begin{equation} \label{lem preserving Holder bounds in outer region, eqn 4}
        \sup_{\tau\in [\tau_0, \tau_1)} \| u(\cdot, \tau) \|_{2, \alpha; (4\Upsilon, \infty)}^{(2)} < C \epsilon 
        \quad \text{and} \quad 
        \sup_{\tau\in [\tau_0, \tau_1)} e^{\kappa \tau} \| u(\cdot, \tau) \|_{2, \alpha; (4 \Upsilon, \infty)}^{(2\kappa + 2)} < C A.
    \end{equation}
\end{lem}
\begin{proof}
    First note that, for all $s \ge \Upsilon \ge 2$, we have
    \begin{gather*}
        A e^{- \kappa \tau} s^{2\kappa + 2} 
        \le e^{-\kappa \tau} \left( 2  A s^{2\kappa+2} - 2  A s^{2\kappa+1} \right) \quad 
        \text{and}\quad  \epsilon s^2 \le 2 \epsilon (s^2 - s).
    \end{gather*}
    Define $\gamma = \gamma (\kappa, A, \epsilon )> 0$ by the equation $2 \epsilon = 2  A \left( \frac \gamma 2 \right)^{2 \kappa}$.
    Note that for fixed $\kappa , A > 0$, $\gamma \to 0$ as $\epsilon \to 0$.
    Using the assumed lateral boundary and initial data estimates for $u$, the maximum principle, and the barriers from Lemma \ref{lem outer region barrier 3},
    it follows that 
    \begin{multline} \label{proof preserving Holder bounds in outer region, eqn 1}
        |u(s, \tau) |
        \le \min \left\{ e^{-\kappa \tau} \left( 2  A s^{2K+2} - 2 A s^{2K+1} \right), \,
        2 \epsilon (s^2 - s) \right\} 
        \\ \forall (s, \tau) \in \left(  \Upsilon , \infty \right) \times (\tau_0, \tau_1)
    \end{multline}
    if $0 < \epsilon \le \epsilon^* ( n, \kappa, A) \ll 1$ and $\Upsilon \ge \Upsilon_*( n, \kappa , C_0, C_1, A) \gg1$.

    \begin{claim} \label{proof preserving Holder bounds in outer region, claim 1}
        There exists a constant $C = C(n, \alpha) > 1$ depending only on $n, \alpha$ such that 
        if $0 < \epsilon \le \epsilon^*( n, \alpha, \kappa  ,A ) \ll 1$, $\Upsilon \ge \Upsilon_*( n,\alpha,  \kappa , C_0, C_1, A) \gg1$, and $\tau_0 \ge \tau_*( n, \alpha, C_0, C_1, \kappa , \epsilon) \gg 1$, 
        then 
        \begin{equation} \label{proof preserving Holder bounds in outer region, eqn 1.1}
            \sup_{\tau \in [\tau_0, \tau_1)} \| u \|_{2, \alpha; (4 \Upsilon, \infty)}^{(2)} \le  C(n, \alpha) \cdot  \epsilon .
        \end{equation}
    \end{claim}
    \begin{claimproof}
        Let $ C_* = C(n, \alpha) $ denote the constant from the conclusion of the interior estimate Proposition \ref{prop outer region interior est} (with $(\gamma, \kappa) = (0, 0) $ therein), and set $\mathbf C_* := \max\{ C_*, 1\}$.
        Suppose for the sake of contradiction that 
            $$\sup_{\tau \in [\tau_0, \tau_1)} \| u \|_{2, \alpha; (4 \Upsilon, \infty) }^{(2)} > 100 \mathbf C_* \epsilon$$
        and define
            $$\tau_1' := \inf \left\{ \tau \in [\tau_0, \tau_1) \colon \| u (\cdot, \tau) \|_{2, \alpha ; (4 \Upsilon, \infty) }^{(2)} > 100 \mathbf C_* \epsilon \right\}. $$
        By the assumed estimates on the initial data \eqref{lem preserving Holder bounds in outer region, eqn 3}, it follows that $\tau_1' > \tau_0$ and
        \begin{equation} \label{proof preserving Holder bounds in outer region, eqn 1.2}
            \sup_{\tau \in [\tau_0, \tau_1')} \| u \|_{2, \alpha; (4 \Upsilon, \infty)}^{(2)} = 100 \mathbf C_* \epsilon.
        \end{equation}
        
        By the interior estimate Proposition \ref{prop outer region interior est} with $\kappa = 0$,
        \begin{multline} \label{proof preserving Holder bounds in outer region, eqn 2}
            \sup_{\tau \in [\tau_0, \tau_1')} \| u \|^{(2)}_{2, \alpha; (4 \Upsilon, \infty) }
            \le \mathbf C_* \left\{ \sup_{\tau \in [\tau_0, \tau_1')} \| u \|^{(2)}_{0; (\Upsilon, \infty) } + \| u ( \cdot , \tau_0) \|_{2, \alpha; (\Upsilon, \infty)}^{(2)} \right.\\
            \left. 
            + \sup_{\tau \in [\tau_0, \tau_1')} \| u \|_{2, \alpha ; (\Upsilon, 8 \Upsilon)}^{(2)}
            + \sup_{\tau \in [\tau_0, \tau_1')} \left\| \left( 1 - 2 \frac {\partial_\tau a}{a} \right) \beta_a + Q_a(u) \right\|_{0, \alpha; (\Upsilon, \infty)}^{(0)} 
            \right\}.
        \end{multline}
        By \eqref{proof preserving Holder bounds in outer region, eqn 1} and the assumed estimates on the initial data \eqref{lem preserving Holder bounds in outer region, eqn 3}
        and lateral boundary \eqref{lem preserving Holder bounds in outer region, eqn 2},
        the first three terms on the right-hand side of \eqref{proof preserving Holder bounds in outer region, eqn 2} can be estimated by
        \begin{equation} \label{proof preserving Holder bounds in outer region, eqn 3}
            \sup_{\tau \in [\tau_0, \tau_1')} \| u \|^{(2)}_{0; (\Upsilon, \infty) } + \| u ( \cdot , \tau_0) \|_{2, \alpha; (\Upsilon, \infty)}^{(2)} + \sup_{\tau \in [\tau_0, \tau_1')} \| u \|_{2, \alpha ; (\Upsilon, 8 \Upsilon)}^{(2)}
            \le 2\epsilon + \epsilon + \epsilon= 4 \epsilon.
        \end{equation}
        Next,
        \begin{gather} \label{proof preserving Holder bounds in outer region, eqn 4} \begin{aligned} 
            &\sup_{\tau \in [\tau_0, \tau_1')} \left\| \left( 1 - 2 \frac{\partial_\tau a}{a} \right) \beta_a \right\|_{0, \alpha; (\Upsilon, \infty)}^{(0)} \\
            \le{}& ( 1 + 2 C_0) \sup_{\tau \in [\tau_0, \tau_1')} \| \beta_a \|_{0, \alpha ; (\Upsilon, \infty) }^{(0)} 
            && ( \text{since } |\partial_\tau a | \le C_0 a ) \\
            \le{}& ( 1 + 2 C_0) \sup_{\tau \in [\tau_0, \tau_1')} C(n) \cdot a(\tau)^2
            && (\text{by Lemma \ref{lem beta properties}}) \\
            \le{}& C(n)\cdot  ( 1 + C_0) C_1^2 e^{-\kappa \tau_0} 
            && \left( \text{since } 0 < a \le C_1 e^{-\frac \kappa 2 \tau} \right) \\
            \le{}& \epsilon 
            && ( \text{if } \tau_0 \ge \tau_*(n, C_0, C_1, \kappa , \epsilon ) \gg 1 ).
        \end{aligned} \end{gather} 
        
        By \eqref{proof preserving Holder bounds in outer region, eqn 1.2} and the thickened lateral boundary estimates \eqref{lem preserving Holder bounds in outer region, eqn 2}, 
            $$\sup_{\tau \in [\tau_0, \tau_1')} \| u \|_{2, \alpha; (\Upsilon, \infty)}^{(2)} \le \frac14 \qquad \text{for all $0 < \epsilon \le \epsilon^*(n, \alpha) \ll 1$.}$$
        Hence, Lemma \ref{lem weighted Holder est for Q} applies to show
        \begin{gather} \label{proof preserving Holder bounds in outer region, eqn 5} \begin{aligned}
            &\sup_{\tau \in [\tau_0, \tau_1')} \left\| Q_a(u) \right\|_{0, \alpha; (\Upsilon, \infty)}^{(0)} \\
            \le{}& C(n, \alpha) \cdot\left(  \sup_{\tau \in [\tau_0, \tau_1')} \| u \|_{2, \alpha; (\Upsilon, \infty)}^{(2)} \right)^2  
            && ( \text{by Lemma \ref{lem weighted Holder est for Q}})\\
            \le{}& C(n, \alpha) \left( \sup_{\tau \in [\tau_0, \tau_1')} \| u \|_{2, \alpha; (\Upsilon, 8 \Upsilon)}^{(2)} \right)^2 \\
            &+ C(n, \alpha) \left( \sup_{\tau \in [\tau_0, \tau_1')} \| u \|_{2, \alpha; (4 \Upsilon, \infty)}^{(2)} \right)^2 
            \\
            \le{}& C(n, \alpha) \epsilon^2 + C(n, \alpha) 100^2 \mathbf C_*^2 \epsilon^2 
            && ( \text{by \eqref{lem preserving Holder bounds in outer region, eqn 2}, \eqref{proof preserving Holder bounds in outer region, eqn 1.2}} ) \\
            \le{}& \epsilon
            && ( \text{if } 0 <  \epsilon \le \epsilon^*( n , \alpha) \ll 1) .
        \end{aligned} \end{gather}
        
        Combining \eqref{proof preserving Holder bounds in outer region, eqn 2}--\eqref{proof preserving Holder bounds in outer region, eqn 5} therefore gives
        \begin{equation} \label{proof preserving Holder bounds in outer region, eqn 6} 
            \sup_{\tau \in [\tau_0, \tau_1')} \| u \|^{(2)}_{2, \alpha; (4 \Upsilon, \infty) }
            \le 6 \mathbf C_* \epsilon < 100\mathbf C_* \epsilon.
        \end{equation}
        This contradicts \eqref{proof preserving Holder bounds in outer region, eqn 1.2} and proves the claim.
    \end{claimproof}

    \begin{claim} \label{proof preserving Holder bounds in outer region, claim 2}
        There exists a constant $C = C(n, \alpha, \kappa) > 1$ (depending only on $n, \alpha, \kappa$) such that 
        if  $\Upsilon \ge \Upsilon_*( n, \alpha, \kappa, C_0, C_1, A) \gg1$, 
        $0 < \epsilon \le \epsilon^*( n, \alpha, \kappa ,A ) \ll 1$,
        and 
        $\tau_0 \ge \tau_*( n, \alpha, \kappa , C_0, C_1, A, \Upsilon, \epsilon) \gg 1$, 
        then 
        \begin{equation}
            \sup_{\tau \in [\tau_0, \tau_1)} e^{\kappa \tau} \| u \|_{2, \alpha; (4 \Upsilon, \infty) }^{(2\kappa+2)} \le C(n, \alpha, \kappa) \cdot   A .
        \end{equation}
    \end{claim}
    \begin{claimproof}
        The proof is similar to the proof of Claim \ref{proof preserving Holder bounds in outer region, claim 1}. We nonetheless provide the details.
        Let $ C_* = C(n, \alpha, \kappa) $ denote the constant from the conclusion of the interior estimate Proposition \ref{prop outer region interior est} with $\gamma = 2 \kappa$, and set $\mathbf C_* := \max\{ C_*, 1\}$.
        Suppose for the sake of contradiction that
        \begin{equation} \label{proof preserving Holder bounds in outer region, proof claim 2, eqn 1}
            \sup_{\tau \in [\tau_0, \tau_1)} e^{\kappa  \tau} \| u \|_{2, \alpha; (4\Upsilon, \infty)}^{(2\kappa +2)} > 100 \mathbf C_*  A.
        \end{equation}
        Fix a small positive constant $\gamma = \gamma(n, \alpha , \kappa , A) > 0$ to be determined.
        Note
        \begin{gather} \label{proof preserving Holder bounds in outer region, proof claim 2, eqn 2} \begin{aligned}
            & \sup_{\tau \in [\tau_0, \tau_1)} e^{\kappa \tau} \| u \|_{2, \alpha; (\frac14 \gamma e^{\tau/2}, \infty) }^{(2 \kappa +2)} \\
            \le{}&  \sup_{[\tau_0, \tau_1)} e^{\kappa  \tau} \left(\frac14  \gamma e^{\tau/2} \right)^{-2 \kappa } \| u \|_{2, \alpha ; (\frac14  \gamma e^{\tau/2} , \infty) }^{(2)} 
            && ( \text{by \eqref{eqn weighted Holder norm 1, outer}--\eqref{eqn weighted Holder norm 3, outer}}) \\
            \le{}& 4^{2\kappa } \gamma^{-2\kappa } \sup_{[\tau_0, \tau_1)} \| u \|_{2, \alpha; (4 \Upsilon, \infty) }^{(2)} 
            && ( \text{for } \tau_0 \ge \tau_*( \gamma, \Upsilon) \gg  1) \\
            \le{}& C(n, \alpha, \kappa )  \gamma^{-2 \kappa} \epsilon && ( \text{by Claim \ref{proof preserving Holder bounds in outer region, claim 1}} ) \\
            <{}& 100 \mathbf C_*  A && ( \text{for } 0 < \epsilon \le \epsilon^* ( n, \alpha, \kappa , A, \gamma) \ll 1 ) .
        \end{aligned} \end{gather}
        Thus, by \eqref{proof preserving Holder bounds in outer region, proof claim 2, eqn 1}, it must be the case that
            $$\sup_{\tau \in [\tau_0, \tau_1)} e^{\kappa \tau} \| u \|_{2, \alpha; (4\Upsilon, \frac14 \gamma e^{\tau/2})}^{(2\kappa+2)} > 100 \mathbf C_* A.$$
        Define
        $$\tau_1' := \inf \{ \tau \in [\tau_0, \tau_1) \colon e^{\kappa \tau} \| u ( \cdot, \tau) \|_{2, \alpha ;(4\Upsilon, \frac14 \gamma e^{\tau/2}) }^{(2\kappa +2)} > 100 \mathbf C_*  A \} .$$
        By the assumed estimates on the initial data \eqref{lem preserving Holder bounds in outer region, eqn 3}, it follows that $\tau_1' > \tau_0$ and
        \begin{equation} \label{proof preserving Holder bounds in outer region, eqn 7} 
            \sup_{\tau \in [\tau_0, \tau_1')} e^{\kappa \tau} \| u \|_{2, \alpha; (4\Upsilon, \frac14  \gamma e^{\tau/2})}^{(2\kappa +2)} = 100 \mathbf C_* A.
        \end{equation}

        By the interior estimate Proposition \ref{prop outer region interior est},
        \begin{multline} \label{proof preserving Holder bounds in outer region, eqn 8}
            \sup_{\tau \in [\tau_0, \tau_1')} e^{\kappa \tau} \| u \|^{(2\kappa +2)}_{2, \alpha; (4 \Upsilon,  \frac14 \gamma e^{\tau/2} ) }
            \le \mathbf C_* \left\{ \sup_{\tau \in [\tau_0, \tau_1')} e^{ \kappa \tau} \| u \|^{(2\kappa+2)}_{0; (\Upsilon,   \gamma e^{\tau/2} ) } + e^{ \kappa \tau_0} \| u ( \cdot , \tau_0) \|_{2, \alpha; (\Upsilon,   \gamma e^{\tau_0/2} )}^{(2\kappa +2)} \right.\\
            \left. 
            + \sup_{\tau \in [\tau_0, \tau_1')} e^{\kappa \tau} \| u \|_{2, \alpha ; (\Upsilon, 8 \Upsilon)}^{(2 \kappa + 2)}
            + \sup_{\tau \in [\tau_0, \tau_1')} e^{ \kappa \tau} \left\| \left( 1 - 2 \frac {\partial_\tau a}{a} \right) \beta_a + Q_a(u) \right\|_{0, \alpha; (\Upsilon,   \gamma e^{\tau/2} )}^{(2 \kappa )}
            \right\}.
        \end{multline}
        By \eqref{proof preserving Holder bounds in outer region, eqn 1} and the assumed estimates on the initial data \eqref{lem preserving Holder bounds in outer region, eqn 3} 
        and lateral boundary \eqref{lem preserving Holder bounds in outer region, eqn 2},
        the first three terms on the right-hand side of \eqref{proof preserving Holder bounds in outer region, eqn 2} can be estimated by
        \begin{multline} \label{proof preserving Holder bounds in outer region, eqn 9}
            \sup_{\tau \in [\tau_0, \tau_1')} e^{\kappa \tau} \| u \|^{(2\kappa+2)}_{0; (\Upsilon, \infty) } + e^{\kappa \tau_0} \| u ( \cdot , \tau_0) \|_{2, \alpha; (\Upsilon, \infty)}^{(2 \kappa+2)}
            + \sup_{\tau \in [\tau_0, \tau_1')} e^{\kappa \tau} \| u \|_{2, \alpha ; (\Upsilon, 8 \Upsilon)}^{(2 \kappa + 2)} \\
            \le 2 A + A +A= 4 A.
        \end{multline}
        Next, 
        \begin{gather} \label{proof preserving Holder bounds in outer region, eqn 10} \begin{aligned}
            &\sup_{\tau \in [\tau_0, \tau_1')} e^{\kappa \tau} \left\| \left( 1 - 2 \frac{\partial_\tau a}a \right) \beta_a \right\|_{0, \alpha; (\Upsilon, \infty)}^{(2\kappa)} \\
            \le{}& \sup_{\tau \in [\tau_0, \tau_1')} e^{\kappa \tau} \sup_{s > \Upsilon }s^{-2\kappa} C(n) \cdot (1+2 C_0)   a^2 
            && ( \text{by Lemma \ref{lem beta properties} and } |\partial_\tau a| \le C_0 a )\\
            \le{}& C(n) \cdot  ( 1 + C_0) C_1^2 \Upsilon^{-2 \kappa}
            && ( \text{since } K > 0 \text{ and } a \le C_1 e^{-\frac \kappa2 \tau} ) \\
            \le{}&  A 
            && ( \text{if } \Upsilon \ge \Upsilon_*( n, \kappa , C_0, C_1, A) \gg 1) .
        \end{aligned} \end{gather}
        By Claim \ref{proof preserving Holder bounds in outer region, claim 1} and the thickened lateral boundary estimates \eqref{lem preserving Holder bounds in outer region, eqn 2},
            $$\sup_{\tau \in [\tau_0, \tau_1')} \| u \|_{2, \alpha; (\Upsilon , \infty)}^{(2)} \le \frac14 \qquad \text{for } 0 < \epsilon \le \epsilon^* (n,\alpha) \ll 1.$$        
        Lemma \ref{lem weighted Holder est for Q} therefore applies to give 
        \begin{gather} \label{proof preserving Holder bounds in outer region, eqn 11} \begin{aligned}
            &\sup_{\tau \in [\tau_0, \tau_1')} e^{\kappa \tau} \left\| Q_a(u)\right\|_{0, \alpha; (\Upsilon,   \gamma e^{\tau/2} )}^{(2 \kappa)} \\
            \le{}& C(n, \alpha, \kappa) \gamma^{2\kappa} \left( \sup_{\tau \in [\tau_0, \tau_1')} e^{\kappa \tau}   \| u \|_{2, \alpha; (\Upsilon,   \gamma e^{\tau/2} )}^{(2\kappa+2)} \right)^2 
            && ( \text{by Lemma \ref{lem weighted Holder est for Q}} ) \\
            \le{}& C(n, \alpha, \kappa) \gamma^{2\kappa} \left( \sup_{\tau \in [\tau_0, \tau_1')} e^{\kappa \tau}  \| u \|_{2, \alpha; (\Upsilon, 8 \Upsilon)}^{(2\kappa+2)} \right)^2 
            \\
            &+ C(n, \alpha, \kappa)\gamma^{2\kappa} \left( \sup_{\tau \in [\tau_0, \tau_1')} e^{\kappa \tau}  \| u \|_{2, \alpha; (4\Upsilon, \infty)}^{(2\kappa+2)} \right)^2  \\
            \le{}& C(n, \alpha, \kappa)\gamma^{2\kappa} \cdot   A^2 \\
            & + C(n, \alpha, \kappa) \gamma^{2\kappa} \cdot 100^2\mathbf C_*^2   A^2
            && ( \text{by \eqref{lem preserving Holder bounds in outer region, eqn 2}, \eqref{proof preserving Holder bounds in outer region, proof claim 2, eqn 2}, \eqref{proof preserving Holder bounds in outer region, eqn 7}} ) \\
            \le{}& A 
        \end{aligned} \end{gather}
        where the last line follows from choosing $\gamma = \gamma (n, \alpha, \kappa, A) > 0$ sufficiently small.

        Combining \eqref{proof preserving Holder bounds in outer region, eqn 8}--\eqref{proof preserving Holder bounds in outer region, eqn 11} then implies that
        \begin{equation}
            \sup_{\tau \in [\tau_0, \tau_1')} e^{\kappa \tau} \| u \|^{(2 \kappa+2)}_{2, \alpha; (4 \Upsilon, \frac14 \gamma e^{\tau/2} ) }
            \le 6 \mathbf C_*  A < 100 \mathbf C_*  A.
        \end{equation}
        This contradicts \eqref{proof preserving Holder bounds in outer region, eqn 7} and proves the claim.
    \end{claimproof}

    Combining the conclusions of Claims \ref{proof preserving Holder bounds in outer region, claim 1} and \ref{proof preserving Holder bounds in outer region, claim 2}
    completes the proof.
\end{proof}

\subsection{Inner Region Estimates} \label{subsect inner region estimates}

Throughout this subsection, we shall use the following spatial weighted H{\"o}lder norms.

\begin{definition} \label{defn weighted Holder norms}
    Given a spacetime domain $\Omega = U\times(\tau_{1}, \tau_{2})\subseteq\mathbb{R}\times[0, \infty)$ and a smooth function $a:(\tau_{1}, \tau_{2})\to\mathbb{R}_{+}$. Define $\rho_{a}:\Omega\to\mathbb{R}_{+}$ by $\rho_{a}(s, \tau) = \rho_{a(\tau)}(s) := \sqrt{a^{2}(\tau) + s^{2}}$. 
    
Let $\gamma\in\mathbb{R}$. For any function $u:\Omega\to\mathbb{R}$,
define the following spatial weighted H\"older norms at $\tau\in(\tau_{1}, \tau_{2})$:
\begin{align}
    \|u(\cdot, \tau)\|^{(\gamma)}_{k;\Omega} &:= \sum_{j=0}^{k}\sup_{s\in\Omega}\rho_{a(\tau)}^{-\gamma+j}(s)|\partial^{j}_{s}u(s, \tau)|,\\
    [\partial^{k}_{s}u(\cdot, \tau)]^{(\gamma)}_{\alpha; \Omega} &:= \sup_{s \in \Omega} \rho_{a(\tau)}^{-\gamma+\alpha}(s) \sup_{\substack{s'\neq s''\in\Omega\\ |s' - s|, |s'' - s| < \frac12 \rho_{a(\tau)}(s)}} \frac{|\partial^{k}_{s}u(s', \tau) - \partial^{k}_{s}u(s'', \tau)|}{|s'-s''|^{\alpha}}
    \\
    \|u(\cdot, \tau)\|^{(\gamma)}_{k, \alpha;\Omega} &:= \|u(\cdot, \tau)\|^{(\gamma)}_{k;\Omega} + [\partial^{k}_{s}u(\cdot, \tau)]^{( \gamma-k)}_{k, \alpha; \Omega}.
\end{align}
\end{definition}

\begin{remark} \label{remark equiv Holder norms}
    Definition \ref{defn weighted Holder norms} resembles Definition \ref{defn weighted Holder norms, outer region}, except the weights $|s|$ are replaced with $\rho_a(s)$.
    However, observe that on regions where $0 < a \le 1$ and $|s| \ge 1$, then $\rho_a(s) \sim |s|$, and so the two definitions define uniformly equivalent norms in this case.
\end{remark}

\subsubsection{Interior Estimates}

On the inner region, we shall make use of an interior Schauder estimate for the weighted H{\"o}lder norms given in Definition \ref{defn weighted Holder norms}.
Its proof is based on a rescaling argument similar to the proof of the corresponding interior Schauder estimate for the outer region Proposition \ref{prop outer region interior est}.

\begin{proposition}\label{prop: weighted Schauder estimate intermediate}
Let $\gamma\in\mathbb{R}$, $s_{0}, \kappa, \kappa' >0$, $C_0 > 1$, $\tau_{2}>\tau_{1}>0$, $\alpha\in(0, 1)$. There exists a constant $C = C(n, \alpha, \gamma, \kappa, s_{0}, C_0, \kappa')>0$ such that if $\tau_1 \ge \tau_*( s_0, C_0, \kappa') \gg 1$ and 
$u$ and $\psi$ satisfy 
\begin{gather}
    \partial_{\tau}u = H_{s, a}u + \psi\quad\mbox{on}\;(-s_{0}, s_{0})\times(\tau_{1}, \tau_{2}), \\
    \text{and } 0 < C_0^{-1} e^{-\kappa' \tau} \le a(\tau) \le C_0 e^{-\kappa' \tau} \qquad \forall \tau \in (\tau_1, \tau_2),
\end{gather}
then there holds
\begin{multline}
 \sup_{\tau\in(\tau_1, \tau_2)}e^{\kappa\tau}\|u(\cdot, \tau)\|^{(\gamma+2)}_{2, \alpha; \{|s|<\frac{7s_{0}}{16}\}} \\
    \leq C\left[ \sup_{\tau\in(\tau_1, \tau_2)}e^{\kappa\tau}\|u(\cdot, \tau)\|^{(\gamma+2)}_{0; \{|s|<s_{0}\}} +  \sup_{\tau\in(\tau_1, \tau_2)}e^{\kappa\tau}\|\psi(\cdot, \tau)\|^{(\gamma)}_{0, \alpha; \{|s|<s_{0}\}}\right. \\
    \left. +e^{\kappa\tau_{1}}\|u(\cdot, \tau_{1})\|^{(\gamma+2)}_{2, \alpha;\{|s|<s_{0}\}} \right].
\end{multline}

\end{proposition}

\begin{proof}

Fix $\tau_{*}\in(\tau_{1}, \tau_{2})$. We first derive the estimate for the `tip region' where $\{|s|<a(\tau_{*})\}$.

Define 
\begin{align*}
    U(x, t) := (-t)u\left(\frac{x}{\sqrt{-t}}, \tau_{*} - \log(-t)\right), \quad \Psi(x, t) := \psi\left(\frac{x}{\sqrt{-t}}, \tau_{*} - \log(-t)\right).
\end{align*}
Then $U$ and $\Psi$ satisfy
\begin{align*}
    &\partial_{t}U = L_{x, \widetilde{a}(t)}U + \Psi\\
    &\mbox{on the spacetime domain}\;\left\{(x, t)\::\:|x|<s_{0}\sqrt{-t},\;-e^{\tau_{*}-\tau_{1}}<t<-e^{\tau_{*}-\tau_{2}}\right\},
\end{align*}
where $\widetilde{a}(t) := \sqrt{-t}\cdot a(\tau_{*}-\log(-t))$. Note that $\widetilde{a}(-1) = a(\tau_{*})\to 0$ as $\tau_{*}\to\infty$. Thus, to fix the length scale, we further rescale $U$ and $\Psi$ by setting
\begin{align*}
    U_{a_{*}}(x, t) := a_{*}^{-2}U\left(a_{*}x, a_{*}^{2}t\right),\quad  \Psi_{a_{*}}(x, t) := \Psi\left(a_{*}x, a_{*}^{2}t\right),
\end{align*}
where $a_{*} := a(\tau_{*})$. Then $U_{a_{*}}$ and $\Psi_{\tau_{*}}$ satisfy
\begin{align*}
    &\partial_{t}U_{a_{*}} = L_{x, \widetilde{a}(a_*^2 t)/a_{*}}U_{a_{*}} + \Psi_{a_{*}}\\
    &\mbox{on the spacetime domain}\;\left\{(x, t)\::\:|x|<s_{0}\sqrt{-t},\;-\frac{e^{\tau_{*}-\tau_{1}}}{a_{*}^{2}}<t<-\frac{e^{\tau_{*}-\tau_{2}}}{a_{*}^{2}}\right\}.
\end{align*}
Let $t_{1} := \max\left\{-\frac{e^{\tau_{*}-\tau_{1}}}{a_{*}^{2}}, -\frac{1}{a_{*}^{2}}-2\right\}$. 
Observe that, for all $t \in \left( t_1, - \frac1{a_*^2} \right]$,
    $$ s_0 \sqrt{-t} \ge \frac{s_0}{a(\tau_*)} \ge s_0 C_0^{-1} e^{\kappa' \tau_1} > 2 \qquad \text{if } \tau_1 \ge \tau_*( s_0 , C_0, \kappa' ) \gg 1,$$
and, since $C_0^{-1} e^{-\kappa' \tau} \le a(\tau) \le C_0 e^{-\kappa' \tau}$,
    $$\frac{\widetilde a(a_*^2 t)}{a_*} = \sqrt{-t}\,  a( \tau_* - \log ( - a_*^2 t))
    \ge \frac{a( \tau_* - \log ( -a_*^2 t))}{a(\tau_*)} \ge \frac1{C_0^2}.$$
Thus, on the spacetime domain $\Omega := \{(x, t)\::\:|x|< 2 ,\;t_{1}<t\leq -\frac{1}{a_{*}^{2}}\}$, the coefficients of the above PDE are bounded in $C^{0}_{t}C^{\alpha}_{x}$ uniformly in $a_{*}$, and hence the non-standard Schauder estimates can be applied. 
Suppose that $t_{1} = -\frac{1}{a_{*}^{2}}-2$. By interior in spacetime non-standard Schauder estimate, there exists a constant $C_{1} = C_{1}(n, \alpha, C_0 )>0$ such that
\begin{align*}
    &\|U_{a_{*}}(\cdot, -a_{*}^{-2})\|_{2, \alpha;\{|x|<1\}}\\
    &\leq C_1\left[ \sup_{-a_{*}^{-2}-2 <t\leq -a_{*}^{-2}}\|U_{a_{*}}(\cdot, t)\|_{0;\{|x|<2\}} + \sup_{-a_{*}^{-2}-2<t\leq a_{*}^{-2}}\|\Psi_{a_{*}}(\cdot, t)\|_{0, \alpha;\{|x|<2\}} \right]\\
    &\leq C_1\left[ \sup_{-1-2a_{*}^{2}<t\leq -1}\|a_{*}^{-2}U(\cdot, t)\|_{0;\{|x|<2a_{*} \}} + \sup_{-1-2a_{*}^{2}<t\leq -1}\|\Psi(\cdot, t)\|_{0, \alpha; \{|x|<2a_{*}\} } \right]\\
    &\leq C_1\left[  \sup_{\tau_{*}-\log(1+2a_{*}^{2})<\tau\leq\tau_{*}}\|(e^{\frac{\tau_{*}-\tau}{2}}a_{*}^{-1})^{2}u(\cdot, \tau)\|_{0;\{|s|< 2a_{*}e^{\frac{\tau-\tau_{*}}{2}}\}}\right.\\
    &\left.\hspace{3cm}+ \sup_{\tau_{*}-\log(1+2a_{*}^{2})<\tau\leq\tau_{*}}\|\psi(\cdot, \tau)\|_{0, \alpha; \{|s|< 2a_{*}e^{\frac{\tau-\tau_{*}}{2}}\}} \right]\\
    &\leq C_1\left[ \sup_{\tau_{*}-\log(1+2a_{*}^{2})<\tau\leq\tau_{*}}\| a_{*}^{-2}u(\cdot, \tau) \|_{0;\{|s|<2a_{*}\}} + \sup_{\tau_{*}-\log(1+2a_{*}^{2})<\tau\leq\tau_{*}}\| \psi(\cdot, \tau) \|_{0;\{|s|<2a_{*}\}}\right]
\end{align*}
Thus,
\begin{align*}
    &\sum_{j=0}^{2}\sup_{|s|<a_{*}}a_{*}^{j-2}|\partial_{s}^{j}u(\cdot, \tau_{*})| + \sup_{\substack{|s|<a_{*},\\ |s'-s|<|s|,\: |s''-s|<|s|}}a_{*}^{\alpha}\frac{|\partial_{s}^{2}u(s', \tau_{*}) - \partial_{s}^{2}u(s'', \tau_{*})|}{|s'-s''|^{\alpha}}\\
    &\leq C_1\left[ \sup_{\tau_{*}-\log(1+2a_{*}^{2})<\tau\leq\tau_{*}}\| a_{*}^{-2}u(\cdot, \tau) \|_{0;\{|s|<2a_{*}\}} + \sup_{\tau_{*}-\log(1+2a_{*}^{2})<\tau\leq\tau_{*}}\| \psi(\cdot, \tau) \|_{0;\{|s|<2a_{*}\}}\right].
\end{align*}
Multiplying both sides by $e^{\kappa\tau_{*}}a_{*}^{-\gamma}$, then it follows that there exists $C_{2} = C_{2}(n, \alpha , \gamma, \kappa, C_0, \kappa')>0$ such that
\begin{align}\label{eq: inner region reg. 1}
    &\sum_{j=0}^{2}e^{\kappa\tau_{*}}\sup_{|s|<a_{*}}a_{*}^{-\gamma-2+j}|\partial_{s}^{j}u(\cdot, \tau_{*})| + \sup_{\substack{|s|<a_{*},\\ |s'-s|<|s|,\: |s''-s|<|s|}}e^{\kappa\tau_{*}} a_{*}^{-\gamma-2+\alpha}\frac{|\partial_{s}^{2}u(s', \tau_{*}) - \partial_{s}^{2}u(s'', \tau_{*})|}{|s'-s''|^{\alpha}}\nonumber\\
    &\leq C_2\left[ \sup_{\tau_{*}-\log(1+2a_{*}^{2})<\tau\leq\tau_{*}}e^{\kappa\tau}\| a_{*}^{-\gamma-2}u(\cdot, \tau) \|_{0;\{|s|<2a_{*}\}} + \sup_{\tau_{*}-\log(1+2a_{*}^{2})<\tau\leq\tau_{*}}e^{\kappa\tau}\|a_{*}^{-\gamma} \psi(\cdot, \tau) \|_{0;\{|s|<2a_{*}\}}\right]\nonumber\\
    &\leq C_2\left[  \sup_{\tau_{1}-\log(1+2a_{*}^{2})<\tau\leq\tau_{*}}e^{\kappa\tau}\| u(\cdot, \tau) \|^{(\gamma+2)}_{0;\{|s|<2a_{*}\}} + \sup_{\tau_{*}-\log(1+2a_{*}^{2})<\tau\leq\tau_{*}}e^{\kappa\tau}\|\psi(\cdot, \tau) \|^{(\gamma)}_{0;\{|s|<2a_{*}\}} \right],
\end{align}
where we used the fact that the time interval is bounded and the distance function $\rho_{a}(s)$ is comparable to $a_{*}$ in the spacetime region we considered.

For the case where $t_{1} = -\frac{e^{\tau_{*}-\tau_{1}}}{a_{*}^{2}}$, we have $\tau_{*}-\tau_{1}\leq\log(2a_{*}^{2}+1)\leq \log 2$ if $\tau_1 \ge \tau_*( C_0, \kappa') \gg 1$. By interior in space non-standard Schauder estimate, there exists $C_{3} = C_{3}(n, \alpha, C_0)>0$ such that
\begin{align*}
    &\|U_{a_{*}}(\cdot, -a_{*}^{-2})\|_{2, \alpha;\{|x|<1\}}\\
    &\leq C_3\left[ \sup_{-\frac{e^{\tau_{*}-\tau_{1}}}{a_{*}^{2}} <t\leq -a_{*}^{-2}}\|U_{a_{*}}(\cdot, t)\|_{0;\{|x|<2\}}\right.\\
    &\left.\hspace{1cm}+ \sup_{-\frac{e^{\tau_{*}-\tau_{1}}}{a_{*}^{2}}<t\leq -a_{*}^{-2}}\|\Psi_{a_{*}}(\cdot, t)\|_{0, \alpha;\{|x|<2\}} + \|U_{a_{*}}(\cdot, t_{1})\|_{2, \alpha;\{|x|<2\}} \right]\\
    &\leq C_3\left[ \sup_{e^{\tau_{*}-\tau_{1}}<t\leq -1}\|a_{*}^{-2}U(\cdot, t)\|_{0;\{|x|<2a_{*} \}}\right.\\
    &\left.\hspace{1cm}+ \sup_{e^{\tau_{*}-\tau_{1}}<t\leq -1}\|\Psi(\cdot, t)\|_{0, \alpha; \{|x|<2a_{*}\} } + \|a_{*}^{-2}U(\cdot, -e^{\tau_{*}-\tau_{1}})\|_{2, \alpha;\{|x|<2a_{*}\}}\right]\\
    &\leq C_3\left[  \sup_{\tau_{1}<\tau\leq\tau_{*}}\|(e^{\frac{\tau_{*}-\tau}{2}}a_{*}^{-1})^{2}u(\cdot, \tau)\|_{0;\{|s|< 2a_{*}e^{\frac{\tau-\tau_{*}}{2}}\}}\right.\\
    &\left.\hspace{1cm}+ \sup_{\tau_{1}<\tau\leq\tau_{*}}\|\psi(\cdot, \tau)\|_{0, \alpha; \{|s|< 2a_{*}e^{\frac{\tau-\tau_{*}}{2}}\}}  + \|(e^{\frac{\tau_{*}-\tau}{2}}a_{*}^{-1})^{2}u(\cdot, \tau_1)\|_{2, \alpha;\{|s|< 2a_{*}e^{\frac{\tau-\tau_{*}}{2}}\}}\right]\\
    &\leq C_3\left[ \sup_{\tau_{1}<\tau\leq\tau_{*}}\| a_{*}^{-2}u(\cdot, \tau) \|_{0;\{|s|<2a_{*}\}}\right.\\
    &\left.\hspace{1cm}+ \sup_{\tau_{*}-\tau_{1}<\tau\leq\tau_{*}}\| \psi(\cdot, \tau) \|_{0;\{|s|<2a_{*}\}} + \|a_{*}^{-2}u(\cdot, \tau_{1})\|_{2, \alpha;\{|s|<2a_{*}\}}\right].
\end{align*}
Thus,
\begin{align*}
    &\sum_{j=0}^{2}\sup_{|s|<a_{*}}a_{*}^{j-2}|\partial_{s}^{j}u(\cdot, \tau_{*})| + \sup_{\substack{|s|<a_{*},\\ |s'-s|<|s|,\: |s''-s|<|s|}}a_{*}^{\alpha}\frac{|\partial_{s}^{2}u(s', \tau_{*}) - \partial_{s}^{2}u(s'', \tau_{*})|}{|s'-s''|^{\alpha}}\\
    &\leq C_1\left[ \sup_{\tau_{1}<\tau\leq\tau_{*}}\| a_{*}^{-2}u(\cdot, \tau) \|_{0;\{|s|<2a_{*}\}}\right.\\
    &\left.\hspace{1cm}+ \sup_{\tau_{1}<\tau\leq\tau_{*}}\| \psi(\cdot, \tau) \|_{0;\{|s|<2a_{*}\}} + \|a_{*}^{-2}u(\cdot, \tau_{1})\|_{2, \alpha;\{|s|<2a_{*}\}}\right].
\end{align*}
Multiplying both sides by $e^{\kappa\tau_{*}}a_{*}^{-\gamma}$ and using that $\tau_{*}-\tau_{1}\leq\log 2$ and the fact that $\rho_{a}(s)$ is comparable to $a_{*}$ yields
\begin{align}\label{eq: inner region reg. 2}
    &\sum_{j=0}^{2}e^{\kappa\tau_{*}}\sup_{|s|<a_{*}}a_{*}^{-\gamma-2+j}|\partial_{s}^{j}u(\cdot, \tau_{*})| + \sup_{\substack{|s|<a_{*},\\ |s'-s|<|s|,\: |s''-s|<|s|}}e^{\kappa\tau_{*}} a_{*}^{-\gamma-2+\alpha}\frac{|\partial_{s}^{2}u(s', \tau_{*}) - \partial_{s}^{2}u(s'', \tau_{*})|}{|s'-s''|^{\alpha}}\nonumber\\
    &\leq C_4\left[  \sup_{\tau_{1}<\tau\leq\tau_{*}}e^{\kappa\tau}\| u(\cdot, \tau) \|^{(\gamma+2)}_{0;\{|s|<2a_{*}\}} + \sup_{\tau_{1}<\tau\leq\tau_{*}}e^{\kappa\tau}\|\psi(\cdot, \tau) \|^{(\gamma)}_{0;\{|s|<2a_{*}\}} \right],
\end{align}
where $C_{4} = C_{4}(n, \alpha, \gamma, \kappa, C_0, \kappa')>0$.

Next, we derive the estimate for the annulus region $\{\frac{3a(\tau_{*})}{4}<|s|<\frac{7s_{0}}{16}\}$. The argument is similar to that for the outer region. For any  $s_* \in ( -\frac{s_{0}}4, -a_{*},)\cup(a_{*}, \frac {s_{0}}{4})$, set
\begin{align*}
    U(x, t) := (-t)\cdot u\left(\frac{x}{\sqrt{-t}}, \tau_{*}-\log(-ts_{*}^{2})\right),\quad \Psi :=\psi\left(\frac{x}{\sqrt{-t}}, \tau_{*}-\log(-ts_{*}^{2})\right).
\end{align*}
Then $U$ and $\Psi$ solve
\begin{align*}
    &\partial_{t}U = L_{x, \widehat{a}(t)}U + \Psi\\
    &\mbox{on the spacetime domain}\;\left\{ (x, t)\::\: |x|<\sqrt{-t}s_{0},\;-\frac{e^{\tau_{*}-\tau_{1}}}{s_{*}^{2}}<t<-\frac{e^{\tau_{*}-\tau_{2}}}{s_{*}^{2}}\right\}, 
\end{align*}
where $\widehat{a}(t) := \sqrt{-t}\:a(\tau_{*} - \log(-ts_{*}^{2}))$.
Set $t_1 = \max \{ - \frac{e^{\tau_*-\tau_1}}{s_*^2} , - \frac1{s_*^2} - 2 \}$ and consider the region
    $$\Omega = \left\{ (x,t) \: : \: \frac12 < |x| < 2 \, , \, t_1 < t \le - \frac1{s_*^2} \right\}.$$
Note that, for all $t \in \left(t_1, - \frac1{s_*^2} \right]$, 
    $$s_0 \sqrt{-t} \ge s_0 \sqrt{  \frac1{s_*^{2}}} = \frac{s_0}{|s_*|} \ge \frac{s_0}{\frac{s_0}4} = 4 > 2,$$
Therefore, $$\Omega \subseteq \left\{ (x, t)\::\: |x|<\sqrt{-t}s_{0},\;-\frac{e^{\tau_{*}-\tau_{1}}}{s_{*}^{2}}<t<-\frac{e^{\tau_{*}-\tau_{2}}}{s_{*}^{2}}\right\},$$
and so $\partial_t U = L_{x, \widehat a(t)} U + \Psi$ on $\Omega$.
Since $|x| > \frac12$ on $\Omega$ and $\widehat{a}\overline{L}$ smoothly converges to a cone on $\frac12 < |x| < 2$ as $\widehat a \searrow 0$, the coefficients of $\partial_t U = L_{x, \widehat a} U  + \Psi$ are uniformly bounded by constants independent of $s_*, \tau_*$ (even if $\hat a$ gets arbitrarily close to 0).
Hence, the non-standard Schauder estimates can be applied.

There are two cases to consider based on whether $t_1 = -\frac{e^{\tau_* - \tau_1}}{s_*^2}$ or $- \frac1{s_*^2}-2$.

If $t_{1} = -\frac{1}{s_{*}^{2}}-2$, we apply the interior in spacetime non-standard Schauder estimate to obtain
\begin{align*}
    &\|U(\cdot, -s_{*}^{-2})\|_{C^{2, \alpha}\left((\frac{3}{4}, \frac{7}{4})\right)}\\
    &\leq C\left[ \sup_{t\in[t_1, -s_{*}^{-2}]}\|U(\cdot, t)\|_{C^{0}\left((\frac{1}{2}, 2)\right)} + \sup_{t\in[t_1, -s_{*}^{-2}]}\|\Psi(\cdot, t)\|_{C^{0, \alpha}\left((\frac{1}{2}, 2)\right)} \right]
\end{align*}
for some $C_{5} = C_{5}(n, \alpha)>0$.
Rescaling back, multiplying both sides by $e^{\kappa\tau_{*}}s_{*}^{-\gamma}$, and using the fact that $\rho_{a}(s_{*})$ is comparable to $s_{*}$ yields
\begin{align*}
   e^{\kappa\tau_{*}}\|u(\cdot, \tau_{*})\|^{(\gamma+2)}_{2, \alpha;(\frac{3s_{*}}{4}, \frac{7s_{*}}{4})} \leq C_{6}&\left[\sup_{\tau\in[\tau_{*}-\log(1+2s_{*}^{2}), \tau_{*}]}e^{\kappa\tau}\|u(\cdot, \tau)\|^{(\gamma+2)}_{0;\left(\frac{s_{*}}{2\sqrt{1+2s_{*}^{2}}}, 2s_{*}\right)}\right.\nonumber\\
   &\left. +\sup_{\tau\in[\tau_{*}-\log(1+2s_{*}^{2}), \tau_{*}]}e^{\kappa\tau}\|\psi(\cdot, \tau)\|^{(\gamma)}_{0, \alpha;\left(\frac{s_{*}}{2\sqrt{1+2s_{*}^{2}}}, 2s_{*}\right)} \right]
\end{align*}
for some $C_{6} = C_{6}(n, \alpha, \gamma, \kappa)>0$.
Taking supremum over $(s_*, \tau_*) \in [( -\frac{s_{0}}4, -a_{*},)\cup(a_{*}, \frac {s_{0}}{4})] \times ( \tau_1, \tau_2)$, we obtain
\begin{align}\label{eq: inner region reg. 3}
    &\sup_{\tau\in(\tau_1, \tau_2)}e^{\kappa\tau}\|u(\cdot, \tau)\|^{(\gamma+2)}_{2, \alpha; \{\frac{3a_{*}}{4}<|s|<\frac{7s_{0}}{16}\}}\nonumber\\
    &\leq C_{6}\left[ \sup_{\tau\in(\tau_1, \tau_2)}e^{\kappa\tau}\|u(\cdot, \tau)\|^{(\gamma+2)}_{0; \{\frac{a_{*}}{2}<|s|<s_{0}\}} +  \sup_{\tau\in(\tau_1, \tau_2)}e^{\kappa\tau}\|\psi(\cdot, \tau)\|^{(\gamma)}_{0, \alpha; \{\frac{a_{*}}{2}<|s|<s_{0}\}} \right].
\end{align}

If $t_1 = -\frac{e^{\tau_* - \tau_1}}{s_*^2}$, then $e^{\tau_{*}-\tau_{1}}\leq 1+2s_{*}^{2}\leq 1+2s_{0}^{2}$. A similar argument, but applying the interior in space non-standard Schauder estimate instead, yields
\begin{align}\label{eq: inner region reg. 4}
    &\sup_{\tau\in(\tau_1, \tau_2)}e^{\kappa\tau}\|u(\cdot, \tau)\|^{(\gamma+2)}_{2, \alpha; \{\frac{3a_{*}}{4}<|s|<\frac{7s_{0}}{16}\}}\nonumber\\
    &\leq C_{6}e^{\kappa(\tau_{*}-\tau_{1})}\left[ \sup_{\tau\in(\tau_1, \tau_2)}e^{\kappa\tau}\|u(\cdot, \tau)\|^{(\gamma+2)}_{0; \{\frac{a_{*}}{2}<|s|<s_{0}\}} +  \sup_{\tau\in(\tau_1, \tau_2)}e^{\kappa\tau}\|\psi(\cdot, \tau)\|^{(\gamma)}_{0, \alpha; \{\frac{a_{*}}{2}<|s|<s_{0}\}}\right.\nonumber\\ 
    &\left.\hspace{7cm} +e^{\kappa\tau_{1}}\|u(\cdot, \tau_{1})\|^{(\gamma+2)}_{2, \alpha;\{\frac{a_{*}}{2}<|s|<s_{0}\}} \right]\nonumber\\
    &\leq C_{7}\left[ \sup_{\tau\in(\tau_1, \tau_2)}e^{\kappa\tau}\|u(\cdot, \tau)\|^{(\gamma+2)}_{0; \{\frac{a_{*}}{2}<|s|<s_{0}\}} +  \sup_{\tau\in(\tau_1, \tau_2)}e^{\kappa\tau}\|\psi(\cdot, \tau)\|^{(\gamma)}_{0, \alpha; \{\frac{a_{*}}{2}<|s|<s_{0}\}}\right.\nonumber\\ 
    &\left.\hspace{7cm} +e^{\kappa\tau_{1}}\|u(\cdot, \tau_{1})\|^{(\gamma+2)}_{2, \alpha;\{\frac{a_{*}}{2}<|s|<s_{0}\}} \right],
\end{align}
where $C_{7} = C_{7}(n, \alpha, \gamma, \kappa, s_{0}) = C_{6}(1+2s_{0}^{2})^{\kappa}>0$.

Putting (\ref{eq: inner region reg. 1}), (\ref{eq: inner region reg. 2}), (\ref{eq: inner region reg. 3}), and (\ref{eq: inner region reg. 4}) together, we obtain the desired estimate.
\end{proof}

\subsubsection{Sup Norm Estimate} \label{subsubsect Sup Norm Est Inner Region}
The weighted sup norm estimate will be obtained by a blow-up argument, inspired by \cite{BK17}*{Proposition~6.2}. The argument relies on Liouville-type results for ancient or immortal solutions of parabolic equations on $\mathbb{R}$ of the form
\begin{align*}
\partial_{t} w = \mathbf{L}w.
\end{align*}
Here, either $\mathbf{L}=L_{\sigma,0}$, in which case $w$ may be regarded as a radial solution of the heat equation on the $G$-invariant cone $\mathcal{C}$, or $\mathbf{L}=L_{\sigma,1}$, in which case $w$ may be regarded as a radial solution of the heat equation on the $G$-invariant SL $\overline{L}$. By Remark~\ref{rem: indep. of G}, we may choose $G=SO(n)$. In this case, $\mathcal{C}=\mathcal{P}_{1}\cup\mathcal{P}_{2}$ is a transverse union of planes, and $\overline{L}$ is the Lagrangian catenoid.

It therefore suffices to establish the following Liouville-type results.

\begin{lemma}[\cite{STW}*{Proposition~6.4}]\label{lem: Liouville ancient heat on SL}
    Let $\overline{L}$ be a Lagrangian catenoid in $\mathbb{C}^{n}$. Suppose that $w:\overline{L}\times(-\infty, T]\to\mathbb{R}$ is an ancient solution to the heat equation
    \begin{align}
        \partial_{t}w = \Delta_{\overline{L}}w,
    \end{align}
    where $\Delta_{\overline{L}}$ is the induced Laplace operator on $\overline{L}$. If in addition $w$ satisfies
    \begin{align}
        \sup_{\mathbf{x}\in\overline{L}}\sqrt{1+|\mathbf{x}|^{2}}^{\:-\gamma}|w(\mathbf{x}, t)|\leq 1,\quad\forall t\in(-\infty, T]
    \end{align}
    for some $\gamma\in(2-n, 0)$, then $w\equiv 0$.
\end{lemma}

\begin{lemma}\label{lem: uniqueness of heat eq on cone}
Let $\mathcal{C} = \mathcal{P}_1\cup \mathcal{P}_2 \subset \mathbb{C}^n$ be a cone consisting of
two transversely intersecting planes. Suppose that $w$ is a solution of
$\partial_t w = \Delta_{\mathcal{C}} w$ on
$\mathcal{C}\setminus\{0\}\times(0,\infty)$, satisfying
\[
|w(\mathbf x,t)|\le |\mathbf x|^\gamma
\]
for some $\gamma\in(2-n,0)$, and such that
\[
w(\cdot,t)\to 0 \quad\text{in }L^1_{\mathrm{loc}}(\mathcal{C})
\quad\text{as }t\to 0.
\]
Then $w\equiv 0$.
\end{lemma}
\begin{proof}
    Denote $w_{j}:=w\big|_{\mathcal{P}_{j}\setminus\{0\}}$. Since
    $|w_{j}(\mathbf{x}, t)|\leq |\mathbf{x}|^{\gamma}$ for some
    $\gamma\in(2-n, 0)$, each $w_{j}$ is locally integrable near the origin
    and therefore is a distributional solution of the heat equation on $\mathcal{P}_{j}$.
    It follows that for every
    $(\mathbf{x}_{0}, t_{0})\in \mathcal{P}_{j}\setminus\{0\}\times(0, \infty)$ and every
    $t\in(0, t_{0})$,
    \begin{align*}
        w_{j}(\mathbf{x}_{0}, t_{0})
        =
        \int_{\mathcal{P}_{j}}
        \frac{1}{(4\pi(t_{0}-t))^{n/2}}
        e^{-\frac{|\mathbf{x} - \mathbf{x}_{0}|^{2}}{4(t_{0}-t)}}
        w_{j}(\mathbf{x}, t)\:dV_{\mathcal{P}_{j}}(\mathbf{x}).
    \end{align*}
    Hence, for every $R>2|\mathbf{x}_{0}|$,
    \begin{align*}
        |w(\mathbf{x}_{0}, t_{0})|\leq I^{1}_{t, R} + I^{2}_{t, R},
    \end{align*}
    where
    \begin{align*}
        I^{1}_{t, R}
        &:=
        \int_{|\mathbf{x}|\leq R}
        \frac{1}{(4\pi(t_{0}-t))^{n/2}}
        e^{-\frac{|\mathbf{x} - \mathbf{x}_{0}|^{2}}{4(t_{0}-t)}}
        |w_{j}(\mathbf{x}, t)|\:dV_{\mathcal{P}_{j}}(\mathbf{x}),\\
        I^{2}_{t, R}
        &:=
        \int_{|\mathbf{x}|> R}
        \frac{1}{(4\pi(t_{0}-t))^{n/2}}
        e^{-\frac{|\mathbf{x} - \mathbf{x}_{0}|^{2}}{4(t_{0}-t)}}
        |w_{j}(\mathbf{x}, t)|\:dV_{\mathcal{P}_{j}}(\mathbf{x}).
    \end{align*}

    Fix $R>2|\mathbf{x}_{0}|$. Since $t_{0}-t\to t_{0}>0$ as $t\to0$, the heat
    kernel is uniformly bounded on $\{|\mathbf{x}|\le R\}$ for all sufficiently
    small $t$. Thus, by the $L^{1}_{\rm loc}$ convergence of $w_j(\cdot,t)$ to
    $0$, we have
    \[
        I^{1}_{t,R}\to 0
        \qquad\text{as } t\to0.
    \]

    On the other hand, using $|w_{j}(\mathbf{x}, t)|\le |\mathbf{x}|^\gamma$
    and $\gamma<0$, we estimate
    \begin{align*}
        I^{2}_{t, R}\leq
        R^{\gamma}\int_{\mathcal{P}_{j}}
        \frac{1}{(4\pi(t_{0}-t))^{n/2}}
        e^{-\frac{|\mathbf{x} - \mathbf{x}_{0}|^{2}}{4(t_{0}-t)}}
        \:dV_{\mathcal{P}_{j}}(\mathbf{x}) = R^{\gamma}.
    \end{align*}
    Therefore, letting $t\to0$ with $R$ fixed, we obtain
    \[
        |w(\mathbf{x}_{0}, t_{0})|\le R^\gamma.
    \]
    Since $\gamma<0$, letting $R\to\infty$ yields
    \[
        |w(\mathbf{x}_{0}, t_{0})|=0.
    \]
    As $(\mathbf{x}_{0}, t_{0})$ was arbitrary, this proves that $w\equiv 0$.
\end{proof}

\begin{lemma}[\cite{BK17}*{Proposition~5.3}, \cite{STW}*{Proposition~6.5}]\label{lem: Liouville ancient heat on the cone}
    Let $\mathcal{C} = \mathcal{P}_{1}\cup \mathcal{P}_{2} \subset \mathbb{C}^n$ be a cone consisting of
two transversely intersecting planes. Suppose that we have an ancient solution $w:\mathcal{C}\setminus\{0\}\times(-\infty, T]\to\mathbb{R}$ of the heat equation
    \begin{align}
        \partial_{t}w = \Delta_{\mathcal{C}}w
    \end{align}
    on $\mathcal{C}\setminus\{0\}$, where $\Delta_{\mathcal{C}}$ is the Laplacian on $\mathcal{C}$ with respect to the induced cone metric. Then, if in addition $w$ satisfies the uniform estimate
    \begin{align}
        |w(\mathbf{x}, t)|\leq |\mathbf{x}|^{\gamma}\quad\forall(\mathbf{x}, t)\in\mathcal{C}\setminus\{0\}\times(-\infty, T],
    \end{align}
    for some $\gamma\in(2-n, 0)$, then $w\equiv 0$.
\end{lemma}

\begin{proposition}[Weighted Sup Norm Estimate]\label{prop: weighted sup norm estimate intermediate}
    Given constants $s_{0}, C_{0}, \kappa, \kappa'>0$, and $\gamma\in(2-n, 0)$. There exist $s_{0}'\in(0, s_{0})$, $\tau_{0}>0$, and $C>0$, with the following significance: For any $\tau_{1}, \tau_{2}>0$ with $\tau_{2}>\tau_{1}\geq\tau_{0}$ and $a:(\tau_{1}, \tau_{2})\to(0, \infty)$ satisfying
    \begin{align}\label{eq: a bound for sup norm interior est}
        0 < C_0^{-1} e^{-\kappa' \tau} \le a(\tau) \le C_0 e^{-\kappa' \tau},
    \end{align}
    if $u$ and $\psi$ are odd functions that solve
    \begin{align*}
        \partial_{\tau}u = H_{a}u  + \psi\quad \mbox{on}\quad(-s_0, s_0)\times(\tau_{1}, \tau_{2}),
    \end{align*}
    then
    \begin{multline}
        \sup_{\tau\in(\tau_{1}, \tau_{2})}e^{\kappa\tau}\|u(\cdot, \tau)\|^{(\gamma)}_{0;\{|s|<s_{0}\}} \\
        \leq C\left[ \sup_{\tau\in(\tau_{1}, \tau_{2})}e^{\kappa\tau}\|u(\cdot, \tau)\|^{(\gamma)}_{0;\{s_{0}'<|s|<s_{0}\}} + e^{\kappa\tau_{1}}\|u(\cdot, \tau_{1})\|^{(\gamma)}_{2, \alpha;\{|s|<s_{0}\}} \right. \\
        \left. + \sup_{\tau\in(\tau_{1}, \tau_{2})}e^{\kappa\tau}\|\psi(\cdot, \tau)\|^{(\gamma-2)}_{0,\alpha;\{|s|<s_{0}\}}\right].
    \end{multline}
\end{proposition}

\begin{proof}
Suppose the assertion does not hold, then there exist sequences of constants $s_{i}'\to 0$, $\tau_{2, i}>\tau_{1, i}\to\infty$, $M_{i}\to\infty$, and sequences of functions $u_{i}, \psi_{i}:(-s_{0}, s_{0})\times(\tau_{1, i}, \tau_{2, i})\to\mathbb{R}$ satisfying
\begin{equation*}
    \partial_{\tau}u_{i} = H_{s, a}u_{i}+ \psi_{i} \quad\mbox{on}\;(-s_{0}, s_{0})\times(\tau_{1, i}, \tau_{2, i}),
\end{equation*} 
and
\begin{multline} \label{weighted sup norm est, eqn contradiction assumption}
        \sup_{\tau\in(\tau_{1,i}, \tau_{2, i})}e^{\kappa\tau}\|u_{i}(\cdot, \tau)\|^{(\gamma)}_{0;\{|s|<s_{0}\}} \\
        > M_{i}\left[ \sup_{\tau\in(\tau_{1}, \tau_{2, i})}e^{\kappa\tau}\|u_{i}(\cdot, \tau)\|^{(\gamma)}_{0;\{s_{i}'<|s|<s_{0}\}} + e^{\kappa\tau_{1,i}}\|u_{i}(\cdot, \tau_{1,i})\|^{(\gamma)}_{2,\alpha;\{|s|<s_{0}\}}\right. \\
        \left.\hspace{6cm}+ \sup_{\tau\in(\tau_{1,i}, \tau_{2, i})}e^{\kappa\tau}\|\psi_{i}(\cdot, \tau)\|^{(\gamma-2)}_{0,\alpha;\{|s|<s_{0}\}}\right].
\end{multline}
After normalization one may assume
\begin{align*}
\sup_{\tau\in(\tau_{1,i}, \tau_{2, i})}e^{\kappa\tau}\|u_{i}(\cdot, \tau)\|^{(\gamma)}_{0;\{|s|<s_{0}\}} = \sup_{(s, \tau)\in (-s_{0}, s_{0})\times(\tau_{1,i}, \tau_{2, i})}e^{\kappa\tau}\rho_{a(\tau)}^{-\gamma}(s)|u_{i}(s, \tau)| = 1\quad\forall i\in\mathbb{N}.
\end{align*}
Then the contradiction assumption \eqref{weighted sup norm est, eqn contradiction assumption} implies
\begin{multline}\label{eq: boundary condition blowup sequence}
\sup_{\tau\in(\tau_{1,i}, \tau_{2, i})}e^{\kappa\tau}\|u_{i}(\cdot, \tau)\|^{(\gamma)}_{0;\{s_{i}'<|s|<s_{0}\}} + e^{\kappa\tau_{1,i}}\|u_{i}(\cdot, \tau_{1,i})\|^{(\gamma)}_{2,\alpha;\{|s|<s_{0}\}} \\
+ \sup_{\tau\in(\tau_{1,i}, \tau_{2, i})}e^{\kappa\tau}\|\psi_{i}(\cdot, \tau)\|^{(\gamma-2)}_{0,\alpha;\{|s|<s_{0}\}}<\frac{1}{M_{i}}.
\end{multline}
By Proposition~\ref{prop: weighted Schauder estimate intermediate}, this implies the uniform bound
\begin{equation}
    \sup_{\tau\in(\tau_{1, i}, \tau_{2, i})}e^{\kappa\tau}\|u_{i}(\cdot, \tau)\|^{(\gamma)}_{2,\alpha;\{|s|<s_{0}/2\}}\leq C
    \qquad (\forall i \in \mathbb N).
\end{equation}

For each $i$, choose $(p_{i}, t_{i})\in [-s_{0}, s_{0}]\times[\tau_{1, i}, \tau_{2, i}]$ such that 
\begin{align}
e^{\kappa t_{i}}\rho_{a(t_{i})}^{-\gamma}(p_{i})|u_{i}(p_{i}, t_{i})| \geq \frac{1}{2},\label{eq: nontrivial pt blowup sequence}
\end{align}
We claim $|p_{i}|\to 0$. If not, then $(p_{i}, t_{i})\in\{(s, \tau)\:|\:s_{i}'\leq|s|\leq s_{0},\;\tau\in[\tau_{1,i}, \tau_{2, i}]\}$ for all sufficiently large $i$. It follows that
\begin{align}
    \frac{1}{2}\leq e^{\kappa t_{i}}\rho_{a(t_{i})}^{-\gamma}(p_{i})|u_{i}(p_{i}, t_{i})| \leq \sup_{\tau\in(\tau_{1,i}, \tau_{2, i})}e^{\kappa\tau}\|u_{i}(\cdot, \tau)\|^{(\gamma)}_{0;\{s_{i}'<|s|<s_{0}\}}<\frac{1}{M_{i}},
\end{align}
which is a contradiction.

According to the convergence rate of $|p_{i}|$ compared with $a(t_{i})$ as $i\to\infty$, we have the following cases:

\noindent{\it Case 1}: ($\limsup_{i\to\infty}\frac{|p_{i}|}{a(t_{i})} <\infty$). \;Define rescaled functions
\begin{align*}
    &w_{i}(\sigma, t) := e^{\kappa t_{i}}a(t_{i})^{-\gamma}u_{i}\left(a(t_{i})\sigma, t_{i} + t a(t_{i})^{2}\right),\\
    &\varphi_{i}(\sigma, t):= e^{\kappa t_{i}}a(t_{i})^{-\gamma-2}\psi_{i}\left(a(t_{i})\sigma, t_{i}+ta(t_{i})^{2}\right).
\end{align*}
For each $i\in\mathbb{N}$, the domain of $w_{i}$ and $\varphi_{i}$ is given by
\begin{align*}
    \mathcal{D}_{i} := \left\{(\sigma, t)\::\:|\sigma|<\tfrac{s_{0}}{a(t_{i})},\;\tfrac{\tau_{1, i}-t_{i}}{a(t_{i})^{2}}\leq t<\tfrac{\tau_{2, i}-t_{i}}{a(t_{i})^{2}}\right\}.
\end{align*}
On $\mathcal{D}_{i}$, $w_{i}$ and $\varphi_{i}$ satisfy the equation
\begin{align}
    \partial_{t}w_{i}(\sigma, t) =  L_{\sigma, \tfrac{a(t_{i}+ta(t_{i})^{2})}{a(t_{i})}}w_{i}(\sigma, t) - a(t_{i})^{2}\left(\frac{\sigma}{2}\partial_{\sigma}w_{i}(\sigma, t) - w_{i}(\sigma, t) \right) +\varphi_{i}(\sigma, t),
\end{align}
and the estimates
\begin{align}
    \sum_{k=0}^{2}\sqrt{\sigma^{2} + \tfrac{a^{2}(t_{i}+ta(t_{i})^{2})}{a(t_{i})^{2}}}^{\:-\gamma+k}|\partial_{\sigma}^{k}w_{i}(\sigma, t)|&\leq Ce^{-\kappa t a(t_{i})^{2}},\label{eq: blowup seq. weighted sup bound tip region}\\
    \sqrt{\sigma^{2} + \tfrac{a(t_{i}+ta(t_{i})^{2})^{2}}{a(t_{i})^{2}} }^{\:-\gamma+2}|\varphi_{i}(\sigma, t)| &\leq\frac{1}{M_{i}}e^{-\kappa ta(t_{i})^{2}}.\label{eq: contrad. assumpt. nonhom. term tip region}
\end{align}
Let $\sigma_{i} := \frac{p_{i}}{a(t_{i})}$. Then $|\sigma_{i}|$ is uniformly bounded. Along the spacetime sequence $(\sigma_{i}, 0)\in\mathcal{D}_{i}$,
\begin{align}\label{eq: nontriviality of limit tip region}
    |w_{i}(\sigma_{i}, 0)| \geq \frac{1}{2}\left(\frac{|p_{i}|^{2} + a(t_{i})^{2}}{a(t_{i})^{2}}\right)^{\gamma/2}\geq r_{0}
\end{align}
for some $r_{0} = r_{0}(\gamma)>0$. 

Assume that the initial time $\frac{\tau_{1, i}-t_{i}}{a(t_{i})^{2}}$ stays bounded as $i\to\infty$. Then there is $I>0$ such that $-Ia(t_{i})^{2}\leq \tau_{1, i} - t_{i}\leq 0$. Therefore, from (\ref{eq: boundary condition blowup sequence}) and (\ref{eq: a bound for sup norm interior est}) we obtain
\begin{align}\label{eq: blow up sequence initial condition tip region}
    \left|w_{i}\left(\sigma, \tfrac{\tau_{1, i}-t_{i}}{a(t_{i})^{2}}\right)\right|\leq \frac{1}{M_{i}}e^{\kappa I a(t_{i})^{2}}\left(\frac{a^{2}(\tau_{1, i})}{a^{2}(t_{i})} + \sigma^{2}\right)^{\gamma/2}\to 0\quad\mbox{as}\; i\to\infty.
\end{align}
After passing to a subsequence, (\ref{eq: blowup seq. weighted sup bound tip region}), (\ref{eq: contrad. assumpt. nonhom. term tip region}), (\ref{eq: nontriviality of limit tip region}), and (\ref{eq: blow up sequence initial condition tip region}) imply that $w_{i}$ converges in $C^{1, \alpha}_{\rm loc}$-sense to a non-trivial solution $w_{\infty}$ of 
\begin{align}\label{eq: radial heat eq case 1.1}
    \partial_{t}w_{\infty} = L_{\sigma, 1}w_{\infty}\quad\mbox{on}\quad\mathbb{R}\times[T_{0}, T_{0}')
\end{align}
for some $T_{0}, T_{0}'$ satisfying $-\infty<T_{0}\leq 0\leq T_{0}'\leq\infty$, with the initial condition $w_{\infty}(\cdot, T_{0})\equiv 0$ and the uniform estimate $|w_{\infty}(\sigma, t)|\leq \sqrt{1+\sigma^{2}}^{\gamma}$ for all $t\in[T_{0}, T_{0}')$. It follows from the maximum principle that $w_{\infty}\equiv 0$, a contradiction.

We may now assume $\frac{\tau_{1, i}-t_{i}}{a(t_{i})^{2}}\to-\infty$ as $i\to\infty$. After passing to a subsequence, (\ref{eq: blowup seq. weighted sup bound tip region}), (\ref{eq: contrad. assumpt. nonhom. term tip region}) and (\ref{eq: nontriviality of limit tip region}) imply that $w_{i}$ converges in $C^{1, \alpha}_{\rm loc}$-sense to a non-trivial ancient solution $w_{\infty}$ of 
\begin{align}\label{eq: radial heat eq case 1.2}
    \partial_{t}w_{\infty} = L_{\sigma, 1}w_{\infty}\quad\mbox{on}\quad\mathbb{R}\times(-\infty, T_{1})
\end{align}
for some $0\leq T_{1}\leq\infty$, with the uniform estimate $|w_{\infty}(\sigma, t)|\leq \sqrt{1+\sigma^{2}}^{\gamma}$ for all $t\in(-\infty, T_{1})$. Since solutions to (\ref{eq: radial heat eq case 1.2}) are radially symmetric solutions to the heat equation on the asymptotically conical special Lagrangian $\overline{L}$, we may apply Lemma~\ref{lem: Liouville ancient heat on SL} to conclude that $w_{\infty}\equiv 0$, a contradiction.

\ 

\noindent{\it Case 2}: ($\limsup_{i\to\infty}\frac{|p_{i}|}{a(t_{i})} =\infty$). \; Define rescaled functions
\begin{align*}
    &w_{i}(\sigma, t) := e^{\kappa t_{i}}|p_{i}|^{-\gamma}u_{i}\left(|p_{i}|\sigma, t_{i} + t |p_{i}|^{2}\right),\\
    &\varphi_{i}(\sigma, t):= e^{\kappa t_{i}}|p_{i}|^{-\gamma-2}\psi_{i}\left(|p_{i}|\sigma, t_{i}+t|p_{i}|^{2}\right).
\end{align*}
For each $i\in\mathbb{N}$, the domain of $w_{i}$ and $\varphi_{i}$ is given by
\begin{align*}
    \mathcal{D}_{i} := \left\{(\sigma, t)\::\:\tfrac{a(t_{i})}{|p_{i}|}<|\sigma|<\tfrac{s_{0}}{|p_{i}|},\;\tfrac{\tau_{1, i}-t_{i}}{|p_{i}|^{2}}\leq t<\tfrac{\tau_{2, i}-t_{i}}{|p_{i}|^{2}}\right\}.
\end{align*}
On $\mathcal{D}_{i}$, $w_{i}$ and $\varphi_{i}$ satisfy the equation
\begin{align}
    \partial_{t}w_{i}(\sigma, t) =  L_{\sigma, \tfrac{a(t_{i}+t|p_{i}|^{2})}{|p_{i}|}}w_{i}(\sigma, t) - |p_{i}|^{2}\left(\frac{\sigma}{2}\partial_{\sigma}w_{i}(\sigma, t) - w_{i}(\sigma, t) \right) +\varphi_{i}(\sigma, t),
\end{align}
and the estimates
\begin{align}
    \sum_{k=0}^{2}\sqrt{\sigma^{2} + \tfrac{a^{2}(t_{i}+t|p_{i}|^{2})}{|p_{i}|^{2}}}^{\:-\gamma+k}|\partial_{\sigma}^{k}w_{i}(\sigma, t)|&\leq Ce^{-\kappa t |p_{i}|^{2}},\label{eq: blowup seq. weighted sup bound interm. region}\\
    \sqrt{\sigma^{2} + \tfrac{a^{2}(t_{i}+t|p_{i}|^{2})}{|p_{i}|^{2}} }^{\:-\gamma+2}|\varphi_{i}(\sigma, t)| &\leq\frac{1}{M_{i}}e^{-\kappa t|p_{i}|^{2}}.\label{eq: contrad. assumpt. nonhom. term interm. region}
\end{align}
Moreover, at $(\sigma, t) = (1, 0)$ we have 
\begin{align}\label{eq: blowup limit nontrivial interm. region}
    |w_{i}(1, 0)| \geq \frac{1}{2}\left(\frac{a(t_{i})^{2}+|p_{i}|^{2}}{|p_{i}|^{2}}\right)^{\gamma/2}\geq r_{1}
\end{align}
for some $r_{1} = r_{1}(\gamma)>0$.

Similar to Case 1, we deal with the case where $\frac{\tau_{1, i} - t_{i}}{|p_{i}|^{2}}$ stays bounded first, that is, there is $J>0$ such that $-J|p_{i}|^{2}\leq\tau_{1, i} - t_{i}\leq 0$. At the initial time, using (\ref{eq: a bound for sup norm interior est}) we have 
\begin{align}\label{eq: contrad. assumpt. inital time interm. region}
\left|w_{i}\left(\sigma, \tfrac{\tau_{1, i} - t_{i}}{|p_{i}|^{2}}\right)\right|\leq \frac{1}{M_{i}}e^{\kappa J|p_{i}|^{2}}\left(\frac{a(\tau_{1, i})^{2}}{|p_{i}|^{2}} + \sigma^{2}\right)^{\gamma/2}\to 0\quad\mbox{on compact subsets of $\mathbb{R}\setminus\{0\}$},
\end{align}
as $i\to\infty$.

After passing to a subsequence, (\ref{eq: blowup seq. weighted sup bound interm. region}), (\ref{eq: contrad. assumpt. nonhom. term interm. region}), (\ref{eq: blowup limit nontrivial interm. region}) and (\ref{eq: contrad. assumpt. inital time interm. region}) imply that $w_{i}$ converges in $C^{1, \alpha}_{\rm loc}(\mathbb{R}\setminus\{0\}\times(-J, \infty) )$ to a non-trivial solution $w_{\infty}$ of 
\begin{align*}
    \partial_{t}w_{\infty} = L_{\sigma, 0}w_{\infty}\quad\mbox{on}\quad\mathbb{R}\setminus\{0\}\times[T_{2}, T_{2}'),
\end{align*}
for some $-\infty<-J\leq T_{2}\leq 0\leq T_{2}'\leq\infty$, with
\begin{align*}
    w_{\infty}(\cdot, t)\to 0\quad\mbox{in}\; C^{0}_{\rm loc}(\mathbb{R}\setminus\{0\}),\quad\mbox{as $t \to T_{2}^{+}$}
\end{align*}
and satisfies the estimate $|w_{\infty}(\sigma, t)|\leq |\sigma|^{\gamma}$ for all $\sigma\neq 0$, $t\in[T_{2}, \infty)$. Viewing $w_{\infty}(\cdot, t)$ as a radially symmetric solution to the heat equation on the cohomogeneity-one special Lagrangian cone $\mathcal{C} = \mathcal{P}_{1}\cup \mathcal{P}_{2}$ consisting of a transverse pair of planes, the growth rate $\gamma\in(2-n, 0)$ implies that $w_{\infty}(\cdot, t)$ is a solution to the heat equation on $\mathcal{P}_{1}\cup \mathcal{P}_{2}$ in the distributional sense, and together with (\ref{eq: contrad. assumpt. inital time interm. region}) yields $w_{\infty}(\cdot, t)\to 0$ in $L^{1}_{\rm loc}$. By Lemma~\ref{lem: uniqueness of heat eq on cone}, $w_{\infty}\equiv 0$. This is a contradiction.

It remains to consider the case where $\frac{\tau_{1, i}-t_{i}}{|p_{i}|^{2}}\to-\infty$ as $i\to\infty$. After passing to a subsequence, (\ref{eq: blowup seq. weighted sup bound interm. region}), (\ref{eq: contrad. assumpt. nonhom. term interm. region}) and (\ref{eq: blowup limit nontrivial interm. region}) imply that $w_{i}$ converges in $C^{1, \alpha}_{\rm loc}$-sense to a non-trivial ancient solution $w_{\infty}$ of 
\begin{align*}
    \partial_{t}w_{\infty} = L_{\sigma, 0}w_{\infty}\quad\mbox{on}\quad\mathbb{R}\setminus\{0\}\times(-\infty, T_{3})
\end{align*}
for some $0\leq T_{3}\leq\infty$, with the uniform estimate $|w_{\infty}(\sigma, t)|\leq |\sigma|^{\gamma}$ for all $t\in(-\infty, T_{3})$. Viewing $w_{\infty}(\cdot, t)$ as a radially symmetric ancient solution to the heat equation on the special Lagrangian cone $\mathcal{C} = \mathcal{P}_{1}\cup \mathcal{P}_{2}$, Lemma~\ref{lem: Liouville ancient heat on the cone} implies that $w_{\infty}\equiv 0$, a contradiction.
\end{proof}


\begin{lemma} \label{lem: inner region Q estimate}
Given $s_{0}>0$, $\tau_{2}>\tau_{1}>0$, $\gamma\in (2-n, 0)$, $\alpha\in(0, 1)$, and $C_0 > 0$, there exist constants $C = C(n, s_0, \alpha, \gamma) > 0$ and $C'= C'(n, s_0, \alpha, \gamma, C_0)>0$ such that the following hold:

If $0 < a(\tau) \le 1$ for all $\tau \in (\tau_1, \tau_2)$ and if $u : (-s_0, s_0) \times (\tau_1, \tau_2) \to \R$ satisfies
$$\sup_{\tau \in (\tau_1, \tau_2)} \| u (\cdot, \tau) \|_{2, \alpha; (-s_0, s_0)}^{(2)} \le \frac14,$$ then 
\begin{equation} \label{eqn: inner region Q est 1}
    \sup_{\tau\in(\tau_{1}, \tau_{2})}e^{\kappa\tau}\|Q_{a}(\cdot, \tau)\|^{(\gamma-2)}_{0,\alpha;(-s_{0}, s_{0})}
        \leq C \sup_{\tau \in (\tau_1, \tau_2)} e^{\kappa \tau} \left( \| u_s \|_{1, \alpha; (-s_0, s_0)}^{(\frac \gamma 2)} \right)^2.
\end{equation}

If additionally $u = \phi + v$ and there exists $K \ge 2$ such that $\kappa \ge \frac K2(2-\gamma)$ and $0 < a(\tau) \le C_0 e^{-\frac K2 \tau}$ for all $\tau \in (\tau_1, \tau_2)$, then
\begin{multline} \label{eqn: inner region Q est 2}
        \sup_{\tau\in(\tau_{1}, \tau_{2})}e^{\kappa\tau}\|Q_{a}(\cdot, \tau)\|^{(\gamma-2)}_{0,\alpha;(-s_{0}, s_{0})}
        \\ 
        \leq C' \sup_{\tau\in(\tau_{1}, \tau_{2})}\left[ e^{\kappa\tau}\|\phi(\cdot, \tau)\|^{2}_{C^{2,\alpha}((-s_{0}, s_{0}))} + \left(e^{\kappa\tau}\|v(\cdot, \tau)\|^{(\gamma)}_{2,\alpha;(-s_{0}, s_{0})}\right)^{2}\right].
\end{multline}
\end{lemma}

\begin{proof}
    Estimate \eqref{eqn: inner region Q est 1} follows directly from Lemma \ref{lem weighted Holder est for Q}.

    Now assume $u = \phi +v$ and $0 < a(\tau) \le C_0 e^{\frac K2 \tau}$ for some $K \ge 2$ with $\kappa \ge \frac K2(2 -\gamma)$.
    Because $\gamma < 0$ and $u = \phi + v$, the definition of the weighted H{\"o}lder norms (Definition \ref{defn weighted Holder norms}) implies the right-hand side can be estimated as
    \begin{align*}
        e^{\kappa \tau} \left( \| u_s \|_{1, \alpha; (-s_0, s_0) }^{(\frac \gamma 2)}\right)^2
        &\leq C \left[ e^{\kappa \tau} \| \phi \|_{C^{2, \alpha}((-s_0, s_0))}^2 + e^{\kappa \tau} \left( \sup_{s \in (-s_0, s_0) } \rho_a(s)^{\gamma - 2} \right) \left( \| v \|_{2, \alpha ; (-s_0, s_0) }^{(\gamma)} \right)^2 \right]\\
        &\leq C \left[ e^{\kappa \tau} \| \phi \|_{C^{2, \alpha}((-s_0, s_0))}^2 + e^{\kappa \tau}  a^{\gamma - 2} \left( \| v \|_{2, \alpha ; (-s_0, s_0) }^{(\gamma)} \right)^2 \right]
    \end{align*}
    Since $\kappa\geq\frac{K}{2}(2-\gamma)$, the asserted estimate \eqref{eqn: inner region Q est 2} follows.
\end{proof}

\subsubsection{H{\"o}lder Estimates in the Inner Region are Preserved}

In this subsection, we combine the results of the previous subsections to show that the box conditions from Definition \ref{defn box} below preserve and improve the H{\"o}lder bounds for $v$ in the inner region $|s| \lesssim 1$.
The subsection's main theorem (Theorem \ref{thm preserving inner Holder condition+}) appears below.
Theorem \ref{thm preserving inner Holder condition+} is based on the following lemma, which says that, up to nonlinear $v$ terms and modulation terms, H{\"o}lder estimates for $v$ on the parabolic boundary of the inner region propagate inwards to H{\"o}lder estimates throughout the inner region.

\begin{lem} \label{lem preserving inner Holder condition}
    Given $n \ge 3$, $\kappa  >0$, $\gamma \in (2-n, 0)$, $\alpha \in (0,1)$, $s_0 > 0$ and $K \ge 1$, 
    there exists $0 < a^* =a^* (n, K, s_0) \ll 1$, $C = C(n, K, s_0, \gamma, \alpha, \kappa ) > 0$, $s_0' = s_0' (n, \alpha, \gamma, \kappa , s_0) \in (0, s_0]$, and $\tau_0 = \tau_0( n, \alpha , \gamma, \kappa , s_0) \gg 1$ such that the following holds:
    
    If $\tau_0 \le \tau_1 < \tau_2 < \infty$, $a : [\tau_1, \tau_2) \to \R_{>0}$, and $u , \phi , v : \R \times [\tau_1, \tau_2) \to \R$ satisfy
    \begin{gather}
        (\partial_\tau - H_a)u = \left( 1 - 2 \frac {\partial_\tau a}a \right) \beta_a   + Q_a(u)  \qquad \text{on } \R \times (\tau_1, \tau_2) , \\
        \text{with } u = \phi + v, \quad \phi = \sum_{i=1}^K b_i(\tau) ( \phi_{i, a} - \phi_{0,a} ) , \quad
        v(\cdot , \tau)  \perp_{L^2_a} \{ \phi_{i,a} \}_{0 \le i \le K} \quad \forall \tau \in [\tau_1, \tau_2),\\
        0 < a(\tau) < a^* \quad \forall \tau \in [\tau_1, \tau_2), \quad \sup_{\tau \in [\tau_1, \tau_2)} \| u ( \cdot, \tau) \|_{2, \alpha; (-4s_0, 4s_0) } ^{(2) } \le \frac14 , \\
        \text{and } \sup_{\tau \in [\tau_1, \tau_2)} \| v( \cdot , \tau) \|_{2, \alpha; \{ s_0' < |s| < 4 s_0\} }^{(0)} \le 1,
    \end{gather}
    then 
    \begin{gather} \label{lem preserving inner Holder condition, eqn 5} \begin{aligned}
        &\sup_{\tau\in [\tau_1, \tau_2)} e^{\kappa \tau} \| v ( \cdot, \tau) \|_{2, \alpha ; (-s_0, s_0)}^{(\gamma)} \\
        \le{}& C \sup_{\tau\in [\tau_1, \tau_2)} e^{\kappa \tau} a(\tau)^{\gamma - 2} \left( \| v(\cdot, \tau) \|_{2, \alpha; (-s_0, s_0)}^{(\gamma)} \right)^2 \\
        &+C \sup_{\tau \in [\tau_1, \tau_2) } e^{\kappa \tau} \| v ( \cdot, \tau) \|_{2, \alpha ; \{ s_0' < |s| < 4s_0\}}^{(0)} 
        + C e^{\kappa \tau_1} \| v ( \cdot , \tau_1) \|_{2, \alpha; (-s_0, s_0)}^{(\gamma)}
        \\
        &+ C \sup_{\tau \in [\tau_1, \tau_2)} e^{\kappa \tau} \left(Mod + \sum_{i=1}^K |b_i| ( a + |\partial_\tau a| + |b_i| )  \right) 
    \end{aligned} \end{gather}
    where $Mod(\tau)$ is defined as in \eqref{eqn defn Mod}.
\end{lem}

\begin{proof}
    The PDE for $u$ implies
        $$(\partial_\tau - H_a) v = \left( 1 - 2 \frac {\partial_\tau a}a \right) \beta_a   - (\partial_\tau - H_a) \phi + Q_a(u).$$
    To simplify the notation, denote
        $$\psi := \left( 1 - 2 \frac {\partial_\tau a}a \right) \beta_a   - (\partial_\tau - H_a) \phi + Q_a(u).$$
    Then, for  constants $C = C(n, \alpha, \gamma , \kappa , s_0) > 0$ and $s_0' = s_0'(n, \alpha, \gamma, \kappa , s_0) > 0$ which may change from line to line, we have
    \begin{gather} \label{proof preserving inner Holder condition, eqn 3}
    \begin{aligned}
        &\sup_{\tau} e^{\kappa \tau} \| v ( \cdot, \tau) \|_{2, \alpha ; (-s_0, s_0)}^{(\gamma)} \\
        \le{}& C \sup_{\tau} e^{\kappa \tau} \| v (\cdot, \tau) \|_{0; (-4s_0, 4s_0)}^{(\gamma)}  + C e^{\kappa \tau_1} \| v ( \cdot, \tau_1) \|_{2, \alpha ; (-4s_0, 4s_0)}^{(\gamma)} 
        + C\sup_{\tau} e^{\kappa \tau} \| \psi \|_{0, \alpha; (-4s_0, 4s_0)}^{(\gamma-2)} \\
        & ( \text{by Proposition \ref{prop: weighted Schauder estimate intermediate}} ) \\
        \le{}& C \sup_{\tau} e^{\kappa \tau} \| v (\cdot, \tau) \|_{0; \{ s_0' < |s| < 4 s_0\} }^{(\gamma)}   + C e^{\kappa \tau_1} \| v ( \cdot, \tau_1) \|_{2, \alpha ; (-4s_0, 4s_0)}^{(\gamma)}  
        + C \sup_{\tau} e^{\kappa \tau} \| \psi \|_{0, \alpha; (-4s_0, 4s_0)}^{(\gamma-2)}  \\
        & ( \text{by Proposition \ref{prop: weighted sup norm estimate intermediate} using $\kappa > 0, \gamma \in (2-n,0), \tau_1 \ge \tau_0$.)} \\
        \le{}& C  \sup_{\tau} e^{\kappa \tau} \| v (\cdot, \tau) \|_{0; \{ s_0' < |s| < 4 s_0\} }^{(\gamma)}   + C e^{\kappa \tau_1} \| v ( \cdot, \tau_1) \|_{2, \alpha ; (-4s_0, 4s_0)}^{(\gamma)}  \\
        & + C\sup_{\tau} e^{\kappa \tau}   \| Q_a(u) \|_{0, \alpha; (-4s_0, 4s_0)}^{(\gamma-2)}  
        + C\sup_{\tau} e^{\kappa \tau}   \left\| ( \partial_\tau - H_a ) \phi - \left( 1 - 2 \frac{\partial_\tau a}a  \right) \beta_a \right\|_{0, \alpha; (-4s_0, 4s_0)}^{(\gamma-2)} .
    \end{aligned} \end{gather}
    We proceed to estimate the terms that appear in the last line.
    
    Using that $0 < a < 1$ and the fact that $s_0' = s_0' (n, \alpha, \gamma, \kappa , s_0) > 0$, it follows that
    \begin{multline} \label{proof preserving inner Holder condition, eqn 4}
        \sup_{\tau} e^{\kappa \tau} \| v (\cdot, \tau) \|_{0; \{ s_0' < |s| < 4 s_0\} }^{(\gamma)}
        \le  C\sup_\tau \sup_{s_0' < |s| < 4 s_0} e^{\kappa \tau} | v ( s , \tau ) | \\
        = C \sup_\tau e^{\kappa \tau} \| v(\cdot , \tau) \|_{0; \{ s_0' < |s| < 4 s_0 \}}^{(0)} 
    \end{multline}
    where $C = C(n, \alpha, \gamma, \kappa, s_0) > 0$.

    \begin{claim} \label{proof preserving inner Holder condition, claim 0}
        \begin{gather}\begin{aligned}
            &\left( \partial_\tau - H_a \right) \phi - \left( 1 -  2 \frac{\partial_\tau a}a \right) \beta_a \\
            ={}& \left[ 2 a \partial_\tau a - a^2 + \sum_{i=1}^K i b_i \right] a^{-2} \beta_a + \sum_{i=1}^K ( \partial_\tau b_i - b_i(1-i) ) ( \phi_{i,a} - \phi_{0,a} ) \\
            &+ \sum_{i=1}^K i b_i \left( \phi_{0,a} - a^{-2} \beta_a \right)      
            - \sum_{i=1}^K b_i  \tilde \lambda_{i,a} \phi_{i,a}  
            + \sum_{i=1}^K b_i \tilde \lambda_{0,a} \phi_{0,a} \\
            &+ \sum_{i=1}^K b_i (\partial_\tau a) (\partial_a \phi_{i,a} - \partial_a \phi_{0,a} )
        \end{aligned} \end{gather}
        where $\lambda_{i,a} = 1-i+\tilde \lambda_{i,a}$ denote the eigenvalues as in Theorem \ref{thm global eigenfunctions and spectral gap}.
    \end{claim}
    \begin{claimproof}
        This follows from a straightforward computation using the relations
        \begin{gather*}
            \phi = \sum_{i=1}^K b_i (\phi_{i,a} - \phi_{0,a} ) , \qquad 
            H_a \phi_{i,a} = \lambda_{i,a} \phi_{i,a} , \qquad \lambda_{i,a} = 1 - i + \tilde \lambda_{i,a} \\
            \text{and } \partial_\tau \phi = \sum_{i=1}^K (\partial_\tau b_i)(\phi_{i,a} - \phi_{0,a} ) 
            + \sum_{i=1}^K b_i (\partial_\tau a)( \partial_a \phi_{i,a} - \partial_a \phi_{0,a} ).
        \end{gather*}
    \end{claimproof}

    \begin{claim} \label{proof preserving inner Holder condition, claim 1}
        If $\gamma - 2 \le 0$, then 
        there exists $C = C(n, s_0, \gamma, \alpha) > 0$ such that, for all $0 < a < 1$,
        \begin{equation}
            \| a^{-2} \beta_a \|_{0, \alpha; (-4s_0, 4s_0)}^{(\gamma-2)} \le C .
        \end{equation}
    \end{claim}
    \begin{claimproof}
        Recall $a^{-2} \beta_a(s) = \beta_1(s/a)$ by Lemma \ref{lem beta properties}.
        By the definition of the weighted H{\"o}lder norms (Definition \ref{defn weighted Holder norms})
        \begin{multline*}
            \| a^{-2} \beta_a \|_{0, \alpha; (-4s_0, 4s_0)}^{(\gamma-2)}
            = a^{-\gamma+2} \sup_{|\sigma| < \frac{4s_0}a} \rho_1(\sigma)^{-\gamma+2} | \beta_1(\sigma)| \\
            +a^{-\gamma+2} \sup_{|\sigma| <  \frac{4s_0}a } \rho_1(\sigma)^{-\gamma+2+\alpha} \sup_{\substack{\sigma' \ne \sigma'' \in \left(-\frac{4s_0}a, \frac{4s_0}a \right),\\ |\sigma' - \sigma| , |\sigma'' - \sigma| \le \frac12 \rho_1(\sigma)}} \frac{|\beta_1(\sigma') - \beta_1(\sigma'')|}{|\sigma' - \sigma''|^\alpha} .
        \end{multline*}
        It follows from the asymptotics of $\beta_1$ (Lemma \ref{lem beta properties}) that the terms on the right-hand side can be bounded by
        $$ C a^{-\gamma+2} \sup_{|\sigma | \le \frac{4s_0}a } \rho_1(\sigma)^{-\gamma+2} \le C' a^{-\gamma+2}  \left( \frac1a \right)^{-\gamma+2} = C'$$
        when $\gamma - 2 \le 0$ and $0 < a < 1$, where $C, C'>0$ are constants both depending only on $n, s_0, \gamma, \alpha$.
    \end{claimproof}

    \begin{claim} \label{proof preserving inner Holder condition, claim 2}
        If $\gamma - 2 \le 0$, then 
        there exists $C = C(n, K, s_0, \gamma , \alpha)> 0  $ such that, for all $0 < a < 1$,
        \begin{gather}
            \label{proof preserving inner Holder condition, claim 2, eqn 1}
            \sum_{i=1}^K \| \phi_{i,a} - \phi_{0, a} \|_{0, \alpha; (-4s_0, 4s_0)}^{(\gamma-2)} \le C \text{ and}\\
            \label{proof preserving inner Holder condition, claim 2, eqn 2}
            \| \phi_{0,a} - a^{-2} \beta_a \|_{0, \alpha; (-4s_0, 4s_0)}^{(\gamma-2)} \le C a.
        \end{gather}
    \end{claim}
    \begin{claimproof}
        \cite{SS26I}*{Lemma \ref{I-lem phi_k,a - phi_0,a pointwise ests}} (with $\delta = \frac12$) gives the pointwise $C^1$ estimate
        \begin{equation} \label{proof preserving inner Holder condition, proof claim 2, eqn 1}
            | \phi_{i,a} - \phi_{0,a} | + \rho_a(s) | \partial_s ( \phi_{i,a} - \phi_{0,a} ) | \le C ( \rho_a(s)^2 + a ) \qquad \forall s \in (-4s_0, 4s_0) , \, \forall 1 \le i \le K 
        \end{equation}
        for some $C = C(n, K, s_0) > 0$.
        This $C^1$-estimate implies the weighted H{\"o}lder estimate \eqref{proof preserving inner Holder condition, claim 2, eqn 1}.

        \cite{SS26I}*{Lemma \ref{I-lem tilde phi_k,a pointwise ests}} (again with $\delta =\frac12)$ gives the pointwise $C^1$ estimate 
            $$| \phi_{0,a} - a^{-2} \beta_a| + \rho_a(s) | \partial_s ( \phi_{0,a} - a^{-2} \beta_a) | \le C a 
            \qquad \forall s \in (-4s_0, 4s_0)$$
        for some $C =C(n, s_0) > 0$.
        This $C^1$-estimate implies the weighted H{\"o}lder estimate \eqref{proof preserving inner Holder condition, claim 2, eqn 2}.
    \end{claimproof}

    \begin{claim} \label{proof preserving inner Holder condition, claim 3}
        If $\gamma - 2 \le 0$, then 
        there exists $C = C(n, K, s_0, \gamma , \alpha)> 0  $ such that, for all $0 < a < 1$,
        \begin{gather}
            \label{proof preserving inner Holder condition, claim 3, eqn 1}
            \sum_{i=0}^K \| \phi_{i,a} \|_{0, \alpha; (-4s_0, 4s_0)}^{(\gamma-2)} \le C  \text{ and} \\
            \label{proof preserving inner Holder condition, claim 3, eqn 2}
            \sum_{i=1}^K \| \partial_a \phi_{i,a} - \partial_a \phi_{0,a} \|_{0, \alpha; (-4s_0, 4s_0)}^{(\gamma-2)} \le C .
        \end{gather}
    \end{claim}
    \begin{claimproof}
        Estimate \eqref{proof preserving inner Holder condition, claim 3, eqn 1} follows from Claims \ref{proof preserving inner Holder condition, claim 1} and \ref{proof preserving inner Holder condition, claim 2} and the triangle inequality.

        \cite{SS26I}*{Lemma \ref{I-lem partial_a phi_a ests}} (with $\delta = 1/2$) gives the pointwise $C^1$ estimate 
        $$ |\partial_a \phi_{i,a} - \partial_a \phi_{0,a} | + \rho_a(s) |\partial_s \partial_a \phi_{i,a} - \partial_s \partial_a \phi_{0,a} | \le C 
        \qquad \forall s \in (-4s_0, 4s_0) ,\, \forall 1 \le i \le K$$
        where $C= C(n, K, s_0) > 0$.
        This $C^1$-estimate implies the weighted H{\"o}lder estimate \eqref{proof preserving inner Holder condition, claim 3, eqn 2}.
    \end{claimproof}
 
    Combining Claims \ref{proof preserving inner Holder condition, claim 0}--\ref{proof preserving inner Holder condition, claim 3} above and using Theorem \ref{thm global eigenfunctions and spectral gap} to estimate $\sup_{0 \le i \le K} | \tilde \lambda_{i,a} | \le Ca$,
    it follows that, if $\gamma - 2\le 0$ and $0 <  a(\tau) < a^* (n, K, s_0 ) \ll 1$ for all $\tau \in [\tau_1, \tau_2)$, then 
    \begin{multline} \label{proof preserving inner Holder condition, eqn 5}
        \sup_{\tau \in (\tau_1, \tau_2)} e^{\kappa \tau}   \left\| ( \partial_\tau - H_a ) \phi - \left( 1 - 2 \frac{\partial_\tau a}a  \right) \beta_a \right\|_{0, \alpha; (-4s_0, 4s_0)}^{(\gamma-2)} \\
        \le C \sup_{\tau \in (\tau_1, \tau_2)} e^{\kappa \tau} \left(Mod + \sum_{i=1}^K |b_i| ( a + |\partial_\tau a| )  \right) 
    \end{multline}
    where $C = C(n, K, s_0, \gamma, \alpha)> 0$.

    \begin{claim} \label{proof preserving inner Holder condition, claim 4}
        If $\gamma \in (2-n, 0)$, then there exists $C = C(n, K, s_0, \gamma, \alpha) > 0$ such that 
        \begin{gather}
            0 < a < a^* (n, K, s_0) \ll 1 \quad \forall \tau \in [\tau_1, \tau_2) \qquad \text{and} \qquad \sup_{\tau\in (\tau_1, \tau_2)} \| u (\cdot, \tau)\|_{2, \alpha; (-4s_0, 4s_0)}^{(2)} \le \frac14 \nonumber \\
            \text{implies} \nonumber \\
            \begin{aligned} 
            \sup_{\tau\in [\tau_1, \tau_2)} e^{\kappa \tau}   \| Q_a(u) \|_{0, \alpha; (-4s_0, 4s_0)}^{(\gamma-2)} 
            \le{}& C \sup_{\tau\in [\tau_1, \tau_2)} e^{\kappa \tau} \sum_{i=1}^K | b_i(\tau)|^2 \\
            &+ C \sup_{\tau \in [\tau_1, \tau_2) } e^{\kappa \tau} a(\tau)^{\gamma - 2} \left( \| v \|_{2, \alpha; (-s_0, s_0)}^{(\gamma)} \right)^2 \\
            &+ C \sup_{\tau \in [\tau_1, \tau_2)} e^{\kappa \tau}  \left( \| v \|_{2, \alpha; \{ s_0 < |s| < 4s_0\} }^{(0)} \right)^2 .
            \end{aligned}
        \end{gather}
    \end{claim}
    \begin{claimproof}
        Throughout the proof of this claim, $C = C(n, K, s_0, \gamma, \alpha) > 0$ denotes a positive constant depending only on $n,K, s_0, \gamma, \alpha$, which may change from line to line.
        
        Assume $0 < a <1$ for all $\tau \in [\tau_1, \tau_2) $ and that
        $$\sup_\tau \| u (\cdot, \tau)\|_{2, \alpha; (-4s_0, 4s_0)}^{(2)} \le \frac14.$$
        Lemma \ref{lem: inner region Q estimate} \eqref{eqn: inner region Q est 1} then implies
        \begin{gather} \label{proof preserving inner Holder condition, proof claim 4, eqn 1} \begin{aligned}
            &\sup_\tau e^{\kappa \tau} \| Q_a(u) \|_{0, \alpha; (-4s_0, 4s_0)}^{(\gamma - 2)} \\
            \le{}& C \sup_{\tau} e^{\kappa \tau} \left( \| u_s \|_{1, \alpha ; (-4s_0, 4s_0)}^{(\frac \gamma 2  )} \right)^2 
            \\
            \le{}& C \sup_{\tau} e^{\kappa \tau} \left( \| u \|_{2, \alpha ; (-4s_0, 4s_0)}^{(\frac \gamma 2 + 1 )} \right)^2 \\
            \le{}& C  \sup_{\tau} e^{\kappa \tau} \sum_{i=1}^K|b_i(\tau)|^2 \left( \| \phi_{i,a} - \phi_{0,a} \|_{2, \alpha; (-4s_0, 4s_0)}^{(\frac \gamma 2 + 1) } \right)^2 
            + C \sup_{\tau} e^{\kappa \tau} \left( \| v \|_{2, \alpha; (-4s_0, 4s_0)}^{(\frac \gamma 2 + 1)} \right)^2 .
        \end{aligned} \end{gather}
        We proceed to estimate the terms that appear on the right-hand side.

        Note that $\phi_{i,a} - \phi_{0,a}$ satisfies
            $$H_a ( \phi_{i,a} - \phi_{0,a} ) = \lambda_{i,a} \phi_{i,a} - \lambda_{0,a} \phi_{0,a},  \qquad \lambda_{i,a} = 1-i + O(a; n, K,s_0),$$
        by Theorem \ref{thm global eigenfunctions and spectral gap} for $0 < a \le a^* (n, K, s_0) \ll 1$.
        An elliptic interior estimate as in Proposition \ref{prop: weighted Schauder estimate intermediate} therefore gives 
        \begin{gather} \label{proof preserving inner Holder condition, proof claim 4, eqn 2} \begin{aligned}
            &\| \phi_{i,a} - \phi_{0,a} \|_{2, \alpha; (-4s_0, 4s_0)}^{( \frac \gamma 2 + 1 )} \\
            \le{}& C \| \phi_{i,a} - \phi_{0,a} \|_{0; (-8s_0, 8s_0)}^{(\frac \gamma 2 + 1)}
            + C \| \lambda_{i,a} \phi_{i,a} - \lambda_{0,a} \phi_{0,a} \|_{0, \alpha; (-8s_0, 8s_0)}^{(\frac \gamma 2 -1)} \\
            \le{}& C \| \phi_{i,a} - \phi_{0,a} \|_{0; (-8s_0, 8s_0)}^{(\frac \gamma 2 + 1)}
            + C \|  \phi_{i,a} \|_{0, \alpha; (-8s_0, 8s_0)}^{(\frac \gamma 2 -1)} + C\| \phi_{0,a} \|_{0, \alpha; (-8s_0, 8s_0)}^{(\frac \gamma 2 -1)} .
        \end{aligned} \end{gather}
        By estimate \eqref{proof preserving inner Holder condition, proof claim 2, eqn 1},
        \begin{gather} \label{proof preserving inner Holder condition, proof claim 4, eqn 3} \begin{aligned}
            \| \phi_{i,a} - \phi_{0,a} \|_{0; (-8s_0, 8s_0)}^{(\frac \gamma 2 + 1)}
            \le{}& C \sup_{s \in (-8s_0, 8s_0)} \rho_a(s)^{- \frac \gamma 2-1} \left( \rho_a(s)^2 + a \right) 
            && (\text{by \eqref{proof preserving inner Holder condition, proof claim 2, eqn 1}} \\
            \le{}& C \sup_{s \in (-8s_0, 8s_0)} \left( \rho_a(s)^{- \frac \gamma 2 + 1 } + \rho_a(s)^{-\frac \gamma 2} \right)
            && ( \text{since } a \le \rho_a(s) )  \\
            \le{}& C && (\text{since } \gamma < 0 ).
        \end{aligned} \end{gather}
        Additionally, $\gamma < 0$ implies $\frac \gamma 2 -1 < 0$ and so
        \begin{equation} \label{proof preserving inner Holder condition, proof claim 4, eqn 4}
            \sum_{i=0}^K \| \phi_{i,a} \|_{0, \alpha; (-8s_0, 8s_0)}^{(\frac \gamma 2 - 1)} \le C 
        \end{equation}
        by Claim \ref{proof preserving inner Holder condition, claim 3}.
        Inserting \eqref{proof preserving inner Holder condition, proof claim 4, eqn 3} and \eqref{proof preserving inner Holder condition, proof claim 4, eqn 4} into \eqref{proof preserving inner Holder condition, proof claim 4, eqn 2} gives
        \begin{equation} \label{proof preserving inner Holder condition, proof claim 4, eqn 5}
            \max_{1 \le i \le K} \| \phi_{i,a} - \phi_{0,a} \|_{2, \alpha; (-4s_0, 4s_0)}^{(\frac \gamma 2 + 1) } \le C .
        \end{equation}

        From the definition of the weighted H{\"o}lder norms, we can estimate the $v$ term as
        \begin{gather} \label{proof preserving inner Holder condition, proof claim 4, eqn 6} \begin{aligned}
            &\| v \|_{2, \alpha; (-4s_0, 4s_0)}^{(\frac \gamma 2 + 1) } \\
            \le{}&  \| v \|_{2, \alpha; (-s_0, s_0)}^{(\frac \gamma 2 +1)}
            + \| v \|_{2, \alpha; \{ s_0 < |s| < 4s_0\} }^{(\frac \gamma 2 +1)} \\
            \le{}& \left( \sup_{|s| < s_0} \rho_a(s)^{\frac \gamma 2 - 1} \right) \| v \|_{2, \alpha; (-s_0, s_0)}^{(\gamma )}
            + \left(\sup_{s_0 < |s| < 4s_0 } \rho_a(s)^{- \frac \gamma 2 - 1} \right) \| v \|_{2, \alpha; \{ s_0 < |s| < 4s_0\} }^{(0)} \\
            \le{}& C a^{\frac \gamma 2 -1} \| v \|_{2, \alpha; (-s_0, s_0)}^{(\gamma )}
            + C \| v \|_{2, \alpha; \{ s_0 < |s| < 4s_0\} }^{(0)} 
        \end{aligned} \end{gather}
        where the last inequality uses $\gamma < 0$ and $\rho_a(s) \ge a$.

        Inserting estimates \eqref{proof preserving inner Holder condition, proof claim 4, eqn 5} and \eqref{proof preserving inner Holder condition, proof claim 4, eqn 6} into the $Q_a$ estimate \eqref{proof preserving inner Holder condition, proof claim 4, eqn 1} yields
        \begin{multline}
            \sup_{\tau} e^{\kappa \tau} \| Q_a(u) \|_{0, \alpha; (-4s_0, 4s_0)}^{( \gamma - 2)} \\
            \le C \sup_{\tau} e^{\kappa \tau} \sum_{i=1}^K | b_i(\tau)|^2
            + C \sup_{\tau} e^{\kappa \tau} a(\tau)^{\gamma - 2} \left( \| v \|_{2, \alpha; (-s_0, s_0)}^{(\gamma)} \right)^2 \\
            + C \sup_{\tau} e^{\kappa \tau}  \left( \| v \|_{2, \alpha; \{ s_0 < |s| < 4s_0\} }^{(0)} \right)^2 ,
        \end{multline}
        which proves the claim.
    \end{claimproof} 

    Inserting estimates \eqref{proof preserving inner Holder condition, eqn 4}, \eqref{proof preserving inner Holder condition, eqn 5}, and Claim \ref{proof preserving inner Holder condition, claim 4} into \eqref{proof preserving inner Holder condition, eqn 3}, it follows that if $\gamma \in (2-n, 0)$,
        $$0 < a(\tau) < a^* (n, K, s_0) \ll 1 \quad \forall \tau \in [\tau_1, \tau_2) , \qquad \text{and} \qquad \sup_{\tau \in [\tau_1, \tau_2)} \| u ( \cdot , \tau) \|_{2, \alpha ;(-4s_0, 4s_0)}^{(2)} \le \frac14 ,$$
    then 
    \begin{gather} \label{proof preserving inner Holder condition, eqn 6} \begin{aligned}
        &\sup_{\tau \in [\tau_1, \tau_2)} e^{\kappa \tau} \| v ( \cdot, \tau) \|_{2, \alpha ; (-s_0, s_0)}^{(\gamma)} \\
        \le{}& C \sup_{\tau \in [\tau_1, \tau_2)} e^{\kappa \tau} a(\tau)^{\gamma - 2} \left( \| v(\cdot, \tau) \|_{2, \alpha; (-s_0, s_0)}^{(\gamma)} \right)^2 \\
        &+C \sup_{\tau \in [\tau_1, \tau_2)} e^{\kappa \tau} \| v ( \cdot, \tau) \|_{0; \{ s_0' < |s| < 4s_0\}}^{(0)} 
        + C \sup_{\tau \in [\tau_1, \tau_2)} e^{\kappa \tau}  \left( \| v(\cdot, \tau) \|_{2, \alpha; \{ s_0 < |s| < 4s_0\} }^{(0)} \right)^2   \\
        & + C e^{\kappa \tau_1} \| v ( \cdot , \tau_1) \|_{2, \alpha; (-4s_0, 4s_0)}^{(\gamma)}
        +  C \sup_{\tau \in [\tau_1, \tau_2)} e^{\kappa \tau} \left(Mod + \sum_{i=1}^K |b_i| ( a + |\partial_\tau a| + |b_i| )  \right) 
    \end{aligned} \end{gather}
    for some $C = C(n, K, s_0, \gamma, \alpha, \kappa ) > 0$ and some $s_0' = s_0' (n,  \alpha , \gamma, \kappa , s_0) > 0$.
    
    By replacing $s_0'$ with $\min\{ s_0' , s_0\}$, we can assume without loss of generality that $0 < s_0' \le s_0$.
    The stated estimate \eqref{lem preserving inner Holder condition, eqn 5} now follows from using the assumption 
        $$\sup_{\tau \in [\tau_1, \tau_2)} \| v( \cdot, \tau) \|_{2, \alpha; \{ s_0' < |s| < 4s_0\}}^{(0)} \le 1$$
    to estimate the second two terms in the right-hand side of \eqref{proof preserving inner Holder condition, eqn 6} by 
    \begin{multline}
        \sup_{\tau} e^{\kappa \tau} \| v ( \cdot, \tau) \|_{0; \{ s_0' < |s| < 4s_0\}}^{(0)} 
        +  \sup_{\tau} e^{\kappa \tau}  \left( \| v(\cdot, \tau) \|_{2, \alpha; \{ s_0 < |s| < 4s_0\} }^{(0)} \right)^2\\
        \le 2 \sup_{\tau} e^{\kappa \tau} \| v( \cdot, \tau) \|_{2, \alpha; \{ s_0' < |s| < 4s_0\}}^{(0)} ,
    \end{multline}
    and by estimating the initial data term in \eqref{proof preserving inner Holder condition, eqn 6} as
    \begin{align*}
        e^{\kappa \tau_1} \| v( \cdot, \tau_1) \|_{2, \alpha; (-4s_0, 4s_0)}^{(\gamma)} 
        \le{}& e^{\kappa \tau_1} \| v ( \cdot, \tau_1) \|_{2 , \alpha ; (-s_0, s_0)}^{(\gamma)} 
        + e^{\kappa \tau_1} \| v ( \cdot, \tau_1) \|_{2 , \alpha ; \{ s_0' < |s| < 4s_0\} }^{(\gamma)}  \\
        \le{}& e^{\kappa \tau_1} \| v ( \cdot, \tau_1) \|_{2 , \alpha ; (-s_0, s_0)}^{(\gamma)} 
        + Ce^{\kappa \tau_1} \| v ( \cdot, \tau_1) \|_{2 , \alpha ; \{ s_0' < |s| < 4s_0\} }^{(0)}  \\
        \le{}& e^{\kappa \tau_1} \| v ( \cdot, \tau_1) \|_{2 , \alpha ; (-s_0, s_0)}^{(\gamma)} 
        +C  \sup_{\tau \in [\tau_1, \tau_2)} e^{\kappa \tau} \| v ( \cdot, \tau) \|_{2 , \alpha ; \{ s_0' < |s| < 4s_0\} }^{(0)}  
    \end{align*}
    where the penultimate estimate follows by similar logic as in \eqref{proof preserving inner Holder condition, eqn 4}.
\end{proof}

We can now apply the previous lemma to prove the main theorem of this subsection.
It shows the box conditions from Definition \ref{defn box} below preserve and improve the H{\"o}lder bounds for $v$ in the inner region $|s| \lesssim 1$.

\begin{theorem} \label{thm preserving inner Holder condition+}
    Let $n \ge 3$, $\kappa > 0$, $\gamma  \in (-1, 0) $, $\alpha \in (0,1)$, $s_0 > 0$, $C_0 \ge 1$, and $K \ge 2$ be constants such that 
    \begin{equation}
        (2-\gamma ) \frac{K-1}2  \le \kappa < 3 \frac{K-1}2.
    \end{equation}
    Let $s_0' = s_0' (n,\alpha, \gamma, \kappa, s_0) > 0$ denote the corresponding constant from Lemma \ref{lem preserving inner Holder condition}.

    For all $0 < \epsilon \le \epsilon^* (\gamma, C_0) \ll 1$, $0 < \delta \le \delta^* (\epsilon, n, K, s_0, \gamma, \alpha, \kappa, C_0 ) \ll 1$, and $\tau_2 > \tau_1 \ge \tau_* (\epsilon, n, K, s_0, \gamma, \alpha, \kappa, C_0 )\gg 1$, the following holds:

    If $a : [\tau_1, \tau_2) \to \R_{>0}$, and $u , \phi , v : \R \times [\tau_1, \tau_2) \to \R$ satisfy
    \begin{gather}
        (\partial_\tau - H_a)u = \left( 1 - 2 \frac {\partial_\tau a}a \right) \beta_a   + Q_a(u)  \qquad \text{on } \R \times (\tau_1, \tau_2) , \\
        \text{with } u = \phi + v, \quad \phi = \sum_{i=1}^K b_i(\tau) ( \phi_{i, a} - \phi_{0,a} ) , \quad
        v(\cdot , \tau)  \perp_{L^2_a} \{ \phi_{i,a} \}_{0 \le i \le K} \quad \forall \tau \in [\tau_1, \tau_2),
    \end{gather}
    the functions $a, b_i$ satisfy estimates
    \begin{multline} \label{thm preserving inner Holder condition+, eqn 1}
        C_0^{-1} e^{\frac{1-K}2 \tau} \le a(\tau) \le C_0 e^{\frac {1-K}2 \tau} , \quad
        |\partial_\tau a(\tau)| \le C_0 a(\tau), \quad \text{and } 
        \sum_{i=1}^K |b_i(\tau)| \le C_0a(\tau)^2 \\
        \forall \tau \in [\tau_1, \tau_2),
    \end{multline}
    the function $Mod(\tau)$ as defined in \eqref{eqn defn Mod} satisfies 
    \begin{equation} \label{thm preserving inner Holder condition+, eqn 2}
        \sup_{\tau \in [\tau_1, \tau_2)} e^{\kappa \tau} Mod(\tau)  \le \delta ,
    \end{equation}
    and $v$ has laterally thickened parabolic boundary estimates
    \begin{equation} \label{thm preserving inner Holder condition+, eqn 3}
        e^{\kappa \tau_1} \| v( \cdot, \tau_1) \|_{2, \alpha; (-s_0, s_0)}^{(\gamma)} + \sup_{\tau \in [\tau_1, \tau_2)} e^{\kappa \tau} \| v( \cdot, \tau) \|_{2, \alpha; \{ s_0' < |s| < 4s_0\} }^{(0)} \le \delta  ,
    \end{equation}
    then 
    \begin{equation} \label{thm preserving inner Holder condition+, eqn 4}
        \sup_{\tau \in [\tau_1, \tau_2)} e^{\kappa \tau} \| v ( \cdot, \tau) \|_{2, \alpha; (-s_0, s_0)}^{(\gamma)} \le \epsilon. 
    \end{equation}
\end{theorem}
\begin{proof}
    Suppose 
        $$\sup_{\tau \in [\tau_1, \tau_2)} e^{\kappa \tau} \| v ( \cdot, \tau) \|_{2, \alpha; (-s_0, s_0)}^{(\gamma)} > \epsilon $$
    for the sake of contradiction, and define
        $$\tau_2' := \inf \left\{ \tau' \in [\tau_1, \tau_2] \, \colon \, \sup_{\tau \in [\tau_1, \tau')} e^{\kappa \tau} \| v ( \cdot, \tau) \|_{2, \alpha; (-s_0, s_0)}^{(\gamma)} > \epsilon \right\} .$$
    So long as $0 < \delta < \frac \epsilon 2$, we have that
        $$e^{\kappa \tau_1} \| v( \cdot, \tau_1) \|_{2, \alpha; (-s_0, s_0)}^{(\gamma)} \le \delta < \frac \epsilon 2 ,$$
    and so therefore $\tau_1 < \tau_2' < \tau_2$.
    It follows that in this case
    \begin{equation} \label{proof preserving inner Holder condition+, eqn 4}
        \sup_{\tau \in [\tau_1, \tau_2')} e^{\kappa \tau} \| v ( \cdot, \tau) \|_{2, \alpha; (-s_0, s_0) }^{(\gamma)} = \epsilon.
    \end{equation}  
    
    Next, note that $\tau_1 \ge \tau_*( C_0, K) \gg 1$ implies
    \begin{equation} \label{proof preserving inner Holder condition+, eqn 5}
        0 < C_0^{-1} e^{\frac{1-K}2 \tau} \le a(\tau) \le C_0 e^{\frac{1-K}2 \tau} < 1 \qquad \forall \tau \in [\tau_1, \tau_2').
    \end{equation}

    \begin{claim} \label{proof preserving inner Holder condition+, claim 1}
        If $0 < \epsilon \le \epsilon^* ( \gamma, C_0) \ll1 $, 
        $0 < \delta \le \delta^* ( n, \alpha, \gamma, \kappa, s_0) \ll 1$, and 
        $\tau_1 \ge \tau_*( n, K, \alpha, C_0) \gg 1$,
        then 
            $$\sup_{\tau \in [\tau_1,\tau_2')} \| u ( \cdot, \tau) \|_{2, \alpha; (-4s_0, 4s_0)}^{(2)} \le \frac14.$$
    \end{claim}
    \begin{claimproof}
        Using $u = \phi +v$, it follows that
        \begin{multline} \label{proof preserving inner Holder condition+, proof claim 1, eqn 1}
            \sup_{\tau  \in [\tau_1,\tau_2')} \| u ( \cdot, \tau) \|_{2, \alpha; (-4s_0, 4s_0)}^{(2)} \\
            \le \sup_{\tau \in [\tau_1, \tau_2')} \sum_{i=1}^K |b_i(\tau) | \| \phi_{i,a}  - \phi_{0, a} \|_{2, \alpha; (-4s_0, 4s_0)}^{(2)} 
            + \sup_{\tau \in [\tau_1, \tau_2')} \| v (\cdot, \tau) \|_{2, \alpha; (-4s_0, 4s_0)}^{(2)} 
        \end{multline}
        We estimate both terms on the right-hand side.
        Using the definition of the weighted H{\"o}lder norms (Definition \ref{defn weighted Holder norms}) and the fact that $\rho_a(s) \ge a$, it follows that
        \begin{gather} \label{proof preserving inner Holder condition+, proof claim 1, eqn 2} \begin{aligned}
            &\sup_{\tau \in [\tau_1, \tau_2')} \sum_{i=1}^K |b_i(\tau) | \| \phi_{i,a}  - \phi_{0, a} \|_{2, \alpha; (-4s_0, 4s_0)}^{(2)} \\
             \le{}& \sup_{\tau} \sum_{i=1}^K |b_i(\tau) | a^{-2+\frac12} \| \phi_{i,a} - \phi_{0,a} \|_{2, \alpha; (-4s_0, 4s_0) }^{\left(\frac12\right)} \\
             \le{}& C(n,K, \alpha) \sup_{\tau} \sum_{i=1}^K |b_i(\tau) | a^{-\frac32} 
             && ( \text{by \eqref{proof preserving inner Holder condition, proof claim 4, eqn 5}} ) \\
             \le{}& C(n, K, \alpha) \sup_\tau C_0 a^{\frac12} 
             && \left( \text{since }  \sum_{i=1}^K |b_i | \le C_0 a^2 \right) \\
             \le{}& C(n, K, \alpha) C_0^2 e^{\frac{1-K}4 \tau_1} \\
             \le{}& \frac18 && ( \text{if } \tau_1 \ge \tau_*( n, K, \alpha, C_0) \gg 1 ) .
        \end{aligned} \end{gather}
        Similarly, the $v$ term can be estimated as
        \begin{gather} \label{proof preserving inner Holder condition+, proof claim 1, eqn 3} \begin{aligned}
            &\sup_{\tau \in [\tau_1, \tau_2')} \| v (\cdot, \tau) \|_{2, \alpha; (-4s_0, 4s_0)}^{(2)}\\
            \le{}& \sup_{\tau} \| v ( \cdot , \tau) \|_{2, \alpha; (-s_0, s_0)}^{(2)} + \sup_{\tau} \| v (\cdot, \tau) \|_{2, \alpha; \{ s_0' < |s| < 4s_0\} }^{(2)} \\
            \le{}& \sup_{\tau}a^{-2+\gamma} \| v ( \cdot , \tau) \|_{2, \alpha; (-s_0, s_0)}^{(\gamma)} \\
            &+ C\sup_{\tau}  \| v (\cdot, \tau) \|_{2, \alpha; \{ s_0' < |s| < 4s_0\} }^{(0)} 
            && ( C = C(n, \alpha, \gamma, \kappa, s_0))\\
            \le{}& \left( \sup_{\tau} a^{-2+\gamma} e^{-\kappa \tau} \right) \left( \sup_{\tau} e^{\kappa \tau} \| v ( \cdot, \tau)\|_{2, \alpha; (-s_0, s_0)}^{(\gamma)} \right) \\
            &+ C \delta e^{-\kappa \tau_1}
            && ( \text{by \eqref{thm preserving inner Holder condition+, eqn 3}}) \\
            \le{}& \epsilon C_0^{2-\gamma} \sup_{\tau} e^{(-2+\gamma)\frac{1-K}2 \tau - \kappa \tau }
            + C \delta e^{-\kappa \tau_1} \\
            \le{}& \epsilon C_0^{2-\gamma} e^{\left[(2-\gamma)\frac{K-1}2 - \kappa \right] \tau_1}
            + C \delta 
            && \left( \text{by } \frac{K-1}2 (2- \gamma ) \le \kappa \right)\\
            \le{}& C_0^{2- \gamma} \epsilon + C \delta && ( \text{if }  \tau_1 \ge 1) \\
            \le{}& \frac18
        \end{aligned} \end{gather}
        where the last inequality holds so long as $0 < \delta \le \delta^* (n, \alpha, \gamma, \kappa, s_0) \ll 1$ and $0 < \epsilon \le \epsilon^*(\gamma, C_0 ) \ll 1$.
        Inserting \eqref{proof preserving inner Holder condition+, proof claim 1, eqn 2} and \eqref{proof preserving inner Holder condition+, proof claim 1, eqn 3} into \eqref{proof preserving inner Holder condition+, proof claim 1, eqn 1} completes the proof.
    \end{claimproof}

    Next, note that \eqref{thm preserving inner Holder condition+, eqn 3} implies
    \begin{equation} \label{proof preserving inner Holder condition+, eqn 6}
        \sup_{\tau \in [\tau_1, \tau_2')} \| v ( \cdot, \tau) \|_{2, \alpha; \{ s_0' < |s | < 4s_0\} }^{(0)}
        \le e^{-\kappa \tau_1} \sup_{\tau \in [\tau_1,\tau_2')} e^{\kappa \tau} \| v ( \cdot, \tau) \|_{2, \alpha; \{ s_0' < |s| < 4s_0\}}^{(0)} \le  \delta e^{-\kappa \tau_1} \le 1  
    \end{equation}
    for all $0 < \delta \le 1$ and $\tau_1 \ge 1$. 

    Claim \ref{proof preserving inner Holder condition+, claim 1} together with \eqref{proof preserving inner Holder condition+, eqn 5} and \eqref{proof preserving inner Holder condition+, eqn 6} state that the hypotheses of Lemma \ref{lem preserving inner Holder condition} hold.
    Therefore, for a constant $C = C(n, K, s_0, \gamma, \alpha ,\kappa, C_0) > 0$ which may change from line to line,
    \begin{align*}
        \epsilon ={}&\sup_{\tau \in [\tau_1, \tau_2')} e^{\kappa \tau} \| v ( \cdot, \tau) \|_{2, \alpha; (-s_0, s_0)}^{(\gamma)}
        && ( \text{by \eqref{proof preserving inner Holder condition+, eqn 4}} ) \\
        \le{}&  C \sup_{\tau \in [\tau_1, \tau_2')} e^{\kappa \tau} a^{\gamma - 2} \left( \| v( \cdot, \tau) \|_{2, \alpha; (-s_0, s_0)}^{(\gamma)} \right)^2 \\
        &+C \sup_{\tau \in [\tau_1, \tau_2) } e^{\kappa \tau} \| v ( \cdot, \tau) \|_{2, \alpha ; \{ s_0' < |s| < 4s_0\}}^{(0)} 
        + C e^{\kappa \tau_1} \| v ( \cdot , \tau_1) \|_{2, \alpha; (-s_0, s_0)}^{(\gamma)}
        \\
        &+ C \sup_{\tau \in [\tau_1, \tau_2)} e^{\kappa \tau} \left(Mod + \sum_{i=1}^K |b_i| ( a + |\partial_\tau a| + |b_i| )  \right) 
        && ( \text{by Lemma \ref{lem preserving inner Holder condition}} ) \\
        \le{}& C \left( \sup_{\tau} e^{-\kappa \tau} a^{\gamma -2} \right) \left( \sup_{\tau} e^{\kappa \tau} \| v \|^{(\gamma)}_{2, \alpha; (-s_0, s_0)} \right)^2 \\
        &+ C\delta  + C \sup_{\tau} e^{\kappa \tau} [\delta e^{-\kappa \tau} + a^3 ]
        && ( \text{by \eqref{thm preserving inner Holder condition+, eqn 1}--\eqref{thm preserving inner Holder condition+, eqn 3}} ) \\
        \le{}& C e^{\left[\frac{K-1}2 (2-\gamma ) - \kappa  \right] \tau_1} \epsilon^2 
        + C\delta + C e^{\left[ \kappa + 3 \frac{1-K}2 \right] \tau_1} \\
        & \left( \text{by \eqref{thm preserving inner Holder condition+, eqn 1}, \eqref{proof preserving inner Holder condition+, eqn 4}, and the fact that } \frac{K-1}2 (2 - \gamma) \le \kappa < 3 \frac{K-1}2 \right) \\
        <{}& \epsilon
    \end{align*}
    where the last inequality holds so long as $0 < \delta \le \delta^* ( n, K, s_0, \gamma, \alpha ,\kappa, C_0, \epsilon ) \ll 1$ and $\tau_1 \ge \tau_* (n, K, s_0, \gamma, \alpha ,\kappa, C_0, \epsilon ) \gg 1$.
    Thus, this estimate gives the contradiction $\epsilon < \epsilon $, and so it must be the case that
        $$\sup_{\tau \in [\tau_1, \tau_2)} e^{\kappa \tau} \| v( \cdot, \tau) \|_{2, \alpha; (-s_0, s_0)}^{(\gamma)} \le \epsilon$$
    as desired.
\end{proof}

\section{The Box Argument} \label{Section Box Argument}

Having completed the necessary integral and H{\"o}lder estimates in Sections \ref{Section Mod Eqns and Integral Ests} and \ref{Section Holder Estimates for the Potential} above, we can now implement the Wa{\.z}ewski Box argument outlined in Subsection \ref{Subsect Proof Outline}.
The section begins by defining the box $\mathcal B$ to be used in the sequel.
Subsection \ref{Subsection Topological Argument} then carries out the proof of the box argument in detail.
Finally, Subsection \ref{Subsect Proof Main Thm Paper II} applies the results of Subsection \ref{Subsection Topological Argument} to prove the paper's main result, Theorem \ref{main thm paper II intro}.

\subsection{Definition of the Box}

\begin{definition} \label{defn box}
    Let $n \ge 3$, $K \in \mathbb N$ with $K \ge 2$, $\alpha \in (0,1)$, and $0 < s_0 < \Upsilon$.
    Let $0 < \eta_-, \eta_+$, $\gamma \in (2-n, 0)$, and $\kappa_{in}, \kappa_{out}, \epsilon_{in}, \epsilon_{par} , \epsilon_{out}, A_{out} > 0$.

    For an interval $I \subset \R$, define 
        $$\mathcal B[I] = \mathcal B [n, K, \alpha, s_0, \Upsilon; \eta_-, \eta_+, \gamma, \kappa_{in} , \kappa_{out};  \epsilon_{in}, \epsilon_{par}, \epsilon_{out}, A_{out}  ; I ]$$
    to be the set of pairs $(u, a)$
    where $a : I \to (0, \infty)$ is a $C^1$-function and
    $u : \R \times I \to \R$, $u = u(s, \tau)$ is a function which is odd in $s$ and
    of the form
    \begin{gather*}
        u = \phi + v , \qquad 
        \phi = \sum_{i=1}^K b_i(\tau) \phi_{i,a} - \sum_{i=1}^K b_i (\tau) \phi_{0,a}, \quad 
        \text{and } v(\cdot, \tau) \perp_{L^2_a} \phi_{i,a} \qquad \forall \tau \in I, 
    \end{gather*}
    such that the following bounds hold:

    \begin{enumerate}
        \item (Fine $L^2_a$ Decay) 
        \begin{enumerate}
            \item (Exit Conditions) For all $\tau \in I$, 
            \begin{gather} \label{defn box, exit L2 conditions}
                (\text{exit conditions})\left\{ 
                \begin{aligned}
                    |b_1(\tau) | &\le a(\tau)^{2 + \eta_-} , \\
                    &\vdots \\
                    |b_i(\tau) | &\le a(\tau)^{2 + \eta_-} \\
                    &\vdots \\
                    |b_{K-1} (\tau) | &\le a(\tau)^{2 + \eta_-} , \\
                    | b_K(\tau) - a^2(\tau)| &\le  K a(\tau)^{2 + \eta_-} ,
                \end{aligned} \right.  
            \end{gather}

            \item (Preserved Conditions) For all $\tau \in I$,
            \begin{gather} \label{defn box, preserved L2 conditions}
                (\text{preserved conditions})\left\{ \begin{aligned}
                    0 < \frac12 e^{\frac{1-K}2 \tau} \le a(\tau) &\le 2 e^{\frac {1-K}2 \tau}  , \\
                    |\partial_\tau a | &\le (K-1) \cdot a(\tau)  ,\\
                    \| v( \cdot, \tau) \|_{L^2_a(\R) } &\le a^{2 + \eta_+} (\tau) , 
                \end{aligned} \right.
            \end{gather}

        \end{enumerate}

        \item (Weighted $C^{2, \alpha}$ Decay in the Inner Region)
            \begin{equation}
                e^{\kappa_{in} \tau} \| v ( \cdot, \tau) \|_{2, \alpha; (-s_0, s_0)}^{(\gamma)} \le \epsilon_{in} \qquad \forall \tau \in I, 
            \end{equation}

        \item (Coarse $C^{2, \alpha}$ Bounds in the Parabolic Region)
            \begin{equation}
                \| u(\cdot, \tau) \|_{C^{2, \alpha}\left( \frac{s_0}2 < |s| < 2 \Upsilon \right) } \le \epsilon_{par} \qquad \forall \tau \in I,
            \end{equation}
        and
        \item (Weighted $C^{2, \alpha}$ Bounds in the Outer Region)
        \begin{equation}
            \| u (\cdot, \tau) \|_{2, \alpha; (\Upsilon, \infty)}^{(2)} \le \epsilon_{out} 
            \quad \text{and} \quad
            \| u (\cdot, \tau) \|_{2, \alpha; (\Upsilon, \infty)}^{(2\kappa_{out} + 2)} \le A_{out} e^{-\kappa_{out} \tau}
            \qquad \forall \tau \in I.
        \end{equation}
    \end{enumerate}
\end{definition}

Here and for the remainder of the paper, the weighted H{\"o}lder norms $\| \cdot \|_{k, \alpha; \Omega}^{(\gamma)}$ are defined as in Definition \ref{defn weighted Holder norms}.
Recall by Remark \ref{remark equiv Holder norms} that these norms are equivalent to those defined in Definition \ref{defn weighted Holder norms, outer region} on regions where $0 < a \le 1$ and $|s| \ge 1$.

\subsection{The Topological Argument} \label{Subsection Topological Argument}

\begin{theorem} \label{thm main box thm}
    If 
    \begin{gather} \label{main box thm, param hypotheses 1}
    \left\{ \begin{gathered}
        n \ge 3, \quad  K \ge 2, \quad \alpha \in (0,1), \\
        0 < s_0 \le s^* ( n, K) \ll 1, \qquad -1 < \gamma < 0,  \\
        \\
        \kappa_{in} = ( 2 + \tilde \kappa_{in} ) \frac{K-1}2, \qquad
        - \gamma \le \tilde \kappa_{in} < \eta_+ < \frac2{K-1}, \\
        \left( \begin{split}
             \eta_+ < \min \left\{ 1, \frac{n-4}2 + 2 (\tilde \kappa_{in} + \gamma ) \right\} < \frac2{K-1} \\
             \text{or}\quad \frac{2}{K-1} \le \min \left\{ 1, \frac{n-4}2 + 2 (\tilde \kappa_{in} + \gamma ) \right\}
        \end{split}\right) ,\\
        \\
        \kappa_{out} = \hat \kappa_{out} \frac{K-1}2, \qquad 
        \frac74 \le \hat \kappa_{out} < 2, \\
        \text{and } 0 < \eta_- < \min \left\{ 1, \frac{n-4}2 + 2 ( \tilde  \kappa_{in} + \gamma) \right\} , \\
    \end{gathered} \right. \end{gather}
    and
    \begin{gather} \label{main box thm, param hypotheses 2}
    \left\{ \begin{gathered}
        A_{out} \ge A_*( n, \alpha, \kappa_{out} ) \gg 1, \qquad 
        \Upsilon \ge \Upsilon_*( n, \alpha, K, \kappa_{out}, A_{out} ) \gg 1, \\
        0 < \epsilon_{in} \le \epsilon_{in}^* (K, \gamma) \ll 1, \qquad 
        0 < \epsilon_{par} \le \epsilon_{par}^* ( n, \alpha, s_0, \Upsilon, \gamma, \kappa_{in} ) \ll 1, \\
        0 < \epsilon_{out} \le \epsilon_{out}^* ( n, \alpha, s_0, \Upsilon, \gamma, \kappa_{in}, \kappa_{out}, A_{out} ) \ll 1, \\
        0 < \delta_{in} \le \delta^* ( n, \alpha, K, s_0, \gamma, \kappa_{in}, \epsilon_{in} ) \ll 1, \\
        \text{and } \tau_0 \ge \tau_*( n, \alpha, K, s_0, \Upsilon, \gamma, \kappa_{in}, \kappa_{out} , \eta_+, \eta_-, \epsilon_{in}, \epsilon_{par}, \epsilon_{out}, A_{out} ) \gg 1,
    \end{gathered} \right. \end{gather}
    then the following holds:

    Assume $u : \R \times [\tau_0, \tau_1) \to \R$, $u = u(s, \tau)$ is an odd function of $s$ (for all $\tau \in [\tau_0, \tau_1)$) and $a : [\tau_0, \tau_1) \to \R$ is a $C^1$-function such that 
    \begin{gather}
        \label{thm main box thm, decomp assump 1}
        \partial_\tau u = \left( 1 - 2 \frac{\partial_\tau a}a \right) \beta_a + H_a u + Q_a(u) \qquad \text{on } \R \times [\tau_0, \tau_1) \\
        \label{thm main box thm, decomp assump 2}
        u =\phi +v, \quad \phi = \sum_{i=1}^K b_i(\tau) \phi_{i,a} - \sum_{i=1}^K b_i (\tau) \phi_{0,a}, \quad \forall \tau \in [\tau_0, \tau_1) \\
        \text{and} \quad v(\cdot, \tau) \perp_{L^2_a} \phi_{i,a} \quad \forall 0 \le i \le K, \, \forall \tau \in [\tau_0, \tau_1), 
    \end{gather}
    with initial data 
    \begin{gather} \label{thm main box thm, init data eqn 1}
        (u(\cdot, \tau_0), a(\cdot , \tau_0))\in \mathcal B \left[  \eta_-, 100 \eta_+, \gamma, \kappa_{in}, \kappa_{out}; \delta_{in} ,  \epsilon_{par} ,   \epsilon_{out}^2 , A_{out}^{1/2} ; \{ \tau_0\} \right]  
        \\
        \label{thm main box thm, init data eqn 2}
        \frac34 e^{\frac{1-K}2 \tau_0} \le a(\tau_0) \le \frac54 e^{\frac{1-K}2 \tau_0} ,\\
        \label{thm main box thm, init data eqn 3}
        \| v ( \cdot, \tau_0 ) \|_{C^{7, \alpha}\left( \left( \frac{s_0'}8, 8 \Upsilon \right) \right)} \le a(\tau_0)^{2 + 100 \eta_+}, 
         \text{ and } \| v ( \cdot, \tau_0) \|_{H^4_a\left( \left( \frac{s_0'}4, 4 \Upsilon\right) \right)} \le a(\tau_0)^{2+100 \eta_+} , 
    \end{gather}
    where $s_0' = s_0'(n, \alpha, \gamma, \kappa_{in} , s_0) \in (0, s_0)$ denotes the corresponding constant from Lemma \ref{lem preserving inner Holder condition}.

    If the ``exit time'' $\tau_{exit} = \tau_{exit}(u,a) \in [\tau_0, \tau_1]$ defined by 
        $$\tau_{exit}(u,a)  := \sup \left\{ \tau_1' \in [\tau_0, \tau_1] \colon (u,a) \in \mathcal B[  \eta_-,  \eta_+, \gamma, \kappa_{in}, \kappa_{out};  \epsilon_{in} , \epsilon_{par} , \epsilon_{out}, A_{out} ; [ \tau_0, \tau_1') ] \right\}$$
    satisfies $\tau_{exit}(u,a) < \tau_1$,
    then
    \begin{enumerate}
        \item 
        \label{main box thm, can't exit thru stable side}
            $$(u,a) \in \mathcal B\left[  \eta_-,  \eta_+, \gamma, \kappa_{in}, \kappa_{out};  \frac12 \epsilon_{in} ,\frac12 \epsilon_{par} ,\frac12 \epsilon_{out}, \frac12 A_{out} ; [\tau_0, \tau_{exit} ] \right],$$
        \item \label{main box thm, must exit thru unstable side}
            $(u,a)$ saturates one of the ``exit condition'' inequalities from \eqref{defn box, exit L2 conditions} at $\tau = \tau_{exit}$, and
        \item \label{main box thm, exit map is cts}
        there exists $\tau_{exit}^+ = \tau_{exit}^+(u,a) > \tau_{exit}$ such that
            \begin{align*} 
                |b_1(\tau) | &> a(\tau)^{2 + \eta_-} && \forall \tau \in (\tau_{exit}, \tau_{exit}^+), \\
                &\vdots \\
                |b_i(\tau) | &> a(\tau)^{2 + \eta_-}  && \forall \tau \in (\tau_{exit}, \tau_{exit}^+),\\
                &\vdots \\
                |b_{K-1} (\tau) | &> a(\tau)^{2 + \eta_-} && \forall \tau \in (\tau_{exit}, \tau_{exit}^+), 
                \\
                &\text{or} \\| b_K(\tau) - a^2(\tau)| &>  K a(\tau)^{2 + \eta_-} && \forall \tau \in (\tau_{exit}, \tau_{exit}^+).
            \end{align*}
    \end{enumerate}
\end{theorem}
\begin{proof}
    Let $(u,a)$ be as in the statement and assume $\tau_{exit} = \tau_{exit}(u,a) \in [\tau_0, \tau_1)$.
    Throughout, we assume $0 < s_0 \le s_*( n, K) \ll 1$ and $\tau_0 \ge \tau_*(n, K, s_0) \gg 1$ so that, for all $\tau \in [\tau_0, \tau_{exit}]$,
        $$a(\tau) \le 2 e^{\frac{1-K}2 \tau} \le 2 e^{\frac{1-K}2 \tau_0} < 1$$
    and the eigenmodes $\{ \phi_{i,a} \}_{0 \le i \le K}$ from \cite{SS26I}*{Theorem \ref{I-thm global eigenfunctions}} are defined.
    Note then, for any $\eta > 0$,
        $$a^{2 + \eta}(\tau) \le a^2(\tau) \qquad \forall \tau \in [\tau_0, \tau_{exit}]$$
    and $\sup_{\tau \in [\tau_0, \tau_{exit} ]} a(\tau) \to 0 $ as $\tau_0 \to +\infty$.
    Define
        $$Mod := \sum_{i=1}^K |\partial_\tau b_i - (1- i) b_i | + \left| \frac d{d\tau} (a^2) - a^2 + \sum_{i=1}^K i b_i \right| $$
    as in \eqref{eqn defn Mod}.
    We also define $\tilde \kappa_{in}, \hat \kappa_{out} \in \R $ via
    \begin{equation} \label{eqn tilde kappa, hat kappa defns}
        \kappa_{in} = ( 2 + \tilde \kappa_{in}) \frac{K-1}2 \quad \text{and} \quad
        \kappa_{out} =  \hat \kappa_{out} \frac{K-1}2
    \end{equation}
    We assume throughout that $\tilde \kappa_{in} , \hat \kappa_{out} > 0$.
    Observe that condition \eqref{defn box, preserved L2 conditions} from the definition of $\mathcal B$ implies
    \begin{align}
        \label{eqn tilde kappa and a}
         C(K, \kappa_{in} )^{-1} e^{-\kappa_{in} \tau} \le a(\tau)^{2 + \tilde \kappa_{in}} &\le C(K, \kappa_{in} ) e^{-\kappa_{in} \tau}
         && \forall \tau \in [\tau_0, \tau_{exit}], \text{ and} \\ 
         \label{eqn hat kappa and a}
        C(K, \kappa_{out})^{-1} e^{- \kappa_{out} \tau} \le a(\tau)^{\hat \kappa_{out}} &\le C(K, \kappa_{out}) e^{- \kappa_{out} \tau}
        && \forall \tau \in [\tau_0, \tau_{exit}].
    \end{align}

    When $0 < s_0 \le s_0^* (n,K ) \ll 1$, $\Upsilon > 1$, $\hat \kappa_{out} > 0$, $0 <  - \gamma \le \tilde \kappa_{in}$, 
    and 
    $$\tau_0 \ge \tau_*( n, K, s_0, \Upsilon, \gamma, \kappa_{in}, \kappa_{out})  \gg 1,$$
    Lemma \ref{lem L2 ests for Q} applies to give
        $$\| Q_a(u (\cdot, \tau)) \|_{L^2_a(\R)} \le C_Q \left( a^{\frac 72} + e^{-2 \kappa_{in} \tau} a^{2 \gamma + \frac n2 - 4} + e^{-2 \kappa_{out} \tau} \right)
        \qquad \forall \tau \in [\tau_0 , \tau_{exit} ]$$
    where $C_Q = C_Q(n, K, s_0, \Upsilon, \gamma, \kappa_{in}, \kappa_{out}) > 0$.
    Using \eqref{eqn tilde kappa, hat kappa defns}--\eqref{eqn hat kappa and a}, it follows that
    \begin{gather}
        \label{proof main box thm, L^2 est for Q}
        \| Q_a(u (\cdot, \tau) ) \|_{L^2_a(\R)} \le C_Q a^{2 + \eta_Q} \qquad \forall \tau \in [\tau_0, \tau_{exit}], \\
        \label{eqn defn eta_Q}
        \text{where } \eta_Q := \min \left\{ \frac32 , \frac{n-4}2 + 2 ( \tilde \kappa_{in} + \gamma) , \, 2 \hat \kappa_{out} - 2\right\}.
    \end{gather}
    Assuming further that $\frac{n-4}2 + 2 (\tilde \kappa_{in} + \gamma) > 0$ and $2 \hat \kappa_{out} - 2 \ge \frac32 $, equation \eqref{eqn defn eta_Q} becomes
    \begin{equation} \label{eqn defn eta_Q 2}
        \eta_Q = \min \left\{ \frac32 , \frac{n-4}2 + 2 ( \tilde \kappa_{in} + \gamma) \right\} > 0.
    \end{equation}

    If $\tau_0 \ge \tau_*(n, K, s_0, \eta_-) \gg 1$, then
    the assumptions \eqref{thm modulation eqns, eqn 4} of Theorem \ref{thm modulation eqns} hold (with $C_0 =2$).
    Thus, estimate \eqref{eqn Mod est 2} of Theorem \ref{thm modulation eqns} (with $C_0 = 2$) gives
    \begin{multline} \label{proof main box thm, Mod est}
        Mod (\tau) \le C(n, K, s_0) \cdot \left[  a^3(\tau) + \sum_{i=0}^K  \| Q_a(\tau) \|_{L^2_a} \| \phi_{i,a} \|_{L^2_a} \right]\\
        \le C(n,K, s_0) \left[ a^3(\tau) + C_Q a^{2 + \eta_Q} (\tau) \right]
        \qquad (\forall \tau \in [\tau_0, \tau_{exit}])
    \end{multline}
    where we used \cite{SS26I}*{Lemma \ref{I-lem L^2_a est for global eigenfunctions}} and \eqref{proof main box thm, L^2 est for Q} in the last inequality.
    Thus,
    \begin{gather}
        \label{proof main box thm, Mod est 2}
        Mod(\tau) \le C(n, K, s_0, C_Q) a^{2 + \eta_Q'} \qquad \forall \tau \in [\tau_0, \tau_{exit}] \\
        \label{eqn defn eta_Q'}
        \text{where } \eta_Q' := \min \{ 1 , \eta_Q \} = \min\left\{ 1 , \frac{n-4}2 + 2 (\tilde \kappa_{in} + \gamma) \right\} > 0.
    \end{gather}

    \textit{(The $a(\tau)$ Estimate is Preserved)} \\
    By \eqref{proof main box thm, Mod est 2} and the fact that $(u,a) \in \mathcal B [\eta_-, \eta_+,  \epsilon_{in}, \epsilon_{par}, \epsilon_{out}; [\tau_0, \tau_{exit} ] ] $, we obtain that
    \begin{gather} \label{proof main box thm, (a^2)' est 1} \begin{aligned}
        \left| \frac d {d\tau} (a^2) - (1-K) a^2 \right| 
        ={}& \left| \frac d {d\tau} (a^2) - a^2 + K a^2 - K b_K + Kb_K + \sum_{i=1}^{K-1} i b_i - \sum_{i=1}^{K-1} i b_i \right| \\
        \le{}& \left| \frac d{d \tau} (a^2) - a^2 + \sum_{i=1}^K i b_i \right| 
        + K |a^2 - b_K| + \sum_{i=1}^{K-1} i |b_i| \\
        \le{}& Mod + K^2 a^{2 + \eta_-} + \frac{K(K-1)}2 a^{2 + \eta_-} \\
        \le{}& C(n, K, s_0, C_Q) (  a^{2 + \eta_Q'} + a^{2 + \eta_-} )  \\
        \le{}& C(n, K, s_0, C_Q, \eta_Q, \eta_-) a^2 \left(  e^{\eta_Q' \frac{1-K}2 \tau} +  e^{\eta_- \frac{1-K}2 \tau} \right)
    \end{aligned} \end{gather}
    for all $\tau \in [\tau_0, \tau_{exit}]$.
    Thus, 
    \begin{gather} \label{proof main box thm, (a^2)' est 2} \begin{gathered}
        (1-K) a^2 -  Da^2 e^{-\delta \tau}  \le \frac{d}{d \tau} a^2 \le (1-K) a^2  + Da^2 e^{-\delta \tau}  < 0 \qquad \forall \tau \in[\tau_0, \tau_{exit} ] \\
        \text{where } D = D(n, K, s_0, C_Q, \eta_Q, \eta_- ) > 0 \quad \text{and} \quad \delta := \min \left\{ 1, \eta_Q  , \eta_-  \right\} \cdot \frac{K-1}2 > 0.
    \end{gathered} \end{gather}
    Integrating this estimate for $\frac d{d\tau} a^2$ gives
        $$a^2(\tau_0) e^{(1-K)(\tau - \tau_0)} e^{-\frac D \delta e^{-\delta \tau_0} }\le a^2(\tau) \le a^2(\tau_0) e^{(1-K)(\tau - \tau_0)} e^{\frac D \delta e^{-\delta \tau_0} } \qquad \forall \tau \in [\tau_0, \tau_{exit}],$$
    which, together with the initial data bound for $a(\tau_0)$ \eqref{thm main box thm, init data eqn 2}, implies
        $$\left( \frac 34 \right)^2 e^{(1-K)\tau } e^{-\frac D \delta e^{-\delta \tau_0} }\le a^2(\tau) \le \left( \frac 54 \right)^2 e^{(1-K)\tau } e^{\frac D \delta e^{-\delta \tau_0} } \qquad \forall \tau \in [\tau_0, \tau_{exit}].$$
    For $\tau_0 \ge \tau_*( n, K, s_0 , C_Q, \eta_Q, \eta_-) \gg 1$, it follows that
    \begin{equation} \label{proof main box thm, eqn a bound preserved}
        0 < \frac12 e^{\frac{1-K}2 \tau } < \frac 58 e^{\frac{1-K}2 \tau } \le a(\tau) \le \frac32 e^{\frac{1-K}2 \tau} < 2 e^{\frac{1-K}2 \tau} \qquad 
        \forall \tau \in [\tau_0, \tau_{exit} ].
    \end{equation}

    \textit{(The $|\partial_\tau a|$ Estimate is Preserved)} \\
    By dividing estimate \eqref{proof main box thm, (a^2)' est 2} through by $2a >0$, it follows that
    \begin{equation} \label{proof main box thm, eqn a' bound+}
        \frac{1-K}2 a  - \frac D2 a e^{-\delta \tau}  \le \frac{d a}{d\tau} \le \frac{1-K}2 a  + \frac D2 a e^{-\delta \tau}   \qquad \forall \tau \in [\tau_0, \tau_{exit} ].
    \end{equation}
    For $\tau_0 \ge \tau_*( n, K, s_0 , C_Q, \eta_Q, \eta_-) \gg 1$, this estimate implies
    \begin{equation} \label{proof main box thm, eqn a' bound preserved}
        \left| \frac{da}{d \tau} \right| \le \frac 34 (K-1) a <  (K-1) a \qquad \forall \tau \in [\tau_0, \tau_{exit} ].
    \end{equation}

    \textit{(The $\| v(\cdot, \tau) \|_{L^2_a(\R)}$ Estimate is Preserved)}
    \begin{claim} \label{claim L^2 norm of v est}
        Assume the parameters satisfy $n \ge 3$, $K \ge 2$, $0 < - \gamma \le \tilde \kappa_{in}$, $\eta_+ \frac{K-1}2 < 1$,
        and either 
        \begin{gather}
            \eta_+ < \eta_Q' < \frac2{K-1} \qquad \text{or} \qquad  \frac2{K-1} \le \eta_Q'.
        \end{gather}
        If $\tau_0 \ge \tau_*( n, K, s_0, \eta_-, \eta_+, \eta_Q, C_Q, \gamma, \kappa_{in} ) \gg 1$, then
            $$\| v ( \cdot, \tau) \|_{L^2_a} \le \frac12 a(\tau)^{2 + \eta_+} \qquad \forall \tau \in [\tau_0, \tau_{exit}].$$
    \end{claim}
    \begin{claimproof}
    Fix  $\epsilon_0 = \epsilon_0 ( K, \eta_+, \eta_Q)   > 0$ to be some small constant which is to be determined.
    If also $\tau_0 \ge \tau_*( n, K, s_0, \eta_-, \epsilon_0) \gg 1$ is sufficiently large, then Lemma \ref{lem L^2_a est for v} (with $C_0 = 2$) applies to give
    \begin{multline}
        \frac12 \frac d{d \tau} \left( \| v \|_{L^2_a(\R)}^2 \right)
        \le (-K + \epsilon_0 ) \| v \|_{L^2_a(\R)}^2  \\
        + C(n, K,s_0) \cdot    \left(1 + a^{\frac{n-3}2 + \gamma} \sup_{|s|\le s_0} \rho_a^{-\gamma} |v|  \right) \left( a^3 + \| Q_a(u) \|_{L^2_a(\R)} \right) \| v \|_{L^2_a(\R)}
    \end{multline}
    for all $\tau \in (\tau_0, \tau_{exit}) $.
    Invoking \eqref{proof main box thm, L^2 est for Q}, \eqref{eqn defn eta_Q'}, and the definition of $\mathcal B$, we deduce that for all $\tau \in (\tau_0, \tau_{exit} ) $
    \begin{multline} \label{proof main box thm, L^2-norm of v diff eqn est 1}
        \frac d{d \tau} \| v \|_{L^2_a} 
        \le ( - K+\epsilon_0) \| v\|_{L^2_a} 
        + C(n, K, s_0) \cdot \left( 1 + a^{\frac{n-3}2 + \gamma} e^{-\kappa_{in} \tau} \epsilon_{in} \right) \left( a^3 + C_Q a^{2 + \eta_Q} \right) \\
        \le ( - K+\epsilon_0) \| v\|_{L^2_a} 
        + C(n, K, s_0, C_Q) \cdot \left( 1 + a^{\frac{n-3}2 + \gamma} e^{-\kappa_{in} \tau} \epsilon_{in} \right) a^{2 + \eta_Q'}  .
    \end{multline}
    Observe that, since $- \gamma \le \tilde \kappa_{in}$  and $n \ge 3$,
        $$\kappa_{in} = (2 + \tilde \kappa_{in} ) \frac{K-1}2 > \left( - \frac{n-3}2- \gamma  \right) \frac{K-1}2,$$
    and thus
    \begin{equation}
        a^{\frac{n-3}2 + \gamma} e^{-\kappa_{in} \tau} 
        \le 2^{\frac{n-3}2 + |\gamma| } e^{ - \kappa_{in} \tau + \left( - \gamma - \frac{n-3} 2 \right) \left( \frac{K-1}2\right) \tau}
        \le 1 \qquad \forall \tau \in [\tau_0, \tau_{exit}]
    \end{equation}
    when $\tau_0 \ge \tau_*( n, K, \gamma, \kappa_{in} ) \gg 1$.
    Incorporating this estimate into \eqref{proof main box thm, L^2-norm of v diff eqn est 1} then gives
    \begin{align*}
        \frac d{d \tau} \| v \|_{L^2_a} 
        \le{}& (-K + \epsilon_0) \| v \|_{L^2_a} + C(n, K, s_0, C_Q) \cdot a^{2 + \eta_Q'}  \\
        \le{}& (-K + \epsilon_0) \| v \|_{L^2_a} + D e^{( 2 + \eta_Q') \frac{1-K}2 \tau} 
        && \left( \text{since }  a \le 2 e^{\frac{1-K}2 \tau} \right)
    \end{align*}
    for all $\tau \in [\tau_0, \tau_{exit}]$, where $D = D(n, K, s_0, C_Q, \eta_Q) > 0$.
    Integrating this estimate in $\tau$ and multiplying both sides by $a(\tau)^{-2 - \eta_+}$ then gives
    \begin{multline} \label{proof main box thm, L^2-norm of v est 1}
        a(\tau)^{-2-\eta_+} \| v ( \cdot, \tau) \|_{L^2_a} 
        \le a(\tau)^{-2-\eta_+}  e^{(-K+\epsilon_0) ( \tau- \tau_0) } \| v (\cdot, \tau_0) \|_{L^2_a} \\
        +  a(\tau)^{-2-\eta_+} \int_{\tau_0}^\tau e^{(-K + \epsilon_0) ( \tau - \tilde \tau) } D e^{( 2 + \eta_Q') \frac{1 - K}2 \tilde \tau } d \tilde \tau 
    \end{multline}
    for all $\tau \in [\tau_0, \tau_{exit} ]$.
    
    The first term on the right-hand side of \eqref{proof main box thm, L^2-norm of v est 1} can be estimated using the initial data bounds \eqref{thm main box thm, init data eqn 1} and the bound $\frac12 e^{\frac{1-K}2 \tau} \le a \le 2 e^{\frac{1-K}2 \tau}$ (see \eqref{defn box, preserved L2 conditions} from the definition of $\mathcal B$) to give
    \begin{gather} \label{proof main box thm, claim L^2_a est v, eqn 10} \begin{aligned}
        a(\tau)^{-2 - \eta_+} e^{(- K + \epsilon_0)(\tau - \tau_0)} \| v ( \cdot, \tau_0) \|_{L^2_a}  
        \le{}& a(\tau)^{-2 - \eta_+} e^{(-K+\epsilon_0)(\tau - \tau_0)} a(\tau_0)^{2 + 100 \eta_+} \\
        \le{}& a( \tau_0)^{99 \eta_+} 4^{2 + \eta_+} e^{(-K+ \epsilon_0) ( \tau - \tau_0)} e^{( 2 + \eta_+) \left( \frac{1-K}2 \right) ( \tau_0 - \tau)} \\
        \le{}& a(\tau_0)^{99\eta_+} 4^{2 + \eta_+} e^{\left( -1 + \epsilon_0 + \eta_+ \cdot\frac{K-1}2 \right) (\tau - \tau_0)}.
    \end{aligned} \end{gather}
    If
    \begin{equation}
        \eta_+ \frac{K-1}2 < 1 
    \end{equation}
    then $\epsilon_0 >0 $ can be chosen sufficiently small (depending only on $K, \eta_+$) so that
    \begin{equation} \label{eqn eta_+, epsilon_0 compatibility}
        \eta_+ \frac{K-1}2 \le  1 - \epsilon_0. 
    \end{equation}
    In this case, \eqref{proof main box thm, claim L^2_a est v, eqn 10} then becomes
    \begin{gather} \label{proof main box thm, claim L^2_a est v, eqn 11} \begin{aligned}
        &a(\tau)^{-2 - \eta_+} e^{(- K + \epsilon_0)(\tau - \tau_0)} \| v ( \cdot, \tau_0) \|_{L^2_a}\\
        \le{}& a(\tau_0)^{99\eta_+} 4^{2 + \eta_+} e^{\left( -1 + \epsilon_0 + \eta_+ \cdot\frac{K-1}2 \right) (\tau - \tau_0)} \\
        \le{}& 2^{99\eta_+} e^{99 \eta_+ \cdot \frac{1-K}2 \tau_0} \cdot 4^{2 + \eta_+}  
        && \left(\text{using } a \le 2 e^{\frac{1-K}2 \tau} \right)\\
        \le{}& \frac1{100} 
    \end{aligned} \end{gather}
    for $\tau \in [ \tau_0, \tau_{exit}]$ and $\tau_0 \ge \tau_*( K, \eta_+) \gg 1$.
    
    The second term on the right-hand side of \eqref{proof main box thm, L^2-norm of v est 1} can be estimated using $\frac12 e^{\frac{1-K}2 \tau} \le a(\tau) \le 2 e^{\frac{1-K}2 \tau}$ to give
    \begin{gather} \label{proof main box thm, claim L2 est for v, eqn 12}\begin{aligned}
        &a(\tau)^{-2-\eta_+} \int_{\tau_0}^\tau e^{(-K + \epsilon_0) ( \tau - \tilde \tau) } D e^{( 2 + \eta_Q') \frac{1 - K}2 \tilde \tau } d \tilde \tau \\
        \le{}& 2^{2 + \eta_+} e^{(-2-\eta_+)\left( \frac{1-K}2 \right) \tau} e^{(-K+\epsilon_0)\tau} D
        \int_{\tau_0}^\tau e^{\left( K - \epsilon_0  + (2 + \eta_Q') \frac{1-K}2 \right) \tilde \tau} d \tilde \tau  \\
        ={}& 2^{2 + \eta_+} e^{\left( - 1 + \eta_+ \frac{K-1}2 + \epsilon_0\right) \tau} D
        \int_{\tau_0}^\tau e^{\left( 1  - \eta_Q' \frac{K-1}2- \epsilon_0 \right) \tilde \tau} d \tilde \tau  .
    \end{aligned} \end{gather}
    There are now two cases to consider.
    
    \textit{(Case 1: $\eta_+ < \eta_Q' < \frac2{K-1}$)}
    Consider the case that
        $$\eta_+ < \eta_Q' < \frac2{K-1}.$$
    Then
        $$1 - \eta_Q' \frac{K-1}2 > 0,$$
    and $0 < \epsilon_0 \le  \epsilon_0^*(K, \eta_Q) \ll 1$ can be chosen sufficiently small so that 
        $$1 -  \eta_Q' \frac{K-1}2  -\epsilon_0 > 0.$$
    Thus, \eqref{proof main box thm, claim L2 est for v, eqn 12} can be bounded as
    \begin{gather} \label{proof main box thm, claim L2 est for v, eqn 13} \begin{aligned}
        &2^{2 + \eta_+} D e^{\left( -1+\eta_+ \frac{K-1}2 + \epsilon_0 \right) \tau} 
        \int_{\tau_0}^\tau e^{\left( 1 - \eta_Q' \frac{K-1}2 - \epsilon_0 \right) \tilde \tau} d \tilde \tau \\
        \le{}& \frac{2^{2 +\eta_+} D}{1 - \eta_Q' \frac{K-1}2 - \epsilon_0 }  e^{\left( -1+\eta_+ \frac{K-1}2 + \epsilon_0 \right) \tau} e^{\left( 1 - \eta_Q' \frac{K-1}2 - \epsilon_0 \right)  \tau} \\
        ={}& \frac{2^{2 +\eta_+} D}{1  - \eta_Q' \frac{K-1}2 - \epsilon_0}  e^{( \eta_+ - \eta_Q') \frac{K-1}2 \tau} \\
        \le{}& \frac1{100} 
    \end{aligned} \end{gather}
    where the last inequality follows so long as 
        $$K \ge 2, \quad \eta_+ < \eta_Q', \quad \text{and} \quad  \tau \ge \tau_0 \ge \tau_*( n, K, s_0 , \eta_+, \eta_Q, C_Q, \epsilon_0)   \gg 1.$$

    \textit{(Case 2: $\frac2{K-1} \le \eta_Q'$)}
    Consider now the case that $\frac2{K-1} \le \eta_Q'$.
    Then 
        $$1 - \eta_Q' \frac{K-1}2 \le 0, $$
    and so
        $$1 - \eta_Q' \frac{K-1}2 - \epsilon_0  < 0\qquad \text{for any } \epsilon_0 > 0.$$
    Hence, \eqref{proof main box thm, claim L2 est for v, eqn 12} can be bounded as
    \begin{align*}
        &2^{2 + \eta_+} D e^{\left( -1+\eta_+ \frac{K-1}2 + \epsilon_0 \right) \tau} 
        \int_{\tau_0}^\tau e^{\left( 1 - \eta_Q' \frac{K-1}2 - \epsilon_0 \right) \tilde \tau} d \tilde \tau  \\
        \le{}& \frac{2^{2 + \eta_+} D}{\left|1  - \eta_Q' \frac{K-1}2 - \epsilon_0 \right|} e^{\left( -1+\eta_+ \frac{K-1}2 + \epsilon_0 \right) \tau} e^{\left( 1 - \eta_Q' \frac{K-1}2 - \epsilon_0 \right) \tau_0} \\
        \le{}& \frac1{100}
    \end{align*}
    where the last inequality follows from using
    \begin{gather*}
        K \ge 2, \quad -1 + \eta_+ \frac{K-1}2+ \epsilon_0  \le 0 \text{ (see  \eqref{eqn eta_+, epsilon_0 compatibility})}, \qquad 
        1 - \eta_Q' \frac{K-1}2-\epsilon_0 <0, \\
        \text{and } \tau \ge \tau_0 \ge \tau_*( n, K, s_0, \eta_+, \eta_Q, C_Q,  \epsilon_0) \gg1.
    \end{gather*}

    In either case, we deduce that,  for $0 < \epsilon_0 \le \epsilon_0^*( K, \eta_Q ) \ll 1$, \eqref{proof main box thm, claim L2 est for v, eqn 12} can be bounded as
    \begin{multline} \label{proof main box thm, claim L^2_a est v, eqn 12}
        a(\tau)^{-2-\eta_+} \int_{\tau_0}^\tau e^{(-K + \epsilon_0) ( \tau - \tilde \tau) } D e^{( 2 + \eta_Q') \frac{1 - K}2 \tilde \tau } d \tilde \tau
        \le \frac1{100}
        \\
        \text{for } \tau_{exit} \ge \tau \ge \tau_0 \ge\tau_*( n, K, s_0, \eta_+, \eta_Q, C_Q) \gg 1.
    \end{multline}
    Inserting this estimate \eqref{proof main box thm, claim L^2_a est v, eqn 12} and \eqref{proof main box thm, claim L^2_a est v, eqn 11} into \eqref{proof main box thm, L^2-norm of v est 1} then gives 
    $$\| v ( \cdot, \tau) \|_{L^2_a} \le \frac1{50} a(\tau)^{2 + \eta_+}$$
    as claimed.
    \end{claimproof}

    \textit{(Coarse $C^{2, \alpha}$ Bounds in the Parabolic Region are Preserved)} \\
    Let $s_0' = s_0'(n, \alpha, s_0 , \gamma, \kappa_{in} ) \in (0, s_0)$ denote the corresponding constant from Lemma \ref{lem preserving inner Holder condition}.
    We seek to apply Theorem \ref{lem preserving Holder bounds in parab region+}, but we first confirm its hypotheses hold in our setting.
    Clearly,
    \begin{equation} \label{proof main box thm, parab region preserved, eqn 1}
        0 < a(\tau) \le 2 e^{\frac{1-K}2 \tau} \qquad \text{and} \qquad |\partial_\tau a | \le (K-1) a \qquad \forall \tau \in [\tau_0, \tau_{exit}]
    \end{equation}
    by the fact that $(u,a) \in \mathcal B[[\tau_0, \tau_{exit}]]$.

    Combining the $C^{2,\alpha}$ bounds for the inner, parabolic, and outer regions that hold for $\tau \in [\tau_0, \tau_{exit} ]$, 
    it follows that, for all $\tau \in [\tau_0, \tau_{exit}]$,
    \begin{gather} \label{proof main box thm, parab region preserved, eqn 2} \begin{aligned}
        &\| u ( \cdot, \tau) \|_{C^{2, \alpha} \left( \frac {s_0'} 8, 8 \Upsilon \right) }\\
        \le{}& \| \phi \|_{C^{2, \alpha} \left( \frac {s_0'} 8, s_0 \right) } + C(s_0, s_0', \gamma) \| v \|_{2, \alpha; \left( \left\{ \frac{s_0'}8 < |s| < s_0 \right\} \right) }^{(\gamma)}+  \epsilon_{par} + C(\Upsilon) \epsilon_{out}\\
        \le{}& C(n, \alpha, K, s_0, \gamma, \kappa_{in}) \left( a^2 +  e^{-\kappa_{in} \tau} \right) + \epsilon_{par} + C(\Upsilon) \epsilon_{out}\\
        \le{}& C(n, \alpha, K, s_0, \gamma, \kappa_{in} ) e^{(1-K) \tau}  + \epsilon_{par} + C(\Upsilon ) \epsilon_{out},
    \end{aligned} \end{gather}
    where the last inequality uses $\kappa_{in} \ge ( 2 - \gamma) \frac{K-1}2$ and $\gamma < 0$.
    Moreover, the initial data estimates \eqref{thm main box thm, init data eqn 1} and \eqref{thm main box thm, init data eqn 3} give
    \begin{multline} \label{proof main box thm, parab region preserved, eqn 3}
        \| u ( \cdot, \tau_0) \|_{C^{7, \alpha} \left( \frac{s_0'}8, 8 \Upsilon \right) } 
        \le \| \phi ( \cdot, \tau_0) \|_{C^{7, \alpha} \left( \frac{s_0'}8, 8 \Upsilon \right) }  + \| v ( \cdot, \tau_0) \|_{C^{7, \alpha} \left( \frac{s_0'}8, 8 \Upsilon \right) } \\
        \le C(n, \alpha , K, s_0, \Upsilon, \gamma, \kappa_{in} )\cdot e^{(1-K) \tau_0}
    \end{multline}
    and
    \begin{multline} \label{proof main box thm, parab region preserved, eqn 4}
        \| u ( \cdot, \tau_0) \|_{H^4_a\left( \left( \frac{s_0'}4 , 4 \Upsilon \right) \right) }
        \le \| \phi ( \cdot, \tau_0) \|_{H^4_a\left( \left( \frac{s_0'}4 , 4 \Upsilon \right) \right) } + \| v ( \cdot, \tau_0) \|_{H^4_a\left( \left( \frac{s_0'}4 , 4 \Upsilon \right) \right) } \\
        \le C(n, \alpha , K, s_0, \Upsilon, \gamma, \kappa_{in}) \cdot e^{(1-K)\tau_0}.
    \end{multline}
    Additionally, \cite{SS26I}*{Lemma \ref{I-lem L^2_a est for global eigenfunctions}} shows
    \begin{multline} \label{proof main box thm, parab region preserved, eqn 5}
        \| u ( \cdot, \tau) \|_{L^2_a ( \R) } 
        \le \sum_{i=1}^K |b_i(\tau) | \| \phi_{i,a} - \phi_{0,a} \|_{L^2_a(\R)} + \| v (\cdot, \tau) \|_{L^2_a(\R)} \\
        \le C(n, K)\cdot  a(\tau)^2 + a(\tau)^{2 + \eta_+} 
        \le C(n, K) \cdot e^{(1-K)\tau} 
        \qquad \forall \tau \in [\tau_0, \tau_{exit}].
    \end{multline}

    Combining estimates \eqref{proof main box thm, parab region preserved, eqn 1}--\eqref{proof main box thm, parab region preserved, eqn 5}, 
    it follows that the hypotheses of Theorem \ref{lem preserving Holder bounds in parab region+} hold 
    (with constants $C_0 = K-1, C_1 =2, \kappa = K-1, C_2 = C(n,K),$ and $C_3 = C(n,K, s_0, \Upsilon, \gamma, \kappa_{in})$ therein)
    so long as
    \begin{gather*}
        0 < \epsilon_{par}, \epsilon_{out} \le \epsilon^*( n, \alpha , s_0, \Upsilon, \gamma, \kappa_{in}) \ll 1 \qquad 
        \text{and} \qquad  \tau_0 \ge \tau_* ( n, \alpha ,K,  s_0, \Upsilon, \gamma, \kappa_{in}) \gg 1.
    \end{gather*}
    In this case, Theorem \ref{lem preserving Holder bounds in parab region+} implies that
    \begin{multline} \label{proof main box thm, strong parab region Holder est}
        \| u ( \cdot, \tau) \|_{C^{2, \alpha}\left( \left( \frac{s_0}2 , 2 \Upsilon \right) \right)}
        \le \| u ( \cdot, \tau) \|_{C^{2, \alpha}\left( \left( \frac{s_0'}2 , 2 \Upsilon \right) \right)} \le C (n, K, s_0, \Upsilon, \gamma, \kappa_{in} ) e^{(1-K) \tau} \le \frac12 \epsilon_{par} \\ \forall \tau \in [\tau_0, \tau_{exit} ]
    \end{multline}
    if also $\tau_0 \ge \tau_*(n, K, s_0, \Upsilon, \gamma, \kappa_{in}, \epsilon_{par} ) \gg 1$.

    We note additionally that Claim \ref{proof preserving Holder bounds in parab region+, claim 1} from the proof of Theorem \ref{lem preserving Holder bounds in parab region+} holds to give that
    \begin{multline} \label{proof main box thm, C^7 est on u in parab region}
        \sup_{\tau \in [\tau_0, \tau_{exit}]} \| u ( \cdot, \tau) \|_{C^7\left( \left( \frac{s_0'}4, 4 \Upsilon \right) \right) } \le C(n, \alpha, s_0, \Upsilon, \gamma, \kappa_{in} ) \cdot \left[  e^{(1-K) \tau_0} + \epsilon_{par} + C(\Upsilon) \epsilon_{out} \right] \\
        \le C(n, \alpha , s_0,\Upsilon, \gamma, \kappa_{in}) \cdot ( \epsilon_{par} + \epsilon_{out} ) 
    \end{multline}
    so long as
    \begin{gather*}
        0 < \epsilon_{par}, \epsilon_{out} \le \epsilon^*( n, \alpha , s_0, \Upsilon, \gamma, \kappa_{in}) \ll 1 \qquad 
        \text{and} \qquad  \tau_0 \ge \tau_* ( n, \alpha ,K,  s_0, \Upsilon, \gamma, \kappa_{in}, \epsilon_{out}) \gg 1.
    \end{gather*}

    \textit{(Weighted $C^{2, \alpha}$ Bounds in the Outer Region are Preserved)}\\
    We seek to apply Lemma \ref{lem preserving Holder bounds in outer region}, but we must first confirm its hypotheses hold.
    Here, let $\mathbf C_* = \mathbf C_*(n, \alpha, \kappa_{out}) > 1$ denote the constant ``$C$'' from the statement of Lemma \ref{lem preserving Holder bounds in outer region}.
    
    Clearly,
    \begin{equation} \label{proof main box thm, outer region preserved, eqn 1}
        0 < a(\tau) \le 2 e^{\frac{1-K}2 \tau} \le \min \left\{ 2 e^{- \frac{\kappa_{out}}2 \tau } , 1 \right\} \quad \text{and} \quad |\partial_\tau a | \le (K-1) a \qquad \forall \tau \in [\tau_0, \tau_{exit} ]
    \end{equation}
    if $\tau_0 \ge 2 \ln (2)$.
    Equation \eqref{proof main box thm, strong parab region Holder est} yields ``thickened'' lateral boundary estimates
    \begin{multline*}
        \sum_{j=0}^2 |s|^j |\partial_s^j u ( s, \tau) | + |s|^{j+\alpha} [ \partial_{ss}u ( \cdot, \tau) ]_{C^{0, \alpha}( \frac12 |s| , 2 |s| ) }
        \le C(n, K, s_0, \Upsilon) e^{(1-K) \tau} \\
        \le \min \left\{ \frac{A_{out}}{2 \mathbf C_*} e^{-\kappa_{out} \tau} s^{2\kappa_{out}+2} , \frac{\epsilon_{out}}{2 \mathbf C_*} s^2 \right\}
        \qquad \forall (s, \tau) \in \left[\frac18 \Upsilon, \Upsilon \right] \times [\tau_0, \tau_{exit} ]
    \end{multline*}
    where the last inequality holds so long as additionally
    $$0< \kappa_{out} < K-1, \quad A_{out} > 1,  \quad \text{and} \quad \tau_0 \ge \tau_*( n, \alpha, K,  s_0, \Upsilon, \kappa_{out}, \epsilon_{out} ) \gg 1.$$

    From the initial data bounds \eqref{thm main box thm, init data eqn 1} and \eqref{thm main box thm, init data eqn 3}, it follows that for all $s \ge \frac18\Upsilon$
    \begin{multline}
        \sum_{j=0}^2 |s|^j |\partial_s^j u ( s, \tau_0) | + |s|^{j+\alpha} [ \partial_{ss} u ( \cdot, \tau_0) ]_{C^{0, \alpha}( \frac12 |s| , 2 |s| ) }
        \le  \min \left\{ A_{out}^{1/2} e^{-\kappa_{out} \tau_0} s^{2\kappa_{out}+2} , \epsilon_{out}^2 s^2 \right\} \\
        \le \min \left\{ \frac{A_{out}}{2 \mathbf C_*} \cdot e^{-\kappa_{out} \tau_0} s^{2\kappa_{out}+2} , \frac{\epsilon_{out}}{2 \mathbf C_*} s^2 \right\} 
    \end{multline}
    if 
    \begin{gather*}
        A_{out} \ge A_* (n, \alpha, \kappa_{out})  \gg 1, 
        \quad
        0 < \epsilon_{out} \le \epsilon_{out}^* ( n, \alpha, \kappa_{out}  ) \ll 1 , \quad\text{and} \\
        \tau_0 \ge \tau_*(n, K, \alpha, s_0, \Upsilon, \kappa_{out}, \epsilon_{out}) \gg 1.
    \end{gather*}

    In this case, Lemma \ref{lem preserving Holder bounds in outer region} (with $C_0 = K-1, C_1 = 2, A = \frac{A_{out}}{2 \mathbf C_*}, \kappa = \kappa_{out}, \epsilon = \epsilon_{out}$) applies to give
    \begin{multline} \label{proof main box thm, preserving outer region Holder bounds}
        \sum_{j=0}^2 |s|^j |\partial_s^j u ( s, \tau_0) | + |s|^{j+\alpha} [ \partial_{ss} u ( \cdot, \tau_0) ]_{C^{0, \alpha}( \frac12 |s| , 2 |s| ) }
        \le  \min \left\{  \frac{A_{out}}2 e^{-\kappa_{out} \tau} s^{2\kappa_{out}+2},     \frac{\epsilon_{out}}2 s^2 \right\}
    \end{multline}
    for all $(s, \tau) \in [ \Upsilon, \infty) \times [\tau_0, \tau_{exit} ]$
    so long as the parameters additionally satisfy
    \begin{gather*}
        0 < \epsilon_{out} \le \epsilon_{out}^*( n, \alpha, \kappa_{out}, A_{out}) \ll 1, \quad
        \Upsilon \ge \Upsilon_*( n, \alpha, K, \kappa_{out}, A_{out} ) \gg 1, \text{ and}\\
        \tau_0 \ge \tau_*( n, \alpha, K, \kappa_{out}, A_{out}, \epsilon_{out} , \Upsilon) \gg 1.
    \end{gather*}

   \textit{(Weighted $C^{2, \alpha}$ Bounds in the Inner Region are Preserved)} \\
   We first need the following claim to improve the time decay of $\| v \|_{C^{2, \alpha}}$ at scale $s_0$.
   \begin{claim} \label{proof main box thm, claim v Holder decay at s_0}
       If additionally
       $$0 < \epsilon_{par}, \epsilon_{out} \le \epsilon^* ( n, \alpha, s_0, \Upsilon, \gamma, \kappa_{in} ) \ll 1 \quad  \text{and} \quad \tau_0 \ge \tau_*( n,\alpha, K, s_0, \Upsilon, \gamma, \kappa_{in}, \epsilon_{out}) \gg 1,$$
       then
        \begin{multline} \label{proof main box thm, claim v Holder decay at s_0, eqn 1}
            \| v ( \cdot, \tau ) \|^{(0)}_{2, \alpha ; \left\{ s_0' < |s| < 4s_0 \right\} }
            \le C(n, \alpha, K, s_0, \gamma, \kappa_{in}, C_Q, \eta_+, \eta_Q) \cdot e^{\left(2 + \min \{ 1, \eta_+,\eta_Q \}  \right) \frac{1-K}2 \tau } \\
            \forall \tau \in [\tau_0, \tau_{exit}]
        \end{multline}
        where $s_0' = s_0'(n, \alpha, s_0 , \gamma, \kappa_{in} ) \in (0, s_0)$ denotes the corresponding constant from Lemma \ref{lem preserving inner Holder condition}.
   \end{claim}
   \begin{claimproof}
        Using \eqref{proof main box thm, C^7 est on u in parab region}, the hypotheses of Lemma \ref{lem local higher order energy ests for v} hold with
        $$k = 4, \quad \Omega = \{ s_0'/2 < |s| < 8 s_0\} , \quad \Omega' = \{ s_0' < |s| < 4s_0 \}, $$
        so long as additionally
        $$0 < \epsilon_{par}, \epsilon_{out} \le \epsilon^* ( n, \alpha, s_0, \Upsilon, \gamma, \kappa_{in} ) \ll 1 \quad \text{and} \quad \tau_0 \ge \tau_*( n,\alpha, K, s_0, \Upsilon, \gamma, \kappa_{in}, \epsilon_{out}) \gg1.$$
        In this case, for any $\tau \in [\tau_0 + 1, \tau_{exit} ]$, 
        \begin{gather} \label{proof main box thm, proof claim v Holder decay at s_0, est 1} \begin{aligned}
            &\| v ( \cdot, \tau ) \|^{(0)}_{2, \alpha ; \left\{ s_0' < |s| < 4s_0 \right\} } \\
            \le{}& C(n, \alpha,  s_0, \gamma, \kappa_{in} ) \| v \|_{L^\infty H^4_a \left( \left\{ s_0' < |s| < 4s_0 \right\} \times [\tau - 1/2, \tau] \right) }
            && ( \text{Sobolev embedding} ) \\
            \le{}& C(n, \alpha, K, s_0, \gamma, \kappa_{in} ) \cdot \left( \| v \|_{L^2 L^2_a \left( \left\{ \frac{s_0'}2 < |s| < 8s_0 \right\} \times [\tau - 1, \tau] \right) }\right.\\
            &\left.\qquad + \| Mod \|_{L^2 ( [\tau-1, \tau]) } + \| a^3 \|_{L^2([\tau-1, \tau])} \right) \\
            & \qquad (\text{by Lemma \ref{lem local higher order energy ests for v}, Remark \ref{remark local higher order energy ests for v+}}) \\
            \le{}& C(n, \alpha, K, s_0, \gamma, \kappa_{in}, C_Q) \cdot \left( a(\tau)^{2 + \eta_+} + a(\tau)^{2 + \eta_Q'}  \right) 
             && (\text{by \eqref{proof main box thm, Mod est 2}} )\\
             \le{}& C(n, \alpha, K, s_0, \gamma, \kappa_{in}, C_Q, \eta_+, \eta_Q) \cdot e^{\left(2 + \min \{ 1, \eta_+, \eta_Q \}  \right) \frac{1-K}2 \tau }\\
             & \qquad \left(\text{since } \frac12 e^{\frac{1-K}2 \tau} \le a \le 2 e^{\frac{1-K}2 \tau} \text{ and } \eta_Q' = \min \{ 1, \eta_Q\} \right) .
        \end{aligned} \end{gather}
        By incorporating the initial data bound $\| v (\cdot, \tau_0) \|_{H^4_a\left( \frac{s_0'}4, 8s_0\right) } \le a(\tau_0)^{2 + 100 \eta_+}$ from \eqref{thm main box thm, init data eqn 3}, an estimate similar to \eqref{proof main box thm, proof claim v Holder decay at s_0, est 1} can also be applied for $\tau \in [\tau_0, \tau_0+1]$.
        It therefore follows that
        \begin{multline*}
            \| v ( \cdot, \tau ) \|^{(0)}_{2, \alpha ; \left\{ s_0' < |s| < 4s_0 \right\} }
            \le C(n, \alpha, K, s_0, \gamma, \kappa_{in}, C_Q, \eta_+, \eta_Q) \cdot e^{\left(2 + \min \{ 1, \eta_+,\eta_Q \}  \right) \frac{1-K}2 \tau } \\
            \forall \tau \in [\tau_0, \tau_{exit}].
        \end{multline*}
   \end{claimproof}

   Assuming that
    \begin{equation}
        \gamma \in (-1, 0) \quad \text{and} \quad 
        (2-\gamma) \frac{K-1}2 \le \kappa_{in} < \left( 2 + \min \{ \eta_+, \eta_Q, 1 \} \right) \frac{K-1}2 \le 3 \frac{K-1}2,
    \end{equation}
    with the estimates \eqref{proof main box thm, Mod est 2}, \eqref{thm main box thm, init data eqn 1}, and \eqref{proof main box thm, claim v Holder decay at s_0, eqn 1},
    it then follows that Theorem \ref{thm preserving inner Holder condition+} can be applied (with $\kappa = \kappa_{in}$ and $\epsilon = \frac12 \epsilon_{in}$) to deduce
    \begin{equation} \label{proof main box thm, preserving inner region Holder bounds}
        \sup_{\tau \in [\tau_0, \tau_{exit}]} e^{\kappa_{in} \tau} \| v ( \cdot, \tau) \|_{2, \alpha; (-s_0, s_0)}^{(\gamma)} \le \frac12 \epsilon_{in}
    \end{equation}
    so long as 
    \begin{gather*}
        0 < \epsilon_{in} \le \epsilon^* (K, \gamma) \ll1, \qquad 
        0 < \delta_{in} \le \delta^* ( n, \alpha, K, s_0, \gamma, \kappa_{in}, \epsilon_{in} ) \ll 1 \\
        \tau_0 \ge \tau_*( n, \alpha, K, s_0, \gamma, \kappa_{in} ,  C_Q, \eta_Q, \eta_+, \epsilon_{in}) \gg 1.
    \end{gather*}

    Combining \eqref{proof main box thm, eqn a bound preserved}, \eqref{proof main box thm, eqn a' bound preserved}, Claim \ref{claim L^2 norm of v est}, \eqref{proof main box thm, strong parab region Holder est}, \eqref{proof main box thm, preserving outer region Holder bounds}, and \eqref{proof main box thm, preserving inner region Holder bounds} proves that 
    $$( u, a) \in \mathcal B \left[ \eta_-, \eta_+,  \gamma, \kappa_{in}, \kappa_{out};  \frac12 \epsilon_{in},  \frac12 \epsilon_{par},  \frac12 \epsilon_{out} , \frac12 A_{out},  ; [\tau_0, \tau_{exit} ] \right]$$
    and that one of the ``exit condition'' inequalities must be saturated at $\tau= \tau_{exit}$.
    In other words, this proves conclusions (\ref{main box thm, can't exit thru stable side}) and (\ref{main box thm, must exit thru unstable side}) from the statement of the theorem.

    \textit{(Proof of Conclusion (\ref{main box thm, exit map is cts}))}\\
    We consider two cases based on which exit condition is saturated at $\tau = \tau_{exit}$.
    Recall that \eqref{proof main box thm, eqn a' bound+} gives
    \begin{equation}\label{proof main box thm, eqn a' bound+ 2}
        \left| \frac1a \frac{da}{d \tau} - \frac{1-K}2  \right| \le D  e^{-\delta \tau} \qquad \forall \tau \in [\tau_0, \tau_{exit}]
    \end{equation}
    where $D = D(n, K, s_0, C_Q, \eta_Q, \eta_-) > 0$ and $\delta = \min \{ 1 , \eta_Q, \eta_-\} \frac{K-1}2 > 0$.
    
    \textit{Case 1:}
    Assume in this case that $|b_k(\tau_{exit})| = a(\tau_{exit})^{2 + \eta_-}$ for some $1 \le k\le K-1$.
    Suppose $b_k (\tau_{exit}) = a(\tau_{exit} )^{2 + \eta_-} > 0$.
    (The subcase where $b_k (\tau_{exit}) = -a(\tau_{exit} )^{2 + \eta_-} < 0$ will follow from an analogous argument.)
    Computing a time derivative at $\tau=\tau_{exit}$ yields
    \begin{align*}
        &\left. \frac{d}{d \tau} \right|_{\tau_{exit}} \left( a^{-2 - \eta_-} b_k \right) \\
        ={}& (-2-\eta_-) a^{-2-\eta_-} b_k \frac{\partial_{\tau} a}{a} + a^{-2-\eta_-} \partial_\tau b_k \\
        ={}& (-2-\eta_-) \frac{\partial_\tau a}{a} + a^{-2-\eta_-} (1-k) b_k + a^{-2-\eta_-} (\partial_\tau b_k - (1-k) b_k) \\
        & \quad ( \text{since } b_k(\tau_{exit}) = a(\tau_{exit})^{2 + \eta_-} )\\
        \ge{}& (-2-\eta_-) \frac{1-K}2 - (2+\eta_-) D e^{-\delta \tau_{exit}}
        + (1-k) - a^{-2-\eta_-} Mod(\tau_{exit})
        && ( \text{by \eqref{proof main box thm, eqn a' bound+ 2}, \eqref{eqn defn Mod}}) \\
        \ge{}& 1 + \eta_- \frac{K-1}2 - ( 2 + \eta_-) D e^{-\delta \tau_0 } - a ^{-2-\eta_-} Mod (\tau_{exit})
        && ( \text{since } 1 \le k \le K-1) \\
        >{}& 1 - (2 + \eta_-) D e^{-\delta \tau_0} - C(n, K, s_0, C_Q)   a(\tau_{exit})^{\eta_Q' - \eta_-} 
        && ( \text{by \eqref{proof main box thm, Mod est 2}}) \\
        >{}& 0 
    \end{align*}
    where the last line holds so long as 
        $$\eta_- < \eta_Q'   \qquad \text{and} \qquad \tau_{exit} \ge \tau_0 \ge \tau_*( n, K, s_0, C_Q, \eta_Q, \eta_-) \gg 1.$$
    The sign of this derivative at $\tau = \tau_{exit}$ therefore implies
        $$b_k(\tau) > a(\tau)^{2 + \eta_-} \quad \text{for all $\tau > \tau_{exit}$ in a neighborhood of } \tau_{exit}.$$

    The subcase where $b_k (\tau_{exit}) = -a(\tau_{exit} )^{2 + \eta_-} < 0$ follows mutatis mutandis.

    \textit{Case 2:}
    Assume in this case that $|b_K(\tau_{exit}) - a^2 (\tau_{exit})| = K a(\tau_{exit})^{2 + \eta_-}$.
    Suppose $b_K(\tau_{exit}) - a^2 (\tau_{exit}) = -K a(\tau_{exit})^{2 + \eta_-} < 0$.
    (The subcase where $b_K(\tau_{exit}) - a^2 (\tau_{exit}) = K a(\tau_{exit})^{2 + \eta_-}>0$ follows from an analogous argument.)
    Computing a time derivative at $\tau = \tau_{exit}$ yields
    \begin{align*}
        &\left. \frac{d}{d \tau} \right|_{\tau = \tau_{exit}} \left[ a^{-2-\eta_-} ( b_K - a^2 ) \right]\\
        ={}& (-2-\eta_-) a^{-2-\eta_-} ( b_K - a^2) \frac{\partial_\tau a}{a} + a^{-2-\eta_-} \left( \partial_\tau b_K - \partial_\tau (a^2) \right) \\
        ={}& ( 2+ \eta_-) K \frac{\partial_\tau a}{a} 
        + a^{-2 -\eta_-} \left[ \partial_\tau b_K - (1-K) b_K - \partial_\tau (a^2) + a^2 - \sum_{i=1}^K i b_i \right] \\
        &+ a^{-2-\eta_-}(b_K -a^2 ) + a^{-2-\eta_-} \sum_{i=1}^{K-1} i b_i \\
        \le{}& ( 2 + \eta_-) K \left( \frac{1-K}2 \right) + (2+\eta_-) K D e^{-\delta \tau_{exit}} + 2 a^{-2-\eta_-} Mod(\tau_{exit})
        && (\text{by \eqref{proof main box thm, eqn a' bound+ 2}} ) \\
        & -K +  \sum_{i=1}^{K-1} i \\
        & ( \text{since } |b_i| \le a^{2 + \eta_-} \text{ for } 1 \le i \le K-1)\\ 
        \le{}& - \frac12 K^2 - \frac12 K - \eta_- \frac{K( K-1)}2 
        + ( 2 + \eta_-) K D e^{-\delta \tau_0} \\
        &+ C(n, K, s_0, C_Q)  a(\tau_{exit})^{\eta_Q' - \eta_-}  \\
        <{}& 0 
    \end{align*}
    where the last line follows from assuming $K \ge 2 > 0$, $0 < \eta_- < \eta_Q'$, and
        $$\tau_{exit} \ge \tau_0 \ge \tau_*( n, K, s_0 , C_Q, \eta_Q, \eta_-) \gg 1.$$
    The sign of this derivative at $\tau = \tau_{exit}$ therefore implies that
        $$b_K(\tau) - a(\tau)^2 < -K a(\tau)^{2 + \eta_-}
        \quad \text{for all $\tau> \tau_{exit}$ in a neighborhood of } \tau_{exit}.$$

    The subcase where $b_K(\tau_{exit}) - a(\tau_{exit})^2 = K a(\tau_{exit})^{2 + \eta_-} > 0$ follows mutatis mutandis.

    \textit{(Numerology)}\\
    In summary, we have completed the proof of the theorem under the following assumptions on the parameters:
    \begin{gather*}
        n \ge 3, \qquad K \ge 2, \qquad \alpha \in (0,1),  \qquad
        0 < s_0 \le s_0^* (n, K) \ll 1, \qquad
        - 1 < \gamma < 0, 
    \end{gather*}
    \begin{gather*}
        ( 2 -\gamma ) \frac{K-1}2 \le \kappa_{in} = ( 2 + \tilde \kappa_{in} ) \frac{K-1}2
        < ( 2 + \min \{ 1 , \eta_+, \eta_Q \} ) \frac{K-1}2 \le 3 \frac{K-1}2,\\
        0 < - \gamma \le \tilde \kappa_{in} , \qquad \frac{n-4}2 + 2 ( \tilde \kappa_{in} + \gamma) > 0,\\
        0 < \kappa_{out} = \hat \kappa_{out} \frac{K-1}2 < K-1, \qquad 
        0 < \hat \kappa_{out} , \qquad 2 \hat \kappa_{out} - 2 \ge \frac32,
    \end{gather*}
    \begin{gather*}
        \eta_Q = \min \left\{ \frac32, \frac{n-4}2 + 2 (\tilde \kappa_{in} + \gamma ) \right\} , \qquad \eta_Q' = \min \{ 1, \eta_Q \} = \min \left\{ 1, \frac{n-4}2 + 2 (\tilde \kappa_{in} + \gamma ) \right\}, \\
        0 < \eta_+ , \qquad \eta_+ \frac{K-1}2 < 1 , \qquad 
        \left( \eta_+ < \eta_Q' < \frac2{K-1} \quad \text{OR} \quad \frac2{K-1} \le \eta_Q' \right) , \qquad 
        0 < \eta_- < \eta_Q',
    \end{gather*}
    \begin{gather*}
        A_{out} \ge A_*( n, \alpha, \kappa_{out} ) \gg 1, \qquad 
        \Upsilon \ge \Upsilon_*( n, \alpha, K, \kappa_{out}, A_{out} ) \gg 1, \\
        C_Q = C_Q(n, K, s_0, \Upsilon, \gamma, \kappa_{in}, \kappa_{out} ) > 0, 
    \end{gather*}
    \begin{gather*}
        0 < \epsilon_{in} \le \epsilon_{in}^* (K, \gamma) \ll 1, \qquad 
        0 < \epsilon_{par} \le \epsilon_{par}^* ( n, \alpha, s_0, \Upsilon, \gamma, \kappa_{in} ) \ll 1, \\
        0 < \epsilon_{out} \le \epsilon_{out}^* ( n, \alpha, s_0, \Upsilon, \gamma, \kappa_{in}, \kappa_{out}, A_{out} ) \ll 1, \\
        0 < \delta_{in} \le \delta^* ( n, \alpha, K, s_0, \gamma, \kappa_{in}, \epsilon_{in} ) \ll 1, \\
        \text{and } \tau_0 \ge \tau_*( n, \alpha, K, s_0, \Upsilon, \gamma, \kappa_{in}, \kappa_{out} , \eta_+, \eta_-, \epsilon_{in}, \epsilon_{par}, \epsilon_{out}, A_{out}, C_Q, \eta_Q ) \gg 1.
    \end{gather*}
    Note that 
    \begin{gather*}
        \eta_+ \frac{K-1}2 < 1 \text{ and } \left( \eta_+ < \eta_Q' < \frac2{K-1} \text{ or } \frac2{K-1} \le \eta_Q' \right) 
        \implies \eta_+ = \min \{ 1 , \eta_+, \eta_Q' \} .
    \end{gather*}
    It is now straightforward to verify that the above conditions on the parameters are equivalent to those listed in the hypotheses of this theorem.
\end{proof}

Theorem \ref{thm main box thm} does not claim there \emph{exist} parameters satisfying \eqref{main box thm, param hypotheses 1} and \eqref{main box thm, param hypotheses 2}.
It is clear there exist parameters satisfying \eqref{main box thm, param hypotheses 2}, since each parameter depends only on parameters preceding it.
The next two propositions show parameters satisfying \eqref{main box thm, param hypotheses 1} do indeed exist when
    $$(n \ge 4 \text{ and } K \ge 2) \text{ or } ( n = 3 \text{ and } K \in \{ 2, 3,4\} ) .$$

\begin{proposition}[Numerology, $n \ge 4$] \label{prop numerology n>3}
    If $n \ge 4$ and $K \ge 2$,
    then there exist parameters $s_0, \alpha, \gamma, \tilde \kappa_{in}, \eta_+ , \hat \kappa_{out} , \eta_-$ satisfying \eqref{main box thm, param hypotheses 1}.
    
    In particular, if $n \ge 4$, $K \ge 2$, $\alpha \in (0,1)$, $0 < s_0 \le s^* (n,K) \ll1 $, 
    $- \frac2{10(K-1)}< \gamma < 0$,
    $\tilde \kappa_{in} =  4|\gamma| $,
    $\eta_+ = 5 |\gamma|$,
    $\hat \kappa_{out} \in \left[ \frac74, 2 \right)$, 
    and $0 < \eta_- < \eta_+ = 5 |\gamma|$,
    then these parameters satisfy \eqref{main box thm, param hypotheses 1}.
\end{proposition}
\begin{proof}
    The proof is a straightforward verification left to the reader.
\end{proof}

\begin{proposition}[Numerology, $n = 3$] \label{prop numerology n=3}
    If $n  =3$ and $K \in \{ 2, 3,4\}$,
    then there exist parameters $s_0, \alpha, \gamma, \tilde \kappa_{in}, \eta_+ , \hat \kappa_{out} , \eta_-$ satisfying \eqref{main box thm, param hypotheses 1}.

    In particular, the following hold when $n=3$ and $\alpha \in (0,1)$:
    \begin{enumerate}
        \item If $K = 2$, $0 < s_0 \le s^*(n,K) \ll 1$, 
        $-\frac 1{10} <   \gamma < 0$,
        $\tilde \kappa_{in} = 1 - 2 |\gamma|$,
        $\eta_+ = 1 - |\gamma|$, 
        $\hat \kappa_{out} \in \left[ \frac74, 2 \right)$, 
        and $0 < \eta_- < \eta_+$,
        then the parameters satisfy \eqref{main box thm, param hypotheses 1}.

        \item If $K = 3$, $0 < s_0 \le s^*(n,K) \ll 1$, 
        $-\frac1{12} \le  \gamma < 0$,
        $\tilde \kappa_{in} = 1 - 2 |\gamma|$,
        $\eta_+ = 1 - |\gamma|$, 
        $\hat \kappa_{out} \in \left[ \frac74, 2 \right)$, 
        and $0 < \eta_- < \eta_+$,
        then the parameters satisfy \eqref{main box thm, param hypotheses 1}.

        \item If $K=4$, $0 < s_0 \le s^*(n,K) \ll 1$, 
        $- \frac1{18} \le  \gamma < 0$,
        $\tilde \kappa_{in} = \frac23 - 2 |\gamma|$,
        $\eta_+ = \frac23 - |\gamma|$, 
        $\hat \kappa_{out} \in \left[ \frac74, 2 \right)$, 
        and $0 < \eta_- < \eta_+$,
        then the parameters satisfy \eqref{main box thm, param hypotheses 1}.
    \end{enumerate}
\end{proposition}
\begin{proof}
    The proof is a straightforward verification left to the reader.
\end{proof}

In light of Theorem \ref{thm main box thm} and Propositions \ref{prop numerology n>3}--\ref{prop numerology n=3}, the topological argument outlined in Subsection \ref{Subsect Proof Outline} can now be applied to prove the existence of the desired flow solutions.

\begin{corollary} \label{cor main box cor}
    If ($n \ge 4$ and $K \ge 2$) or ($n =3$ and $K \in \{2, 3,4\}$),
    then there exist parameters $s_0, \alpha, \gamma , \tilde \kappa_{in}, \kappa_{in} = ( 2 + \tilde \kappa_{in} ) \frac{K-1}2, \eta_+, \hat \kappa_{out}, \kappa_{out} = \hat \kappa_{out} \frac{K-1}2, \eta_-$ as in Propositions \ref{prop numerology n>3}--\ref{prop numerology n=3} which satisfy \eqref{main box thm, param hypotheses 1}, and 
    there exist parameters $A_{out}, \Upsilon, \epsilon_{in}, \epsilon_{par}, \epsilon_{out} , \delta_{in}, \tau_0$ satisfying \eqref{main box thm, param hypotheses 2}
    such that the following holds:

    For any $v_0(\cdot)$ and $a_0 \in \R $ such that
    \begin{gather}
        \left( a_0^2 ( \phi_{K,a_0}- \phi_{0,a_0}) + v_0 , a_0 \right) \in \mathcal B \left[ \eta_-, 100 \eta_+, \gamma, \kappa_{in}, \kappa_{out}; \delta_{in} ,  \epsilon_{par} ,   \epsilon_{out}^2 , A_{out}^{1/2} ; \{ \tau_0\} \right] , \\
        \frac34 e^{\frac{1-K}2 \tau_0} \le a_0 \le \frac54 e^{\frac{1-K}2 \tau_0} ,\\
        \| v_0  \|_{C^{7, \alpha}\left( \left( \frac{s_0'}8, 8 \Upsilon \right) \right)} \le a_0^{2 + 100 \eta_+}, 
         \text{ and } \| v_0  \|_{H^4_{a_0}\left( \left( \frac{s_0'}4, 4 \Upsilon\right) \right)} \le a_0^{2+100 \eta_+} ,
    \end{gather}
    there exist functions $u : \R \times [\tau_0, \infty) \to \R$, $u = u(s, \tau)$ an odd function of $s$, and $a : [\tau_0, \infty) \to (0, \infty)$ such that 
    \begin{gather}
        \label{cor main box cor, eqn 4}
        \partial_\tau u = \left( 1 - 2 \frac{\partial_\tau a}a \right) \beta_a + H_a u + Q_a(u) \qquad \text{on } \R \times [\tau_0, \infty) \\
        \qquad 
        \label{cor main box cor, eqn 5}
        u =\phi +v, \quad \phi = \sum_{i=1}^K b_i(\tau) \phi_{i,a} - \sum_{i=1}^K b_i (\tau) \phi_{0,a}, \quad 
        v(\cdot, \tau) \perp_{L^2_a} \phi_{i,a} \quad \forall \tau \in [\tau_0, \infty), \\
        \label{cor main box cor, eqn 6}
        v(\cdot, \tau_0) = v_0, \qquad \text{and} \qquad a(\tau_0) = a_0,
    \end{gather}
    and 
    \begin{equation}
        \label{cor main box cor, eqn 7}
        (u, a) \in \mathcal B[ n, K, \alpha, s_0, \Upsilon;  \eta_-,  \eta_+, \gamma, \kappa_{in}, \kappa_{out};  \epsilon_{in} , \epsilon_{par} , \epsilon_{out}, A_{out} ; [ \tau_0, \infty) ].
    \end{equation}
\end{corollary}
\begin{proof}
    Throughout the proof, we simplify the notation by writing
        $$\mathcal B[\tau_0, \tau_1)  = \mathcal B[ n, K, \alpha, s_0, \Upsilon;  \eta_-,  \eta_+, \gamma, \kappa_{in}, \kappa_{out};  \epsilon_{in} , \epsilon_{par} , \epsilon_{out}, A_{out} ; [ \tau_0, \tau_1) ].$$

    Let $\mathbb D$ denote the closed topological $K$-ball
        $$\mathbb D := \left\{ \mathbf p = (p_1, \dots, p_K) \in \R^K \, \colon \, \max_{1 \le i \le K} |p_i| \le 1 \right\} .$$
    Given $v_0, a_0$ as in the statement of the corollary, 
    consider, for each $\mathbf p  \in \mathbb D$, the perturbed initial data
    \begin{gather} 
    \label{proof main box cor, eqn 1}
        u_{\mathbf p} ( \cdot, \tau_0) := \sum_{i=1}^{K-1} p_i a_0^{2+\eta_-}  ( \phi_{i, a_0} - \phi_{0, a_0} ) + ( a_0^2 + K p_K a_0^{2 + \eta_-} ) ( \phi_{K, a_0} - \phi_{0, a_0} )  + v_0 , \\
    \label{proof main box cor, eqn 1.2}
        a (\tau_0) = a_0, \qquad 
        h_{\mathbf p} ( \cdot, \tau_0) = f_{a_0} + \partial_s u_{\mathbf p}( \cdot, \tau_0) .
    \end{gather}
    There then exists a unique classical solution $h_{\mathbf p}$ to the rescaled Lagrangian mean curvature flow equation \eqref{eq: rescaled flow of profile function} for short time (see e.g. \cite{Neves07}*{Proof of Theorem 4.1}).
    The geometric decomposition Proposition \ref{Prop Geometric Decomposition} can then be applied to write $h_{\mathbf p}$ for short time $\tau \in [\tau_0, \tau_1(\mathbf p))$ as 
    \begin{equation} \label{proof main box cor, eqn 2.0}
        h_{\mathbf p}  = f_{a_{\mathbf p}} + \partial_s ( \phi_{\mathbf p} + v_{\mathbf p})
    \end{equation}
    where
    \begin{gather}
        \label{proof main box cor, eqn 2.1}
        a_{\mathbf p} = a_{\mathbf p}(\tau) \text{ is } C^1, \quad 
        \phi_{\mathbf p} = \sum_{i=1}^K b_{\mathbf p, i}(\tau) ( \phi_{i,a_{\mathbf p}} - \phi_{0,a_{\mathbf p}} ) \, \,   \text{with } b_{\mathbf p, i} (\tau) \in C^1 \text{ for all } 1 \le i \le K ,\\
        \label{proof main box cor, eqn 2.2}
        v_{\mathbf p} \text{ is an odd function of $s$ such that } \langle v_{\mathbf p} , \phi_{i,a} \rangle_{L^2_a} = 0 \text{ for all } 1 \le i \le K,\\
        \label{proof main box cor, eqn 2.3}
        a_{\mathbf p}(\tau_0) = a_0, \qquad
        v_{\mathbf p}(\cdot, \tau_0) = v_0, \\
        \label{proof main box cor, eqn 2.4}
        b_{\mathbf p, i}(\tau_0) = p_i a_0^{2 + \eta_-} \, (\forall 1 \le i \le K-1), \quad \text{and } \quad b_{\mathbf p, K}(\tau_0) = a_0^2 + K p_K a_0^{2 + \eta_-}.
    \end{gather}
    Let $u_{\mathbf p }( \cdot, \tau) =  \phi_{\mathbf p} + v_{\mathbf p}$ denote the sum of these functions $\phi_{\mathbf p}, v_{\mathbf p}$ given above.
    Then $u_{\mathbf p}$ solves \eqref{eq: linearization at Lawlor at potential level with constant c}.
    Since $u_{\mathbf p}$ is an odd function of $s$, evaluating \eqref{eq: linearization at Lawlor at potential level with constant c} at $s=0$ reveals $u_{\mathbf p}$ in fact solves 
    \begin{equation}  \label{proof main box cor, eqn 2.5}
        \partial_\tau u_{\mathbf p}  = \left( 1 - 2 \frac{\partial_\tau a_{\mathbf p}}{a_{\mathbf p}} \right) \beta_{a_{\mathbf p}} + H_{a_{\mathbf p}} u_{\mathbf p} + Q_{a_{\mathbf p}}(u_{\mathbf p}) \qquad \text{on } \R \times [\tau_0, \tau_1(\mathbf p)).
    \end{equation}

    Without loss of generality, we can take $\tau_1(\mathbf p)$ (which also depends on the initial data) to be the maximal time such that the above conditions \eqref{proof main box cor, eqn 2.0}--\eqref{proof main box cor, eqn 2.5} hold.
    It then follows that 
    $$\tau_1(\mathbf p) < \infty \quad \implies \quad (u_{\mathbf p} , a_{\mathbf p}) \notin \mathcal B  \big[ \tau_0, \tau_1( \mathbf p) \big) ,$$
    else the arguments above could be repeated to contradict the choice of the maximality of $\tau_1(\mathbf p)$.

    For each $\mathbf p \in \mathbb D$, define the ``exit time'' $\tau_{exit} = \tau_{exit}(\mathbf p)$ as in Theorem \ref{thm main box thm}, that is,
    $$\tau_{exit}(\mathbf p) := \sup \left\{ \tau_1' \in [\tau_0, \tau_1(\mathbf p)] \colon (u_{\mathbf p},a_{\mathbf p}) \in \mathcal B [ \tau_0, \tau_1')  \right\}.$$
    Suppose for the sake of contradiction that $\tau_{exit}(\mathbf p) < \infty$ for every $\mathbf p \in \mathbb D$.
    Consider the function
    \begin{gather*}
        \Phi : \mathbb D \to \R^K , \\
        \Phi ( \mathbf p ) := \frac1{a_{\mathbf p}(\tau_{exit})^{2 + \eta_-}}\left( b_{\mathbf p, 1}(\tau_{exit}), \dots , b_{\mathbf p, K-1} ( \tau_{exit}), \frac{b_{\mathbf p, K}(\tau_{exit}) - a_{\mathbf p}(\tau_{exit})^2}{K} \right) .
    \end{gather*}
    It then follows from Theorem \ref{thm main box thm} that $\Phi$ is a continuous function $\mathbb D \to \partial \mathbb D$ such that $\Phi |_{\partial \mathbb D} = Id_{\partial \mathbb D}$.
    However, the $K-1$ homology of $\mathbb D$ and $\partial \mathbb D$ prevents the existence of such a function $\Phi$.
    This is a contradiction, and thus there exists some $\mathbf p_* \in \mathbb D$ such that $\tau_{exit}(\mathbf p_*) = \tau_1(\mathbf p_*) = \infty$.
    By definition of $\tau_{exit}$, we then have that
    \begin{equation*}
        (u_{\mathbf p_*}, a_{\mathbf p_*} ) \in \mathcal B [ \tau_0, \infty ),
    \end{equation*}
    which completes the proof of the corollary.    
\end{proof}

\subsection{Proof of the Main Theorem} \label{Subsect Proof Main Thm Paper II}

\begin{proof}[Proof of Theorem \ref{main thm paper II intro}]
    Let ($n \ge 4$ and $K \ge 2$) or ($n = 3$ and $K \in \{ 2,3,4\}$) and fix $G \le SU(n)$ as in the statement of Theorem \ref{main thm paper II intro}.
    Fix parameters $$s_0, \alpha, \tilde \kappa_{in}, \eta_+, \hat \kappa_{out} , \eta_-, A_{out}, \Upsilon, \epsilon_{in}, \epsilon_{par}, \epsilon_{out} , \delta_{in}, \tau_0$$ 
    and initial data $v_0(\cdot), a_0$ as in Corollary \ref{cor main box cor}.
    Without loss of generality, we can further choose the initial data so that 
    \begin{gather} 
        \label{proof of main thm paper II, initial data assumption 1}
        \| v(\cdot, \tau_0) \|_{3; (-s_0, s_0)}^{(\gamma)} \le \epsilon_{in} e^{-\kappa_{in} \tau_0},
        \quad
        \| u ( \cdot, \tau_0) \| _{C^3\left( \frac{s_0}2 < |s| <  2 \Upsilon \right)} \le \epsilon_{par},\\
        \label{proof of main thm paper II, initial data assumption 2}
         \text{and} \quad 
        \| v(\cdot, \tau_0) \|_{3; (\Upsilon, \infty)}^{(2)} \le \epsilon_{out}.
    \end{gather}
    
    Throughout the proof, we will simplify the notation by writing
        $$\mathcal B[\tau_0, \tau_1)  = \mathcal B[ n, K, \alpha, s_0, \Upsilon;  \eta_-,  \eta_+, \gamma, \kappa_{in}, \kappa_{out};  \epsilon_{in} , \epsilon_{par} , \epsilon_{out}, A_{out} ; [ \tau_0, \tau_1) ].$$
    Corollary \ref{cor main box cor} can then be applied to obtain functions $u(s,\tau)$ and $a(\tau)$ satisfying \eqref{cor main box cor, eqn 4}--\eqref{cor main box cor, eqn 7}. 
    The arguments in Section \ref{section Prelims} ensure $h(s,\tau) = f_{a(\tau)}(s) + \partial_s u(s, \tau)$ is the profile function for a parabolically rescaled $G$-invariant Lagrangian mean curvature flow $M(\tau) = e^{\tau/2} L (T - e^{-\tau})$ defined for $\tau \in [\tau_0, \infty)$.
    Because $(u,a)$ is in the set $\mathcal B[\tau_0, \infty)$, it follows that $h$ has $C^1$-bounds, which, together with solving the quasilinear PDE \eqref{eq: LMCF of profile function}, suffices to deduce that $h$ is smooth on $[\tau_0+1, \infty)$ \cite{LSU88}*{Ch. VI, Sect. 1}.
    Hence, $L(t)$ is a smooth, non-compact, properly embedded, $G$-invariant, Lagrangian mean curvature flow for $t \in [T - e^{-\tau_0-1}, T)$, and item \ref{main thm paper II intro, item profile} of Theorem \ref{main thm paper II intro} follows from the definition of $\mathcal B$.
    These estimates on $h$ further ensure that the tangent flow of $L_K(t)$ at $(\mathbf 0, T)$ is unique and given by the stationary flow of a pair of $G$-invariant special Lagrangian cones (Theorem \ref{main thm paper II intro}, Item \ref{main thm paper II intro, item tangent flow}).
    In particular, the flow $L_K(t)$ forms a finite-time singularity at $(\mathbf 0, T)$, which proves Theorem \ref{main thm paper II intro}, Item \ref{main thm paper II intro, item finite-time sing}.

    We proceed to estimate the curvature blow-up rate.
    \begin{claim} \label{proof of main thm paper II, claim 1}
        There exists $C (n, K, \alpha, s_0, \gamma) > 0$ such that 
        $$\| u (\cdot, \tau) \|_{3, \alpha; (-s_0, s_0) }^{(2)} \le C a(\tau)^{ \tilde \kappa_{in} + \gamma} \qquad \forall \tau \in [\tau_0, \infty).$$
    \end{claim}
    \begin{claimproof}
        Recall from Corollary \ref{cor main box cor} that $u$ decomposes as
            $$u = \sum_{i=1}^K b_i(\tau) ( \phi_{i,a} - \phi_{0,a} ) + v.$$
        Using the definition of the box $\mathcal B$ and \eqref{proof preserving inner Holder condition, proof claim 4, eqn 5}, we can estimate
        \begin{multline*}
            \left\| \sum_{i=1}^K b_i(\tau) ( \phi_{i,a} - \phi_{0,a} ) \right\|_{2, \alpha; (-s_0, s_0) }^{(2)}
            \le C \sum_{i=1}^K a^2 \| \phi_{i,a} - \phi_{0,a} \|_{2, \alpha; (-s_0, s_0)}^{(2)} \\
            \le C \sum_{i=1}^K a^2 a^{-1 + \frac \gamma 2} \| \phi_{i,a} - \phi_{0,a} \|_{2 , \alpha ; (-s_0, s_0)}^{(\frac \gamma 2 + 1) } 
            \le C a^{1+ \frac \gamma 2}
            \le C a^{ \tilde \kappa_{in} + \gamma}
        \end{multline*}
        for some $C = C(n, K , \alpha, s_0, \gamma) > 0$ which may change from line to line.
            
        The definition of $\mathcal B$ and the weighted H{\"o}lder norm ensures, for a universal constant $C$,
        \begin{equation*}
            \| v (\cdot, \tau) \|_{2, \alpha; (-s_0,s_0)}^{(2)} 
            \le e^{-\kappa_{in} \tau} a^{-2+\gamma} \| v ( \cdot, \tau) \|_{2, \alpha; (-s_0, s_0)}^{(\gamma)} 
            \le  C \epsilon_{in} a^{ \tilde \kappa_{in} + \gamma}.
        \end{equation*}        
        
        The weighted $C^{2, \alpha}$-estimate for $u$ now follows.
        By differentiating the PDE for $u$ with respect to $s$ and appealing to the interior estimate Proposition \ref{prop: weighted Schauder estimate intermediate}, it follows that the weighted $C^{2, \alpha}$ estimate bootstraps to the claimed weighted $C^{3, \alpha}$-estimate.
    \end{claimproof}

    By \cite{MadnickWood25}*{Proposition 4.13}, the second fundamental form of the rescaled Lagrangian $M(\tau) = e^{\tau/2} L (T - e^{-\tau})$ satisfies, for some $C= C (n, G) > 0$,
    \begin{equation} \label{proof of main thm paper II, eqn 1}
        |\mathbf A_{M}|^2 \le \frac{h_{ss}^2}{(1 +h_s^2)^3} + \frac C{s^2 + h^2} \le h_{ss}^2 + \frac C{h^2}.
    \end{equation}
    By Lemma \ref{lem: profile function of Lawlor neck}, there exists $C = C(n) > 1$ such that 
    \begin{equation} \label{proof of main thm paper II, eqn 2}
        C^{-1} \rho_a(s) \le f_a \le C \rho_a(s), \quad
        |\partial_s f_a| \le C , 
        \qquad \text{and} \qquad C^{-1} \rho_a(s)^{-1} \le |\partial_{ss} f_a | \le C \rho_a(s)^{-1}.
    \end{equation}
    Combining Claim \ref{proof of main thm paper II, claim 1} and \eqref{proof of main thm paper II, eqn 1}--\eqref{proof of main thm paper II, eqn 2}, it follows that for some some $C = C(n, G, K, \alpha, s_0, \gamma) > 0$ and all $s \in (-s_0, s_0)$ 
    \begin{multline*}
        |\mathbf A_{M}|^2
        \le | \partial_{ss} f_a + \partial_{sss} u|^2 + \frac C{( f_a + \partial_s u)^2}\\
        \le C \rho_a^{-2} + C a^{2 (\tilde \kappa_{in} + \gamma)} \rho_a^{-2} + \frac C{\rho_a^2 ( 1 - C a^{\tilde \kappa_{in} + \gamma} )^2 } 
        \le \frac C{\rho_a^2}
        \le \frac C{a^2} 
        \\ \text{for all } \tau \ge \tau_*( n, G, K, \alpha, s_0, \gamma) \gg 1.
    \end{multline*}
    
    This proves that on the inner region $|s| < s_0$ and for $\tau$ sufficiently large, $|\mathbf A_M| \lesssim a^{-1}$.
    Therefore, for $t$ sufficiently close to $T$,
    $$\sup_{|x| \le s_0 \sqrt{T-t}}|\mathbf A_{L(t)} | = \frac1{\sqrt{T-t}} \sup_{|s| \le s_0} |\mathbf A_{M(\tau(t))} | \lesssim \frac1{\sqrt{T-t}} \frac1 {a(\tau(t))} \lesssim  (T-t)^{-K/2}. $$
    The initial data assumptions \eqref{proof of main thm paper II, initial data assumption 1}--\eqref{proof of main thm paper II, initial data assumption 2} ensure Lemma \ref{lem C^2 Profile Ests in outer region} applies.
    By \eqref{proof of main thm paper II, eqn 1} and the $C^2$-estimates from Lemmas \ref{lem C^2 Profile Ests on outer-parab overlap}--\ref{lem C^2 Profile Ests in outer region}, for $t$ sufficiently close to $T$,
    \begin{equation*}
        \sup_{|x| \ge s_0 \sqrt{T-t}} | \mathbf A_{L(t)} |^2
        \le \sup_{|x| \ge s_0 \sqrt{T-t}} |\tilde h_{xx}|^2 +  \frac C{x^2 + \tilde h^2}
        \lesssim \sup_{|x| \ge s_0 \sqrt{t-t}}  \frac1{x^2} 
        \lesssim \frac1{T-t}
    \end{equation*}
    where $\tilde h (x,t) = \sqrt{T-t} \, h (x / \sqrt{T-t}, -\ln (T-t) )$.
    Combining these last two estimates for $\mathbf A_{L(t)}$ thereby proves $\sup_{x \in \R} |\mathbf A_{L(t)}| \lesssim (T-t)^{-K/2}$.
   
    Recall the curvature $\boldsymbol \kappa$ of the curve $\{ (s, h(s)) \, | \,  s \in \R \} \subset \R^2$ is given by $\frac{\partial_{ss} h}{\sqrt {1 + (\partial_s h)^2}}.$
    \cite{MadnickWood25}*{Proof of Proposition 4.13} shows that 
        $$|\mathbf A_{M(\tau)} | \ge |\boldsymbol \kappa| = \left| \frac{\partial_{ss} f_a + \partial_{sss} u}{\sqrt{1 + (\partial_s f_a + \partial_{ss} u)^2}} \right|.$$
    Evaluating the right-hand side at $s=0$ where $\partial_s f_a (0)  +\partial_{ss} u(0) = 0$ and using \eqref{proof of main thm paper II, eqn 2} and Claim \ref{proof of main thm paper II, claim 1} shows 
    \begin{equation*}
        \sup_{M(\tau)} |\mathbf A_{M(\tau)}| \ge |\mathbf A_{M(\tau)}| (\mathbf x_0(\tau)) \ge  |\partial_{ss} f_a(0) + \partial_{sss} u(0) | \gtrsim a^{-1} + a^{-1 + \tilde \kappa_{in} + \gamma} .
    \end{equation*}
    where $\mathbf x_0(\tau) \in M(\tau)$ is a point in the $G$-orbit of $(0, f_{a(\tau)}(0) + \partial_s u(0, \tau)) = (0, a(\tau) + \partial_s u(0, \tau))$.
    Since $\tilde \kappa_{in} + \gamma > 0$, it follows that, for $t$ sufficiently close to $T$,
    \begin{multline} \label{proof of main thm paper II, eqn 3}
        \sup_{L(t)} |\mathbf A_{L(t)} | \ge |\mathbf A_{L(t)}| \left( \sqrt{T-t} \cdot \mathbf x_0(\tau(t)) \right)
        = \frac1{\sqrt{T-t}} | \mathbf A_{M(\tau(t))} | \left( \mathbf x_0(\tau(t)) \right) \\
        \gtrsim (T-t)^{-1/2} a^{-1} \gtrsim (T-t)^{-K/2}.
    \end{multline}
    This completes the proof of the curvature blow-up rate claimed in Theorem \ref{main thm paper II intro}, Item \ref{main thm paper II intro, item A blow-up rate}.
    Observe additionally that $\sqrt{T-t} \, \mathbf x_0(\tau(t)) \to \mathbf 0$ as $t \nearrow T$.
    
    Recall from Proposition \ref{prop: PDE for the potential} that the Lagrangian angle of $L(t)$ or equivalently $M(\tau(t))$ is given by
    \begin{equation} \label{proof of main thm paper II, eqn 4}
        \Theta[ f_a + u_s] = (n-1) \frac \pi 2 + L_{s,a} u + Q_{s,a} (u).
    \end{equation}
    By \eqref{proof of main thm paper II, eqn 2}, we have
        $$|L_{s,a}u| = \left| \frac{\partial_{ss} u}{1 + (\partial_s f_a)^2} + ( n-1) \frac{ s \partial_s u}{s^2 + f_a^2} \right| 
        \le C(n) | \partial_{ss} u | + C(n) \frac{| \partial_s u|}{\rho_a(s)}.$$
    Claim \ref{proof of main thm paper II, claim 1} and the definition of the set $\mathcal B$ then imply
    \begin{equation} \label{proof of main thm paper II, eqn 5}
        \sup_{s \in \R} |L_{s,a} u| \le C a(\tau)^{\tilde \kappa_{in} + \gamma} + C \epsilon_{par} + C \epsilon_{out} 
    \end{equation}
    where $C = C(n,K,\alpha, \gamma, s_0, \Upsilon )>0$.
    The $Q_{s,a}(u)$ term can similarly be estimated using Lemma \ref{lem weighted Holder est for Q}, Claim \ref{proof of main thm paper II, claim 1}, and the definition of the set $\mathcal B$ to give
    \begin{multline} \label{proof of main thm paper II, eqn 6}
        \sup_{s \in \R} |Q_{s,a}(u)| 
        \le C  \left( \left\| \partial_s u \right\|_{1, \alpha; \R}^{(1)} \right)^2 
        \le C \left( a(\tau)^{\tilde \kappa_{in} + \gamma} + \epsilon_{par} + \epsilon_{out} \right)^2   \\
        \le C \left( a(\tau)^{\tilde \kappa_{in} + \gamma} + \epsilon_{par} + \epsilon_{out} \right)
    \end{multline}
    where $C = C(n, K, \alpha , \gamma, s_0, \Upsilon) > 0$.
    Combining \eqref{proof of main thm paper II, eqn 4}--\eqref{proof of main thm paper II, eqn 6} shows 
        $$ \osc \Theta \le C \left( \sup_{\tau\ge \tau_0} a( \tau)^{\tilde \kappa_{in} + \gamma} + \epsilon_{par} + \epsilon_{out} \right) .$$
    Since $\tilde \kappa_{in} + \gamma > 0$ and the parameters $\epsilon_{par}, \epsilon_{out}$ can be taken arbitrarily small in Corollary \ref{cor main box cor}, the oscillation of the Lagrangian angle of the initial Lagrangian $L_K(T - e^{- \tau_0})$ can be made arbitrarily small.
    This completes the proof of Theorem \ref{main thm paper II intro}, Item \ref{main thm paper II intro, item angle osc}.

    Finally, Theorem \ref{main thm paper II intro}, Item \ref{main thm paper II intro, item profile} ensures the rescalings $\frac1{(T-t)^{K/2}} L_K(t)$ converge to $a \overline L$ in $C^0_{loc}(\C^n)$ as $t \nearrow T$.
    The curvature blow-up rates (Theorem \ref{main thm paper II intro}, Item \ref{main thm paper II intro, item A blow-up rate}) imply that the sequence $\frac1{(T-t)^{K/2}} L_K(t)$ has local curvature bounds that are uniform in $t$.
    The curvature bounds therefore imply that the $C^0_{loc}(\C^n)$-convergence $\frac1{(T-t)^{K/2}} L_K(t) \to a \overline L$ improves to $C^\infty_{loc}(\C^n)$-convergence.
\end{proof}

\begin{remark}\label{rem: compactification}
    Finally, we give a heuristic way of modifying our flow into a LMCF of compact, immersed, zero-Maslov LMCF with prescribed Type II singularity at the origin. In Figure~\ref{fig: compactification}, the initial profile curve is truncated at a sufficiently large radius and then capped off by a large figure-eight-type curve. The self-intersection makes the Lagrangian immersed but remains zero-Maslov. 

\begin{figure}[hbtp]
    \centering
    \includegraphics[width=0.35\linewidth]{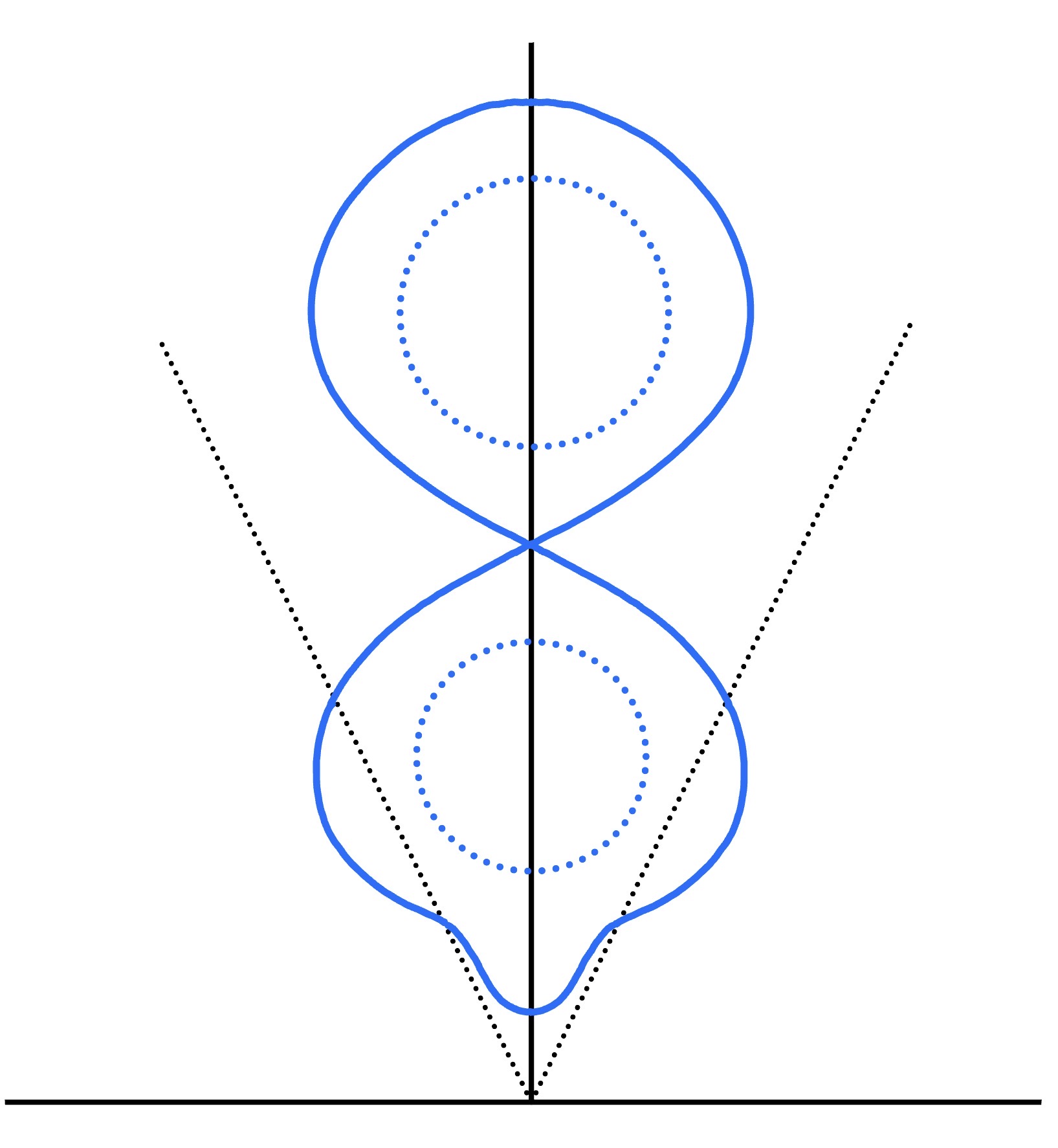}
    \caption{A heuristic compactification of the profile curve. The solid figure-eight-type curve is the profile curve of a compact immersed zero-Maslov Lagrangian. The dotted circles indicate interior barriers for the lobes, chosen sufficiently large so that these lobes cannot collapse before the neck pinch at the origin occurs.}
    \label{fig: compactification}
\end{figure}

    If the compact lobes are chosen sufficiently large, then the neck region near the origin develops the first singularity. 
    Indeed, the Pseudolocality Theorem controls the dynamics of the two compact lobes away from the origin and ensures the lobes remains smooth up to the time the neck pinch at the origin occurs.
    The box conditions and associated estimates are sufficiently robust to allow for the incorporation of cutoff functions and a suitable adaptation of the box argument in Section \ref{Section Box Argument} to apply in this setting.
    Related constructions have been carried out for the mean curvature flow of hypersurfaces and Ricci flow \cites{Liu24, Stolarski24b}.
\end{remark}

\appendix

\section{Assorted Formulas and Estimates} \label{Appendix Assorted Formulas and Estimates}

\subsection{Estimates for $H_a$ and Integral Weights}

The estimates of this subsection were obtained in the companion paper \cite{SS26I}*{Appendix \ref{I-Appendix Formulas and Estimates}}. We list the statements here for the readers' convenience and refer to \cite{SS26I} for proofs.

\begin{lem}[{\cite{SS26I}*{Lemma \ref{I-lem H_a - H_0 pointwise est}}}] \label{lem H_a - H_0 pointwise est}
    Let $H_a = H_{s,a}$ be given by \eqref{eqn H_a defn} 
    and let $H_0 = H_{s,0}$ be the operator
        $$H_{s,0} u := \frac1{1 + \overline c_0^2} \left( \partial_{ss} u + \frac{n-1}s \partial_s u \right) - \frac s2 \partial_s u + u$$
    where $\overline c_0 = \overline c_0(n) = \tan \left( \frac \pi 2 - \frac \pi {2n} \right)$ as usual.
    
    For any $n \ge 3$ and $s_0 > 0$, there exists $C = C(n, s_0) > 0$ such that 
    \begin{equation} \label{eqn H_a - H_0 pointwise est}
        |(H_a - H_0) u | \le C \left( a^n s^{-n} | \partial_{ss} u | + a^n s^{-n-1} |\partial_s u| \right) \qquad \forall s \in (s_0, \infty) , \, a \in (0,1).
    \end{equation}
\end{lem}

\begin{lem}[{\cite{SS26I}*{Lemma \ref{I-lem partial_a H_a pointwise est}}}] \label{lem partial_a H_a pointwise est}
    Let $H_a = H_{s,a}$ be the operator given by \eqref{eqn H_a defn}.
    For any $n \ge 3$ and $s_0 > 0$, there exists $C = C(n, s_0) > 0$ such that for all smooth $u : (s_0, \infty) \to \R $
    \begin{equation} \label{eqn partial_a H_a pointwise est}
        | \partial_a H_a  u | \le C \left( a^{n-1} s^{-n} | \partial_{ss} u | + a^{n-1} s^{-n-1} |\partial_s u| \right) \qquad \forall s \in (s_0, \infty) , \, a \in (0,1).
    \end{equation}
\end{lem}

\begin{proposition}[{\cite{SS26I}*{Proposition \ref{I-prop exp -F_a bounds}}}] \label{prop exp -F_a bounds}
    For $n \ge 3$, there exists a dimensional constant $C_n > 0$ such that
    \begin{multline} \label{eqn Gaussian weight est}
        \max\left\{ e^{-C_n s^2}, e^{-\frac{(1+\ol{c}_0^2)s^2}4 - C_n a^2} \right\} \le e^{-F_a(s)} =  e^{- \frac{s^2}4 - \frac12 \int_0^s \tilde s ( \partial_s f_a)^2 d \tilde s} \\ \le \min \left\{ e^{-\frac{s^2}4 }, e^{-\frac{(1+\ol{c}_0^2)s^2}4 + C_n a^2} \right\} \qquad \forall s \in \R        
    \end{multline}
    where $\overline c_0 = \overline c_0(n) = \tan \left( \frac \pi 2 - \frac \pi {2n} \right)$ as usual.
\end{proposition}

\begin{proposition}[{\cite{SS26I}*{Proposition \ref{I-prop vol element bounds}}}] \label{prop vol element bounds}
    For $n\ge 3$, there exists a dimensional constant $C_n > 1$ such that 
    \begin{equation} \label{eqn volume form est}
        C_n^{-1} \rho_a(s)^{n-1} \le \sqrt{ ( 1 + \partial_s f_a^2 ) ( s^2 + f_a^2 )^{n-1} } \le C_n \rho_a(s)^{n-1}
        \qquad \forall s \in \R; \, \forall 0 < a \le 1.
    \end{equation}

    Consequently, for any $\kappa > 0$, there exists $C = C(n, \kappa) > 0$ such that
    \begin{equation} \label{eqn L^2_a est for polynomials}
        \| |s|^\kappa \|_{L^2_a} \le C(n, \kappa) \qquad \forall 0  < a \le 1.
    \end{equation}
\end{proposition}




\begin{proposition}[{Computation of $\partial_a d\mu_a$, \cite{SS26I}*{Proposition \ref{I-prop computation of partial_a d mu_a}}}] \label{prop computation of partial_a d mu_a}
    For $n \ge 3$,
    \begin{multline} \label{eqn partial_a d mu_a}
        \partial_a \ln \left( e^{- \frac{s^2}4 - \frac12 \int_0^s \tilde s ( \partial_s f_a)^2 d \tilde s} \sqrt{ ( 1 + (\partial_s f_a)^2 )( s^2 + f_a^2 )^{n-1} } \right) \\
        = \frac1a \int_0^s \tilde s^2 (\partial_s f_a) (\partial_{ss}f_a) d \tilde s + \frac{n-1}a \frac{ (s \partial_s f_a - f_a )^2 }{s^2 + f_a^2} .
    \end{multline}

    Moreover, there exists a dimensional constant $C_n > 0$ such that for all $s \in \R$
    \begin{align}
        \label{prop computation of partial_a d mu_a est 1}
        \left| \frac1a \int_0^s \tilde s^2 (\partial_s f_a) (\partial_{ss}f_a) d \tilde s \right| &\le C_n a,\\
        \label{prop computation of partial_a d mu_a est 2}
        \text{and } \left| \frac{n-1}a \frac{ (s \partial_s f_a - f_a )^2 }{s^2 + f_a^2} \right| &\le \frac{C_n}a \frac1{1 + ( s/a)^{n+1}}.
    \end{align}
\end{proposition}

\begin{proposition}[{\cite{SS26I}*{Proposition \ref{I-prop H_a symmetric}}}] \label{prop H_a symmetric}
    For $u, w \in C^\infty_c(\R)$,
    \begin{equation}
        \int_\R w H_a u \, d\mu_a = - \int_\R \frac{(\partial_s w) (\partial_s u)}{1 + (\partial_s f_a)^2} d \mu_a + \int_\R w u \, d \mu_a.
    \end{equation}
\end{proposition}


\begin{proposition}[{\cite{SS26I}*{Proposition \ref{I-prop Ecker's Sobolev ineq}}}] \label{prop Ecker's Sobolev ineq}
    There exists $C = C(n, a) > 0$ such that
    \begin{equation} \label{eqn Ecker's Sobolev ineq}
        \| s u \|_{L^2_a(\R)} \le C \| u \|_{H^1_a(\R) } \qquad \forall u \in C^\infty_c(\R) .
    \end{equation}
\end{proposition}

\subsection{Estimates for the Error Term $Q_a(u)$}

Recall $Q_a(u)$ is defined as in \eqref{eqn Q_a defn} to be
    \begin{equation} \tag{\ref{eqn Q_a defn}}
        Q_a(u) = Q_{s,a}(u) = \arctan ( \partial_s f_a + u_{ss} ) + (n-1) \arctan \left( \frac {f_a+u_s}s \right) - (n-1)\frac{\pi}{2} - L_{s, a} u .
    \end{equation}
In this subsection, we obtain estimates for $Q_a(u)$ and its derivatives in terms of $u$.

\subsubsection{Pointwise Estimates}
\begin{lemma}[Weighted $C^{0}$ Estimate for $Q_{a}(u)$]\label{lem: error term C0}
    There exists a universal constant $C>1$ such that 
    \begin{equation} \label{eqn Q_a pointwise est 1}
        |Q_{s, a}(u)|\leq nC \left(\frac{u_s^2}{s^2} + u_{ss}^{2}\right)
        \qquad \forall s \in \R.
    \end{equation}
\end{lemma}
\begin{proof}
    By Taylor's Theorem we have
    \begin{align}
        \arctan(\partial_{s}f_{a} + u_{ss}) = \arctan(\partial_{s}f_{a}) + \frac{u_{ss}}{1+(\partial_{s}f_{a})^{2}} + R_{1}u_{ss}^2,
    \end{align}
    where $R_{1}(s) = \frac{-\xi}{(1+\xi^{2})^{2}}$ for some $\xi$ between $\partial_s f_a$ and $\partial_s f_a + u_{ss}$.
    Similarly,
    \begin{align}
        \arctan\left(\frac{f_{a} + u_{s}}{s}\right) = \arctan\left(\frac{f_{a}}{s}\right) + \frac{su_{s}}{s^{2}+f_{a}^{2}} + R_{2}\frac{u_s^2}{s^2},
    \end{align}
    where $R_{2}(s) = \frac{-\zeta}{(1+\zeta^{2})^{2}}$ for some $\zeta$ between $\frac{f_a}s$ and $\frac{f_a + u_s}s$.
    Combining these yields
    \begin{align}
        Q_{s, a}(u) 
        ={}&\arctan(\partial_{s}f_{a}) + (n-1)\arctan\left(\frac{f_{a}} {s}\right) - (n-1)\frac{\pi}{2}\\
        & + \frac{u_{ss}}{1+(\partial_{s}f_{a})^{2}} + (n-1)\frac{su_{s}}{s^{2}+f_{a}^{2}} - L_{s, a}u +R_{1}u_{ss}^2 + (n-1) R_{2} \frac{u_s^2}{s^2}. \nonumber
    \end{align}
    Observe that $\arctan(\partial_{s}f_{a}) + (n-1)\arctan\left(\frac{f_{a}}{s}\right)$ is the Lagrangian angle of the special Lagrangian $a\overline{L}$, which is equal to $(n-1)\frac{\pi}{2}$, and that $L_{s, a}u = \frac{u_{ss}}{1+(\partial_{s}f_{a})^{2}} + (n-1)\frac{su_{s}}{s^{2}+f_{a}^{2}}$. Hence,
    \begin{align}\label{eq: expression of Q in terms of squares}
        Q_{s, a}(u) = R_{1}u_{ss}^2 + (n-1) R_{2}\frac{u_s^2}{s^2}.
    \end{align}
    It is easy to see that $x \mapsto \frac{x}{(1+x^{2})^{2}}$ is uniformly bounded for all $x\in\mathbb{R}$ by some universal constant $C>1$. 
    Estimate \eqref{eqn Q_a pointwise est 1} now follows.
\end{proof}


The $C^0$ estimate from the last lemma can be improved to a $C^{0, \alpha}$ estimate.

\begin{lemma}[Weighted $C^{0, \alpha}$ Estimate for $Q_{a}(u)$] \label{lem weighted Holder est for Q}
Let $\Omega\subseteq\mathbb{R}$, $\gamma\in\mathbb{R}$, $\tau_{1}<\tau_{2}$, and $u: \Omega \times(\tau_{1}, \tau_{2})\to\mathbb{R}$. Suppose that $\sup_{\tau\in(\tau_{1}, \tau_{2})}\|u(\cdot, \tau)\|^{(2)}_{2,\alpha;\Omega}\leq\frac{1}{4}$. Then there is a constant $C = C(n, \alpha)>0$ such that
    \begin{equation}
        \|Q_{a}(\cdot, \tau)\|^{(\gamma-2)}_{0,\alpha; \Omega}\leq C \left(\|u_{s}(\cdot, \tau)\|_{1, \alpha;\Omega}^{(\frac{\gamma}{2})}\right)^2
        \qquad \forall \tau \in (\tau_1, \tau_2) , 
    \end{equation}
    and
\begin{equation}
        \|Q_{a}(\cdot, \tau)\|^{(\gamma-2)}_{0,\alpha; \Omega}\leq C \left(\sup_{s\in\Omega}\rho_{a(\tau)}^{\gamma-2}(s)\right) \left( \| u(\cdot, \tau) \|_{2, \alpha; \Omega }^{(\gamma)}\right)^2 \qquad \forall \tau \in (\tau_1, \tau_2) .
    \end{equation}
\end{lemma}

\begin{proof}
    Since $f_{a}$ represents a special Lagrangian with phase $(n-1)\frac{\pi}{2}$, we have
    \begin{align*}
        \arctan\left(\partial_{s}f_{a}\right) + (n-1)\arctan\left(\frac{f_{a}}{s}\right) = (n-1)\frac{\pi}{2}.
    \end{align*}
    By Taylor expansion, we write $Q_{a}(u) = Q^{(1)} + Q^{(2)}$, where
    \begin{align*}
        Q^{(1)} &:= \arctan\left(\partial_{s}f_{a}+u_{ss}\right) - \arctan\left(\partial_{s}f_{a}\right) - \frac{u_{ss}}{1+(\partial_{s}f_{a})^{2}}\\
        &=\left[\int_{0}^{1}\frac{-2(\partial_{s}f_{a} + \zeta u_{ss})}{(1+(\partial_{s}f_{a} + \zeta u_{ss})^{2})^{2}}(1-\zeta)\:\dd\zeta\right]\:u_{ss}^{2},\\
        Q^{(2)} &:= (n-1)\left\{\arctan\left(\frac{f_{a}+u_{s}}{s}\right) - \arctan\left(\frac{f_{a}}{s}\right) - \frac{s^{2}+f_{a}^{2}}{su_{s}}\right\}\\
         &= \left[\int_{0}^{1}\frac{-2s(f_{a} + \zeta u_{s})}{[s^{2} + (f_{a} + \zeta u_{s})^{2}]^{2}}(1-\zeta)\:\dd\zeta\right]\:u_{s}^{2}.
    \end{align*}
    
    Since $\sup_{\tau\in(\tau_{1}, \tau_{2})}\|u(\cdot, \tau)\|^{(2)}_{2,\alpha;\Omega}\leq\frac{1}{4}$, we have
    \begin{align*}
        |u_{s}(s, \tau)|\leq\frac{1}{4}\rho_{a(\tau)}(s)\quad\mbox{and}\quad |u_{ss}(s, \tau)|\leq\frac{1}{4} \qquad \text{on } \Omega\times(\tau_{1}, \tau_{2}).
    \end{align*}
    Let
    \begin{align*}
        &U^{1} := \left\{(s, \tau, p)\::\:(s, \tau)\in\Omega\times(\tau_{1}, \tau_{2}),\;|p|<\tfrac{1}{4}\right\},\\
        &U^{2} := \left\{(s, \tau, q)\::\:(s, \tau)\in\Omega\times(\tau_{1}, \tau_{2}),\;|q|<\tfrac{1}{4}\rho_{a(\tau)}(s)\right\}.
    \end{align*}
    and define functions $F^{i}: U^{i}\to\mathbb{R}$, $i=1, 2$, by
    \begin{align*}
        &F^{1}(s, \tau, p) := \int_{0}^{1}\frac{-2(\partial_{s}f_{a}(s, \tau)+\zeta p)}{(1+(\partial_{s}f_{a}(s, \tau)+\zeta p)^{2})^{2}}(1-\zeta)\:\dd\zeta,\\
        &F^{2}(s, \tau, q) := \int_{0}^{1}\frac{-2s(f_{a}(s, \tau) + \zeta q)}{(s^{2}+(f_{a}(s, \tau) + \zeta q)^{2})^{2}}(1-\zeta)\:\dd\zeta.
    \end{align*}

    We first estimate the $C^{0}$-norm of $F^{1}$ and $F^{2}$. For $F^{1}$, using $|\frac{-2x}{(1+x^{2})^{2}}|\leq\frac{3\sqrt{3}}{8}$ we have
    \begin{align*}
        |F^{1}(s, \tau, p)|\leq 2\int_{0}^{1}\left|\frac{\partial_{s}f_{a}+\zeta p}{(1+(\partial_{s}f_{a} + \zeta p)^{2})^{2}}\right|(1-\zeta)\:\dd\zeta\leq \frac{3\sqrt{3}}{8}.
    \end{align*}
    For $F^{2}$, we first consider the case $|s|<a(\tau)$. Note that in this case $|q|<\frac{1}{2}a$ and thus $|f_{a}+\zeta q|\geq|f_{a} - |q||\geq \frac{1}{2}a$. Using this we have
    \begin{align*}
        |F^{2}(s, \tau, q)|\leq 2a^{-3}|s|(\tau)\int_{0}^{1}(1-\zeta)\:\dd\zeta = a^{-2}(\tau)\leq 2\rho_{a(\tau)}^{-2}(s)\quad\mbox{if}\;|s|<a(\tau).
    \end{align*}
    For the case $|s|\geq a(\tau)$, we use $|f_{a}+\zeta q|\leq 2|s|$ to get
    \begin{align*}
        |F^{2}(s, \tau, q)|\leq \frac{4}{|s|^{2}}\int_{0}^{1}(1-\zeta)\:\dd\zeta =  2|s|^{-2}\leq 4\rho_{a(\tau)}^{-2}(s)\quad\mbox{if}\;|s|\geq a(\tau).
    \end{align*}
    Combining these estimates yields
    \begin{align*}
        |Q_{a}(s, \tau)|&\leq 5\left[\rho_{a(\tau)}^{-2}(s)|u_{s}(s, \tau)|^{2} + |u_{ss}(s, \tau)|^{2}\right].
    \end{align*}
    Multiplying both sides by $\rho_{a}^{-\gamma+2}$ yields
    \begin{align} \label{eq: weighted sup norm of Q_a ver 1}
        \rho_{a(\tau)}^{-\gamma+2}(s)|Q_{a}(s, \tau)|&\leq 5\left[\left(\rho_{a}^{-\frac{\gamma}{2}}(s)|u_{s}(s, \tau)|\right)^{2} + \left(\rho_{a}^{-\frac{\gamma}{2}+1}(s)|u_{ss}(s, \tau)|\right)^{2}\right]
    \end{align}
    for all $(s, \tau ) \in \Omega\times (\tau_1, \tau_2)$.
    Rearranging the weights in a different way yields
    \begin{align} \label{eq: weighted sup norm of Q_a ver 2}
        \rho_{a(\tau)}^{-\gamma+2}(s)|Q_{a}(s, \tau)|&\leq 5\rho_{a}^{\gamma-2}(s)\left[\left(\rho_{a}^{-\gamma+1}(s)|u_{s}(s, \tau)|\right)^{2} + \left(\rho_{a}^{-\gamma+2}(s)|u_{ss}(s, \tau)|\right)^{2}\right]
    \end{align}
    for all $(s, \tau ) \in \Omega\times (\tau_1, \tau_2)$.

    To estimate the H\"older semi-norm of $Q^{(1)}$, we compute the derivatives of $F^{1}$ as follows.
    \begin{align*}
        \partial_{s}F^{1} &= -2\partial_{ss}f_{a}\int_{0}^{1}\frac{1-3(\partial_{s}f_{a}+\zeta p)^{2}}{[1+(\partial_{s}f_{a}+\zeta p)^{2}]^{3}}(1-\zeta)\:\dd\zeta &&\implies\;|\partial_{s}F^{1}(s, \tau, p)|\leq |\partial_{ss}f_{a}(s, \tau)|,\\
        \partial_{p}F^{1} &= -2\int_{0}^{1}\frac{1-3(\partial_{s}f_{a}+\zeta p)^{2}}{[1+(\partial_{s}f_{a}+\zeta p)^{2}]^{3}}\zeta(1-\zeta)\:\dd\zeta &&\implies\;|\partial_{p}F^{1}(s, \tau, p)|\leq\frac{1}{3}.
    \end{align*}
    By the derivative estimate (\ref{eq: asymptotic of SL profile function}), there exists a dimensional constant $D_{1}>0$ such that
    \begin{align*}
        |\partial_{ss}f_{a}(s, \tau)|\leq D_{1}a^{n}(\tau)\rho^{-n-1}_{a(\tau)}(s)
    \end{align*}
    If follows that there is $C_{2} = C_{2}(n,\alpha)>0$ such that if $|s' - s|<\frac{1}{2}\rho_{a(\tau)}(s)$ and $|s'' - s|<\frac{1}{2}\rho_{a(\tau)}(s)$,
    \begin{align*}
        &\frac{|F^{1}(s', \tau, u_{ss}(s', \tau)) - F^{1}(s'', \tau, u_{ss}(s'', \tau))|}{|s'-s''|^{\alpha}}\nonumber\\
        &\leq C_{2}\left[a^{n}(\tau)\rho^{-n-1}_{a(\tau)}(s)|s'-s''|^{1-\alpha} + \frac{|u_{ss}(s', \tau) - u_{ss}(s'', \tau)|}{|s'-s''|^{\alpha}}\right]\nonumber\\
        &\leq C_{2}\left[\rho_{a(\tau)}^{-\alpha}(s) + \tfrac{1}{4}\rho_{a(\tau)}^{-\alpha}(s)\right]\nonumber
        &&  \left(\text{using } \| u \|_{2, \alpha; \Omega}^{(2)} \le \frac14 \right)\\
        &\leq 2C_{2}\rho_{a(\tau)}^{-\alpha}(s).
    \end{align*}
    Hence, the H\"older semi-norm of $Q^{(1)}$ can be estimated by
    \begin{align}\label{eq: semi-norm of Q1}
        \frac{|Q^{(1)}(s', \tau) - Q^{(1)}(s'', \tau)|}{|s'-s''|^{\alpha}}&\leq \frac{|F^{1}(s', \tau, u_{ss}(s', \tau)) - F^{1}(s'', \tau, u_{ss}(s'', \tau))|}{|s'-s''|^{\alpha}}|u_{ss}(s', \tau)|^{2} \nonumber\\
        &\quad+ |F^{1}(s'', \tau)|(|u_{ss}(s', \tau)|+ |u_{ss}(s'', \tau)|)\frac{|u_{ss}(s', \tau) - u_{ss}(s'', \tau)|}{|s'-s''|^{\alpha}}\nonumber\\
        &\leq 2C_{2}\rho_{a}^{-\alpha}(s)|u_{ss}(s', \tau)|^{2}\nonumber\\
        &\quad  + \tfrac{3\sqrt{3}}{4}(|u_{ss}(s', \tau)|+ |u_{ss}(s'', \tau)|)\frac{|u_{ss}(s', \tau) - u_{ss}(s'', \tau)|}{|s'-s''|^{\alpha}}.
    \end{align}
    Note that $|u_{ss}(s', \tau)|$ can be estimated by
    \begin{align*}
        |u_{ss}(s', \tau)|&\leq |u_{ss}(s', \tau) - u_{ss}(s, \tau)| + |u_{ss}(s, \tau)|\\
        &\leq 2^{-\alpha}\rho_{a}^{\alpha}(s)\frac{|u_{ss}(s', \tau) - u_{ss}(s, \tau)|}{|s' - s|^{\alpha}} + |u_{ss}(s, \tau)|,
    \end{align*}
    and the same estimate holds for $|u_{ss}(s'', \tau)|$. Thus, multiplying both sides of (\ref{eq: semi-norm of Q1}) by $\rho_{a}^{-\gamma+2+\alpha}$ and taking sup yields
    \begin{align}\label{eq: weighted semi-norm of Q1 ver 1}
        \sup_{\substack{|s'-s|, |s''-s| \le \frac12 \rho_a(s)\\s' \ne s'' \in \Omega}}\rho_{a}^{-\gamma + 2 + \alpha}(s)\frac{|Q^{(1)}(s', \tau) - Q^{(1)}(s'', \tau)|}{|s'-s''|^{\alpha}}\leq C_{3}\left(\|u_{ss}(\cdot, \tau)\|_{0,\alpha;\Omega}^{(-\frac{\gamma}{2}+1)}\right)^{2}
    \end{align}
    for some $C_{3} = C_{3}(n, \alpha)>0$. Rearranging the weights in a different way yields
    \begin{align}\label{eq: weighted semi-norm of Q1 ver 2}
        \sup_{\substack{|s'-s|, |s''-s| \le \frac12 \rho_a(s)\\s' \ne s'' \in \Omega}}\rho_{a}^{-\gamma + 2 + \alpha}(s)\frac{|Q^{(1)}(s', \tau) - Q^{(1)}(s'', \tau)|}{|s'-s''|^{\alpha}}\leq C_{4}\left(\sup_{s\in\Omega}\rho_{a(\tau)}^{\gamma-2}(s)\right)\left(\|u(\cdot, \tau)\|^{(\gamma)}_{2,\alpha;\Omega}\right)^{2}
    \end{align}
    for some $C_{4} = C_{4}(n, \alpha)>0$.

    The estimate for H\"older semi-norm of $Q^{(2)}$ can be done similarly. The derivatives of $F^{2}$ are given by
    \begin{align*}
        \partial_{s}F^{2} &= -2\int_{0}^{1}\frac{s\partial_{s}f_{a}(s^{2}-3(f_{a}+\zeta q)^{2}) + (f_{a}+\zeta q)((f_{a}+\zeta q)^{2} - 3s^{2})}{[s^{2}+(f_{a}+\zeta q)^{2}]^{3}}(1-\zeta)\:\dd\zeta,\\
        \partial_{q}F^{2} &= -2s\int_{0}^{1}\frac{s^{2} - 3(f_{a}+\zeta q)^{2}}{[s^{2}+(f_{a}+\zeta q)^{2}]^{3}}\zeta(1-\zeta)\:\dd\zeta.
    \end{align*}
    Using $\frac{a}{2}<|f_{a}+\zeta q|<2a$ and (\ref{eq: asymptotic of SL profile function}), there are dimensional constants $D_{2}, C_{5}>0$ such that
    \begin{align*}
        |\partial_{s}F^{2}(s, \tau, q)|\leq\begin{cases}
            4D_{2}a^{-3}(\tau)\leq C_{5}\rho_{a(\tau)}^{-3}(s), &|s|<a(\tau),\\
            4D_{2}|s|^{-3}\leq C_{5}\rho_{a(\tau)}^{-3}(s), &|s|\geq a(\tau).
        \end{cases}
    \end{align*}
    Using $\frac{a}{2}<|f_{a}+\zeta q|<2a$, there are dimensional constants $D_{3}, C_{6}>0$ such that
    \begin{align*}
        |\partial_{q}F^{2}(s, \tau, q)|\leq\begin{cases}
            4D_{3}a^{-3}(\tau)\leq C_{6}\rho_{a(\tau)}^{-3}(s), &|s|<a(\tau),\\
            4D_{3}|s|^{-3}\leq C_{6}\rho_{a(\tau)}^{-3}(s), &|s|\geq a(\tau).
        \end{cases}
    \end{align*}
    Combining the derivative estimates of $F^{2}$ and using the mean value theorem yields the following: If $|s' - s|<\frac{1}{2}\rho_{a(\tau)}(s)$ and $|s'' - s|<\frac{1}{2}\rho_{a(\tau)}(s)$, there exists a constant $C_{7} = C_{7}(n, \alpha)>0$ such that
    \begin{align*}
        &\frac{|F^{2}(s', \tau, u_{s}(s', \tau)) - F^{2}(s'', \tau, u_{s}(s'', \tau))|}{|s-s'|^{\alpha}}\nonumber\\
        &\leq C_{7}\left[\rho_{a(\tau)}^{-3}(s)|s' - s''|^{1-\alpha} + \rho_{a(\tau)}^{-3}(s)\frac{|u_{s}(s', \tau) - u_{s}(s'', \tau')|}{|s'-s''|^{\alpha}}\right]\nonumber\\
        &\leq 2C_{7}\rho_{a(\tau)}^{-2-\alpha}(s)
        && \left( \text{using } \| u \|_{2,\alpha; \Omega}^{(2)} \le \frac14 \right) 
    \end{align*}
    Hence, the H\"older semi-norm of $Q^{(2)}$ can be estimated by
    \begin{align}\label{eq: semi-norm of Q2}
        \frac{|Q^{(2)}(s', \tau) - Q^{(2)}(s'', \tau)|}{|s'-s''|^{\alpha}}
        &\leq \frac{|F^{2}(s', \tau, u_{s}(s', \tau)) - F^{2}(s'', \tau, u_{s}(s'', \tau))|}{|s'-s''|^{\alpha}}|u_{s}(s', \tau)|^{2} \nonumber\\
        &\quad+ |F^{2}(s'', \tau)|(|u_{s}(s', \tau)|+ |u_{s}(s'', \tau)|)\frac{|u_{s}(s', \tau) - u_{s}(s'', \tau)|}{|s'-s''|^{\alpha}}\nonumber\\
        &\leq 2C_{7}\rho_{a(\tau)}^{-2-\alpha}(s)|u_{s}(s', \tau)|^{2} \nonumber\\
        &\quad+ 4\rho_{a(\tau)}^{-2}(s)(|u_{s}(s', \tau)|+ |u_{s}(s'', \tau)|)\frac{|u_{s}(s', \tau) - u_{s}(s'', \tau)|}{|s'-s''|^{\alpha}}.
    \end{align}
    Multiplying both sides of (\ref{eq: semi-norm of Q2}) by $\rho_{a}^{-\gamma+2+\alpha}$ and taking sup yields
    \begin{align}\label{eq: weighted semi-norm of Q2 ver 1}
        \sup_{\substack{|s'-s|, |s''-s| \le \frac12 \rho_a(s)\\s' \ne s'' \in \Omega}}\rho_{a(\tau)}^{-\gamma+2+\alpha}(s) \frac{|Q^{(2)}(s', \tau)- Q^{(2)}(s'', \tau)|}{|s'-s''|^{\alpha}}
        \leq C_{8}\left(\|u_{s}(\cdot, \tau)\|_{0, \alpha;\Omega}^{(\frac{\gamma}{2})}\right)^{2}
    \end{align}
    for some $C_{8} = C_{8}(n, \alpha)>0$. Rearranging the weights differently yields
    \begin{align}\label{eq: weighted semi-norm of Q2 ver 2}
        \sup_{\substack{|s'-s|, |s''-s| \le \frac12 \rho_a(s)\\s' \ne s'' \in \Omega}}\rho_{a(\tau)}^{-\gamma+2+\alpha}(s) \frac{|Q^{(2)}(s', \tau)- Q^{(2)}(s'', \tau)|}{|s'-s''|^{\alpha}}
        \leq C_{9}\left(\sup_{s\in\Omega}\rho_{a(\tau)}^{\gamma-2}(s)\right)\left( \| u(\cdot, \tau)\|_{2, \alpha; \Omega }^{(\gamma)}\right)^2
    \end{align}
    for some $C_{9} = C_{9}(n, \alpha)>0$.

    Finally, combining (\ref{eq: weighted sup norm of Q_a ver 1}), (\ref{eq: weighted semi-norm of Q1 ver 1}) and (\ref{eq: weighted semi-norm of Q2 ver 1}) gives the following estimate
    \begin{align*}
        \|Q_{a}(\cdot, \tau)\|^{(\gamma-2)}_{0,\alpha; \Omega}\leq C_{10} \left(\|u_{s}(\cdot, \tau)\|^{(\frac{\gamma}{2})}_{1,\alpha;\Omega}\right)^2
    \end{align*}
    for some $C_{10} = C_{10}(n, \alpha)>0$; while 
    combining (\ref{eq: weighted sup norm of Q_a ver 2}), (\ref{eq: weighted semi-norm of Q1 ver 2}) and (\ref{eq: weighted semi-norm of Q2 ver 2}) gives the following estimate
    \begin{align*}
        \|Q_{a}(\cdot, \tau)\|^{(\gamma-2)}_{0,\alpha; \Omega}\leq C_{11} \left(\sup_{s\in\Omega}\rho_{a(\tau)}^{\gamma-2}(s)\right) \left( \| u(\cdot, \tau) \|_{2, \alpha; \Omega }^{(\gamma)}\right)^2
    \end{align*}
    for some $C_{11} = C_{11}(n, \alpha)>0$. 
\end{proof}

We shall also need estimates for higher order derivatives of $Q_{s,a}(u)$.
To that end, we first present the following elementary lemma.

\begin{lem}[Derivatives of Linearizations] \label{lem derivs of linearizations}
    Let $k \ge 1$ and 
    let $F : \R \to \R$, $F = F(y)$ and $y,h : (a,b) \to \R$, $y = y(x), h= h(x) $ both be $C^k$ functions.
    Then
    \begin{align*}
        &\frac{d^k}{dx^k} \left[ F(y+h) - F(y) - \partial_y F (y) \cdot h \right] \\
        ={}& \sum_{i=1}^k \left[ \partial_y^i F(y+h) - \partial_y^i F(y) - \partial_y^{i+1} F (y) \cdot h \right] \sum_{ \substack{j_1+ \dots +j_i = k \\ 1 \le j_1, \dots j_i} } C_{j_1, \dots, j_i} \partial_x^{j_1} y \cdot \, \dots \, \cdot \partial_x^{j_i} y \\
        &+ \sum_{i=1}^k \left[ \partial_y^i F(y+h) - \partial_y^i F(y)  \right] \sum_{ \substack{j_1+ \dots +j_i = k \\ 1 \le j_1, \dots j_i \\ 1 \le \alpha } } C_{j_1, \dots, j_i, \alpha}' \partial_x^{j_1} h \cdot \, \dots \, \partial_x^{j_\alpha} h \cdot \partial_x^{j_{\alpha+1}} y \cdot \, \dots \, \cdot \partial_x^{j_i} y \\
        &+ \sum_{i=1}^k \left[ \partial_y^i F(y+h)  \right] \sum_{ \substack{j_1+ \dots +j_i = k \\ 1 \le j_1, \dots j_i \\ 2 \le \alpha } } C_{j_1, \dots, j_i, \alpha }'' \partial_x^{j_1} h \cdot \, \dots \, \partial_x^{j_\alpha} h \cdot \partial_x^{j_{\alpha+1}} y \cdot \, \dots \, \cdot \partial_x^{j_i} y 
    \end{align*}
    for some constants $C_{j_1, \dots, j_i}, C'_{j_1, \dots , j_i, \alpha}, C''_{j_1, \dots, j_i, \alpha} \in \mathbb{N}$.
\end{lem}
\begin{proof}
    The proof follows from a straightforward induction argument using the product rule, chain rule, and the following claims applied to $F$ and its derivatives.

    \begin{claim}
        For any $C^1$ functions $F, y,h$ as in the statement of the lemma,
        \begin{equation}
            \frac{d}{dx} \left[ F (y+h) - F(y) \right] 
            = [ \partial_y F (y+h) - \partial_y F(y)] \cdot \partial_x y 
            + \partial_y F(y+h) \cdot \partial_x h.
        \end{equation}
    \end{claim}

    \begin{claim}
        For any $C^1$ functions $F, y, h$ as in the statement of the lemma,
        \begin{multline}
            \frac{d}{dx} \left[ F(y+h) - F(y) - \partial_y F (y) \cdot h \right] \\
            = \left[ \partial_yF(y+h) - \partial_y F(y) - \partial_y^2 F (y) \cdot h \right]\cdot \partial_x y
            + \left[ \partial_y F(y+h) - \partial_y F(y) \right] \cdot \partial_x h.
        \end{multline}
    \end{claim}

    The claims are straightforward to verify.
\end{proof}

\begin{lem}[Higher Order Estimates for the Error Term] \label{lem: error term Ck}
    Let $Q_a(u)$ be defined as in \eqref{eqn Q_a defn}.
    For any $k \ge 1$, there exists a constant $C_k  \ge 1$ depending only on $k$ such that
    \begin{multline} \label{eqn: error term Ck}
         \| Q_a(u) \|_{C^k(\Omega)} \le  n C_k \left( \| u_{ss} \|_{C^k(\Omega)}^2 + \| u_{ss} \|_{C^k(\Omega)}^{k+1} + \left\| \frac{u_s}s \right\|_{C^k(\Omega)}^2 + \left\| \frac{u_s}s \right\|_{C^k(\Omega)}^{k+1} \right) \\
         \cdot \left( 1 + \| \partial_s f_a \|_{C^{k}(\Omega)}^k + \left\| \frac{f_a}s \right\|_{C^k(\Omega)}^k \right)
    \end{multline}
    for any $\Omega \subset \R$ and any $a >0$.
\end{lem}
\begin{proof}
    Throughout, $C_k \ge 1$ denotes a constant depending only on $k \ge 1$ that may change from line to line.
    Let $F(x) = \arctan(x)$ denote the arctan function.
    Then $Q_a(u)$ can be written as 
    \begin{multline}
        Q_a(u) = Q_{a;1}(u) + (n -1) Q_{a;2}(u) \\:= 
        \left[ F(\partial_s f_a + u_{ss} ) - F(\partial_s f_a) - F'(\partial_s f_a) u_{ss} \right] \\
        + (n-1) \left[ F( f_a/s + u_s/a ) - F(f_a/s) - F'(f_a/s) \frac{u_s}s \right].
    \end{multline}

    By Lemma \ref{lem derivs of linearizations}, for any $s \in \Omega$ and $a > 0$,
    \begin{align*}
        &\left| \frac{d^k}{ds^k} Q_{a;1}(u) \right|\\
        \le{}& C_k \sum_{i=1}^k \left| F^{(i)} (\partial_s f_a + u_{ss} ) - F^{(i)} (\partial_s f_a +u_{ss} ) - F^{(i+1)}(\partial_s f_a ) u_{ss} \right| \\
        & \qquad \cdot \sum_{ \substack{j_1 + \dots + j_i = k \\ 1 \le j_1, \dots , j_i }} | \partial_s^{j_1+1} f_a | \cdot \, \dots \cdot | \partial_s^{j_i+1} f_a | \\
        &+ C_k \sum_{i=1}^k \left| F^{(i)}(\partial_s f_a + u_{ss} ) - F^{(i)}(\partial_s f_a ) \right| \\
        & \qquad \cdot \sum_{\substack{j_1 + \dots + j_i = k \\ 1 \le j_1 , \dots , j_i \\ 1 \le \alpha }} |\partial_s^{j_1} u_{ss} | \cdot \, \dots \, \cdot | \partial_s^{j_{\alpha}} u_{ss} | \cdot | \partial_s^{j_{\alpha +1} +1} f_a | \cdot \, \dots \, \cdot | \partial_s^{j_i+1} f_a |  \\
        &+ C_k \sum_{i=1}^k | F^{(i)}(\partial_s f_a + u_{ss} ) | \sum_{\substack{j_1 + \dots + j_i = k \\ 1 \le j_1 , \dots , j_i \\ 1 \le \alpha }} |\partial_s^{j_1} u_{ss} | \cdot \, \dots \, \cdot | \partial_s^{j_{\alpha}} u_{ss} | \cdot | \partial_s^{j_{\alpha +1} +1} f_a | \cdot \, \dots \, \cdot | \partial_s^{j_i+1} f_a |
    \end{align*}
    \begin{align*}
        \le{}& C_k \left(\| \partial_s f_a \|_{C^k(\Omega)} + \| \partial_s f_a \|_{C^k(\Omega)}^k \right) \sum_{i=1}^k \left| F^{(i)} (\partial_s f_a + u_{ss} ) - F^{(i)} (\partial_s f_a +u_{ss} ) - F^{(i+1)}(\partial_s f_a ) u_{ss} \right| \\
        &+ C_k \left( \| u_{ss} \|_{C^k(\Omega)} + \| u_{ss} \|_{C^k(\Omega)}^{k} \right) \left( 1 + \| \partial_s f_a \|_{C^k(\Omega)}^k \right) \sum_{i=1}^k \left| F^{(i)}(\partial_s f_a + u_{ss} ) - F^{(i)}(\partial_s f_a ) \right| \\
        &+ C_k \left( \| u_{ss} \|_{C^k(\Omega)}^2 + \| u_{ss} \|_{C^k(\Omega)}^{\max\{ 2, k \} } \right) \left( 1 + \| \partial_s f_a \|_{C^k(\Omega)}^k \right) \sum_{i=1}^k | F^{(i)}(\partial_s f_a + u_{ss} ) | .
    \end{align*}
    By Taylor's theorem, for any $i$,
    \begin{multline*}
       F^{(i)} (\partial_s f_a + u_{ss} ) - F^{(i)} (\partial_s f_a +u_{ss} ) - F^{(i+1)}(\partial_s f_a ) u_{ss} = F^{(i+2)}(\zeta_i) u_{ss}^2 \\ \text{for some } \zeta_i \text{ between $\partial_s f_a$ and $\partial_s f_a + u_{ss}$} ,
    \end{multline*}
    \begin{multline*}
        \text{and } F^{(i)} (\partial_s f_a + u_{ss} ) - F^{(i)} (\partial_s f_a +u_{ss} )  = F^{(i+1)}(\xi_i) u_{ss} \\ \text{for some } \xi_i \text{ between $\partial_s f_a$ and $\partial_s f_a + u_{ss}$}  .
    \end{multline*}
    Note that the derivatives of $F(x) = \arctan(x)$ up to order $k+2$ are uniformly bounded, that is,
    $$\sup_{0 \le i \le k+2} \sup_{x \in \R} |F^{(i)} (x) | \le C_k.$$
    Combining the estimates and equalities above, it therefore follows that, for any $s \in \Omega$ and $a > 0$,
    \begin{align*}
        &\left| \frac{d^k}{ds^k} Q_{a;1}(u) \right|\\
        \le{}& C_k \left(\| \partial_s f_a \|_{C^k(\Omega)} + \| \partial_s f_a \|_{C^k(\Omega)}^k \right) \sum_{i=1}^k \left| F^{(i)} (\partial_s f_a + u_{ss} ) - F^{(i)} (\partial_s f_a +u_{ss} ) - F^{(i+1)}(\partial_s f_a ) u_{ss} \right| \\
        &+ C_k \left( \| u_{ss} \|_{C^k(\Omega)} + \| u_{ss} \|_{C^k(\Omega)}^{k} \right) \left( 1 + \| \partial_s f_a \|_{C^k(\Omega)}^k \right) \sum_{i=1}^k \left| F^{(i)}(\partial_s f_a + u_{ss} ) - F^{(i)}(\partial_s f_a ) \right| \\
        &+ C_k \left( \| u_{ss} \|_{C^k(\Omega)}^2 + \| u_{ss} \|_{C^k(\Omega)}^{\max\{ 2, k \} } \right) \left( 1 + \| \partial_s f_a \|_{C^k(\Omega)}^k \right) \sum_{i=1}^k | F^{(i)}(\partial_s f_a + u_{ss} ) | \\
        \le{}& C_k \left(\| \partial_s f_a \|_{C^k(\Omega)} + \| \partial_s f_a \|_{C^k(\Omega)}^k \right) u_{ss}^2 \\
        &+ C_k \left( \| u_{ss} \|_{C^k(\Omega)} + \| u_{ss} \|_{C^k(\Omega)}^{k} \right) \left( 1 + \| \partial_s f_a \|_{C^k(\Omega)}^k \right) |u_s|  \\
        &+ C_k \left( \| u_{ss} \|_{C^k(\Omega)}^2 + \| u_{ss} \|_{C^k(\Omega)}^{\max\{ 2, k \} } \right) \left( 1 + \| \partial_s f_a \|_{C^k(\Omega)}^k \right)  \\
        \le{}&  C_k \left( \| u_{ss} \|_{C^k(\Omega)}^2 + \| u_{ss} \|_{C^k(\Omega)}^{k+1} \right) \left( 1 + \| \partial_s f_a \|_{C^k(\Omega)}^k \right) .
    \end{align*}

    By similar logic, for any $s \in \Omega$ and $a > 0$,
    \begin{equation*}
        \left|(n-1) \frac{d^k}{ds^k} Q_{a; 2} (u) \right| 
        \le (n-1) C_k \left( \left\| \frac{u_s}s \right\|_{C^k(\Omega)}^2 + \left\| \frac{u_s}s \right\|_{C^k(\Omega)}^{k+1} \right) \left( 1 + \left\| \frac{f_a}s \right\|_{C^k(\Omega)}^k \right) .
    \end{equation*}

    The estimate \eqref{eqn: error term Ck} on the $C^{k_0}(\Omega)$-norm of $Q_a(u) = Q_{a;1}(u) + (n-1)Q_{a; 2}(u)$ now follows from combining these estimates for $1 \le k \le k_0$ and using Lemma \ref{lem: error term C0} for the zeroth order estimate.
\end{proof}

\subsubsection{Integral Estimates}
Before estimating the $L^{2}_a$ norm of $Q_{a}$, we shall need the following simple observation.
\begin{lemma}\label{lem: integral est. for rho}
    Let $n\geq 3$, $k\in \R$, and $0 < s_0 < \frac12$.
    As usual let $\rho = \rho_a(s) = \sqrt{ a^2 + s^2} $. 
    Then there exists $C = C(n, k, s_{0})>0$ such that, for any $0 < a < s_0$, we have
    \begin{align}
        \int_{-s_{0}}^{s_{0}}\rho_{a}^{n-k}(s)\:ds\leq\begin{cases}
            C, & k <  n+1,\\
            C |\log(a)| , & k = n+1,\\
            Ca^{n-k+1}, & k>  n+1.
        \end{cases}
    \end{align}
\end{lemma}
\begin{proof}
    Throughout the proof, $C = C(n,k, s_0) > 0$ is a positive constant depending only on $n, k, s_0$ and which may change from line to line.

    When $k < n+1$, we have
    \begin{multline*}
        \int_{-s_0}^{s_0} \rho^{n-k} ds 
        ={} 2 \int_0^a \rho^{n-k} ds + 2 \int_a^{s_0} \rho^{n-k} ds 
        \le{} C \int_0^a a^{n-k} ds + C \int_a^{s_0} s^{n-k} ds \\
        ={} C a^{n-k+1} + \frac{C}{n-k+1} \left( s_0^{n-k+1} - a^{n-k+1} \right) 
        \le{} C
    \end{multline*}
    where the last inequality uses $n+1 -k > 0$ and $0 < a  < 1$.

    If $k = n+1$,
    \begin{multline*}
        \int_{-s_0}^{s_0} \rho^{n-k} ds 
        =\int_{-s_{0}}^{s_{0}}\rho_{a}^{-1}(s)\:ds 
        = \int_{-s_{0}}^{s_{0}}\frac{ds}{\sqrt{s^{2}+a^{2}}} 
        = 2\left[\log\left(s_{0} + \sqrt{s_{0}^{2}+a^{2}}\right) - \log(a)\right] \\
        \le C  |\log a|  
    \end{multline*}
    where the last line uses $0 < a < s_0 < \frac12$.
    
    If $k >  n+1$, then
    \begin{multline*}
        \int_{-s_{0}}^{s_{0}}\rho_{a}^{n-k}(s)\:ds = 2\int_{0}^{a}\rho_{a}^{n-k}(s)\:ds + 2\int_{a}^{s_{0}}\rho_{a}^{n-k}(s)\:ds
        \leq C\int_{0}^{a}a^{n-k}\:ds + C\int_{a}^{s_{0}}s^{n-k}\:ds\\
        \leq Ca^{n-k+1} + \frac{C}{n-k+1}\left(s_{0}^{n-k+1} - a^{n-k+1}\right)
        \le C a^{n-k+1}.
    \end{multline*}
\end{proof}

\begin{lem}[$L^2_a$ Estimates for $Q_a$] \label{lem L2 ests for Q}
    Given $n\geq 3$, $K \in \mathbb N$ with $K \ge 2$, $0 < s_0 \le s_0^*(n, K) \ll  1$, $\Upsilon >  1$, $C_0 > 1$, $\gamma < 0$, and $\kappa_{in}, \kappa_{out} > 0$ such that 
        $$0 < (2 - \gamma) \frac{K-1}2 \le \kappa_{in},$$
    there exists a constant $ C= C(n, K, s_0, \Upsilon, C_0, \gamma, \kappa_{in}, \kappa_{out}) > 0$ such that the following holds for all $\tau \ge \tau_*( n, K, s_0, \Upsilon, C_0, \gamma, \kappa_{in}, \kappa_{out} )\gg 1$:
    
    Assume $u = \phi + v$ is an odd function where
        $$\phi = \sum_{i=1}^K b_i ( \phi_{i,a} - \phi_{0,a} ), \qquad v \perp_{L^2_a} \phi_{i,a} \quad \forall 0 \le i \le K$$
    and that the following estimates hold
    \begin{gather*}
        0 < C_0^{-1} e^{\frac{1-K}2 \tau} \le a(\tau) \le C_0 e^{ \frac{1-K}2 \tau}, \qquad |\partial_\tau a | \le C_0 a, \qquad
        \sum_{i=1}^K |b_i| \le C_0 a^2 , \\
        \| u \|_{C^2 \left( \left(\frac12 s_0, 2\Upsilon \right) \right)} \le C_0 a^2, \quad
        \sum_{i=0}^2 \sup_{s \ge \Upsilon } s^{-2K -2 +i} |\partial_{s}^i u | \le C_0 e^{-\kappa_{out}\tau}, \quad
        \text{and } e^{\kappa_{in} \tau} \| v \|_{2; (-s_0, s_0)}^{(\gamma)} \le 1.
    \end{gather*}
    Then 
    \begin{equation*}
        \| Q_a(u) \|_{L^2_a(\R)}^2 \le C \left( a^7 +  e^{-4 \kappa_{in} \tau}  a^{4\gamma + n - 8}  + e^{-4\kappa_{out}\tau}\right).
    \end{equation*}
\end{lem}

\begin{proof}
Throughout the proof, unless noted otherwise, $C = C(n, K, s_0, \Upsilon, C_0, \gamma, \kappa_{in}, \kappa_{out}) > 0$ is a positive constant depending on $n, K, s_0, \Upsilon, C_0, \gamma, \kappa_{in}, \kappa_{out}$ and which may vary from line to line.

We first derive a pointwise estimate for $Q_{a}$ in the inner region $s\in(-s_{0}, s_{0})$. Pointwise estimates and elliptic regularity as in \eqref{proof preserving inner Holder condition, proof claim 4, eqn 2}--\eqref{proof preserving inner Holder condition, proof claim 4, eqn 5} gives
\begin{align*}
    \rho_{a}^{-1}|\partial_{s}(\phi_{i, a} - \phi_{0, a})| + |\partial_{ss}(\phi_{i, a} - \phi_{0, a})|\leq C\rho_{a}^{-1},\quad\forall s\in(-s_{0}, s_{0}). 
\end{align*}
Thus, using $\sum_{i=1}^K |b_{i}|\leq C_{0}a^{2}$, it follows that
\begin{align*}
    \rho_{a}^{-1}|\phi_{s}| + |\phi_{ss}|\leq Ca^{2}\rho_{a}^{-1}.
\end{align*}
On the other hand, $e^{\kappa_{in} \tau} \| v \|_{2; (-s_0, s_0)}^{(\gamma)} \le 1$ implies
\begin{align*}
    \rho_{a}^{-1}|v_{s}| + |v_{ss}|\leq e^{-\kappa_{in}\tau}\rho_{a}^{\gamma-2},\quad\forall s\in(-s_{0}, s_{0}). 
\end{align*}
Combining these give the pointwise estimate 
\begin{equation*}
    |Q_{a}|\leq C\left(\rho_{a}^{-2}|\phi_{s}|^{2} + |\phi_{ss}|^{2} + \rho_{a}^{-2}|v_{s}|^{2} + |v_{ss}|^{2}\right)
    \leq C\left(a^{4}\rho_{a}^{-2} + e^{-2\kappa_{in} \tau}\rho_{a}^{2\gamma-4}\right),\quad\forall s\in(-s_{0}, s_{0}). 
\end{equation*}

Using Lemma \ref{lem: integral est. for rho} and Propositions ~\ref{prop exp -F_a bounds} and ~\ref{prop vol element bounds}, the $L^{2}$ norm of $Q_{a}$ in the inner region can be estimated as follows:
\begin{gather} \label{eq: L2 est of Q inner region} \begin{aligned}
    \int_{-s_{0}}^{s_{0}}|Q_{a}|^{2}\:d\mu_{a} 
    \le{}& C\int_{-s_{0}}^{s_{0}}\left(a^{8}\rho_{a}^{-4} + e^{-4\kappa_{in}\tau}\rho_{a}^{4\gamma-8}\right)\rho_{a}^{n-1}\:ds\\
    \le{}& Ca^{8}\int_{-s_{0}}^{s_{0}}\rho_{a}^{n-5}\:ds + C e^{-4\kappa_{in} \tau}\int_{-s_{0}}^{s_{0}}\rho_{a}^{n + 4 \gamma -9}\:ds \\
    \le{}& C a^8 ( 1 + |\ln a | + a^{3-5+1} ) 
    + C e^{-4 \kappa_{in} \tau} ( 1 + |\ln a| + a^{4 \gamma + n- 8} ) \\
    \le{}& C a^7 + C e^{-4 \kappa_{in} \tau} \left( 1 + |\ln a| + a^{4 \gamma + n -8} \right) .
\end{aligned} \end{gather}
Finally, one can use
\begin{equation*}
    \gamma < 0 , \qquad (2- \gamma) \frac{K-1}2 \le \kappa_{in} , \quad \text{and} \quad C_0^{-1} e^{- \frac{K-1}2 \tau} \le a \le C_0 e^{- \frac{K-1}2 \tau}
\end{equation*}
to deduce 
    $$C e^{-4 \kappa_{in} \tau} ( 1 + | \ln a | ) \le C a^7,$$
whence \eqref{eq: L2 est of Q inner region} therefore becomes 
\begin{equation} \label{eq: L2 est of Q inner region 2}
    \int_{-s_{0}}^{s_{0}}|Q_{a}|^{2}\:d\mu_{a}
    \le C a^7 + C e^{-4 \kappa_{in} \tau} \left( 1 + |\ln a| + a^{4 \gamma + n -8} \right)
    \le C a^7 + C e^{-4 \kappa_{in} \tau} a^{4 \gamma + n -8}.
\end{equation}

For the parabolic region $s\in(\frac{1}{2}s_{0}, 2\Upsilon)$, we use the pointwise estimate $\|u\|_{C^{2}((\frac{1}{2}s_{0}, 2\Upsilon))}\leq C_{0}a^{2}$ and $\frac{1}{4}s_{0}^{2}\leq\rho_{a}\leq\sqrt{4\Upsilon^{2} + 1}$ to get
\begin{equation*}
    \rho_{a}^{-1}|u_{s}| + |u_{ss}|\leq Ca^{2},\quad\forall s\in \left(\frac{1}{2}s_{0}, 2\Upsilon \right).
\end{equation*}
Thus, 
\begin{align}\label{eq: L2 est of Q parabolic region}
    \int_{\frac{1}{2}s_{0}}^{2\Upsilon}|Q_{a}|^{2}\:d\mu_{a}\leq Ca^{8}.
\end{align}

For the outer region $s\geq \Upsilon$, we have the pointwise estimate
\begin{align*}
    \rho_{a}^{-1}|u_{s}| + |u_{ss}|\leq C_{0}s^{2K}e^{-\kappa_{out} \tau},\quad\forall s\geq \Upsilon.
\end{align*}
Thus, using Proposition~\ref{prop exp -F_a bounds} and Proposition~\ref{prop vol element bounds} to estimate the volume element, it follows that
\begin{align}\label{eq: L2 est of Q outer region}
    \int_{\Upsilon}^{\infty}|Q_{a}|^{2}\:d\mu_{a}
    \leq C e^{-4 \kappa_{out} \tau} \int_{\Upsilon}^{\infty}s^{8K + n-1} e^{-C_{n} s^{2}}\:ds\leq C e^{-4 \kappa_{out} \tau} ,
\end{align}

Finally, combining (\ref{eq: L2 est of Q inner region 2}), (\ref{eq: L2 est of Q parabolic region}) and (\ref{eq: L2 est of Q outer region}) yields the claimed result.
\end{proof}

\section{Higher Order Estimates for the Profile}

In this appendix, we collect higher order estimates for the profile function $h = f_a + \partial_s u$ of the solution obtained in Corollary \ref{cor main box cor}.
The definition of the box $\mathcal B$ ensures $h$ satisfies certain weighted $C^{1, \alpha}$-estimates.
We will invoke rescaling arguments to obtain higher order estimates for the profile in certain regions.

\begin{lem}[$C^2$ Estimates for the Profile at Scales $\sqrt{T-t} \lesssim | x| \lesssim 1$] \label{lem C^2 Profile Ests on outer-parab overlap}
    For parameters as in Corollary \ref{cor main box cor}, let $u(s,\tau)$ and $a(\tau)$ be the  solutions obtained in Corollary \ref{cor main box cor}.
    Let $h(s, \tau) = f_{a(\tau)}(s) + \partial_s u(s,\tau)$ be the corresponding profile function for the rescaled LMCF.
    With $\tau_0$ as in Corollary \ref{cor main box cor}, set $T:= e^{-\tau_0}$ and $\tilde h (x,t) := \sqrt{T - t} \cdot h \left( x/ \sqrt{T-t}, -\ln (T-t) \right)$, the profile function for the associated LMCF.

    Then, for any $0 < \epsilon < 1$, there exists $C = C(n, K, \gamma, s_0, \Upsilon, \kappa_{in}, \kappa_{out}, A_{out}, \epsilon, \tau_0) > 0$ such that 
    \begin{equation}
        \sum_{j=0}^2 |x|^{-1+j} \left|\partial_x^j \left( \tilde h (x, t) - \cot\left( \frac \pi{2n} \right) \cdot  x \right)  \right| \le C |x|^{2 \kappa_{out} }
        \qquad \forall \epsilon \sqrt{T-t} \le x \le \sqrt{t} , \, t \in [0, T).
    \end{equation}
\end{lem}

\begin{proof}
    Let $0 < \epsilon < 1$.
    Fix $x_*, t_*$ such that $\epsilon \sqrt{ T- t_*}\le x_* \le \sqrt{t_*}$ and $t_* \in (0, T)$.
    Define
    \begin{equation} \label{proof finite-time sing, large scale H bound, eqn 0}
        w( \sigma, \xi) := \frac1{|x_*|} \left( \tilde h ( x_* \sigma , t_* + \xi x_*^2) - \cot\left( \frac \pi{2n} \right) \cdot |x_* \sigma| \right) ,
    \end{equation}
    which is a rescaling of the difference of the profile $\tilde h$ from the profile for a special Lagrangian pair of planes.
    By Corollary \ref{cor: PDE for the potential}, $h$ satisfies \eqref{eq: rescaled flow of profile function}.
    Hence, $\tilde h(x,t)$ satisfies the PDE \eqref{eq: LMCF of profile function}
    \begin{equation*} 
        \partial_t \tilde h = \frac{\partial_{xx} \tilde h}{1 + (\partial_x \tilde h)^2} + (n-1) \frac{x \partial_x \tilde h - \tilde h}{x^2 + \tilde h^2}
        \qquad \forall (x, t) \in \R \times [0, T).
    \end{equation*}
    Therefore, $w(\sigma, \xi)$ satisfies
    \begin{multline} \label{proof finite-time sing, large scale H bound, eqn 1.5}
        \partial_\xi w = \frac{\partial_{\sigma \sigma} w}{1 + ( \overline c_0 +  \partial_\sigma w)^2 } + (n-1) \frac{\sigma \partial_\sigma w - w}{\sigma^2 + ( \overline c_0 \sigma +  w)^2} \\
        \text{on the region } \Omega := \left\{ (\sigma, \xi) \in \left(\frac12, \frac32 \right) \times \left[- \frac{t_*}{x_*^2}, 0 \right] \right\},
    \end{multline}
    where here and throughout the proof we denote $\overline c_0 = \cot\left( \frac \pi{2n} \right) $ for simplicity.

    Corollary \ref{cor main box cor} ensures $(u,a) \in \mathcal B$ for all $\tau \in [\tau_0, \infty)$.
    By the definition of $\mathcal B$ (Definition \ref{defn box}), $u$ satisfies weighted $C^2$-estimates that in particular imply
    \begin{equation} \label{proof finite-time sing, large scale H bound, eqn 2}
        | \partial_{ss} u | + \left| \frac1s \partial_s u  \right|  \le C_* \left( s e^{-\tau/2} \right)^{2 \kappa_{out}} 
        \qquad \forall s \ge \frac \epsilon {10}, \, \tau \ge \tau_0.
    \end{equation}
    for some $C_* = C_*( n,K, \gamma, s_0, \Upsilon, \kappa_{in} , \kappa_{out} , A_{out}, \epsilon) > 0$ independent of $(x_*, t_*)$.
    The asymptotics of $f_a$ (Lemma \ref{lem: profile function of Lawlor neck}) and the fact $a(\tau) \sim e^{\frac{K-1}2 \tau}$ (Definition \ref{defn box}) also yield
    \begin{equation} \label{proof finite-time sing, large scale H bound, eqn 3}
        \left| \partial_s f_{a(\tau)} - \overline c_0 \right| + \left| \frac1s f_{a(\tau)} - \overline c_0 \right| \le C_* \left( s e^{-\tau/2} \right)^{2 \kappa_{out}} \qquad \forall s \ge \frac \epsilon {10}, \, \tau \ge -\ln T.
    \end{equation}
    after possibly increasing $C_*$.
    Combining \eqref{proof finite-time sing, large scale H bound, eqn 2}--\eqref{proof finite-time sing, large scale H bound, eqn 3}, it follows that 
    \begin{multline} \label{proof finite-time sing, large scale H bound, eqn 4}
        |\partial_\sigma w | + \left| \frac1 \sigma w \right| 
        \le C_* |x_*|^{2 \kappa_{out}} \sigma^{2 \kappa_{out} } \le C_* T^{\kappa_{out}} \sigma^{2 \kappa_{out}}\\ 
        \text{on the region } \Omega':= \left\{ (\sigma, \xi) \colon \sigma  \ge \frac \epsilon{10} \sqrt{\frac{T-t_*}{x_*^2} - \xi  } \text{ and } \xi \in \left[ - \frac{t_*}{x_*^2}, 0\right] \right\}.
    \end{multline}
    Straightforward estimates show
    \begin{equation*}
        0 < \epsilon \sqrt{T-t_*} \le x_* \le \sqrt{t_*} \quad \text{and} \quad 0 < \epsilon < 1
    \end{equation*}
    imply $\left( \frac12, \frac32 \right) \times \left[- 1, 0 \right] \subset \Omega \cap \Omega'$.

    Therefore, 
    $w$ satisfies the quasilinear PDE \eqref{proof finite-time sing, large scale H bound, eqn 1.5} on $\left( \frac12, \frac32 \right) \times \left[- 1, 0 \right]$ with $C^1$-bounds \eqref{proof finite-time sing, large scale H bound, eqn 4} that are independent of $x_*, t_*$.
    These gradient bounds suffice to obtain interior $C^{2, \alpha}$-estimates for $w$ \cite{LSU88}*{Ch. VI, Section 1}. 
    Namely, since $|x_*| \le \sqrt{t_*}$, these interior estimates imply 
    \begin{multline} \label{proof finite-time sing, large scale H bound, eqn 5}
        |x_*| | \partial_{xx} \tilde h ( x_*, t_*) | 
        =|\partial_{\sigma \sigma}  w (1, 0) | 
        \lesssim_{n, C_*, \kappa_{out}, T} \sup_{(\sigma, \xi) \in \left( \frac12, \frac32 \right) \times \left[- 1, 0 \right] } | w (\sigma, \xi) | \\
        \le C_* |x_*|^{2\kappa_{out}} \sup_{\sigma \in \left( \frac12, \frac32 \right)} |\sigma|^{2 \kappa_{out} +1} =  C_* |x_*|^{2 \kappa} \left( \frac32 \right)^{2 \kappa_{out} +1}
    \end{multline}
    where the last line uses \eqref{proof finite-time sing, large scale H bound, eqn 4}.
    Since $C_*$ is independent of $(x_*, t_*)$, it follows from \eqref{proof finite-time sing, large scale H bound, eqn 2}, \eqref{proof finite-time sing, large scale H bound, eqn 3}, and \eqref{proof finite-time sing, large scale H bound, eqn 5} that 
    \begin{equation*}
        \sum_{j=0}^2 |x|^{-1+j} \left|\partial_x^j \left( \tilde h (x, t) - \overline c_0 x \right)  \right| \lesssim_{n, C_*, \kappa_{out}, T} |x|^{2 \kappa_{out} }
        \qquad \forall \epsilon \sqrt{T-t} \le x \le \sqrt{t} , \, t \in [0, T).
    \end{equation*}
\end{proof}

\begin{lem}[$C^2$ Estimates for the Profile at Scales $| x| \gtrsim 1$] \label{lem C^2 Profile Ests in outer region}
    For parameters as in Corollary \ref{cor main box cor}, let $u(s,\tau)$ and $a(\tau)$ be the  solutions obtained in Corollary \ref{cor main box cor}.
    Let $h(s, \tau) = f_{a(\tau)}(s) + \partial_s u(s,\tau)$ be the corresponding profile function for the rescaled LMCF.
    With $\tau_0$ as in Corollary \ref{cor main box cor}, set $T:= e^{-\tau_0}$ and $\tilde h (x,t) := \sqrt{T - t} \cdot h \left( x/ \sqrt{T-t}, -\ln (T-t) \right)$, the profile function for the associated LMCF.

    Assume additionally that the potential function $u$ at initial time $\tau_0$ has $C^3$ bounds
    \begin{gather*}
        \| v(\cdot, \tau_0) \|_{3; (-s_0, s_0)}^{(\gamma)} \le \epsilon_{in} e^{-\kappa_{in} \tau_0},
        \quad
        \| u ( \cdot, \tau_0) \| _{C^3\left( \frac{s_0}2 < |s| <  2 \Upsilon \right)} \le \epsilon_{par},\\
         \text{and} \quad 
        \| v(\cdot, \tau_0) \|_{3; (\Upsilon, \infty)}^{(2)} \le \epsilon_{out}.
    \end{gather*}
    Then, for any $\epsilon > 0$,
    \begin{equation}
        \sup_{t \in [0,T)} 
        \sup_{|x| \ge \epsilon}
        \sum_{j=0}^2 |x|^{-1+j} \left|\partial_x^j  \tilde h (x, t)   \right| < \infty
    \end{equation}
\end{lem}

The proof uses a rescaling argument similar to Lemma \ref{lem C^2 Profile Ests on outer-parab overlap}. We include the details nonetheless.

\begin{proof}
    Throughout the proof, $C = C(n, K, \gamma, s_0, \Upsilon, \kappa_{in}, \epsilon, \tau_0) > 0$ denotes a constant, which may change from line to line, but always depends only on $n, K, \gamma, s_0, \Upsilon, \kappa_{in}, \epsilon, \tau_0$.

    As in the proof of Lemma \ref{lem C^2 Profile Ests on outer-parab overlap},
    the profile function $\tilde h(x,t)$ satisfies the quasilinear PDE \eqref{eq: LMCF of profile function}
    \begin{equation} \label{proof C^2 Profile Ests in outer region, eqn 1}
        \partial_t \tilde h = \frac{\partial_{xx} \tilde h}{1 + (\partial_x \tilde h)^2} + (n-1) \frac{x \partial_x \tilde h - \tilde h}{x^2 + \tilde h^2}
        \qquad \forall (x, t) \in \R \times [0, T).
    \end{equation}
    The functions $u,a$ are in $\mathcal B$, which implies estimates for $a(\tau)$ and weighted $C^2$ estimates for $u$ (Definition \ref{defn box}).
    Together with the asymptotics of $f_{a}$ (Lemma \ref{lem: profile function of Lawlor neck}), it follows that $h(s, \tau) = f_{a(\tau)}(s) + \partial_s u(s, \tau)$ satisfies
    \begin{equation} \label{proof C^2 Profile Ests in outer region, eqn 2}
        | \partial_{s} h| + \left| \frac1s  h\right| \le C \qquad \forall s \ge \frac \epsilon {10 \sqrt{T}}, \, \tau \ge -\ln T.
    \end{equation}
    The assumption on the initial data further implies
    \begin{equation}
        \sum_{j=0}^3 |s|^{-1+j} | \partial_{s}^j h|  \le C \qquad \forall s \ge \frac \epsilon {10\sqrt{T}}, \, \tau = -\ln T.
    \end{equation}
    It follows that 
    \begin{gather} \label{proof C^2 Profile Ests in outer region, eqn 3}
        | \partial_x \tilde h| + \left| \frac1x  \tilde h \right| \le C \qquad \forall x \ge \frac{\epsilon}{10} , \, t \in [0, T) \\
        \label{proof C^2 Profile Ests in outer region, eqn 3.5}
        \text{and } \sum_{j=0}^3 |x|^{-1+j} | \partial_{x}^j \tilde h|  \le C \qquad \forall x \ge \frac \epsilon {10},\, t = 0.
    \end{gather}
    
    Fix $x_*, t_*$ such that $x_* \ge \epsilon$ and $t_* \in (0, T)$.
    Define
        $$w(\sigma, \xi) := \frac1{x_*} \tilde h( x_* \sigma, t_* + \xi x_*^2) .$$
    The evolution equation \eqref{proof C^2 Profile Ests in outer region, eqn 1} for $\tilde h$ implies $w$ satisfies
    \begin{equation} \label{proof C^2 Profile Ests in outer region, eqn 4}
        \partial_\xi w = \frac{\partial_{\sigma \sigma} w}{1 + (\partial_\sigma w)^2 } + (n-1) \frac{\sigma \partial_\sigma w - w}{\sigma^2 + w^2}
        \quad \text{on }
        \left( \frac12, \frac32 \right) \times \left[-\frac{t_*}{x_*^2}, 0 \right].
    \end{equation}
    Estimates \eqref{proof C^2 Profile Ests in outer region, eqn 3}--\eqref{proof C^2 Profile Ests in outer region, eqn 3.5} for $\tilde h$ imply
    \begin{equation} \label{proof C^2 Profile Ests in outer region, eqn 4.1}
        | \partial_\sigma w | + \left| \frac1\sigma w \right| < C 
        \quad  \text{for all } (\sigma, \xi ) \in \left[ \frac\epsilon{10 x_*} , \infty \right) \times \left[ - \frac{t_*}{x_*^2}, 0 \right]
    \end{equation}
    with initial data estimates
    \begin{equation} \label{proof C^2 Profile Ests in outer region, eqn 4.2}
        \sum_{j=0}^3 |\sigma|^{-1+j} |\partial_\sigma w | < C 
        \quad \text{for all } (\sigma, \xi ) \in  \left[ \frac\epsilon{10 x_*} , \infty \right) \times \left\{ - \frac{t_*}{x_*^2} \right\}.
    \end{equation}

    Because $x_* \ge \epsilon $,     
    $\left( \frac12, \frac32 \right) \subset\left[ \frac\epsilon{10 x_*} , \infty \right)$.
    Since $w$ solves the quasilinear PDE \eqref{proof C^2 Profile Ests in outer region, eqn 4}, these gradient estimates \eqref{proof C^2 Profile Ests in outer region, eqn 4.1} and higher order initial data estimates \eqref{proof C^2 Profile Ests in outer region, eqn 4.2} for $w$ suffice to apply interior $C^{2, \alpha}$-estimates to $w$ \cite{LSU88}*{Ch. VI, Section 1}. 
    Namely,
    \begin{multline} \label{proof C^2 Profile Ests in outer region, eqn 5}
        |x_*| |\partial_{xx} \tilde h (x_*, t_*)| 
        =|\partial_{\sigma \sigma} w (1,0)|  \\
        \le C \sup_{(\sigma, \xi) \in \left(\frac12, \frac32\right) \times \left[ -\frac{t_*}{x_*^2}, 0 \right]} | w| + C \sum_{j=0}^3 \sup_{(\sigma, \xi) \in \left(\frac12, \frac32\right) \times \left\{ -\frac{t_*}{x_*^2} \right\}} |\sigma|^{-1+j} |\partial_\sigma^j w| 
        < C < \infty.
    \end{multline}
    As $(x_*, t_*) \in [\epsilon, \infty) \times (0, T)$ were arbitrary and $C$ is independent of $(x_*, t_*)$, 
    The estimate \eqref{proof C^2 Profile Ests in outer region, eqn 5} together with \eqref{proof C^2 Profile Ests in outer region, eqn 3} proves the statement.
\end{proof}

\bibliography{BibFile}

\end{document}